\documentclass{article}
\usepackage{graphicx,amssymb,amsmath,amsthm,geometry,xcolor,mathrsfs,mathtools,todonotes,hyperref,tikz,tikz-cd, qtree, algorithm2e, enumitem} 
\usepackage{forest}
\theoremstyle{plain}

\newtheorem{theorem}{Theorem}
\newtheorem{definition}[theorem]{Definition}

\newtheorem{proposition}[theorem]{Proposition}
\newtheorem{lemma}[theorem]{Lemma}
\newtheorem{corollary}[theorem]{Corollary}
\newtheorem{example}[theorem]{Example}
\newtheorem{remark}[theorem]{Remark}
\newtheorem{question}[theorem]{Question}
\newtheorem{claim}{Claim}[theorem]
\newenvironment{claimproof}[1]{\par\noindent\underline{Proof of claim:}\space#1}{\hfill $\blacksquare$}

\RestyleAlgo{ruled}
\SetKwComment{Comment}{/* }{ */}
\SetKw{Make}{Make}
\SetKw{Set}{Set}

\DeclareMathOperator{\Inv}{Inv}

\DeclareMathOperator{\Simple}{\mathbf{Simple}}

\DeclareMathOperator{\FL}{FL}

\DeclareMathOperator{\Csp}{CSP}
\DeclareMathOperator{\Pol}{Pol}
\DeclareMathOperator{\ar}{ar}

\DeclareMathOperator{\Lin}{Lin}
\DeclareMathOperator{\Eq}{Eq}
\DeclareMathOperator{\id}{id}

\DeclareMathOperator{\Sup}{Sup}

\DeclareMathOperator{\Mut}{Mut}
\DeclareMathOperator{\SimpleLocal}{\mathbf{SimpleLocal}}

\DeclareMathOperator{\Solve}{Solve}

\DeclareMathOperator{\Cons}{Cons}

\DeclareMathOperator{\Full}{Full}
\DeclareMathOperator{\Reduce}{Reduce}

\DeclareMathOperator{\Mod}{Mod}
\DeclareMathOperator{\Lclass}{L}
\DeclareMathOperator{\Trunk}{Trunk}
\DeclareMathOperator{\NReduce}{NReduce}

\newlist{primenumerate}{enumerate}{1}
\setlist[primenumerate,1]{label={(\alph*$'$)}}

\newcommand{\ignore}[1]{}

\newcommand{\bA}{\mathfrak A}
\newcommand{\bB}{\mathfrak B}
\newcommand{\bC}{\mathfrak C}
\newcommand{\bI}{\mathfrak I}
\newcommand{\bE}{\mathfrak E}

\newcommand{\bS}{\mathfrak S}
\newcommand{\bT}{\mathfrak T}
\newcommand{\bP}{\mathfrak P}
\newcommand{\fA}{{\bf A}}
\newcommand{\fC}{{\bf C}}

\newcommand{\fR}{{\bf R}}

\newcommand{\fB}{{\bf B}}
\DeclareMathOperator{\Clo}{Clo}

\tikzset{myStyle/.style={baseline=(center.base), font=\small,
    every node/.style={inner sep=0.25em} }}

\newcommand{\SquareUnwrapped}[4]{ 
  \node (center) at (0.5,-0.5) {\phantom{$\cdot$}}; 
  \path (0,0)  node (nw) {$#2$}
      ++(1,0)  node (ne) {$#4$}
      ++(0,-1) node (se) {$#3$}
      ++(-1,0) node (sw) {$#1$};
  \draw (nw) -- (ne) -- (se) -- (sw) -- (nw);
}  
\NewDocumentCommand{\SquareXY}{ O{} O{} O{} O{} O{1} O{1} }{ 
  \begin{tikzpicture}[myStyle, xscale=#5*1, yscale=#6*1 ]
    \SquareUnwrapped{#1}{#2}{#3}{#4}
  \end{tikzpicture}
}  
\NewDocumentCommand{\Square}{ O{} O{} O{} O{} O{1} }{ 
  \SquareXY[#1][#2][#3][#4][#5][#5]
}

\title{Conservative Maltsev Constraint Satisfaction Problems}
\author{Manuel Bodirsky 
\and Andrew Moorhead\thanks{Institut f\"ur Algebra, TU Dresden. Email: $\{$manuel.bodirsky,andrew\textunderscore paul.moorhead$\}$@tu-dresden.de. Both authors received funding from the ERC (Grant Agreement no. 101071674, POCOCOP). Views and opinions expressed are however
those of the authors only and do not necessarily reflect those of the European Union or the European Research
Council Executive Agency.}}

\date{}

\begin{document}

\maketitle

\begin{abstract}
       One of the central open problems to classify the computational complexity of finite-domain constraint satisfaction problems within P is to prove better algorithmic results 
    for CSPs with a Maltsev polymorphism; we do not even know whether these CSPs are in NC. Relatedly, the descriptive complexity of these problems is open as well.
    An important special case, previously studied by Carbonell from the perspective of uniform polynomial time-algorithms, are CSPs with a \emph{conservative} Maltsev polymorphism. 
    We show that for every finite structure $\mathfrak B$ with a conservative Maltsev polymorphism,
    the CSP for $\mathfrak B$ can be solved by a \emph{symmetric linear ${\mathbb Z}_2$-Datalog program},
    and in particular is in the complexity class $\oplus L$. Previously, the best known algorithms just showed containment in P. In our proof we develop a structure theory for conservative Maltsev algebras which might be of independent interest. 
\end{abstract}

\ignore{
\begin{abstract}
    We show that for every finite structure $\mathfrak B$ with a conservative Maltsev polymorphism,
    the constraint satisfaction problem for $\mathfrak B$ can be solved by a symmetric linear ${\mathbb Z}_2$-Datalog program,
    and in particular is in the complexity class $\oplus L$. 
    The proof has two steps:
    we first present the result 
    for a certain subclass 
     whose  polymorphism algebras are hereditarily subdirectly irreducible. 
    We then show that every other structure in our class can be primitively positively constructed from one of the structures in the subclass. The second step requires different techniques and will be presented in a companion article.  
\end{abstract}
}


\tableofcontents


\section{Introduction} 
For many years, the P-versus-NP-complete dichotomy conjecture of Feder and Vardi~\cite{FederVardi} has been the holy grail of constraint satisfaction research in several research communities, e.g., graph homomorphisms, finite model theory, parts of theoretical artificial intelligence, and universal algebra. After the conjecture has been confirmed~\cite{Zhuk20} (announced in 2017 independently by Bulatov~\cite{BulatovFVConjecture} and by Zhuk~\cite{ZhukFVConjecture}), 
there are three important and closely related research directions in finite-domain constraint satisfaction with numerous problems that remain wide open: 
\begin{enumerate}
    \item Find a simpler algorithm or simpler correctness proofs of the existing algorithms. Both proofs of the dichotomy theorem are complicated and have been verified only by a handful of researchers. 
    Moreover, it is not known whether 
    there exists an algorithm 
    which is 
    \emph{uniform polynomial-time} in the sense that it has polynomial running time in the input size even if the template structure that defines the CSP is part of the input.
    \item The second direction is to classify the computational complexity of finite-domain CSPs within P up to logspace reductions. In particular, this can clarify which CSPs can be parallelised efficiently, and for which CSPs we may not hope to find parallel algorithms (e.g., if they are P-complete). 
    It is well-known that finite-domain CSPs can be complete for the complexity classes L, NL, Mod$_p$L for some prime $p$, P, and NP. A complete classification is known for CSPs over two-element domains~\cite{AllenderSchaefer}, but for larger domains our knowledge is very limited.\footnote{A complexity classification within P is trivial for undirected graphs~\cite{HellNesetril} and more generally for smooth digraphs~\cite{BartoKozikNiven}, but already open for orientations of trees~\cite{BulinSmallestHardTriad}.} 
    \item The third research direction is to clarify the \emph{descriptive} complexity of finite-domain CSPs: express polynomial-time tractable CSPs in logics that are designed to capture polynomial-time computation, such as Fixed Point Logic and its various extensions, and Choiceless Polynomial Time; see~\cite{DawarPado} for a recent survey. Also expressiveness in more restrictive logics (e.g., in linear Datalog, or symmetric linear Datalog, etc.) is interesting and has been studied intensively~\cite{LinearDatalog,DalmauLarose08,Kazda-n-permute,CarvalhoDalmauKrokhin}.  The search for logics that capture P or subclasses of P is an important theme in finite model theory, and CSPs provide a fruitful and rich microcosm to test conjectures in this area. There are numerous connections to the previous research direction~\cite{LaroseTesson,EgriLaroseTessonLogspace}.  
\end{enumerate} 

Most of the central open problems in these research directions
are already open for a certain important
subclass of the class of finite-domain CSPs, namely the class of CSPs for structures with a \emph{Maltsev polymorphism}; we will call them \emph{Maltsev CSPs} for short. If a relational structure has a Maltsev polymorphism, then this has quite remarkable mathematical consequences. An important source of such structures comes from finite groups; but there are many structures with a Maltsev polymorphism that do not come from groups. 
These structures played an important role in history of the study of the CSP, with Maltsev CSPs coming from groups and generalisations already appearing in the Feder-Vardi paper~\cite{FederVardi}.
Maltsev CSPs can be solved in polynomial time, by the famous 
 Bulatov-Dalmau algorithm~\cite{Maltsev}. 
This algorithm has later been generalised, culminating in the algorithm for CSPs with few subpowers~\cite{IMMVW}, one of the black-box ingredients for Bulatov's approach to the Feder-Vardi dichotomy conjecture~\cite{BulatovFVConjecture}.  
Zhuk's algorithm~\cite{ZhukFVConjecture} solves Maltsev constraints as well, but with an algorithmically fundamentally different approach.
All of the known algorithms for Maltsev CSPs are \emph{not} uniformly polynomial in the sense of the first research direction. 
Maltsev CSPs are also interesting for the second research direction mentioned above, 
because the known algorithms do not even show containment in the complexity class NC (which contains all the classes mentioned above, except for P and NP). At the same time, it appears to be unlikely that Maltsev CSPs can be P-hard
(e.g., they can provably not simulate the P-complete HornSAT problem via gadget-reductions~\cite{wonderland,DalmauOprsalLocal}; all other known finite-domain CSPs that are P-hard can simulate HornSAT via such reductions). 
It is also open whether Maltsev CSPs can be expressed in the extension of Fixed Point Logic with Rank Operators, or in Choiceless Polynomial Time (some other weaker logics fail to capture some basic Maltsev CSPs~\cite{AtseriasBulatovDawar,DawarPado,Lichter,LichterPagoSeppelt}). 
Interestingly, the algorithm of Zhuk~\cite{Zhuk20} does not rely on the Bulatov-Dalmau algorithm,
but rather solves repeatedly linear equations over finite fields, and hence provides some hope that with a better analysis of this algorithm we can make progress in the mentioned research directions. 

Another historically important restriction of the class of all finite-domain CSPs is the class of \emph{conservative CSPs}: here we assume that in instances of the CSP, we can restrict variables to take values from an arbitrary subset of the domain. 
The name comes from the observation that the corresponding finite structures are precisely those where all polymorphisms are \emph{conservative}, i.e., satisfy $f(x_1,\dots,x_n) \in \{x_1,\dots,x_n\}$, for all elements $x_1,\dots,x_n$ of the structure. 
Conservativity turns many easy CSPs into hard ones, and makes the complexity classification task significantly easier. Note that the class of conservative CSPs includes the class of  \emph{list homomorphism problems} from the graph theory community~\cite{FederHH99,DalmauEHLR15}. 
The P-versus-NP-complete dichotomy for conservative CSPs was shown by Bulatov~\cite{Conservative} already in 2003 (also 
see~\cite{Bulatov-Conservative-Revisited,Barto-Conservative}). 
Carbonell~\cite{Carbonnel16} showed that the tractability condition for conservative CSPs given by Bulatov can be tested in polynomial time, and he points out that the open question whether there is a uniform algorithm for tractable finite-domain CSPs is 
already open for conservative CSPs. 

The mentioned research questions are open even if we combinte the two restrictions, i.e., for the class of \emph{conservative Maltsev CSPs}; however, there is one exception: a uniform polynomial-time algorithm for such CSPs has been discovered by Carbonell~\cite{Carbonnel16b}. His algorithm uses the arc-consistency algorithm as a sub-procedure;
since arc-consistency solves P-complete problems such as Horn-SAT, the algorithm is not suited to analyse the complexity of such problems within P.


\subsection*{Our contributions}
We show that conservative Maltsev CSPs can be expressed in an extension of symmetric linear Datalog by operators for testing feasability of linear equations over the two-element field $\mathbb Z_2$. 
In particular, they can be expressed in the already mentioned Least Fixed Point Logic with a rank operator for ${\mathbb Z}_2$. 
Every problem that can be expressed in this variant of Datalog is in the complexity class $\oplus$L; hence, our results show that conservative Maltsev CSPs are in the complexity class $\oplus$L.
This complexity class contains all problems that can be solved by a Turing machine with a work tape of logarithmic size 
in the sense that the positive instances of the problem are precisely those where the machine has an even number of accepting computation paths. 
In fact, conservative Maltsev CSPs are either in L (deterministic logspace) or they are complete for $\oplus$L, so we have a full complexity classification up to logspace reductions (Theorem~\ref{thm:dicho}). 
Our proof has several parts: 

\begin{enumerate}
    \item In the first part we 
develop some general structure theory for finite structures with a conservative Maltsev polymorphism.
It is known that 
every such structure also has a (conservative) minority polymorphism; we show that algebras generated by a single conservative minority operation 
can be decomposed in a tree-like fashion (Section~\ref{sect:tree-reps}). 
\item 
Building on our results for algebras with a single conservative minority operation, we show that any structure with a conservative Maltsev polymorphism has a primitive positive definition in a structure with a specific set of at most ternary relations (Section~\ref{sect:inv-rels}).
\item We isolate an important `basic' subclass of these structures, where the mentioned trees of their polymorphism algebras are essentially paths, decorated by some algebras that are hereditarily simple and particularly well-behaved  (Section~\ref{sec:algebrasP_nk}), and we present an algorithm for their CSPs (Section~\ref{sect:alg}). 
\item We show that every finite structure with a conservative Maltsev polymorphism can be primitively positively constructed from the mentioned basic ones (Section~\ref{sect:MainPP}).
\item We introduce \emph{Cyclic Group Datalog}, an extension of Datalog
by a mechanism to solve systems of linear equations. We show that 
symmetric linear ${\mathbb Z}_2$-Datalog programs can be evaluated in the complexity $\oplus$L, and implement our algorithm in symmetric linear ${\mathbb Z}_2$-Datalog (Section~\ref{sect:descr}). 
Expressibility in symmetric linear Datalog and symmetric linear ${\mathbb Z}_2$-Datalog is preserved 
by primitive positive constructions (for symmetric linear Datalog, see~\cite{StarkeDiss}), and hence our main results follow.  
\end{enumerate}

\ignore{
\subsection*{Acknowledgements}
This project has received funding from the European Union
(Project POCOCOP, ERC Synergy Grant 101071674).
Views and opinions expressed are however those of the author(s) only and do not
necessarily reflect those of the European Union or the European Research
Council Executive Agency. Neither the European Union nor the granting
authority can be held responsible for them.
}

\section{Overview}\label{sect:overview}

The comparison of some well-known facts from the literature of fixed template finite domain CSPs and Universal Algebra can serve as a point of entry to an overview of our investigation. 
Dalmau and Larose showed that if a fixed template CSP is solvable with a \emph{Datalog program} 
and the template has a Maltsev polymorphism, then the CSP is solvable with a \emph{symmetric linear Datalog program}~\cite{DalmauLarose08}. 
Datalog can be viewed as an existential-positive fragment of fixed point logic;
every problem that can be expressed in this logic can be solved in polynomial time (we refer to~\cite{Libkin} for an introduction to finite model theory). Problems in symmetric linear Datalog can even be solved in L (deterministic logarithmic space). 

Larose and Z{\' a}dori, building on results of Feder and Vardi~\cite{FederVardi}, showed that if a fixed template CSP can simulate linear equations over a finite ring (here, `simulate' is meant in a specific technically precise sense, but we refrain from going into the details, because they are not relevant later) then it cannot be solved by Datalog~\cite[proof of Theorem 4.2]{LaroseZadori} (also see~\cite{AtseriasBulatovDawar} for an even stronger result). 
Barto and Kozik later proved that the converse also holds~\cite{BartoKozikFOCS09}. 
Hence, if a Maltsev CSP is not solvable with symmetric linear Datalog, then 
it can express linear equations.

Therefore, we arrive at an intuition that finite domain Maltsev CSPs should be solvable with an extension of symmetric linear Datalog which is capable of interleaving local consistency checking with some linear system consistency queries. The obstacle to formalizing this intuition into a concrete complexity classification is that the ways that linear systems manifest in CSPs can be quite complicated, even for the Maltsev case. However, when we further impose the condition of conservativity on the template, we are able to obtain a manageable characterization of the kinds of linear systems that arise and how they connect to the other relations in the template. This characterization is then refined into an exact characterization for a specific class of structures from which all other conservative Maltsev templates can be pp-constructed. It is for this special class that we exhibit a concrete $\oplus$L algorithm, which we implement in an extension of Datalog that we call $\mathbb{Z}_2$-Datalog. We provide in this section a high level summary of these results, for which detailed proofs are provided in Sections~\ref{sect:cons-mino},~\ref{sect:MainPP},~\ref{sect:alg}, and~\ref{sect:descr}.

\subsection{Section~\ref{sect:cons-mino}: The structure of conservative minority algebras}

 At the outset, we invoke Lemma~\ref{lem:Carbonnel}, due to Carbonnel~\cite{Carbonnel16b}, which states that every conservative Maltsev function clone contains a conservative minority operation. As a consequence, we may assume that a template $\mathfrak{A}$ can pp-define all relations which are invariant under a particular conservative minority operation and we henceforth focus our attention on building a structure theory for such templates. We call an algebra with a single basic conservative minority operation a \emph{conservative minority algebra}. 

The homomorphism kernels (which are called \emph{congruences} in Universal Algebra) of conservative minority algebras satisfy an extremely strong property, which we describe in Proposition~\ref{prop:blocksinsideblocks}. Roughly, equivalence classes of congruences of conservative minority algebras can be collapsed independently. Hence, understanding congruences of conservative minority algebras is reduced to understanding the congruences with at most one nontrivial class, which we call \emph{block} congruences. Actually, Proposition~\ref{prop:blocksinsideblocks} asserts an even stronger property, which is that this block decomposition property is transitive, in the sense that a class of a congruence of the subalgebra determined by a congruence class also determines a congruence of the global algebra. We then show in Proposition~\ref{prop:blockintersection} that congruence classes $[a]_{\theta_1}$ and $[b]_{\theta_2}$ (from potentially distinct $\theta_1$ and $\theta_2$) cannot have nontrivial symmetric difference, which has the consequence that every conservative minority algebra possesses a unique maximal congruence~(Corollary~\ref{corollary:maximalcongruence}).

So, instead of working with congruences via their lattice structure as is customary in Universal Algebra, we should encode congruences as finite trees, which are constructed iteratively using the classes of the unique maximal congruence guaranteed by Corollary~\ref{corollary:maximalcongruence}. If we keep track of the quotient structure at each stage of this construction, we obtain a tree whose vertices are decorated with the structure of various simple conservative minority algebras, which we call \emph{local} algebras. Turning this around, we provide in Definition~\ref{def:consmintree} a definition of a \emph{conservative minority tree} $\mathcal{T}$ and define in ~\ref{def:leafalgebra} the method to obtain a conservative minority algebra $\mathbf{A}_\mathcal{T}$ defined over the leaves of $\mathcal{T}$, which we call the algebra \emph{represented} by $\mathcal{T}$. We prove in Theorem~\ref{thm:representwithtreealgebras} that every finite conservative minority algebra can be represented in this fashion. 

Hence, we are able to work exclusively with conservative minority trees and the algebras they represent. The reason to adopt this perspective is purely practical, since a conservative minority tree provides a unified environment which encodes both an algebra and the structure of the congruences, and almost all of our results depend on our ability to articulate precise features of congruences. We go on to develop several technical features of conservative minority trees. Definition~\ref{def:treevertexcongruences} and Lemma~\ref{lemma:congruencesoftreealgebras} establish how certain block congruences of $\mathbf{A}_\mathcal{T}$ are encoded in $\mathcal{T}$, while Lemma~\ref{lem:subalgebrasuccessor} and Lemma~\ref{lem:CongruencesofSubalgebras} establish a connection between subalgebras of $\mathbf{A}_\mathcal{T}$ and $\mathcal{T}$. 

In Section~\ref{sec:saplings}, we develop some even more specialized notation and definitions for the purpose of identifying and manipulating subtrees of a conservative minority tree $\mathcal{T}$ which represent the \emph{subdirectly irreducible subalgebras} of $\mathbf{A}_\mathcal{T}$ (see Section~\ref{sect:congruences} for a definition). Subdirectly irreducible algebras play an important role in characterizing how linear systems can occur among the invariant relations of an algebra. We established during our inquiry into the properties of block congruences that the congruences of subdirectly irreducible conservative minority algebras are linearly ordered under inclusion (see Proposition~\ref{prop:subdirectlyirreduciblearelinearchains}), which means that only those subtrees of $\mathcal{T}$ whose nontrivially branching vertices form a linearly ordered chain can represent subdirectly irreducible algebras of $\mathbf{A}_\mathcal{T}$ (see Lemma~\ref{lem:saplingsrepresentsubdirectirreducible}). We call such conservative minority trees \emph{saplings}. The remainder of the Section~\ref{sect:tree-reps} develops more technical machinery for conservative minority trees, which we exclude from this overview.

In Section~\ref{sect:inv-rels}, we turn to the task of understanding the invariant relations of conservative minority algebras. We are assisted in this task by the theory of the \emph{critical} relations of an algebra $\mathbf{A}$ (see Definition~\ref{def:criticalrelation}), which are easily seen to primitively positively define all of $\Inv(\mathbf{A})$. Hence, much of this section is devoted to performing a distillation of the critical relations of a conservative minority algebra, which is achieved through applying properties of conservative algebras, block congruences, and \emph{rectangular} relations (see Definition~\ref{sect:rectangularrelations}). These efforts culminate in the explicit list of relations provided in Theorem~\ref{thm:consminorityrelbasis}. For a conservative minority algebra $\mathbf{A}$ with domain $A$, aside from the obvious unary relations, there are three flavors of at most ternary relations which, taken together with unary relations, primitively positively define $\Inv(\mathbf{A})$. The relations with flavors \emph{(b$'$)} and \emph{(c$'$)} are certain endomorphism graphs and isomorphism graphs between subdirectly irreducible subalgebras. We remark that, since these binary relations are also preserved by a majority function, it follows from Dalmau and Larose (see~\ref{thm:DL}) that any template which is primitively positively definable from them has a CSP solvable in symmetric linear Datalog. Hence, it is the third nontrivial flavor of relation \emph{(d$'$)} which prevents solvability by a Datalog program. Roughly, these ternary `linear' relations listed in \emph{(d$'$)} of Theorem~\ref{thm:consminorityrelbasis} are maximal affine $\mathbb{Z}_2$-subspaces which have been enclosed by the ternary equality relation on some subdirectly irreducible subalgebra of $\mathbf{A}$. Arbitrary systems of $\mathbf{Z}_2$-linear equations can be primitively positively defined with such linear relations, and they can interact in nontrivial ways with the other relations. Thus, the remainder of the paper is focused on disentangling these interactions. 

Our next task is describing a class of conservative minority structures that is robust enough to primitively positively construct the class of all finite conservative Maltsev templates, while also tame enough to allow us to produce an ergonomic description of the relations of its members. We describe this class algebraically in ~\ref{sec:algebrasP_nk}, where we define for each $n \geq 2$ and $k \geq 1$ a subdirectly irreducible algebra $\mathbf{P}^{n,k}$, and then relationally in ~\ref{sect:structureP_nk}, where we define for each $\mathbf{P}^{n,k}$ its companion structure $\mathfrak{P}^{n,k}$. Section~\ref{sect:minimal-taylor} characterizes the \emph{minimal Taylor} conservative minority algebras (see Definition~\ref{def:minTaylorclone} and Theorem~\ref{prop:MinTayConsMalChar}). Since our special class of algebras consists of minimal Taylor algebras, reducing to the minimal Taylor setting is an important step towards our final reduction, although we remark that the minimal Taylor reduction could be performed after our main primitive positive construction (see Theorem~\ref{thm:FinalPPConstructionTheorem}). In any case, it is in Section~\ref{sect:minimal-taylor} that we define the \emph{projection minority} algebras $\mathbf{P}_n$, for each $n \geq 2$, which we argue are exactly the \emph{hereditarily simple} conservative minority algebras (Proposition~\ref{prop:MinTayConsMalChar}). Projection minority algebras are the simple quotients of the algebras $\mathbf{P}^{n,k}$ built in Section~\ref{sec:algebrasP_nk}. 
We conclude the section with Corollary~\ref{cor:minTaylorConsMinTreeConstOthers} and Lemma`\ref{lem:SaplingTreesConstructGeneralOnes}, each of which further restricts the class of conservative minority trees we need to consider in the pursuit of our main primitive positive construction result. 

\subsection{Section~\ref{sect:MainPP}: Main primitive positive construction}

Before describing the contents of this section, we remark that we will be somewhat informal in our exposition here, since the machinery at this point in the paper has become quite technical. With that said, let us describe our main primitive positive construction. Essentially, we provide a recursive scheme to replace `problematic' simple algebras which decorate the vertices of a given conservative minority tree $\mathcal{T}$, where a `problematic' simple algebra is one which is not hereditarily simple. The idea is to isolate a particular such algebra, say $\mathbf{B}$, among the local algebras of $\mathcal{T}$ and replace it with a new algebra which encloses each maximal proper subalgebra of $\mathbf{B}$ in a distinguished congruence class, so that the resulting quotient is hereditarily simple. We also include a new special congruence class which encloses a hereditarily simple algebra whose elements are in bijection with the elements of $\mathbf{B}$. 

If we perform this replacement for all occurrences of $\mathbf{B}$ in $\mathcal{T}$, we obtain a tree $\mathcal{T}^{\mathbf{B} \downarrow}$, which we call the \emph{$\mathbf{B}$-unpacking} of $\mathcal{T}$ (see Definition~\ref{def:splicing}, Figure~\ref{fig:unpacking1}, and Figure~\ref{fig:unpacking2}). The reason this works is that we have exposed the subalgebras of $\mathbf{B}$, which means that we can apply a refinement procedure on $\mathcal{T}^{\mathbf{B}\downarrow}$ which does not alter the isomorphism type of the algebra it represents (see Section~\ref{sec:treetransformations}). Any relations of $\mathbf{A}_\mathcal{T}$ which rely on these congruences have counterparts in $\mathbf{A}_{\mathcal{T}^{\mathbf{B} \downarrow}}$, and the extra congruence class which encloses a hereditarily simple algebra handles the other possibilities. 

We can now use our understanding of the invariant relations of conservative minority algebras to prove that the invariant relations of $\mathbf{A}_{\mathcal{T}^{\mathbf{B}}}$ have a primitive positive construction in the invariant relations of $\mathbf{A}_{\mathcal{T}^{\mathbf{B}\downarrow}}$. We do so by exhibiting a \emph{minion homomorphism} (see Section~\ref{sect:minhomppconstructions}) from the clone of polymorphisms of $\mathbf{A}_{\mathcal{T}^{\mathbf{B} \downarrow}}$ into the clone of polymorphisms of $\mathbf{A}_\mathcal{T}$ (see Theorem~\ref{thm:DeltaisMinionHom}). Although the tree structure of the refinement of $\mathcal{T}^{\mathbf{B} \downarrow}$ is more complicated than that of $\mathcal{T}$, its set of local algebras is now closer to consisting entirely of hereditarily simple algebras. Hence, we can iterate this construction until all problematic simple algebras are eliminated and our main primitive positive construction result follows~\ref{thm:FinalPPConstructionTheorem}.

\subsection{Section~\ref{sect:alg}: Solving conservative Maltsev CSPs and Section~\ref{sect:descr}: Descriptive complexity}

Our efforts up to this point have culminated in the reduction of conservative Maltsev CSP to the CSP for structures $\mathfrak{P}^{n,k}$, for each $n \geq 2$ and $k \geq 1$ (see the beginning of Section~\ref{sect:structureP_nk}). Algorithm~\ref{alg:solve} solves the CSP for members of this class by proceeding recursively along the $k$-many congruences of the respective polymorphism algebras $\mathbf{P}^{n,k}$. The basis of the recursion solves $\Csp(\mathfrak{P}^{n,1})$, where the polymorphism algebra $\mathbf{P}^{n,1}$ is hereditarily simple. In this case, we create from an instance $\mathfrak{I}$ instances $\mathfrak{E}$ and $\mathfrak{T}$, which we call the \emph{system of $\mathbb{Z}_2$-linear equations associated to $\mathfrak{I}$} and the \emph{permutational system associated to $\mathfrak{I}$}, respectively. 

The idea is that connected components of $\mathfrak{I}$ either involve the linear constraint $\Lin$, or they do not. In the former case, a satisfiability equivalent $\mathbb{Z}_2$-system is obtained by propagation of the unary and binary constraints of $\mathfrak{P}^{n,1}$ along paths which connect variables within the scope of $\Lin$. In the latter case, there are no linear equations, and so $\mathfrak{I}$ restricted to such a component is solvable in symmetric linear Datalog (recall our discussion at the beginning of this overview~\ref{sect:overview}). We collect together the connected components of $\mathfrak{I}$ which involve linear equations into the $\mathbb{Z}_2$-linear system for $\mathfrak{I}$ and the others are collected together to form the permutational system for $\mathfrak{I}$. If either of these systems is unsatisfiable, the algorithm rejects $\mathfrak{I}$. 

For larger values of $k$, we invoke another recursive algorithm $\Reduce_{n,k}$~\ref{alg:reduce}, which produces a satisfiability equivalent instance of $\Csp(\mathfrak{P}^{n,k-1})$ from an instance of $\Csp(\mathfrak{P}^{n,k})$. Essentially, the procedure works by analyzing instances obtained by restricting the constraints of certain \emph{uniform connected components}~\ref{def:components} to the maximal congruence block of $\mathbf{P}^{n,k}$, and then reducing the obtained instances with recursive calls to $\Reduce_{n,k-1}$. The various reduced instances that are obtained are then combined with the instances on the components from which they came, and then some final further connections are formed using a recursively formed function we call the \emph{reduction atlas}~\ref{def:(k-1)reduction}:. We provide in Figure~\ref{fig:alg} a dependency diagram for our two main classes of procedures $\Solve_{n,k}$ and $\Reduce_{n,k}$. 

Finally, in Section~\ref{sect:descr} we define for a cyclic group $\mathbb{G}$ an extension of Datalog which we call $\mathbb{G}$-Datalog~\ref{sect:RDatalog}, provide a $\mathbb{Z}_2$-Datalog implemententation of Algorithm~\ref{alg:solve} (Propostion~\ref{prop:reduce-symlin}), and argue that any problem which is solved by a $\mathbb{Z}_2$-Datalog program belongs to the class $\oplus$L (Theorem~\ref{thm:Z2Datalog}).

\section{Preliminaries}
In this section, we recall some basic terminology and concepts from model theory and universal algebra. This section can probably be skipped by many readers, and consulted if needed later.

If $v = (v_1,\dots,v_n)$ and $w = (w_1,\dots,w_m)$ are finite sequences of elements from some set $A$, then 
$v^\frown w$ denotes the concatenation of $v$ and $w$, i.e., the finite sequence $(v_1,\dots,v_n,w_1,\dots,w_m)$.  
We write $(a^k)$ for the constant tuple $(\underbrace{a, \dots, a}_k)$. Hence, 
\[
(a^k)^\frown(b^j) = (\underbrace{a, \dots,a}_k, \underbrace{b, \dots, b}_j).
\]
In the specific situation where $a=0$ we also write $0^j$ instead of $(0^j)$ since there is no danger of confusion.

If $f \colon A \to B$ is a function, then the relation $\{(a,f(a)) \mid a \in A\} \subseteq A \times B$ will be called the \emph{graph} of $f$. The graph of a bijection will also be called \emph{bijection graph}.
If $R \subseteq A^n$ is a relation and $i \in \{1,\dots,n\}$, then $i$ is called a \emph{dummy coordinate of $R$} if for all $b \in A$ and 
$(a_1,\dots,a_n) \in R$ the tuple $(a_1,\dots,a_{i-1},b,a_{i+1},\dots,a_n)$ is contained in $R$ as well. 
For $i_1,\dots,i_k \in \{1,\dots,n\}$ such that $i_1 < \dots < i_k$, 
the \emph{projection of $R$ onto $(i_1,\dots,i_k)$} is the relation 
$$\pi_{i_1,\dots,i_k}(R) := \{(t_{i_1},\dots,t_{i_k}) \mid (t_1,\dots,t_n) \in R\}.$$

\subsection{Structures}
A \emph{signature} $\tau$ is a set of relation and function symbols, each equipped 
with an arity $\ar(f) \in {\mathbb N}$. 
A \emph{$\tau$-structure} $\bA$ is a set $A$ (the \emph{domain} of $\bA$) 
together with 
\begin{itemize}
    \item an operation $f^{\bA} \colon A^k \to A$ for each
$f \in \tau$ of arity $k$ (also called a \emph{basic operation} of $\bA$), 
    \item a relation $R^{\bA} \colon A^k \to A$  for each $R \in \tau$ of arity $k$ (also called a \emph{basic relation} of $\bA$).
\end{itemize}
Sometimes, we use the same letter $R$ to denote a symbol $R \in \tau$ and the respective relation $R^{\bA}$ or, to denote a symbol $f \in \tau$
and the respective operation $f^{\bA}$. 
Structures are then specified as tuples $\bA := (A;R_1,R_2,\dots,f_1,f_2,\dots)$. 
A structure is called \emph{finite}
if its domain is finite. 
Note that function symbols of arity $0$ are allowed; they are also called \emph{constant symbols}. The operations denoted by constant symbols in $\fA$ can be identified with elements of $A$, and will be called \emph{constants}. 

If $f \colon A \to B$ is a function, and $t = (t_1,\dots,t_k) \in A^k$ is a $k$-tuple, then $f(t)$ denotes the $k$-tuple $(f(t_1),\dots,f(t_k))$.
A \emph{homomorphism} from a 
$\tau$-structure $\bA$ to a $\tau$-structure $\bB$ is a map $h$ from $A$ to $B$ such that 
\begin{itemize}
    \item for every relation symbol $R \in \tau$, if $a \in R^{\bA}$, then $h(a) \in R^{\bB}$. 
    \item for every function symbol $R \in \tau$ of arity $k$ and $a \in A^k$ we have $f^{\bA}(a) = f^{\bB}(h(a))$. 
\end{itemize}
A bijective homomorphism $f \colon \bA \to \bB$ whose inverse is a homomorphism as well is called an \emph{isomorphism} between $\bA$ and $\bB$; if there exists an isomorphism between $\bA$ and $\bB$, we say that $\bA$ and $\bB$ are \emph{isomorphic}, and write $\bA \cong \bB$. An \emph{embedding} of $\bA$ into $\bB$ is an injective homomorphism which is an isomorphism between $\bA$ and a substructure of $\bB$.

If $\bA$ and $\bB$ are $\tau$-structures, then the \emph{product} $\bA \times \bB$ of $\bA$ and $\bB$ is the $\tau$-structure  with domain $A \times B$ such that
\begin{itemize}
    \item for every relation symbol $R \in \tau$ of arity $k$
    $$R^{\bA \times \bB} = \{((a_1,b_1),\dots,(a_k,b_k)) \mid (a_1,\dots,a_k) \in R^{\bA}, (b_1,\dots,b_k) \in R^{\bB} \}.$$
    \item for every function symbol $R \in \tau$ of arity $k$ and $(a_1,b_1),\dots,(a_k,b_k) \in A \times B$ we have 
    $f^{\bA \times \bB}(((a_1,b_1),\dots,(a_k,b_k)) = (f^{\bA}(a_1,\dots,a_k),  f^{\bB}(b_1,\dots,b_k))$.
\end{itemize}
We will not deal with signatures which have both function and relation symbols. A signature $\tau$ is called \emph{relational} if it only contains relation symbols and no function symbols, and in this situation a $\tau$-structure is called a \emph{relational structure}. Often, we refer to relational structures as simply \emph{structures}, in contrast to \emph{algebras}, which are understood to be structures in a purely functional signature. 

If $\tau$ is a finite relational signature
and $\bB$ is a $\tau$-structure, then 
the \emph{constraint satisfaction problem for $\bB$}, denoted by 
$\Csp(\bB)$, is the 
the class of finite
$\tau$-structures $\bA$ that have a homomorphism to $\bB$. This can be viewed as a computational problem: given a finite $\tau$-structure $\bA$, decide whether there exists a homomorphism to $\bB$. 
The tuples in the relations of $\bA$ are referred to as \emph{constraints}.

A first-order $\tau$-formula is called \emph{primitive positive} 
if it is of the form
$$ \exists x_1,\dots,x_n (\psi_1 \wedge \cdots \wedge \psi_m)$$
where $\psi_1,\dots,\psi_m$ are atomic $\tau$-formulas. 
In other words, disjunction, universal quantification, and negation from first-order logic are not permitted. 
If $\mathcal R$ is a set of relations over a common domain $A$, then we say that $R$ is \emph{primitively positively definable from ${\mathcal R}$} if $R$ has a primitive positive definition in a relational structure with domain $A$ obtained from ${\mathcal R}$ by assigning each of the relations of ${\mathcal R}$ a relation symbol. 


(Symmetric linear) Datalog reductions will be introduced in Section~\ref{sect:descr}. 
However, we state already here the following to motivate primitive positive definability. 

\begin{lemma}[\cite{StarkeDiss}]
If $\bA$ and $\bB$ are structures with finite relational signatures, and every relation of $\fA$ has a primitive positive definition in $\bB$, then there is a 
symmetric linear Datalog reduction reduction from $\Csp(\bA)$ to $\Csp(\bB)$. 
\end{lemma}


\subsection{Algebras}
A \emph{functional signature} is a signature $\tau$ which only contains function symbols and no relation symbols. In this case, 
a $\tau$-structure $\fA$ is also called a \emph{$\tau$-algebra}. 
In this case, substructures of $\fA$ 
are also called \emph{subalgebras} of $\fB$; we then write $A \leq \fB$ (and $\fA \leq \fB)$. The smallest subalgebra of $\fB$ which contains $A \subseteq B$ is called the \emph{subalgebra of $\fB$ generated by $A$}. 
A class of $\tau$-algebras ${\mathcal C}$ is called a \emph{pseudo-variety} if it is closed under homomorphic images, subalgebras, and finite products, and a \emph{variety} if it is additionally closed under infinite products. The smallest pseudo-variety (variety) that contains a class ${\mathcal C}$ of $\tau$-algebras is 
called the \emph{pseudo-variety} (or \emph{variety}, respectively) \emph{generated by ${\mathcal C}$}.

We write 
$\Inv(\fA)$ for the set of all \emph{invariant} relations $R \subseteq A^k$, for some $k \in {\mathbb N}$, i.e., if $t_1,\dots,t_k \in R$ and $f \in \tau$ has arity $k$, then $f^{\fA}(t_1,\dots,t_k) \in R$.
Note that this is precisely the case when $R \leq \fA^k$. 
In this case we also say that $f^{\fA}$ \emph{preserves} $R$,
and that $R$ is \emph{compatible} with $f^{\fB}$.

If $t := t(x_1,\dots,x_n)$ is a $\tau$-term (as in first-order logic (we refer to~\cite{Hodges}) over the variables $x_1,\dots,x_n$,
and $\bA$ is a $\tau$-structure, 
then the \emph{term operation} 
is defined to be the operation 
$t^{\fA} \colon A^k \to A$ 
where $(a_1,\dots,a_k) \in A^k$
is mapped to the element of $A$ obtained by evaluating $t$ at $(a_1,\dots,a_k)$ in the usual way. 
The set of all term operations of 
$\fA$ is a \emph{clone}, i.e., it is closed under composition 
and contains the projections, and it is denoted by $\Clo(\fA)$. 

\begin{example}
    If $\fA$ is an algebra with the empty signature, then $\Clo(\fA)$ just contains the projections; this clone is called the \emph{clone of projections}. 
\end{example}

If $t_1(x_1,\dots,x_k),\dots,t_k(x_1,\dots,x_k)$ are $\tau$-terms over the common variables $x_1,\dots,x_k$, then we write
$\fA \models t_1 \approx \cdots \approx t_k$ if 
$t_1^{\fA}(a_1,\dots,a_k) = \cdots = t_k^{\fA}(a_1,\dots,a_k)$ for all $a_1,\dots,a_k \in A$, and say that
$\fA$ \emph{satisfies} the identity $t_1 \approx \cdots \approx t_k$. 
For example, $({\mathbb Z_2};+) \models x+x+y \approx y+x+x$. 
The following identities are important for this article. 
\begin{align*}
    m(x,x,y) & \approx m(y,x,x) \approx y &&  \text{($m$ is a Maltsev operation)} \\
    f(x,x,y) & \approx f(x,y,x) \approx f(y,x,x) \approx y &&
    \text{($f$ is a minority operation)} \\
    g(x,x,y) & \approx g(x,y,x) \approx g(y,x,x) \approx y && \text{($g$ is a majority  operation).}
\end{align*}

An operation $f$ is called \emph{idempotent} if $f(x,\dots,x) \approx x$; note that Maltsev, minority, and majority operations are idempotent. Clearly, conservative operations are idempotent. A clone is called \emph{idempotent} (\emph{conservative}) if all of its operations are idempotent (conservative, respectively).

A \emph{polymorphism} of a relational structure $\bA$ is an operation $f \colon A^k \to A$, for some $k \geq 1$, which preserves all relations of $\bA$. In other words, it is a homomorphism from $\bA^k$ to $\bA$. 
The set of all polymorphisms of $\bA$ is denoted by $\Pol(\bA)$; it is easy to see that $\Pol(\bA)$ is a clone as well. 
Previously, we have defined $\Inv(\bA)$ for an algebra $\bA$; we also use the same notation $\Inv({\mathcal F})$ 
if ${\mathcal F}$ is a set of operations on a set $A$, and we write
$\Inv({\mathcal F})$ 
for the set of all relations on $A$ that are preserved by all operations from ${\mathcal F}$. Note that
$\Inv(\bA) = \Inv(\Clo(\bA))$. 
We need the following fundamental fact.

\begin{theorem}[\cite{Geiger,BoKaKoRo}]\label{thm:inv-pol}
    Let $\bA$ be a finite relational structure.
    Then a relation is primitively positively definable in $\bA$ if and only if it is preserved by all polymorphisms of $\bA$. 
\end{theorem}


The set of all subalgebras 
of an algebras $\fA$ form a lattice,
where the meet $A_1 \wedge A_2$ of $A_1,A_2 \leq \fA$ is the subalgebra of $\fA$ with domain $A_1 \cap A_2$, and where the join $A_1 \vee A_2$ of $A_1,A_2 \leq \fA$ is the 
subalgebra of $\fA$ generated by $A_1 \cup A_2$. 
An element $x$ of a lattice $(L,\wedge, \vee)$ is called 
\emph{meet irreducible} 
if for all $a,b \in L$ such that $(a \wedge b)=x$ we have 
$x = a$ or $x = b$.  
If $a,b \in L$, we write $a \leq b$ if $(a \wedge b) = a$.
If $L$ is a finite lattice, it has a largest element, i.e., an element $t$ such that $a \leq t$ for all $a \in L$.
Dually, it has a smallest element, defined analogously. 
A \emph{cover} of an element $x$
is an element $a$ such that $x \leq a$ and for every $b \in L$ such that $x \leq b \leq a$ we have $b=x$ or $b=a$. 
We later need that if $|A| \geq 2$, then a subalgebra $R \leq {\fA}^n$ is meet irreducible (in the lattice of subalgebras of $\fA^n$) if 
 and only if $R$ has a unique upper cover among subalgebras of $\fA^n$, which we then denote by $R^* = \bigcap \{ S \in \Inv(\fA): R \subsetneq S \}$.

\subsection{Primitive positive constructions and minion homomorphisms}\label{sect:minhomppconstructions}
In order to describe finite structures with a conservative Maltsev polymorphism we need a more powerful concept than 
primitive positive definitions 
to related structures on different domains, namely \emph{primitive positive constructions}~\cite{wonderland}.
Two structures $\bA$ and $\bB$ are \emph{homomorphically equivalent} if 
there exists a homomorphism from $\bA$ to $\bB$ and vice versa.
A \emph{$d$-th pp-power} of $\bA$, for $d \in {\mathbb N}$, is a structure with domain $A^d$ whose $k$-ary relations, when viewed as $dk$-ary relations over $A$, are pp-definable in $\bA$. A structure $\bB$ is called \emph{pp-constructible} in $\bA$ if it is homomorphically equivalent to a pp-power of $\bA$.

Similarly as pp-definability, it is possible to characterise pp-constructibilty of finite structures algebraically. 
For this we need the concept of a \emph{minion homomorphism}. 
If $f \colon A^k \to A$ is an operation and $\alpha \colon \{1,\dots,k\} \to \{1,\dots,\ell\}$ is a function, then 
$f_\alpha \colon A^{\ell} \to A$ denotes the operation 
given by $$f_{\alpha}(x_1,\dots,x_{\ell}) \mapsto f(x_{\alpha(1)},\dots,x_{\alpha(\ell)}),$$ 
which is also called a \emph{minor} of $f$. If $\mathcal C$ and $\mathcal D$ are clones, then a map $\xi \colon {\mathcal C} \to {\mathcal D}$ 
is called a \emph{minion homomorphism} if 
\begin{itemize}
    \item $\xi$ \emph{preserves arities}, i.e., maps operations in ${\mathcal C}$ of arity $k$ to operations of arity $k$ in ${\mathcal D}$, and 
    \item $\xi$ \emph{preserves minors}, i.e., if $f \in {\mathcal C}$ has arity $k$ and $\alpha \colon \{1,\dots,k\} \to \{1,\dots,\ell\}$ is a map, then the operation $f_{\alpha}$ is an operation in ${\mathcal D}$. 
\end{itemize}

\begin{example}\label{expl:hsp}
It is well-known and easy to see that if $\fB$ is a $\tau$-algebra and
$\fA$ is in the variety generated by $\{\fB\}$, 
then 
there exists a minion homomorphism from 
$\Clo(\fB)$ to $\Clo(\fA)$.
\end{example} 

\begin{theorem}[\cite{wonderland}]\label{thm:pp-constr}
Let $\bA$ and $\bB$ be finite structures. Then the following are equivalent. 
\begin{itemize}
    \item $\bA$ pp-constructs $\bB$. 
    \item $\Pol(\bA)$ has a minion homomorphism to $\bB$.  
\end{itemize}
 \end{theorem}

\subsection{Minimal Taylor}

A function $f \colon B^n \to B$, for $n \geq 2$, 
is called a \emph{Taylor operation}
if for every $i \in [n]$ there are $\alpha,\beta \colon [n] \to [2]$ such that $f_\alpha = f_\beta$ and $\alpha(i) \neq \beta(i)$. It is well-known that a finite structure $\bA$ such that $\Pol(\bA)$ is idempotent has \emph{no} minion homomorphism to the clone of projections on a set of two elements if and if it has a Taylor polymorphism (see, e.g.,~\cite{GraphHomomorphisms}).

\begin{definition}\label{def:minTaylorclone}
    An clone is called \emph{minimal Taylor} if it is Taylor and every proper subclone is not Taylor. 
    An algebra is called \emph{minimal Taylor} if its clone of term operations is minimal Taylor. 
\end{definition}

Clearly, every operation that generates a Taylor minimal clone must be Taylor. We also call such an operation a \emph{minimal Taylor operation}. 
The following is shown in~\cite{MinimalTaylor}, using the cyclic terms theorem~\cite{Cyclic}. 

\begin{theorem}
\label{thm:taylor-min}
Every Taylor clone ${\mathscr C}$ over a finite domain contains a minimal Taylor clone 
${\mathscr C}'$. 
\end{theorem}

\begin{example}\label{expl:mino}
    The clone $\mathcal C$ generated by the Boolean minority operation is minimal Taylor, because every proper subclone of ${\mathcal C}$ only contains the projections. 
\end{example}

\begin{example}
    The operation $f \colon \{0,1\}^3 \to \{0,1\}$ given by $f(x,x,y) = f(y,x,x) = y$ and
    $f(x,y,x) = x$ is Maltsev, and therefore Taylor, but not minimal Taylor.   
    To see this, note that $h(x,y,z) := f(x,f(x,y,z),z)$ is a majority operation.  It is easy to prove by induction on the generation process in the clone that a majority operation cannot generate a Maltsev operation. 
    This shows that $f$ is not minimal Taylor.
\end{example}

\begin{lemma}[Proposition 5.4 in~\cite{MinimalTaylor}]
\label{lem:minTaylorSubAlg}
    Every subalgebra of a minimal Taylor algebra is minimal Taylor. 
\end{lemma}

\subsection{Congruences}\label{sect:congruences}
A good reference for the material in this section is \cite{McKenzieMcNultyTaylor}. 
For an algebra $\fA$, an equivalence relation $\theta \in \Inv(\fA)$ is called a \emph{congruence}.
The \emph{congruence generated by $S \subseteq A^2$} is 
the smallest congruence $\theta$ of $\fA$ containing $S$; 
if $S = \{(a,b)\}$ we also say that \emph{$\theta$ is generated by $(a,b)$}.
It is well-known and easy to see that an  equivalence relation $\theta$ 
on $A$ is a congruence of $\fA$ 
if and only if $\theta$ is preserved by all unary term functions in the algebra obtained from $\fA$ by adding all constants. 

Given a congruence $\theta$ and an element $a \in A$, we denote by $[a]_\theta$ the set of all $b$ such that $(a,b) \in \theta$ and call this set a \emph{block}, or a \emph{$\theta$-class}. A block is called \emph{trivial} if it contains only one element. A congruence is called \emph{nontrivial} if it has at least one nontrivial class. It is a well-known and fundamental fact that for any homomorphism $h \colon \fA \to \fB$, the set 
\[
\ker(h) := \{ (a,b) \in A^2 \mid h(a) = h(b) \}
\]
is a congruence of $\fA$. Furthermore, given a congruence $\theta$ of $\fA$, one can define a \emph{quotient algebra} $\fA/\theta$ with underlying set equal to all $\theta$-classes by setting 
\[
f([a_1]_\theta, \dots, [a_n]_\theta) = [f(a_1, \dots, a_n)]_\theta
\]
for a basic operation $f$ of $\fA$. Since $\theta$ is an equivalence relation belonging to $\Inv(\fA)$, it is easy to see that the quotient algebra is well-defined. The \emph{first isomorphism theorem} of universal algebra states that 
\[
\fA/\theta \simeq h(\fA),
\]
where $\theta$ is the kernel of $h$. Therefore, the isomorphism types of the homomorphic images of an algebra are totally determined by its congruences. We denote by $\eta_\theta \colon \fA \to \fA / \theta$ the \emph{natural map}, which is the homomorphism which maps each element of $\fA$ to the $\theta$-class containing it.  

If $\theta$ and $\theta'$ are congruences such that $\theta \subseteq \theta'$, then we say that $\theta$ is \emph{finer} than $\theta'$, and that
$\theta'$ is \emph{coarser} than $\theta$.
Every algebra $\fA$ has at least two congruences, namely the graph of the equality relation $\{(a,a): a \in A \} $ and the full relation $A \times A$. We will denote these two congruences as $0$ and $1$, respectively, and they will be called the \emph{trivial} congruences of $\fA$. The collection of all congruences of an algebra forms 
a lattice, where the meet of two congruences $\theta_1$ and $\theta_2$ equals $\theta_1 \cap \theta_2$, and the joint of $\theta_1$ and $\theta_2$ equals the smallest congruence of $\fA$ that contains $\theta_1 \cup \theta_2$.
This lattice has 
the largest element $1$ and the smallest element $0$.
An algebra is called \emph{simple} if it has exactly these two congruences and no others. 
A congruence of $\fA$ is \emph{maximal} if is is distinct from 
$A \times A$ and maximal with respect to inclusion. 
Note that a congruence $\theta$ of $\fA$ is maximal if and only if $\fA/\theta$ is simple.

A congruence $\theta$ is called \emph{atomic} if it is non-trivial and every finer non-trivial congruence of $\fA$ equals $\theta$. An algebra is called \emph{subdirectly irreducible} if it has exactly one atomic congruence, which is then called the \emph{monolith} of the algebra. The \emph{correspondence theorem} of universal algebra states that the lattice of congruences of a quotient algebra $\fA/ \theta$ is isomorphic to the sublattice of congruences of $\fA$ that contain $\theta$.

If $\alpha_1 \colon A_1 \to B_1,\dots,\alpha_n \colon A_n \to B_n$ are maps, then $\alpha_1 \times \dots \times \alpha_n$ denotes the map with domain $A_1 \times \cdots \times A_n$  and range $B_1 \times \cdots \times B_n$ given by $(a_1,\dots,a_n) \mapsto (\alpha_1(a_1),\dots,\alpha_n(a_n))$. Given two sequences of algebras $\fA_1, \dots, \fA_n$ and congruences $\theta_1, \dots, \theta_n$ so that $\theta_i$ is a congruence of $\fA_i$ for each $1 \leq i \leq n$, we denote by $\theta_1 \times \dots \times \theta_n$ the \emph{product congruence} of $\fA_1 \times \dots \times \fA_n$ which relates a pair of elements if and only if they are related in the $i$-th coordinate by $\theta_i$, for each $1 \leq i \leq n$. It is immediate that the kernel of the product of the natural maps $\eta_{\theta_1} \times \dots \times \eta_{\theta_n} \colon A_1 \times \dots \times A_n \to (A_1/ \theta_1) \times \dots \times (A_n/\theta_n)$ is equal to $\theta_1 \times \dots \times \theta_n$.


\section{The structure of conservative minority algebras}
\label{sect:cons-mino}
In this section we will investigate finite algebras $\fA = (A; f)$ where $f \colon A^3 \to A$ is a conservative minority operation. 
We call such algebras \emph{conservative minority algebras}. The following lemma, due to Carbonnel~\cite{Carbonnel16b}, implies 
that this also solves the case where we just assume that $f$ is a conservative Matlsev operation.


\begin{lemma}[Carbonnel~\cite{Carbonnel16b} (Lemma 1)]\label{lem:Carbonnel}
    If $f \colon A^3 \to A$ is a conservative Maltsev operation, then $\Clo(A;f)$ contains 
    a conservative minority operation. 
\end{lemma}


\subsection{Block congruences}

We will call a congruence with at most one nontrivial class a \emph{block congruence}. In the setting of conservative minority algebras, block congruences have some nice properties.

\begin{proposition}\label{prop:blocksinsideblocks}
Let $f \colon  A^3 \to A$ be a conservative minority operation. 
Let $\alpha$ be a congruence of $\fA = (A;f)$ and let $B$ be a congruence class. 
Let $\theta$ be a congruence
of the subalgebra of $\fA$ with domain $B$.
Then for each $b \in B$, the equivalence relation $\theta_b$ on $A$ defined
by
\[
\theta_b := \{(x,y) \in A^2 \mid  x,y \in [b]_\theta \text{ or } x=y \}
\]
is a block congruence of $\fA$.
\end{proposition}

\begin{proof}
It is well known that to check the compatibility of an equivalence relation with a function, it suffices to check the compatibility with the basic translations of that function, i.e., those unary functions obtained by substituting all but a single variable with a constant from $A$. We show the argument for when the second two variables of $f$ are set to constants. The other cases are analogous.

So, let us take $(u,v) \in \theta_b$ and consider its image under the unary mapping $x \mapsto f(x,c_1, c_2)$ for constants $c_1,c_2 \in A$. If $u=v$ then of course $f(u,c_1, c_2) = f(v, c_1, c_2) \in \theta_b$. If $u \neq v$, then they both belong to $[b]_\theta$. If also both $c_1, c_2 \in [b]_\theta$, then compatibility is immediate because $[b]_\theta$ is a subuniverse of $\fA$.
So let us assume that $c_1,c_2$ do not both belong to $[b]_{\theta}$. 
We now consider cases corresponding to the other possible configurations of $c_1, c_2$ with respect to $\theta_b$-classes. Notice that since $\fA$ is conservative, it must be that $(f(u,c_1, c_2), f(v, c_1, c_2)) \in \{u, c_1, c_2 \} \times \{v, c_1, c_2 \}$. We show that actually $(f(u,c_1, c_2), f(v, c_1, c_2)) \in \{(u,v), (c_1, c_1), (c_2, c_2) \}$.
\begin{itemize}
\item Suppose that $(c_1, c_2) \in \theta$. Because $(u,u) = (u,f(u,c_1,c_1)) \in \theta$, we obtain $(u,f(u,c_1, c_2)) \in \theta$. Because $\fA$ is conservative, $f(u, c_1, c_2) \in \{u, c_1, c_2\}$. Since 
$(c_1,c_2) \in \theta$, and one of $c_1$ and $c_2$ does not belong to $[b]_{\theta}$, they both 
do not belong to $[b]_\theta$. This forces $f(u, c_1, c_2) = u$. This argument works for $v$ as well, so we conclude that $$(f(u, c_1, c_2), f(v, c_1, c_2) ) = (u, v) \in \theta_b.$$

\item Suppose that $(c_1, c_2) \notin \theta$, and precisely one of $c_1, c_2$ belongs to $[b]_\theta$. Without loss, suppose that $c_1 \in [b]_\theta = [u]_{\theta}$. Suppose additionally that $c_2 \in B$. Since $(c_2,c_2) = (c_2,f(u,u,c_2)) \in \theta$, we obtain $(c_2,f(u,c_1, c_2)) \in \theta$. This forces $f(u, c_1, c_2) = c_2$. This argument works for $v$ as well, so we conclude that 
$$(f(u, c_1, c_2), f(v, c_1, c_2) ) = (c_2, c_2) \in \theta_b.$$
If $c_2 \in A \setminus B$, the same argument works with $\alpha$ playing the role of $\theta$.

\item Suppose that $(c_1, c_2) \notin \theta$, $c_1, c_2 \in B$, 
and neither 
$c_1$ nor $c_2$ belongs to $[b]_\theta$. It follows that the classes $[b]_\theta$, $[c_1]_\theta$, and $[c_2]_\theta$ are pairwise distinct, so in this case the only pairs from the product $ \{u, c_1, c_2 \} \times \{v, c_1, c_2 \}$ that are $\theta$-related are $(u,v), (c_1, c_1), (c_2,c_2)$. Therefore, 
\[
(f(u, c_1, c_2), f(v, c_1, c_2) ) \in 
\{
(u,v), (c_1, c_1), (c_2, c_2)
\} \subseteq \theta_b. 
\]
\item Suppose that $(c_1, c_2) \notin \theta$, $c_1\in B$, and $ c_2 \in A \setminus B$. Since $(c_2,c_2) = (c_2,f(u,u,c_2)) \in \alpha$, we obtain $(c_2,f(u,c_1, c_2)) \in \alpha$. This forces $f(u, c_1, c_2) = c_2$. This argument works for $v$ as well, so we conclude that 
$$(f(u, c_1, c_2), f(v, c_1, c_2) ) = (c_2, c_2) \in \theta_b.$$ 
\end{itemize}
We finish by remarking that this argument works for the other basic translations of $f$ because we assume that it is a  conservative minority operation. Therefore, the use of the Maltsev identities in the cases above is permissible also for the unary polynomial $f(c_1, x, c_2)$.
\end{proof}

\begin{corollary}\label{cor:blocksarecongruences}
Let $f \colon A^3 \to A$ be a conservative minority operation, let $\theta$ be a congruence of $\fA = (A;f)$, and let $a \in A$. Then the equivalence relation 
$\theta_{a} := \{(x,y) \in A^2 \mid x,y \in [a]_\theta \text{ or } x =y \}$ is a block congruence of $\fA$.
\end{corollary}
\begin{proof}
Apply Proposition~\ref{prop:blocksinsideblocks} for $\alpha := A^2$ (so that $\fB = \fA$) and $b := a$.
\end{proof}

\begin{corollary}
Let $f \colon A^3 \to A$ be a conservative minority operation. The atomic congruences of $\fA = (A;f)$ are exactly the nontrivial block congruences $\theta$ for which the 
subalgebra whose domain is the nontrivial congruence class is simple. 
\end{corollary}

\begin{proof}
Suppose that $\theta$ is an atomic congruence of $\fA$. By Corollary~\ref{cor:blocksarecongruences}, $\theta$ cannot have more than one nontrivial class. By Proposition \ref{prop:blocksinsideblocks}, the nontrivial block of $\theta$ cannot have a nontrivial congruence. 
Conversely, suppose that $\theta$ is a non-trivial block congruence with non-trivial block $B$. 
and that the subalgebra $\fB$ of $\fA$ with domain $B$ is simple. 
Let $\theta'$ be a finer congruence than $\theta$; then $\theta' \cap B^2$ is a congruence of $\fB$, which by the simplicity of $\fB$ equals $0$. Since $\theta$ is a block congruence with non-trivial block $B$, and $\theta' \subseteq \theta$, it follows that $\theta' = 0$ as well, which shows that $\theta$ is atomic.
\end{proof}

\begin{proposition}\label{prop:blockintersection}
Let $f \colon A^3 \to A$ be a conservative Maltsev operation and let $\theta_1, \theta_2$ be block congruences of $(A,f)$. If two classes $[a]_{\theta_1}$ and $[b]_{\theta_2}$ have a nonempty intersection, then $[a]_{\theta_1} \subseteq [b]_{\theta_2}$ or vice versa.
\end{proposition}
\begin{proof}
Let $c \in [a]_{\theta_1} \cap [b]_{\theta_2}$. 
If $[a]_{\theta_1} \subseteq [b]_{\theta_2}$ is not true, we may suppose
that $a \notin [b]_{\theta_2}$,
and if $[b]_{\theta_2} \subseteq [a]_{\theta_1}$ is not true, we may suppose
that $b \notin [a]_{\theta_1}$. 
Thus, we suppose for contradiction 
that $a \notin [b]_{\theta_2}$
and that $b \notin [a]_{\theta_1}$. 
Define $d := f(a,c,b)$.
Then 
$(f(c,c,b),f(a,c,b)) = (b,d) \in \theta_1$ and $(f(a,c,b),f(a,b,b)) = (d,a) \in \theta_2$. 
We have $d \notin [a]_{\theta_1}$ since
otherwise $(a,b) \in \theta_1$,
contrary to our assumptions. 
Since $\theta_1$ is a block congruence, 
$(b,d) \in \theta_1$ thus implies that 
$b=d$. 
Likewise, $d \notin [b]_{\theta_2}$ since
otherwise $(a,b) \in \theta_2$,
contrary to our assumptions. 
Since $\theta_2$ is a block congruence, 
$(d,a) \in \theta_2$ thus implies that 
$d=a$. Hence, $a=d=b$, a contradiction.
\end{proof}

\begin{proposition}\label{prop:subdirectlyirreduciblearelinearchains}
Let $f \colon A^3 \to A$ be a conservative Maltsev operation and let $\fA = (A;f)$. Then the  following are equivalent:
\begin{enumerate}
\item
$\fA $ is subdirectly irreducible.
\item The congruences of $\fA$ are linearly ordered by the subset relation.
\item 
Every congruence of $\fA$ is a block congruence. 
\end{enumerate}
\end{proposition}

\begin{proof}
\emph{1.} implies \emph{3.}
Suppose that there exists a congruence $\theta$ of $\fA$ that is not a block congruence. Then there exist two nontrivial $\theta$-classes, call them $[a]_\theta$ and $[b]_\theta$. In this situation there are two nontrivial block congruences $\theta_a$ and $\theta_b$ whose nontrivial blocks are disjoint. 
We may choose these congruences to be atomic, and 
hence $\fA$ is not subdirectly irreducible. 

\emph{3.} implies \emph{2.}
If every congruence of $\fA$ is a block congruence, then any two congruences $\theta_1, \theta_2$ are comparable, because otherwise Proposition \ref{prop:blockintersection} indicates that their respective nontrivial blocks are disjoint, from which it follows that $\theta_1 \vee \theta_2 = \theta_1 \cup \theta_2$ is not a block congruence.

\emph{2.} implies \emph{1.} 
This is obvious.
\end{proof}

\begin{proposition}\label{prop:uniquemaxblock}
Let $f \colon A^3 \to A$ be a conservative minority operation. For every $a\in A$, there exists a unique maximal proper subset $B_a \subseteq A$ that contains $a$ and is a block of a congruence. 
\end{proposition}

\begin{proof}
Suppose that $B_1,B_2 \subseteq A$ are congruence classes of $\theta_1, \theta_2$, respectively, with $a\in B_1 \cap B_2$. By Corollary~\ref{cor:blocksarecongruences}, both $(\theta_1)_a$ and 
$(\theta_2)_a$ are block congruences. By Proposition \ref{prop:blockintersection}, either $B_1 \subseteq B_2$ or vice versa. If both $B_1, B_2$ are proper and maximal with respect to the subset relation, then they are equal. 
\end{proof}

\begin{corollary}\label{corollary:maximalcongruence}
Let $A$ be a finite set and $f \colon A^3 \to A$ a conservative minority operation. Then there exists a unique maximal proper congruence of $(A, f)$.
\end{corollary}

\begin{proof}
This follows immediately from Proposition \ref{prop:uniquemaxblock}. 
\end{proof}


\subsection{Tree representations of conservative minority algebras}
\label{sect:tree-reps}

In this section we apply the insights of the previous section to represent a conservative minority algebra with a tree equipped with additional structure. The additional structure specifies for the set of children of every vertex a conservative minority algebra. There is a natural multiplication that can then be defined on the set of leaves of such a tree and in doing so one obtains a conservative minority algebra. On the other hand, we will see that every conservative minority algebra has such a tree representation.

\begin{definition}\label{def:consmintree}
    Let $\mathcal{T} = (T, \leq_\mathcal{T}, S)$ be a triple whose entries satisfy the following condition:
    \begin{enumerate}
        \item $S = \{ \mathbf{A}_1, \dots, \mathbf{A}_s \}$ is a finite collection of conservative minority algebras $\mathbf{A}_k = (A_k;m_k)$, for each $1 \leq k \leq s$,  which satisfies the property that the minority operations $m_k$ and $m_l$ agree on $A_k \cap A_l$, for all $1 \leq k \leq l \leq s$.
        \item $T \subseteq  \left( \bigcup_{1 \leq k \leq s} A_k \right)^{< \omega}$ is a set of finite tuples which is closed under prefixes (the empty tuple $\epsilon$ is a prefix of every tuple). 
        \item $\leq_\mathcal{T}$ is the tree order defined on $T$ by $w_1 \leq_\mathcal{T}  w_2$ if and only if $w_2$ is a prefix of $w_1$.
        \item For every $v \in T$ that is not a leaf, there exists $k \in \{1,\dots,s\}$ and $D \subseteq A_k$ such that the set $C_v$ of children of $v$ equals
        \[
        v^\frown D := \{ v^\frown (a): a \in D \}.
        \]
    \end{enumerate}
    We call such an $S$ a collection of \emph{local} algebras and we call such a $\mathcal{T}$ a \emph{conservative minority tree (over $S$)}. 

\end{definition}

\begin{remark}
    When working with a conservative minority tree $\mathcal{T}$ over a set of local algebras $S$, we will often only refer to $\mathcal{T}$, but include the possibility of emphasizing that $\mathcal{T}$ is a conservative minority tree `over $S$' for situations where it is important to say something specific about the local algebras that are being used. 
\end{remark}

To avoid excessive notation, we will usually just refer to $\mathcal{T}$ instead of $(T; \leq_\mathcal{T})$ when working with the underlying tree structure and we will often write $\leq$ instead of $\leq_\mathcal{T}$ when the context permits it. We will also allow ourselves to write $v \in \mathcal{T}$ and $X \subseteq \mathcal{T}$, respectively, when $v$ is a vertex belonging to $T$ and $X$ is a subset of vertices of the underlying tree $T$ of $\mathcal{T}$.

Now, we introduce some auxiliary notation and definitions for conservative minority trees. 

\begin{itemize}
    \item We denote for a set of vertices $X \subseteq \mathcal{T}$ the longest common prefix among its elements by $\bigvee X$, i.e. the least upper bound or join in the poset $(T, \leq_\mathcal{T})$. 
    \item We denote by $A_\mathcal{T}$ the set of leaves of the underlying tree of $\mathcal{T}$. 
    \item In \emph{4.}\ of Definition~\ref{def:consmintree}, the set of children of a vertex $v$ which is not a leaf, which we denote $C_v$, is formally the set of tuples which extend $v$ by one extra entry. Since this set is equal to $v^\frown D$ a subalgebra domain $D \subseteq A_k$ for some $1\leq k \leq s$, the set $C_v$ implicitly has the structure of the conservative minority algebra $\mathbf{D} \leq \mathbf{A}_k$. It is actually more convenient to work directly with the algebra $\mathbf{D}$, which we call the \emph{$v$-successor algebra} (in $\mathcal{T}$), and denote it by $\mathbf{v}^{+_\mathcal{T}}$, or just write $v^{+_\mathcal{T}}$ for the domain $A_k$. In some uses of this successor notation, the tree we are working in is clear from the context in which it is used, and so we may write $+$ instead of $+_\mathcal{T}$, e.g.\ $v^{+}$ instead of $v^{+_\mathcal{T}}$.
    \item Given distinct vertices $w,v \in \mathcal{T}$ such that $w \leq v$, we define the \emph{$v$-successor value above $w$} to be the unique $v$-successor value $a \in v^{+_\mathcal{T}}$ such that $w \leq v^\frown (a)$, and denote this element by $w^{v^+}$. 
    \item Given a set of leaves $D \subseteq A_\mathcal{T}$ of $\mathcal{T}$ over $S$, we denote by $\mathcal{T}_D$ the conservative minority \emph{subtree} of $\mathcal{T}$ (over $S$) whose underlying tree structure is obtained by closing $D$ under prefixes. We say that, for a vertex $v \in \mathcal{T}_D$, that $\mathcal{T}_D$ is \emph{full in $\mathcal{T}$ at $v$} if $v^{+_{\mathcal{T}_D}}=v^{+_{\mathcal{T}}}$, that is, the $v$-successor values in $\mathcal{T}_D$ are equal to the $v$-successor values in $\mathcal{T}$ (note that $v^{+_{\mathcal{T}_D}} \subseteq v^{+_{\mathcal{T}}}$ always). 
     \item We say that a vertex $v \in \mathcal{T}$ is \emph{full} if $\mathbf{v}^+ \in S$. We say that $\mathcal{T}$ is \emph{full} if every vertex which is not a leaf is full. 
    \item We say that the conservative minority tree $\mathcal{T}$ (over $S$) is \emph{simple} (or \emph{hereditarily simple}) if every successor algebra $\mathbf{v}^+$ for $v \in \mathcal{T}$ is simple (or hereditarily simple, respectively).
    \item We say that $\mathcal{T}$ is \emph{reduced} if every vertex in $\mathcal{T}$ which is not a leaf has at least two children.

\end{itemize}

\begin{definition}\label{def:leafalgebra}

Let $\mathcal{T} = (T, \leq, S)$ be a conservative minority tree over $S = \{\mathbf{A}_1, \dots, \mathbf{A}_s\}$ and let $A_\mathcal{T}$ be the set of leaves of the underlying tree of $\mathcal{T}$. 
We define $\mathbf{A}_\mathcal{T} := (A_\mathcal{T}, m_{\mathcal{T}})$ to be the conservative minority algebra whose minority operation $m_\mathcal{T}$ is specified as follows. Suppose that $w,u,z \in A_\mathcal{T}$ are leaves of $\mathcal{T}$. Let $v = \bigvee \{w,u,z\}$ in $\mathcal{T}$. Note that if $v$ is equal to $w$, $u$, or $z$, then $v = w = u = z$, so in this case we set $m_\mathcal{T}(w,u,z)$ to the common value $w=u=z$. Otherwise, $v$ is a proper prefix of all three leaves $w,u,z$, and is therefore not a leaf itself. Hence, $\mathbf{v}^+$ exists and is equal to $\mathbf{D} \leq \mathbf{A}_k$ for some $1 \leq k \leq s$ . We then define 
\[
m_\mathcal{T}(w,u,z) = w,u, \text{ or } z \text{ if } m_k(w^{v^+},u^{v^+},z^{v^+}) = w^{v^+}, u^{v^+}, \text{ or } z^{v^+}, \text{ respectively.}
\]

If $\mathbf{A}$ is a conservative minority algebra and $\mathcal{T}$ is a conservative minority tree such that $\mathbf{A}_\mathcal{T}$ is isomorphic to $\mathbf{A}$, we say that $\mathcal{T}$ \emph{represents} $\mathbf{A}$. 
\end{definition}
\begin{remark}
It is straightforward to check that $m_\mathcal{T}$ is a well-defined conservative minority operation and we leave these details to the reader. 
\end{remark}

We say that a conservative minority algebra $\mathbf{B}$ is \emph{represented} by a conservative minority tree $\mathcal{T}$ if $\mathbf{B}$ is isomorphic to $\mathbf{A}_\mathcal{T}$. Given a conservative minority tree $\mathcal{T}$, we denote by $\mathfrak{S}_{\mathbf{A}_\mathcal{T}}$ the structure whose invariant relations are those given by \emph{2.} of Theorem~\ref{thm:consminorityrelbasis} for $\mathbf{A}_\mathcal{T}$ (see Definition~\ref{def:relationalbasisStructureNamed}). 

\begin{lemma}\label{lem:subalgebrarepresentingtree}
    Let $\mathcal{T}$ be a conservative minority tree and let $D \subseteq A_\mathcal{T}$. Then $D$ is the domain of a subalgebra $\mathbf{D} \leq \fA_\mathcal{T}$ which is represented by $\mathcal{T}_D$.
\end{lemma}
\begin{proof}
    Since $\fA_\mathcal{T}$ is conservative, it is immediate that $D$ is the domain of a subalgebra $\mathbf{D} \leq \fA_\mathcal{T}$. It is also immediate that $\mathcal{T}_D$ represents $\mathbf{D}$, since 

    \begin{itemize}
    \item
    the leaves of $\mathcal{T}_D$ are exactly the elements of $D$, 
    \item $\mathcal{T}_D$ contains all least upper bounds (maximal common prefixes) of subsets of $S$, and 
    \item for $a,b,c \in A_{\mathcal{T}_S}$ and $v = \bigvee \{a,b,c\}$, the values $a^{v^{+}}, b^{v^{+}}$, and $c^{v^{+}}$ are intrinsic to $a,b,c$, hence
    \[
    m_{\mathcal{T}}(a,b,c) = m_{\mathcal{T}_D}(a,b,c). 
    \qedhere 
    \]
    \end{itemize}
\end{proof}

\subsubsection{Existence of representations}

We establish in this section that every conservative minority algebra can be represented by a simple reduced full conservative minority tree (Theorem~\ref{thm:representwithtreealgebras}). We then observe some key connections between the tree structure of a conservative minority tree $\mathcal{T}$ and the block congruences of the subalgebras of the represented algebra $\mathbf{A}_\mathcal{T}$ (Lemma~\ref{lem:CongruencesofSubalgebras}).

\begin{theorem}\label{thm:representwithtreealgebras}
Let $\fA = (A;m)$ be a conservative minority algebra. Then there exists a simple reduced full conservative minority tree $\mathcal{T}$ over a collection $S$ of pairwise nonisomorphic (simple) conservative minority algebras such that $\mathcal{T}$ represents $\mathbf{A}$. 
\end{theorem}
\begin{proof}
Let $X = \{[a]_\theta: \text{$a \in A$ and $\theta$ is a block congruence of $\fA$}
\}$. It follows from Proposition \ref{prop:blockintersection} that $X$ ordered by set inclusion is a tree. Indeed, the equality relation is a block congruence with no nontrivial classes, hence the minimal elements of $X$ are exactly the singleton subsets of $A$. The full relation is also a block congruence, with a single nontrivial class equal to $A$. Proposition \ref{prop:blockintersection} indicates that two classes of a block congruence must either have empty intersection or be comparable in the subset order, so there is a unique path 
in $X$ from a given $[a]_\theta \in X$ to the root $A \in X$, which we denote by $v_{[a]_\theta}$.

Now we define 
\[
T':= \{ v_{[a]_\theta}: a \in A \text{ and } \theta \text{ is a block congruence of $\mathbf{A}$} \} 
\]
to be the tree of all such paths in $X$, now ordered by the prefix order we use for conservative minority trees. Let $\theta$ be a nontrivial block congruence of $\mathbf{A}$ with nontrivial class $[a]_\theta$. Clearly, $v_{[a]_\theta}$ is not a leaf in $T$. Since $\mathbf{A}$ is conservative, $[a]_\theta$ is the domain of a subalgebra $\mathbf{[a]_\theta} \leq \mathbf{A}$. If we denote by $\lambda$ the unique maximal congruence of $\mathbf{[a]_\theta}$, then the set of successors of $v$ in $T$ is exactly 
\[
v_{[a]_\theta}^+ = \{ [b]_{\lambda}: b \in [a]_\theta \},
\]
and this successor set is naturally equipped with the quotient structure of $\mathbf{[a]_\theta} /\lambda$, which is simple, since $\lambda$ is a maximal congruence. Hence, we define for such $[a]_\theta$ the successor algebra $\mathbf{v_{[a]_\theta}^+} := \mathbf{[a]_\theta} /\lambda$ and denote by $S'$ the collection of all such successor algebras. 

Taken all together, we have specified a full simple conservative minority tree $\mathcal{T}' = (T', \leq, S')$. We claim that $\mathcal{T}'$ represents $\mathbf{A}$, which we will demonstrate by arguing that the mapping $\phi \colon \fA \to \fA_{\mathcal{T}'}$ defined by $\phi(a) := v_{\{a\}}$ is an isomorphism. Indeed, take $a,b,c \in A$. We want to show that $\phi(m(a,b,c)) = m_{\mathcal{T}'}(v_{\{a\}},v_{\{b\}},v_{\{c\}})$. This is obvious if $a = b = c$, so suppose without loss of generality that $a \neq b$. In this case, there exists a minimal nontrivial block congruence $\theta$ with $[a]_\theta = [b]_\theta = [c]_\theta$ (apply Corollary~\ref{cor:blocksarecongruences} to the least congruence of $\mathbf{A}$ which collapses $\{a,b,c\}$ ). Observe that the nontrivial class of $\theta$ is precisely equal to the least upper bound of the set $ \{\{a\},\{b\},\{c\}\}$ in $X$, or equivalently, that \[v_{[a]_\theta} = \bigvee \{ v_{\{a\}},v_{\{b\}},v_{\{c\}} \} \] in $\mathcal{T}'$. Let $\lambda$ be the maximal congruence of the subalgebra $\mathbf{[a]_\theta} \leq \fA$. It follows readily from the definitions that 
for  $x \in \{a,b,c\}$: 
\begin{align*}
m(a,b,c) &= x \\
\iff m([a]_\lambda, [b]_\lambda, [c]_\lambda) &= [x]_\lambda
\\
\iff m_{\mathcal{T}'}(v_{\{a\}},v_{\{b\}},v_{\{c\}}) &= v_{\{x\}}. 
\end{align*}

Finally, it is obvious that $\mathcal{T}$ is reduced, because $[a]_\theta$ is not a leaf in $X$ if and only if $|[a]_\theta| > 1$ if and only if $[a]_\theta$ has a maximal proper congruence, which guarantees that all vertices of $\mathcal{T}$ that are not leaves have at least two children.

To finish the proof of the lemma, we choose a representative of each isomorphism class among the local algebras $S'$ of $\mathcal{T}'$, collect these representatives in a set of simple algebras $S$, and relabel the entries of the sequences in $T'$ with their counterparts in the algebras among $S$ to produce a tree $T$. Now set $\mathcal{T} = (T, \leq, S)$. Obviously, $\mathbf{A}_{\mathcal{T}'}$ and $\mathbf{A}_\mathcal{T}$ are isomorphic, but now the simple algebras in $S$ are pairwise nonisomorphic, as desired. 
\end{proof}

We now take a moment to discuss the unsurprising limitations of Theorem~\ref{thm:representwithtreealgebras}. Its  utility is that it gives a recipe for building nonsimple conservative minority algebras from simple conservative minority algebras. One might hope that the class of simple conservative minority algebras itself is easy to describe, but this is unfortunately not the case, since any algebra with congruences is quite easy to hide as subalgebra of a simple conservative minority algebra. This means that there is a limit to the amount of congruence information that can be stored in a conservative minority tree over a collection of simple algebras, since these simple algebras themselves can possess subalgebras with a richer collection of congruences. 

We make no attempt to understand all of the different ways that a conservative minority algebra can be embedded into a simple conservative minority algebra. Rather, we produce a recipe for one such embedding in the proof of the following proposition, which is only included to satisfy the curious reader and is not needed for the main results.

\begin{proposition}\label{prop:AnythingEmbedsInSimple}
    Let $n \geq 2$ and let $\mathbf{A} = (A, f)$ be a conservative minority algebra of size $n$. There exists a simple conservative minority algebra $\mathbf{B} = (B, \overline{f})$ of size $n+1$ such that $\mathbf{A} \leq \mathbf{B}$.
\end{proposition}

\begin{proof}
    The proof proceeds by induction on $n \geq 2$. We first prove the basis. Let $\mathbf{A}$ be the unique two element minority algebra. Let $\mathbf{P}_3 : = (\{0,1,2\} ; p_3)$ be the three element projection minority algebra, where $p_3(x,y,z) = y$ for injective triples $(x,y,z)$. We prove in Lemma~\ref{lem:brady1ishsimple} that $\mathbf{P}_3$ is simple, and since $\mathbf{A}$ embeds in any conservative minority algebra, it certainly embeds in $\mathbf{P}_3$. Now suppose the proposition is true for $n \geq 2$. Let $B = A \cup \{z\}$ be the domain of $A$ along with an additional element $z$. The idea is to extend the definition of $f$ to all three-element subsets of the form $X \cup \{z\}$, where $X \subseteq A$ has cardinality two, in such a way to prohibit nontrivial congruences. 

    So, let $\theta$ be the unique maximal proper congruence of $\mathbf{A}$ which is guaranteed by Corollary~\ref{corollary:maximalcongruence} (note that $\theta$ could be the equality relation and is not the full relation), let $S_1, \dots, S_k$ be the $\theta$-classes and let $\mathbf{S}_1, \dots, \mathbf{S}_k$ be the respective subalgebras on these classes. We define an extension $\overline{f}:B^3 \to B$ of $f$ as follows. 

    \begin{itemize}
        \item For each $1 \leq j \leq k$, define $\overline{f}$ so that its restriction to $S_j':=S_j \cup {z}$ gives a simple conservative minority algebra $\mathbf{S}_j'$ (note that here we apply the induction hypothesis).
        \item On all other subsets $\{a,b\ \} \cup \{z\}$ for which $\overline{f}$ has not yet been specified (this is the same thing as saying $(a,b) \notin \theta$), define it so that the resulting three-element subalgebra is isomorphic to $\mathbf{P}_3$.
    \end{itemize}
    Now set $\mathbf{B} = (B, \overline{f})$. Obviously, $\mathbf{A} \leq \mathbf{B}$. 
    
    We claim that $\mathbf{B}$ is simple. Suppose towards a contradiction that $\lambda$ is a congruence of $\mathbf{B}$ which has a nontrivial class $C$. First, we note that for each $1\leq j \leq k$,
    \[
    | C \cap S_j'| \in \{0,1, |S_j'|\}
    \]
    holds, as otherwise $\lambda$ restricted to $S_j'$ is a nontrivial congruence of $\mathbf{S}_j'$, which contradicts simplicity. We now analyze the following cases. 
    \begin{itemize}
        \item Suppose there exist distinct $1 \leq j_1 \neq j_2 \leq k$ so that 
            $| C \cap S_{j_1}'| = |S_{j_1}'|$ and $| C \cap S_{j_2}'| = |S_{j_2}'|$ (note this implies that $z \in C$). Then $\lambda$ restricted to $A$ is a congruence $\eta$ with a block which contains $S_{j_1} \cup S_{j_2}$. It is therefore impossible that $\eta$ is a proper congruence, since $\theta$ is the unique maximal proper congruence, so we deduce that $\eta$ must be the full congruence on $\mathbf{A}$. It follows that $A \subseteq C$. Since $z \in C$ in this case also, we obtain that $\lambda$ has only one class $C = A \cup \{z\} = B$, which is contrary to the assumption that $\lambda$ is proper. 

        \item Suppose there is exactly one $1\leq j \leq k$ such that $ | C \cap S_j'| = |S_j'|$. It follows that $z \in C$. Let $a \in S_j$. Then $a \in C$ also, hence $(a,z) \in \lambda$. Pick $b \in A$ so that $(a,b) \notin \theta$. Then $b \notin C$, as otherwise $(a,b) \in \lambda$, which implies that $(a,b) \in \theta$, because $\lambda$ restricted to $A$ is contained in $\theta$. So, the situation is that $\lambda$ restricted to the three-element set $\{a,b,z\}$ is a nontrivial congruence of the associated subalgebra of $\mathbf{B}$, which is impossible, since this subalgebra is isomorphic to $\mathbf{P}_3$. 

        \item Suppose that there is no $1\leq j \leq k$ such that $ | C \cap S_j'| = |S_j'|$, i.e.\ 
        $
    | C \cap S_j'| \in \{0,1\}.
    $
    for every $1 \leq j \leq k$. Observe that $z \notin C$. Indeed, if $z \in C$, then since $C$ is a nontrivial block of $\lambda$, there exists $a \in C$ distinct from $z$. Let $1 \leq j \leq k$ be such that $a \in S_j$. By assumption, $\lambda$ restricted to $S_j'$ is nontrivial, which contradicts the simplicity of $\mathbf{S}_j'$. On the other hand, it is impossible that $z \notin C$. Indeed, $z\notin C$ implies that $C \subseteq A$, from which it follows that $C$ is a class of $\lambda$ restricted to $A$, which is the congruence $\eta$ of $\mathbf{A}$. Since $\eta$ is nontrivial ($C$ has more than one element), we deduce that either $\eta \subseteq \theta$, or $\eta$ is the full relation. Either one of these possibilities contradicts the assumption that $C$ can only intersect $\theta$-classes in at most one element. \qedhere
    \end{itemize}
\end{proof}

Hence, simple conservative minority algebras can possess subalgebras with a rich collection of congruences. Luckily, the congruences of subalgebras of conservative minority algebras which are not simple can be understood in terms of the congruences of subalgebras of the local algebras of a representing tree. We finish this section with some discussion of subalgebras of represented algebras and their congruences. Recall Lemma~\ref{lem:subalgebrarepresentingtree}, which states that given $\mathcal{T}$ a conservative minority tree and a subalgebra $\mathbf{D} \leq \mathbf{A}_\mathcal{T}$, the algebra $\mathbf{D}$ is represented by $\mathcal{T}_D$.

\begin{definition}\label{def:treevertexcongruences}
Let $\mathcal{T}$ be a conservative minority tree. We define the following two equivalence relations on $A_\mathcal{T}$.
\begin{enumerate}
\item 
For $v \in \mathcal{T}$, we define 
$
\sim_v \; := \{ (w,u) \in (A_\mathcal{T})^2: \text{ $w =u$ or both $ w \leq v$ and $u \leq v$ }\}.
$
\item 
For $v \in \mathcal{T}$ not a leaf, we  define 
$
\sim_{C_v} \; := \bigcup_{w \in C_v} \sim_w .
$
\end{enumerate}
\end{definition}

\begin{lemma}\label{lemma:congruencesoftreealgebras}
Let $\mathcal{T}$ be a conservative minority tree over $S = \{\mathbf{A}_1, \dots, \mathbf{A}_s\}$. The following hold.
\begin{enumerate}
\item 
For every vertex $v \in \mathcal{T}$, the equivalence relation 
$
\sim_v 
$
is a block congruence of $\fA_\mathcal{T}$.
\item 
For every $v \in \mathcal{T}$ not a leaf, the equivalence relation
$
\sim_{C_v}
$
is a congruence of $\mathbf{A}_\mathcal{T}$ and is equal to the join of the set 
$
\{ \sim_w: w \in C_v\}
$
in the congruence lattice of $\mathbf{A}_{\mathcal{T}}$
\end{enumerate}
\end{lemma}

\begin{proof}
We prove \emph{1.} first. It is clear that $\sim_v$ is an equivalence relation. Indeed, if there are distinct $a,b,c \in A_\mathcal{T}$ with $(a,b), (a,c) \in \; \sim_v$, then it is immediate that $(a,c) \in \; \sim_v$. Let us show that $\sim_v$ is preserved by $m_\mathcal{T}$. To do so, it suffices to show that all basic translations of $m_\mathcal{T}$ preserve $\sim_v$, i.e., those unary functions obtained from $m_\mathcal{T}$ by evaluating two arguments at constants.
Take $a,b,c,d \in A_{\mathcal{T}}$
so that $(a,b) \in \; \sim_v$. We argue that $(m_{\mathcal{T}}(a,c,d), m_{\mathcal{T}}(b,c,d)) \in \; \sim_v$. If $a=b$, there is nothing to prove, so we assume $a\neq b$. There are two cases to consider. 
\begin{itemize}
    \item Suppose that $c \leq v$ and $d \leq v$. It follows that $m_\mathcal{T}(a,c,d) \leq v$ and $m_\mathcal{T}(b,c,d) \leq v$, hence $(m_{\bT}(a,c,d), m_{\bT}(b,c,d)) \in \; \sim_v$.
    \item Suppose that at least one of $c$ or $d$ does not have $v$ as an upper bound. It follows that $\bigvee \{a,c,d\} = \bigvee \{b,c,d\} = \bigvee \{v,c,d\} =: w$. We have that $v \leq w$ and $v \neq w$, so $a^{w^+} = b^{w^+}$. In this case we have 
    \[
    m_{k}(a^{w^+}, c^{w^+}, d^{w^+}) = m_{k}(b^{w^+}, c^{w^+}, d^{w^+}),
    \]
    where $1\leq k \leq s$ is an index of a local algebra $\mathbf{A}_k$ such that $\mathbf{v}^+ \leq \mathbf{A}_k$.
    It follows from Definition~\ref{def:leafalgebra} that $(m_\mathcal{T}(a,c,d), m_\mathcal{T}(b,c,d)) \in \{(a,b),(c,c),(d,d)\} \subseteq \; \sim_v$.
\end{itemize}

The reasoning we just gave does not depend on the coordinates of $m_{\mathcal{T}}$ that are evaluated at constants, so we have shown that all basic translations of $m_{\mathcal{T}}$ preserve $\sim_v$. It follows that $m_{\mathcal{T}}$ preserves $\sim_v$. 

Now we prove that \emph{2.} holds. To see that $\sim_{C_v}$ is the least congruence of $\mathbf{A}_\mathcal{T}$ which contains each $\sim_w$, for $w \in C_v$. In general, the least congruence $\alpha$ containing a collection of congruences $\{\alpha_1, \dots, \alpha_l\}$ is given by the iterated relational composition
\[
\alpha = \bigcup_{i \in \mathbb{N}} (\alpha_1 \circ \dots \circ \alpha_l)^i.
\]
Because each of the $\sim_w$ is a block congruences whose nontrivial block is disjoint from the others, the above iterated relational composition cannot produce any new pairs other than those already identified by a particular nontrivial $\sim_w$-class and the result follows. 

\end{proof}
\begin{remark}
Throughout this article we include pictures of conservative minority trees (which are drawn as trees) and the algebras that they represent (which are drawn in the colloquial `potato diagram' style). The reader can consult Figure \ref{fig:treesandpotatoes} for an illustration of the relationship between the two kinds of pictures. Notice how the circles in the potato diagram correspond exactly to the nontrivial classes of the block congruences $\sim_v$ for a vertex $v$ that is not a leaf.
\end{remark}

\begin{figure}
    \centering
    \includegraphics[width=0.5\linewidth]{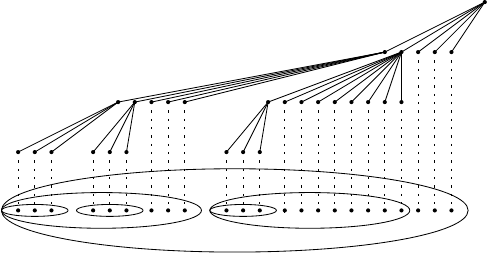}
    \caption{A tree and a potato diagram.}
    \label{fig:treesandpotatoes}
\end{figure}

\begin{lemma}\label{lem:subalgebrasuccessor}
    Let $\mathcal{T}$ be a conservative minority tree over $S = \{ \mathbf{A}_1, \dots, \mathbf{A}_S \}$. Let $\mathbf{D} \leq \mathbf{A}_\mathcal{T}$ be a subalgebra with underlying set $D \subseteq A_\mathcal{T}$. Let $d$ be the least upper bound of $D$ in $\mathcal{T}$, i.e.\ $d$ is the maximal prefix common to all elements of $D$. The function 
    \begin{align*}
    \phi: D &\to d^{+_{\mathcal{T}_D}}\\
     w &\mapsto w^{d^+}
    \end{align*}
    is a homomorphism from $\mathbf{D}$ onto $\mathbf{d^{+_{\mathcal{T}_D}}}$ and the kernel of $\phi$ is $\sim_{C_d}$ restricted to $D$.
\end{lemma}

\begin{proof}
    We show that $\phi$ is a homomorphism. Indeed, take $v,w,u \in D$. Since $d$ is an upper bound of $D$, it follows that $d$ is an upper bound of $z = \bigvee \{v, w,u\}$. Let $1\leq l \leq s$ be the index such that $\mathbf{d}^+ = \mathbf{A}_l = (A_l, m_l)$ and let $1\leq k \leq s$ be an index such that $\mathbf{z}^+ \subseteq \mathbf{A}_k = (A_k, m_k)$. There are two cases 
    \begin{itemize}
        \item Suppose that $z \neq d$, so $z$ is a proper prefix of $d$. Then $z^{d^+}=v^{d^+} = w^{d^+}=u^{d^+}$, hence $\phi$ maps $u,v,w$ to the same value $z^{d^+} \in D^{d^+}$. On the other hand, 
        $
        m_\mathcal{T}(v,w,u)^{d^+} = z^{d^+},
        $
         hence
        \[
        \phi(m_\mathcal{T}(v,w,u)) =  m_\mathcal{T}(v,w,u)^{d^+} = z^{d^+} = m_l( z^{d^+} , z^{d^+} , z^{d^+} )=
        m_l(\phi(v), \phi(w), \phi(u)).
        \]
        \item Suppose that $z=d$. In this case, the homomorphism property for $\phi$ follows directly from the definition of the minorty operation $m_\mathcal{T}$. 
    \end{itemize}
    That the kernel of $\phi$ is $\sim_{C_d}$ restricted to $D$ follows immediately from the definitions. 
\end{proof}
\begin{figure}
    \centering
    \includegraphics[width=.82\linewidth]{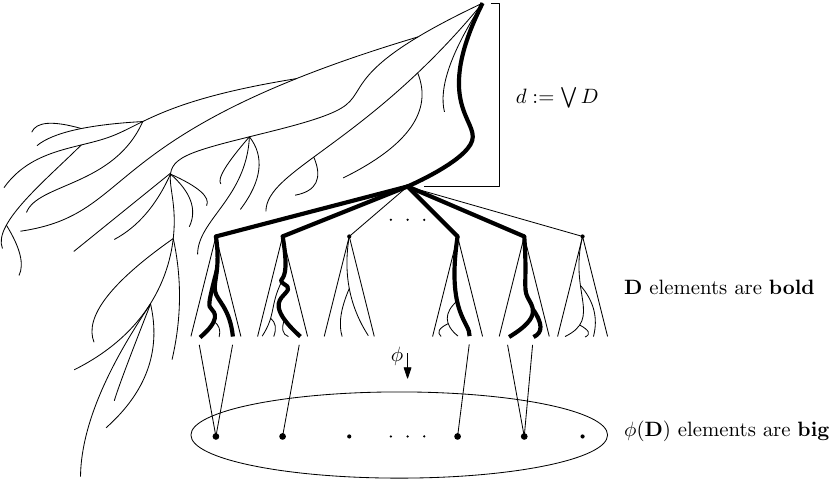}
    \caption{Subalgebra prefix successor map}
    \label{fig:prefixhom}
\end{figure}

We provide in Figure~\ref{fig:prefixhom} a general picture of the situation described by Lemma~\ref{lem:subalgebrasuccessor}. We call the map $\phi$ given in the lemma statement the \emph{subalgebra prefix successor map} for $\mathbf{D}$. The following technical lemma is crucially used in the proof of Lemma~\ref{lem:MutationsofBlockCongruencesAreInvariant}, which is applied to one of our main results, Theorem~\ref{thm:DeltaisMinionHom}.
\begin{lemma}\label{lem:CongruencesofSubalgebras}
    Let $\mathcal{T}$ be a conservative minority tree over a collection of (not necessarily simple) algebras $S = \{\mathbf{A}_1, \dots, \mathbf{A}_s \}$.
    Suppose that $\theta$ is an equivalence relation on a set $U \subseteq A_\mathcal{T}$ with exactly one nontrivial class $D$. Let
    \[
    d = \bigvee D
    \]
    be the maximal prefix common to all elements of $U$. Then $\theta$ is a block congruence of the subalgebra $\mathbf{U} \leq \mathbf{A}_\mathcal{T}$ with domain $U$ if and only if $d^{+_{\mathcal{T}_D}}$ is the nontrivial class of a block congruence $\overline{\theta}$ on the subalgebra $\mathbf{d^{+_{\mathcal{T}_U}}}\leq \mathbf{A}_l$, where $1\leq l \leq s$ is such that $\mathbf{u}^+ = \mathbf{A}_l$.

\end{lemma}

\begin{proof}
     Let $U' = \{ u\in U: u \leq_{\mathcal{T}} d \}$ be the domain of the subalgebra $\mathbf{U}' \leq \mathbf{A}_\mathcal{T}$ consisting of those leaves of $\mathcal{T}$ which have $d$ as a prefix. Then, the least upper bound of $U'$ is also $d$, since $D \subseteq U'$. Let $\theta'$ be the restriction of $\theta$ to $U'$. 
     Observe that $d^{+_{\mathcal{T}_D}}$ contains at least two elements, since otherwise $d$ is not the least upper bound of $D$. By Lemma~\ref{lem:subalgebrasuccessor}, the subalgebra prefix successor map
     \begin{align*}
    \phi: U' &\to d^{+_{\mathcal{T}_{U}}}\\
     w &\mapsto w^{d^+}
    \end{align*}
    is a homomorphism with kernel $\sim_{C_d}$ restricted to $U'$, which we will call $\sim_{C_u}'$. 
    
     We first prove the forward direction and suppose that $\theta$ is a block congruence of $\mathbf{U}$. Then $\theta'$ is a nontrivial block congruence of $\mathbf{U}'$ with nontrivial class $D$. We then define the congruence
    \[
    \overline{\theta} := \phi(\theta ' \vee \sim_{C_u}'),
    \]
    and observe that $\overline{\theta}$ is a block congruence of $\phi(\mathbf{U}') = {d^{+_{\mathcal{T}_U}}}$ whose nontrivial block is $\phi(U) = d^{+_{\mathcal{T}_D}}$.

    For the other direction, suppose that $\overline{\theta}$ is a congruence of $\mathbf{d}^{+_{\mathbf{T}_U}}$ with nontrivial class $d^{+_{\mathcal{T}_D}}$. Then $\phi^{-1}(\overline{\theta})$ is a congruence of $\mathbf{U}'$, hence it decomposes as the join of nontrivial block congruences, so in particular $D$ is the nontrivial class of a block congruence on $\mathbf{D'}$, since $D=\phi^{-1}(d^{+_{\mathcal{T}_D}})$. Since $U'$ is the nontrivial class of the congruence $\sim_u$ restricted to $\mathbf{D}$, it follows from Proposition~\ref{prop:blocksinsideblocks} that $D$ is also the nontrivial class of a block congruence on $\mathbf{U}$, which evidently is $\theta$.  
\end{proof}

\begin{corollary}\label{cor:blockcongofSimpleTreeAlgebrasAreSim}
    If $\mathcal{T}$ is a simple conservative minority tree, then every block congruence of $\mathbf{A}_\mathcal{T}$ is of the form $\sim_v$ for some $v \in \mathcal{T}$. 
\end{corollary}
\begin{proof}
    In view of Lemma~\ref{lem:CongruencesofSubalgebras}, this is immediate. Indeed, let $\theta$ be a block congruence on $\mathbf{A}_\mathcal{T}$. Let $D$ be the nontrivial $\theta$-class and let $d$ be the least upper bound of $D$ in $\mathcal{T}$. We conclude that $d^{+_{\mathbf{T}_D}}$ is the nontrivial class of a congruence of the successor algebra $\mathbf{d}^{+_\mathcal{T}}$. Since this successor algebra is assumed to be simple, it must be that $d^{+_{\mathbf{T}_D}} = d^{+_\mathcal{T}}$. Suppose towards a contradiction that $\sim_d$ is not equal to $\theta$. Then there exists $b \in \mathcal{T}$ such that $b \leq d$ but $b \notin D$. Since $d^{+_{\mathbf{T}_D}} = d^{+_\mathcal{T}}$, there exists $c \in D$ such that $z= c^{d^+} = b^{d^+}$. Consider now the vertex $v= d^\frown(z)$ and the congruence $\sim_v$. Now $c$ belongs to the nontrivial $\sim_v$-class and also $c$ belongs $D$, so it follows from Proposition~\ref{prop:blockintersection} that the nontrivial $\sim_v$-class is contained in $D$. This contradicts our choice of $b$. 
\end{proof}

The reader can consult Figure~\ref{fig:congruencesofsubalgebras} for an illustration of the main idea behind Lemma~\ref{lem:CongruencesofSubalgebras}. The same conservative minority tree picture that is used in Figure~\ref{fig:prefixhom} is used here, but now, although the least upper bound of $D$ is the same, the set $D$ is smaller. The branches leading to leaves which belong to $U$ are drawn with bold, while the elements of $D$ are drawn in heavier bold. The two elements of $\phi(U')$ which are circled are the elements of the nontrivial block of the congruence $\overline{\theta}$ of $\mathbf{d}^{+_{\mathcal{T}_U}}$.

So, the simplicity of all successor algebras of a conservative minority tree $\mathcal{T}$ need not imply the simplicity of the successor algebras of the conservative minority tree $\mathcal{T}_D$ which represents a subalgebra $\mathbf{D} \leq \fA_\mathcal{T}$. Indeed, there might exist a successor algebra $\mathbf{v}^{+_{\mathcal{T}_S}} \leq \mathbf{v}^{+_{\mathcal{T}}}$ which is no longer simple and this situation happens whenever $\mathbf{v}^{+_{\mathcal{T}}}$ is simple but not hereditarily simple. Lemma~\ref{lem:CongruencesofSubalgebras} provides a complete description of the interaction between congruences of subalgebras of $\mathbf{A}_\mathcal{T}$ and congruences of subalgebras of the local algebras $S$ of $\mathcal{T}$. 

The fact that successor algebra simplicity is not always inherited by subtrees highlights the complexity of the structures that we wish to understand and in general there is no hope to obtain from the congruences of an algebra a complete description of its invariant relations. However,  simplicity of successor algebras \emph{is} inherited by subtrees of conservative minority trees whose local algebras are all hereditarily simple. This follows directly from the definitions and we record this fact in the following remark. 

\begin{remark}\label{rem:subtreesofhsimpletreesarehsimple}
If  $\mathcal{T}$ is a hereditarily simple conservative minority tree, then so is $\mathcal{T}_D$, for any subalgebra $\mathbf{D} \leq \mathbf{A}_\mathcal{T}$.
\end{remark}

\begin{figure}
    \centering
    \includegraphics[width=0.64\linewidth]{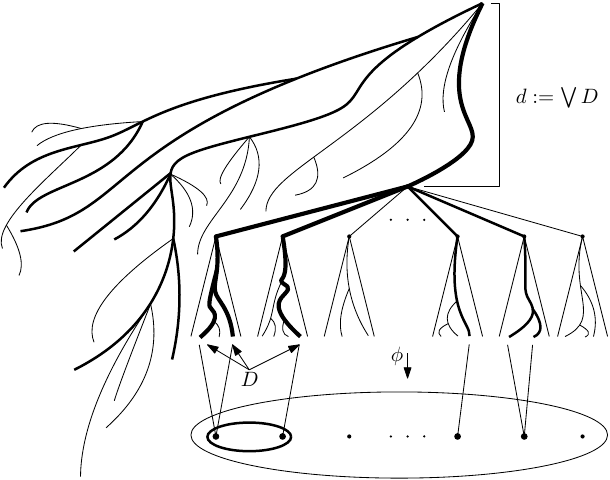}
    \caption{Block congruence of a subalgebra corresponds to block congruence of successor algebra subalgebra (cf.\ Lemma~\ref{lem:CongruencesofSubalgebras}).}
    \label{fig:congruencesofsubalgebras}
\end{figure}


\subsubsection{Saplings}\label{sec:saplings}

The subdirectly irreducible subalgebras of conservative minority algebras play a crucial role in this paper. Here we apply Proposition \ref{prop:subdirectlyirreduciblearelinearchains} to obtain a characterization of those subtrees of conservative minority trees which represent subdirectly irreducible subalgebras of the leaf algebra. We call a conservative minority tree $\mathcal{T}$ a \emph{sapling} if for all $v \in T$, there is at most one $w \in C_v$ so that $\{ u : u \leq_\mathcal{T} w \}$ is not a linearly ordered chain.
A \emph{sapling of $\mathcal{T}$} is a conservative minority subtree of $\mathcal{T}$ which is a sapling. If $\mathcal{T}$ is a conservative minority tree, we say that a sapling $\mathcal{T}_S$ of $\mathcal{T}$ is 
\begin{itemize}
    \item \emph{full} if every $\mathcal{T}_S$ is full in $\mathcal{T}$ at any vertex $v$ which satisfies $|v^{+_{\mathcal{T}_S}}| \geq 2$. 
\item \emph{maximal} if there is no 
sapling of $\mathcal{T}$ which has 
$\mathcal{T}_S$ as a proper subtree.
\end{itemize}
Note that for a sapling $\mathcal{T}_S$, the set of vertices $\{v: |v^{+_{\mathcal{T}_S}}| \geq 2\}$ is linearly ordered by $\leq_\mathcal{T}$. We call this linearly ordered set of vertices the \emph{trunk} of the sapling $\mathcal{T}_S$ and denote this set by $\Trunk(\mathcal{T}_S)$. When the sapling representing some subdirectly irreducible conservative minority algebra $\mathbf{A}$ is clear, we may simply write $\Trunk(\fA)$ instead. Note that the isomorphism class of the algebra represented by a full sapling is uniquely determined by its trunk. We define the \emph{height} of a subdirectly irreducible conservative minority algebra $\fA$ to be the size of the trunk of any tree representing it (note that this is an invariant of the algebra). If a subalgebra $\mathbf{S} \leq \mathbf{A}_\mathcal{T}$ is represented by a full sapling $\mathcal{T}_S$, we say that $\mathbf{S}$ is a \emph{full subalgebra}. An isomorphism graph $G \leq \mathbf{S}_1 \times \mathbf{S}_2$ is called a \emph{full isomorphism graph} if both $\mathbf{S}_1$ and $\mathbf{S}_2$ are full. 
Figure \ref{fig:saplings} provides some pictures of these definitions. 

\begin{remark}\label{rem:emptytrunk}
    We emphasize that a sapling $\mathcal{T}$ can have empty trunk, in which case its underlying tree structure is a linearly ordered chain. Obviously, a subdirectly irreducible algebra represented by a sapling with empty trunk is a trivial one element algebra. 
\end{remark}

\begin{figure}
    \centering
    \includegraphics[width=0.8\linewidth]{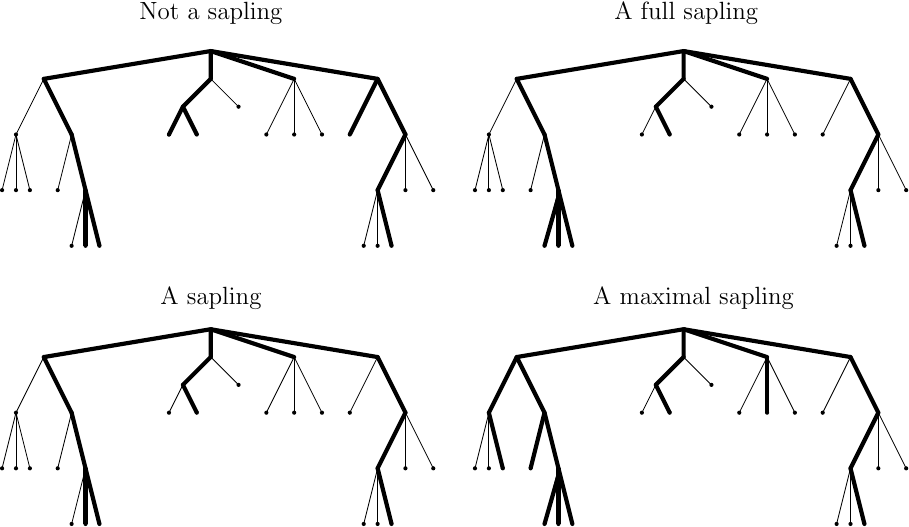}
    \caption{(Full and maximal) Saplings.}
    \label{fig:saplings}
\end{figure}

\begin{lemma}\label{lem:saplingsrepresentsubdirectirreducible}
Let $\mathcal{T}$ be a conservative minority tree, let $\mathbf{D} \leq \fA_\mathcal{T}$, and let $\mathcal{T}_D$ be the conservative minority subtree of $\mathcal{T}$ which represents $\mathbf{D}$. The following hold:
\begin{enumerate}
\item If $\mathbf{D}$ is subdirectly irreducible, then $\mathcal{T}_D$ is a sapling.
\item If $\mathcal{T}$ is a simple conservative minority tree, then every full sapling of $\mathcal{T}$  represents a subdirectly irreducible algebra.
\item If $\mathcal{T}$ is a hereditarily simple conservative minority tree, then every subtree of $\mathcal{T}$ is a hereditarily simple conservative minority tree and every sapling of $\mathcal{T}$  represents a subdirectly irreducible algebra. 
\end{enumerate}
\end{lemma}

\begin{proof}
    To establish \emph{1.}, suppose that $\mathcal{T}_D$ is not a sapling. Then there exists $v \in T_D$ with children $w_1, w_2$ so that the each of the sets $\{u : u \leq w_1 \}$ and $\{ u: u \leq w_2\}$ is not a linearly ordered chain. This implies that each  of $\sim_{w_1}$ and $\sim_{w_2}$ is a nontrivial block congruence of $\fA_{\mathcal{T}_D}$, but they are incomparable because $w_1$ and $w_2$ are incomparable.  Proposition \ref{prop:subdirectlyirreduciblearelinearchains} indicates that the congruences of $\mathbf{D}$ are block congruences that are linearly ordered by inclusion, so $\mathbf{D}$ is not subdirectly irreducible. 

    To establish \emph{2.}, note that if $\mathcal{T}$ is a simple conservative minority tree, then so is any full sapling $\mathcal{T}_S$, because by definition every successor algebra $\mathbf{v}^{+_{\mathcal{T}_S}}$ is either a simple one element algebra or the simple algebra $\mathbf{v}^{+_{\mathcal{T}}}$. By Corollary~\ref{cor:blockcongofSimpleTreeAlgebrasAreSim}, the block congruences of $\mathbf{S}$ are all of the form $\sim_v$ for $v \in T_S$. Note that $\sim_v$ is the equality relation if $\{u: u \leq v\}$ is a linearly ordered chain. Since $\mathcal{T}_S$ is a sapling, it follows that the set of congruences $\{\sim_v: v \in T_S\}$ is linearly ordered by inclusion. Proposition \ref{prop:subdirectlyirreduciblearelinearchains} now indicates that $\mathbf{S}$ is subdirectly irreducible. 

    To establish \emph{3.}, take a sapling $\mathcal{T}_D$ of $\mathcal{T}$. By Remark \ref{rem:subtreesofhsimpletreesarehsimple}, we know that $\mathcal{T}_D$ is a hereditarily simple conservative minority tree. Since $\mathcal{T}_D$ is a full sapling of itself, it follows from 2. above that $\fA_{\mathcal{T}_S}$ subdirectly irreducible.
\end{proof}

\subsubsection{Tree transformations}\label{sec:treetransformations}
In this section we describe some basic ways to locally manipulate trees. We can obtain a simple reduced tree representing a given conservative minority algebra $\fA$ 
from any other representing tree by applying a sequence of local transformations.

The first transformation is identifying an only child and its parent. Let $\mathcal{T} = (T, \leq, S)$ be a conservative minority tree and suppose that there exist vertices $v, w$ of $\mathcal{T}$ with $w$ the unique child of $v$, so that $w = v^\frown(a)$ and $v^{+_\mathcal{T}} = \{a\}$. Let 
\[
T^{v=w}: \{ u \in T: u \nleq w\} \cup \{ v^\frown r: v^\frown (a)^\frown r  \in T\}
\]
and the set $\mathcal{T}^{v=w} = (T^{v=w},\leq, S)$. 
Intuitively, we are deleting a trivial one element successor algebra with domain $v^+ = \{a\}$ which doesn't provide any new information. It is straightforward to check that $\fA_{\mathcal{T}^{v=w}} $ is isomorphic to $\fA_\mathcal{T}$.

The second transformation refines a given $\mathcal{T}$ by introducing additional vertices representing classes of a nontrivial proper congruence of a successor algebra. Suppose that $v$ is a vertex of $\mathcal{T}$ such that $v^{+_\mathcal{T}}$ is not a simple algebra. Let $\lambda$ be a nontrivial proper congruence of $\mathbf{v}^{+_\mathcal{T}}$ with classes $W_1, \dots, W_k$ for some $ k \geq 2$. We define a new set of tree vertices 
\[
T^{\lambda \to v} := \{ u \in T: u \nless v\} \cup \{ v^\frown(W_i)^\frown (a)^\frown z:  v^\frown (a)^\frown z \in T \text{ and } a \in W_i \},
\]
update $S$ to $S^{\lambda \to v} = S \cup \{\mathbf{v}^{+_\mathcal{T}}/ \lambda \} \cup \{\mathbf{W}_1, \dots, \mathbf{W}_k\}$, and then set $\mathcal{T}^{\lambda \to v}:= (T^{\lambda \to v} , \leq , S^{\lambda \to v})$. 
It is also a routine application of the definitions to check that this transformation also does not impact the isomorphism type of the represented algebra.  
\begin{remark}\label{rem:NormofSdoesNotIncrease}
We observe that the sizes of the new local algebras $\{\mathbf{v}^{+_\mathcal{T}}/ \lambda \} \cup \{\mathbf{W}_1, \dots, \mathbf{W}_k\}$ added in the second transformation above are all strictly less than $|\mathbf{v}^{+_\mathcal{T}}|$. This becomes important in one of our main results. 
\end{remark}

\subsection{Invariant relations}\label{sect{invariantrelations}}
\label{sect:inv-rels}
In this section we will produce for an arbitrary finite conservative minority algebra $\fA$ a  finite set of relations whose primitive positive closure equals $\Inv(\fA)$. 

\begin{remark}
Kearnes and Szendrei~\cite{KearnesSzendrei-Parallologram} apply some of the deeper aspects of the modular commutator theory to obtain a nice set of relations for an algebra with a parallelogram term, which is a generalization of a Maltsev operation which characterizes the property of having `few subpowers'~\cite{IMMVW}. 
In fact, it is possible to combine the Kearnes and Szendrei structure theorem with our analysis of the congruences of conservative minority algebras to obtain the relational basis that we present here. Indeed, this is how we first obtained these results. However, this line of reasoning relies on the Freese notion of a `joint similarity' and the complexity of this theory prohibits us from presenting a self-contained explanation. Therefore, to make the presentation as accessible as possible, we prefer to present a development which avoids the commutator and the notion of joint similarity, replacing these parts instead with a direct analysis of conservative minority algebras. We emphasize that commutator theory is `lurking in the background' throughout and that we will use important ideas from the Kearnes and Szendrei paper on critical relations. 
\end{remark}

\subsubsection{Rectangularity and the linkedness congruence}\label{sect:rectangularrelations}
Let $k \geq 1$ and $\fA_1,\dots,\fA_k$ be algebras with the same signature. 
Then a relation $R \leq \fA_1 \times \cdots \times \fA_k$ is called \emph{subdirect}
if $\pi_i(R) = A_i$ for every $i \in \{1,\dots,k\}$. 
A relation $R \subseteq A \times B$
is called \emph{rectangular} if
$(a,b) \in R$, $(c,b) \in R$, and 
$(c,d) \in R$ implies that $(a,d) \in R$. Note that if $R \leq \fA \times \fB$ and $\fA \times \fB$ contains a Maltsev term, then $R$ is rectangular. 
The following is well known and added for the convenience of the reader.

\begin{lemma}\label{lem:congr}
Let $\fA$ and $\fB$ be algebras with the same signature $\tau$ such that $\fA \times \fB$ is Maltsev and let 
$R \leq \fA \times \fB$ be sudirect. 
Then $C_1 := R \circ R^{-1}$ is a congruence of $\fA$ and $C_2 := R \circ R^{-1}$ is a congruence of $\fB$, and 
the algebras $\fA' := \fA/C_1$ and $\fB' := \fB/C_2$ are isomorphic. 

\end{lemma}
\begin{proof}
Clearly, $C_1$ is a subuniverse of $\fA^2$, is reflexive and symmetric, so we are left with showing transitivity. Let $a_1,a_2,a_3 \in A$ be such that
$(a_1,a_2) \in C_1$ and $(a_2,a_3) \in C_1$. By the definition of $C_1$, there are $b_1,b_2 \in B$ such that $(a_1,b_1),(a_2,b_1),(a_2,b_2),(a_2,b_3) \in R$. 
Then rectangularity implies that
$(a_1,b_2)\in R$, and hence
$(a_1,a_3) \in C_1$. The argument for $C_2$ is symmetric. 

    Let $a \in A$; since $R$ is subdirect, there exists $b \in B$ such that $(a,b)$. 
    Then the map $\eta$ that sends 
    the congruence class $[a]_{C_1}$
    to $[b]_{C_2}$ is well-defined, and an isomorphism between 
    $\fA' := \fA/C_1$ and $\fB' := \fB/C_2$. Indeed, if $a' \in A$ 
    is such that $(a,a') \in C_1$,
    and $b' \in B$ is such that 
    $(a',b') \in R$, 
    they by the rectangularity of $R$ we have $(a,b') \in R$,
    and hence $(b,b') \in C_2$, 
    so $\eta(a) = [b]_{C_2} = [b']_{C_2} = \eta(a')$. 
    Finally, since $R$ is a subalgebra 
of $\fA \times \fB$, for every $f \in \tau$ of arity $k$ we have 
that $\eta(f^{\fA}(a_1,\dots,a_k)) = f^{\fB}(\eta(a_1),\dots,\eta(a_2))$,
and hence $\eta$ is an isomorphism between $\fA$ and $\fB$. 
\end{proof}

\subsubsection{Critical relations}
Critical relations have been introduced by Kearnes and Szendrei~\cite{KearnesSzendrei-Parallologram} and have been used by Zhuk~\cite{ZhukCritical}. 
The definition given below (unlike~\cite{KearnesSzendrei-Parallologram}, we assume that the underlying algebras is finite) is well-known to be equivalent to the one given in~\cite{KearnesSzendrei-Parallologram} (the equivalence follows from Lemma~\ref{lem:criticaltuple} below). 

\begin{definition}\label{def:criticalrelation}
    Let $\fA$ be a finite algebra and let $R \leq \fA^n$ be an element of $\Inv(\fA)$. 
 Then $R$ is called \emph{critical} if it has no dummy coordinates and is meet irreducible in the lattice of subalgebras of $\fA^n$. 
\end{definition}
Note that the set of critical relations of an algebra $\fA$ primitively positively define all of $\Inv(\fA)$.


\begin{lemma}[c.f.\ Lemma 2.1 of \cite{KearnesSzendrei-Parallologram}]\label{lem:criticaltuple}
Let $\fA$ be a finite algebra and let $R \in \Inv(\fA)$ be of arity $n \geq 2$. Then $R$ is critical if and only if 
there exists a tuple $s \in A^n$ (called a \emph{critical tuple}) such that 
\begin{enumerate}
    \item $R$ is maximal among the $n$-ary members of $\Inv(\fA)$ such that $s \notin R$, and
    \item for every $1 \leq i \leq n$, there exists a tuple $t \in R$ such that $s_j = t_j$ for all $ j\neq i$ (we call such a tuple $t$ an \emph{$i$-approximation} of $s$).
\end{enumerate}
Furthermore, if $R$ is critical and $R^*$ is its unique upper cover
in the lattice of $n$-ary members of $\Inv(\fA)$, then every tuple $s \in R^* \setminus R$ has the above two properties. 
\end{lemma}
\begin{proof}
    Suppose that $R$ is critical and let $R^*$ be its unique upper cover among the $n$-ary members of $\Inv(\fA)$. Choose $s \in R^* \setminus R$. Obviously, $R$ is maximal with respect to the property that $s \notin R$. Let us show that $R$ contains a $1$-approximation to $s$. The arguments for the other coordinates is similar. We claim that $\pi_{2, \dots, n} (R) = \pi_{2, \dots, n} (R^*)$. One containment is obvious, since $R \subseteq R^*$. To see the other containment, notice that $R^* \subseteq A \times \pi_{2, \dots, n} (R)$, because the first coordinate of $ A \times \pi_{2, \dots, n} (R)$ is dummy. From this we obtain that $\pi_{2, \dots, n}(R^*) \subseteq \pi_{2, \dots, n}(\fA \times \pi_{2, \dots, n} (R)) = \pi_{2, \dots, n} (R)$. It now follows that there exists a tuple $(t_1, t_2, \dots, t_n) \in R$ so that $(t_2, \dots, t_n) = (s_2, \dots, s_n)$. Observe that we have proven also the final statement of the lemma. 

    We now show that $R$ is critical when \emph{1.}\ and \emph{2.}\ are satisfied for some tuple $s \in A^n$. Indeed, suppose that $R$ is meet reducible, so that there exist $R_1$ and $R_2$ that each properly contain $R$ and $R = R_1 \cap R_2$. Since $R$ is maximal with respect to omitting $s$, it follows that $s$ belongs to $R_1$ and $R_2$, but then $s \in R$, which is a contradiction. Therefore, $R$ is meet irreducible. Suppose that the first coordinate of $R$ is dummy, i.e., $R = A \times \pi_{2, \dots, n} (R)$. We assume that there exists $(t_1, s_2, \dots, s_n) \in R$, but this is impossible because we obtain $(s_1, s_2, \dots, s_n) \in R$. This argument works for any coordinate $1 \leq i \leq n$, so we conclude that $R$ has no dummy coordinates. Thus, the lemma is proved.
\end{proof}

\subsubsection{Multisorted critical relations}
Notice that the definition of an $n$-ary critical relation of an algebra $\fA$ refers to the subalgebra lattice of $\fA^n$. It will be convenient to relax this assumption. In view of Lemma~\ref{lem:criticaltuple}, we make the following definition. 

\begin{definition}\label{def:multisortedcritical}
    Let $\fA_1, \dots, \fA_n$ be algebras of the same signature. A relation $R \leq \fA_1 \times \dots \times \fA_n$ is called \emph{multisorted critical} if there exists a tuple $s \in A_1 \times \dots \times A_n$ (again called the critical tuple) such that 

    \begin{enumerate}
    \item $R$ is maximal among the subalgebras of $\fA_1 \times \dots \times \fA_n$ with respect to the property that $s \notin R$, and
    \item for every $1 \leq i \leq n$, there exists a tuple $t \in R$ such that $s_j = t_j$ for all $ j\neq i$ (we again call such a tuple $t$ an \emph{$i$-approximation} of $s$).
\end{enumerate}
\end{definition}

\begin{definition}
A relation $R \subseteq A^n$ is called \emph{functional} if 
for every $i \in \{1,\dots,n\}$ we have 
$a_i = b_i$ whenever 
\begin{align*}
    (a_1, \dots, a_{i-1}, a_i, a_{i+1}, \dots, a_n) & \in R \\
    \text{and } (a_1, \dots, a_{i-1}, b_i, a_{i+1}, \dots, a_n) & \in R. 
    \end{align*}
\end{definition}

\begin{lemma}\label{criticalfunctionalfactorsareSI}
    Let $\fA_1, \dots, \fA_n$ be finite conservative minority algebras and let $R \leq \fA_1 \times \dots \times \fA_n$ be a subdirect, functional, and multisorted critical relation of arity $n \geq 2$. Then $ \fA_i$ is subdirectly irreducible for every $i \in \{1,\dots, n\}$. Furthermore, if $s$ is a critical tuple for $R$ and $t$ is an $i$-approximation of $s$, then the monolith $\mu_i$ of $\fA_i$ equals the congruence generated by the pair $(s_i, t_i)$. 
\end{lemma}

\begin{proof}
    We will argue that $\fA_1$ is subdirectly irreducibly. The argument is analogous for the other coordinates. Let $s = (s_1, \dots, s_n) $ be a critical tuple for $R$ and let $t = (t_1, s_2, \dots, s_n) \in R$ be a $1$-approximation of $s$. Let $\psi$ be a nontrivial congruence of $\fA_1$. We will show that $(s_1, t_1) \in \psi$, which will establish that $\fA_1$ is subdirectly irreducible, since it implies that every nontrivial congruence of $\fA_1$ contains the congruence generated by $(s_1, t_1)$.

    Let $\psi' = \psi \times 0 \times \dots \times 0 $ be the congruence of $\fA_1 \times \dots \times \fA_n$ which identifies tuples $(x_1, \dots, x_n)$ and $(y_1, \dots, y_n)$ if and only if $(x_1, y_1) \in \psi $ and $x_i = y_i $ for $2 \leq i \leq n$. Let $$R^{\psi'} := \bigcup \{ [(x_1, \dots, x_n)]_{\psi'}: (x_1, \dots, x_n) \in R \} $$ be the union of $\psi'$-classes which have nonempty intersection with $R$. Since $\psi'$ is a congruence, its equivalence classes are compatible with the conservative minority operation for $\fA_1 \times \dots \times \fA_n$, so it follows that $R^{\psi'} \leq \fA_1 \times \dots \times \fA_n$. It is clear that $R \subseteq R^{\psi'}$. We claim that the containment is proper. Indeed, since $\psi$ is a nontrivial congrence, there exists $(a_1, b_1) \in \psi$ with $a_1 \neq b_1$. Since $R$ is subdirect in $\fA_1 \times \dots \times \fA_n$, there is $(a_2,\dots,a_n) \in A_2 \times \cdots \times A_n$ such that
    $(a_1, a_2, \dots, a_n) \in R$. It follows from the definitions of $\psi'$ and $R^{\psi'}$ that $(b_1, a_2, \dots, a_n) \in R^{\psi'}$. If $R = R^{\psi'}$, then $a_1 = b_1$, since $R$ is functional, but this contradicts the choice of $a_1$ and $b_1$. 

    It follows that $s \in R^{\psi}$. Hence, there exists $(c_1, \dots, c_n) \in R$ that is $\psi$-related to $(s_1, \dots, s_n)$, i.e., $c_i = s_i$ for $2 \leq i \leq n$. Since $R$ is functional, we deduce that $c_1 = t_1$. We conclude that $(s_1, t_1) \in \psi$, which is what we wanted to show. 
\end{proof}

For any subdirect $R \leq \fA_1 \times \dots \times \fA_n$ and $1\leq i \leq n$, we may consider $R$ as a binary relation $R^i \leq \pi_i(R) \times \pi_{1,\dots,i-1,i+1,\dots,k}(R).$ We then define the relation 
\[
\theta_i = R^i \circ (R^i)^{-1},
\]
which is a congruence of $\fA_i$ by Lemma \ref{lem:congr}. Following Kearnes and Szendrei in~\cite{KearnesSzendrei-Parallologram}, we call $\theta_i$ the \emph{$i$-th coordinate kernel (of $R$)}. 
 Let $\fA_1,\dots,\fA_n$ be algebras and $\theta_i$ a congruence of $\fA_i$, for every $i \in \{1,\dots,n\}$. Then $\theta := \theta_1 \times \cdots \times \theta_n$ denotes the binary relation $$\{((a_1,\dots,a_n),(b_1,\dots,b_n)) \mid (a_1,b_1) \in \theta_1, \dots, (a_n,b_n) \in \theta_k\}.$$ Note that $\theta$ is a congruence of $\fA_1 \times \cdots \times \fA_n$.
 We write $\eta_\theta \colon \fA_1 \times \dots \times \fA_n \to (\fA_1 \times \dots \times \fA_n) / \theta$ for  the natural map which maps a tuple to the $\theta$-class containing it. 

\begin{lemma}\label{lemma:relationsaresaturated}
Let $\fA$ be a finite conservative minority algebra and let $R \leq \fA^n$ for $n \geq 2$. Let $\fA_i = \pi_i(R)$ for $1 \leq i \leq n$; so in particular $R \leq \fA_1 \times \dots \times \fA_n$ is subdirect. Let $\theta = \theta_1 \times \dots \times \theta_n$ be the product of the coordinate kernels of $R$. Then the following hold.
\begin{enumerate}
    \item If $((a_1, \dots, a_k),(b_1, \dots, b_n)) \in \theta$, then $(a_1, \dots, a_n) \in R$ implies $(b_1, \dots, b_n ) \in R$. In particular, $R = \eta_{\theta}^{-1}(\eta_{\theta}(R))$.
    \item $\eta_\theta(R)$ is functional and subdirect.
    \item If $R$ is a critical relation, the $\eta_\theta(R)$ is a multisorted critical relation.
\end{enumerate}

\end{lemma}

\begin{proof}

We prove the items in order.
The proof of \emph{1.}\ proceeds by an induction on the number of coordinates in which $(a_1, \dots, a_n)$ and $(b_1, \dots, b_n)$ differ. The basis of the induction is trivial, because if the two tuples do not differ at any coordinates, then they are equal, hence $(b_1, \dots, b_n)= (a_1, \dots, a_n) \in R$. 
For the inductive step, suppose that $(a_1, \dots, a_n)$ and $(b_1, \dots, b_n)$ differ at $(i+1)$-many coordinates for some $ 0 \leq i < n$. Without loss of generality, we may assume that $a_1 \neq b_1$. We assume that $((a_1, \dots, a_n), (b_1, \dots, b_n)) \in \theta$, so by definition we have that $(a_1, b_1) \in \theta_1$. By the definition of $\theta_1$, there exists a tuple $(c_2, \dots, c_n) \in A_2 \times \dots \times A_n$ so that $(a_1, c_2, \dots, c_n) \in R$ and $(b_1, c_2, \dots, c_n) \in R$. Let $m$ be the Maltsev operation of the product $\fA_1 \times \dots \times \fA_n$. We have that
\[
m
\left(
(a_1, a_2, a_3, \dots, a_n),
(a_1, c_2, c_3, \dots, c_n),
(b_1, c_2, c_3, \dots, c_n)
\right) = (b_1, a_2, \dots, a_n) 
\in R,
\]
because each of the three input tuples belongs to $R$. 

Now $(b_1,a_2, \dots, a_n)$ and $(b_1, \dots, b_n)$
differ in $i$-many coordinates, so the induction hypothesis applies. We conclude that $(b_1, \dots, b_n) \in R$. It is now easy to see that $R = \eta_\theta^{-1}(R)$.

To prove \emph{2.}, let $\overline{R} = \eta_\theta(R)$. Suppose that $([a_1]_{\theta_1}, \dots, [a_n]_{\theta_n}) \in \overline{R}$. We argue that \[ ([b_1]_{\theta_1}, [a_2]_{\theta_2} \dots, [a_k]_{\theta_n}) \in \overline{R} \] implies $[b_1]_{\theta_1} = [a_1]_{\theta_1}$. 
Let $(a_1, \dots, a_n), (b_1, \dots, b_n) \in R$ be such that $(a_i,b_i) \in \theta_i$ for all $2 \leq i \leq n$. 
We have to show that $(a_1, b_1) \in \theta_1$. Indeed, the pair $((a_1, a_2, \dots, a_n), (a_1, b_2, \dots, b_n))$ satisfies the hypothesis of \emph{1.}, so we conclude that $(a_1, b_2, \dots, b_k) \in R$. Now the pair $((a_1, b_2, \dots, b_n), (b_1, b_2, \dots, b_n))$ witnesses that $(a_1, b_1)$ is contained in the first coordinate kernel $\theta_1$. The above reasoning does not depend on which projection onto $(n-1)$-many coordinates is chosen. It is obvious that $\eta_\theta(R)$ is subdirect, hence the proof of \emph{2.} is finished.

Now we prove \emph{3.} of the lemma. Let $R^*$ be the upper cover of $R$ in the lattice of subalgebras of $\fA^n$, let $s \in R^* \setminus R$ be a critical tuple for $R$, let $\theta= \theta_1 \times \dots \times \theta_n$ be the product of the coordinate kernels for $R$, let $\eta_\theta$ be the natural map, let $\overline{R} = \eta_{\theta}(R)$ be the reduced representation of $R$, and let $\overline{R^*} = \eta_\theta(R^*)$. In the proof of Lemma \ref{lem:criticaltuple} we argued that $\pi_{1, \dots,i-1, i+1, \dots n}(R) = \pi_{1, \dots,i-1, i+1, \dots n}(R^*)$, so it follows that $\fA_i = \pi_i(R) = \pi_i(R^*)$ for all $1 \leq i \leq n$. Hence, $R^* \leq \fA_1 \times \dots \times \fA_n$, so in particular $\eta_\theta$ is defined for every element of $R^*$.

We argue that $\overline{R}$ is maximal among relations $S \leq (\fA_1 /\theta_1) \times \dots \times (\fA_n/ \theta_n)$ with respect to omitting $\eta_\theta(s)$ and that for each $i$-approximation $t$ of $s$ for $R$, the tuple $\eta_\theta(t)$ is an $i$-approximation to $\eta_\theta(s)$ for $\overline{R}$. First, we establish that $\eta_\theta(s) \notin \overline{R}$. Indeed, by Lemma ~\ref{lemma:relationsaresaturated}, $R = \eta_\theta^{-1}(\overline{R})$, so if $\eta_\theta(s) \in \overline{R}$, then $s \in R$, which is false. For $t \in R$ an $i$-approximation of $s$, it is immediate that $\eta_\theta(t) \in \overline{R}$ and is an $i$-approximation to $\eta_\theta(s)$. Suppose that there is $S \leq (\fA_1 / \theta_1) \times \dots \times (\fA_n / \theta_n)$ such that $\overline{R} \subseteq S$ and $\eta_\theta(s) \notin S$. Then $s \notin \eta_\theta^{-1}(S)$, but $R \subseteq \eta_\theta^{-1}(S)$, hence $R = \eta_\theta^{-1}(S)$. It follows that $\overline{R} = S$, hence $\overline{R}$ is maximal among subalgebras of $(\fA_1 /\theta_1) \times \dots \times (\theta_n / \theta_n)$ with respect to omitting $\eta_\theta(s)$. Hence, any $S$ that properly contains $\overline{R}$ must contain $\eta_\theta(s)$ and so for such $S$ we obtain that $R^* \subseteq \nu_\theta^{-1}(S)$, from which it follows that $\overline{R^*} =\eta_\theta(R^*)$ is the upper cover of $\overline{R}$ among subalgebras of $(\fA_1 /\theta_1) \times \dots \times (\fA_n / \theta_n)$.
\end{proof}

\subsubsection{Reduced representations}
We call the image of $R$ under $\eta_\theta$ its \emph{reduced representation} and denote this image by $\overline{R}$. Next, we argue that for each $1 \leq i \leq n$ there exists a subalgebra $\fB_i \leq \fA_i$ that is isomorphic to $\fA_i / \theta_i$ (this relies on the conservativity of $\fA$, since for an arbitrary algebra it need not hold that quotients of subalgebras are isomorphic to some subalgebra). This allows us to encode each natural map $\eta_{\theta_i} \colon \fA_i \to \fA_i / \theta_i$ as a binary relation $S_i \leq \fA_i \times \fB_i$, which will in turn allow us to recover a relation $R$ from its reduced representation $\overline{R}$.

\begin{definition}\label{def:transversalendograph}
   Let $\fA$ be a conservative minority algebra and let $\fB \leq \fA$. We say that $S \leq \fA \times \fB$ is a \emph{transversal endomorphism graph} if $S$ is the graph of an endomorphism from $\fA$ onto $\fB$ with the property that $S \subseteq \theta$, where $\theta = S \circ S^{-1}$ (i.e., $\theta$ is the kernel of the endomorphism associated to $S$). 
\end{definition}

\begin{lemma}\label{lem:quotientsaresubalgebras}
Let $f \colon A^3 \to A$ be a conservative minority operation, let $\fA = (A; f)$, and let $\theta$ be a congruence of $\fA$. Then there exists a subalgebra $\fB \leq \fA$ such that $\fA/ \theta \simeq \fB$. In this situation, the natural map $\eta \colon \fA \to \fA/ \theta$ can be factored into the composition of an isomorphism with an endomorphism of $\fA$, the graph of which is a transversal endomorphism graph $S \leq \fA \times \fB$. 
\end{lemma}

\begin{proof}
    Let $B \subseteq A$ be any \emph{transversal} of $\theta$,
    i.e., a set of representatives of the classes of $\theta$. Then $\fB = (B; f)$ is a subalgebra of $\fA$, because $\fA$ is conservative. It is easy to see that the function which maps each equivalence class of $\fA / \theta $ to its unique representative in $\fB$ is an isomorphism and we denote this isomorphism by $\nu$. We define the binary relation $S \leq \fA \times \fB$ to be the graph of $\mu := \nu \circ \eta$. 
    So $\eta = \nu^{-1} \circ \mu$, as we wanted to show.
\end{proof}

\begin{lemma}\label{lem:connectingreducedreptoR}
Let $\fA = (A; f)$ be a finite conservative minority operation, let $\fR \leq \fA^n$, let $\fA_i = \pi_i(R)$ for each $1\leq i \leq n$, and let $\theta_i$ be the $i$-th coordinate kernel for $R$ for each $1\leq i \leq n$. For each $1\leq i \leq n$, let $\fB_i \leq \fA_i$ be a subalgebra of representatives of the $\theta_i$-classes. Let $\nu_i \colon \fA_i/ \theta_i \to \fB_i$ be the canonical isomorphism mapping each $\theta_i$-class to its representative in $\fB_i$, let $\nu = \nu_1 \times \dots \times \nu_n$ and let $\widetilde{R} = \nu(\overline{R})$.   Let $S_i \leq \fA_i \times \fB_i$ be the transversal endomorphism graph $\mu_i = \nu_i \circ \eta_{\theta_i}$. Then the  following hold.
\begin{enumerate}
    \item $\overline{\fR} \simeq \widetilde{\fR}$. 
    \item $R$ is defined by the formula
    $
    \exists y_1, \dots, y_n
    \left(
    \widetilde{R}(y_1, \dots, y_n) \wedge
    \bigwedge_{1\leq i \leq n}S_i(x_i,y_i)
    \right). 
    $
    \item If $R$ is a critical relation, then $\widetilde{R}$ is a subdirect, functional, multisorted critical relation, and each $\fB_i$ is subdirectly irreducible.
\end{enumerate}
\end{lemma}

\begin{proof}
    Obviously the function $\nu$
    is an isomorphism between $(\fA_1 \times \cdots \times \fA_n)/(\theta_1 \times \cdots \times \theta_n)$ and $\fB_1 \times \dots \times \fB_n$. This establishes \emph{1.}\ of the lemma. To see that the formula given in \emph{2.}\ defines $R$, notice that it follows from Lemma \ref{lemma:relationsaresaturated} that $R = \eta^{-1}\circ \eta (R)$. Since $\nu$ is an isomorphism, we deduce that
    \[R = \eta^{-1}\circ \eta (R) = (\nu \circ \eta)^{-1} \circ (\nu \circ \eta)(R)  = (\nu \circ \eta)^{-1}(\nu(\bar{R})) = (\nu \circ \eta)^{-1}(\widetilde{R}). 
    \]
    Now the tuples $(a_1, \dots, a_n)$ that belong to the set $(\eta \circ \nu)^{-1}(\widetilde{R})$ are easily seen to be exactly the tuples that satisfy the formula given in the lemma statement. For Item \emph{3.},
    note that if $R$ is critical, then $\widetilde{R} = \nu(\eta_{\theta}(R))$ is subdirect, functional, and multisorted critical by Lemma~\ref{lemma:relationsaresaturated}
    and the fact that $\nu$ is an isomorphism, 
    and in this case $\fB_1,\dots,\fB_n$ are subdirectly irreducible by Lemma~\ref{criticalfunctionalfactorsareSI}.  
\end{proof}

\begin{lemma}\label{lem:preliminarybasis}
    Let $\fA = (A; m)$ be a finite conservative minority algebra. The following are equivalent.
    \begin{enumerate}
        \item $U \in \Inv(\fA)$, and 
        \item $U$ has a primitive positive definition in the following set of relations:
        \begin{enumerate}
            \item all unary relations $X \subseteq A$,
            \item all subdirect transversal endomorphism graphs $S \leq \fA' \times \fB$ for $\mathbf{B} \leq \mathbf{A}' \leq \mathbf{A}$ and $\mathbf{B}$ subdirectly irreducible, and
            \item subdirect functional multisorted critical relations $T \leq \fA_1 \times \dots \times \fA_n$, for $n \geq 2$ and subdirectly irreducible $\fA_1, \dots, \fA_n \leq \fA$.
        \end{enumerate}
    \end{enumerate}
\end{lemma}

\begin{proof}
    The direction \emph{2.}\ implies \emph{1.}\ is trivial, since $m$ preserves any relation that has a primitive positive definition in a set of its invariant relations. It preserves the relations in \emph{(a)}, because it is conservative, and it preserves 
    the relations in \emph{(b)} and \emph{(c)} by definition. 
    
    For the nontrivial direction \emph{1.}\ implies \emph{2.}\ it suffices to show that every critical relation $R \in \Inv(\fA)$ is definable from the listed relations. This is obvious if $R$ is unary. If the arity of $R$ is $n \geq 2$, let $\fA_i = \pi_i(R)$ for $1 \leq i \leq n$, and let $\theta = \theta_1 \times \dots \times \theta_n$ be the product of the coordinate kernels of $R$. For each $1 \leq i \leq n$, choose a subalgebra $\mathbf{B}_i \leq \fA_i$ of representatives of $\theta_i$-classes and let $S_i \leq \fA_i \times \mathbf{B}_i$ be the transversal endomorphism graph which maps each element to the representative of its $\theta_i$-class. Lemma~\ref{lem:connectingreducedreptoR} applies to exactly such a collection of data, hence we obtain 
    from \emph{2.} of Lemma~\ref{lem:connectingreducedreptoR} that $R$ can be defined primitive positively from $\widetilde{R}$ and the transversal endomorphism graphs $S_1, \dots, S_n$, which are each listed above in \emph{(b)}. Since $R$ is critical, it follows from 
    \emph{3.}\ of Lemma~\ref{lem:connectingreducedreptoR} that 
    $\widetilde{R}$ is a subdirect, functional,  multisorted critical relation 
    and that each $\mathbf{B}_i$ is subdirectly irreducible. 
    Therefore, $\widetilde{R}$ is listed above in \emph{(c)}.
\end{proof}

\subsubsection{Characterizing subdirect functional multisorted critical relations}
In Lemma~\ref{lem:preliminarybasis} we gave a description of the invariant relations of a finite conservative minority algebra in terms of a finite number of invariant relations and additionally subdirect functional multisorted critical relations.
In this section we further analyze subdirect functional multisorted critical relations to show that overall a finite number of invariant relations primitively positively defines all others.

\begin{lemma}\label{lem:criticalmonolithblockssizetwo}
    Let $\fA = (A,f)$ be a conservative minority algebra, let $\fA_1, \dots, \fA_n \leq \fA$ be subdirectly irreducible for $n \geq 3$, and let $R \leq \fA_1 \times \dots \times \fA_n$ be a subdirect functional multisorted critical relation for $n \geq 3$. For each $1\leq i \leq n$, let $B_i$ be the nontrivial block of the monolith $\mu_i$ of $\fA_i$. Then $|B_i| = 2$, for each $1 \leq i \leq n$.
\end{lemma}

\begin{proof}
Let $s \in \fA_1 \times \dots \times \fA_n$ be a critical tuple for $R$ (see Lemma~\ref{lem:criticaltuple}). We showed in the proof of Lemma~\ref{criticalfunctionalfactorsareSI} that the pair $(s_i, t_i)$ generates $\mu_i$, for any $i$-approximation of $s$. Since $\mu_i$ is a block congruence, this forces $s_i, t_i \in B_i$, because $B_i$ is the only nontrivial class of $\mu_i$. Therefore, we label the elements of the factors so that $(s_1, \dots, s_n) = (0_1, \dots, 0_n)$ 
is the critical tuple and $(0_1, \dots, 0_{i-1}, 1_i, 0_{i+1}, \dots, 0_n)$ is the $i$-approximation to $s$, for each $1 \leq i \leq n$. Under this labeling convention,
the algebra $\fA_1 \times \dots \times \fA_n$ restricted to $\{0_1,1_1\} \times \dots \times \{0_n, 1_n\}$ becomes a $\mathbb{Z}_2$-vector space, with the Maltsev operation on $\fA_1 \times \dots \times \fA_n$ equal to $f(x,y,z) = x-y+z$ on the restriction. The critical tuple is the origin of this vector space and the $i$-approximations are standard basis vectors. It is easy to see now that $R$ must contain the affine subspace spanned by the $i$-approximations, which is the set $L$ of solutions to the $\mathbb{Z}_2$-linear equation $x_1 + \dots + x_n = 1$.

Our goal is to show that each $B_i = \{0_i,1_i\}$. If this is not the case, then after reordering coordinates we may assume that $|B_1| \geq 3$. Notice that there must exist $b \in B_1 \setminus \{0_1,1_1\}$ so that $f(u,v,w) \in \{0_1,1_1\}$ for distinct $u,v,w \in \{0_1,1_1,b\}$, since otherwise $\{0_1,1_1\}$ is a nontrivial congruence of $\fA$ restricted to $B_1$, which in view of Lemma \ref{prop:blocksinsideblocks} contradicts that $B_1$ is a minimal nontrivial block of a congruence of $\fA$. Going forward, we suppose without loss of generality that $f(0_1,1_1,b) = 1_1$. The arguments for all other cases are similar.

Let $S = \pi_{1,3, \dots, n}(R)$ be the projection of $R$ onto all coordinates except the second and let $S' = \{ (x_1, (x_3, \dots, x_n)) : (x_1, \dots, x_n)  \in S \}$ be the interpretation of $S$ as a binary relation between its projection onto the first coordinate and its projection onto the remaining coordinates. Since $R$ contains the solutions to the equation $x_1 + \dots + x_n = 1$, it follows that $S$ restricted to $B_1 \times B_3 \times \dots \times B_n$ contains $\{0_1,1_1\} \times \{0_3, 1_3 \} \times \dots \times \{0_n, 1_n \}$. By Lemma \ref{lem:congr}, the relation $\alpha = S' \circ (S')^{-1}$ of $\fA_1$ is a congruence and it follows from what we just observed that this congruence is not the equality relation. Therefore, all elements of $B_1$ are collapsed by $\alpha$, so $(0_1,b) \in \alpha$. By definition of $\alpha$, there exist tuples $(0_1,b_3, \dots, b_n), (b, b_3, \dots, b_n) \in S$. Since the tuple $(0_1,0_3, \dots, 0_{n}) \in S$, we obtain that from the rectangularity of $S'$ that $(b, 0_3, \dots, 0_{n}, ) \in S$. By definition of $S$, this means there exists $c \in A_2$ with $(b,c,0_3, \dots, 0_{n}) \in R$. Notice that $c \notin \{0_2,1_2\}$, because otherwise the functionality of $R$ would force $b \in \{0_1,1_1\}$. Since $(1_1, 0_2, 0_3,  \dots, 0_n)$ and $(0_1,0_2,1_3,0_4, \dots, 0_n) \in R $ and $f$ preserves $R$, we deduce that $(1_1,c,1_3, 0_4 \dots, 0_n) \in R$. This contradicts that $R$ is functional, because $(1_1,1_2,1_3, 0_4\dots, 0_n)$ already belongs to $R$. 
\end{proof}

\begin{lemma}\label{lem:tuplesdontcrossminimalboundary}
    Let $\fA$ be a finite conservative minority algebra and let $\fA_1, \dots, \fA_n \leq \fA$ be subdirectly irreducible subalgebras of $\fA$. For each $i \in \{1,\dots,n\}$, let $\mu_i$ be the monolith of $\fA_i$ and suppose that the nontrivial block of each $\mu_i$ has exactly two elements $a_i$ and $b_i$. Suppose that $R \leq \fA_1 \times \dots \times \fA_n$ satisfies the following properties:
    \begin{enumerate}
    \item $R$ is functional, and
    \item 
    $R$ restricted to $\mu_1 \times \dots \times \mu_n$ is the solution set to the $\mathbb{Z}_2$-linear equation $x_1 + \dots + x_n = 1$, for some labeling of the elements of the nontrivial block of $\mu_i = \{a_i, b_i\}$ with $0$ and $1$, for each $1 \leq i \leq n$.
    \end{enumerate}
    Then each tuple $(c_1, \dots, c_n) \in R$ satisfies the following implication:
    \begin{itemize}
    \item
    if there is an $i \in \{1,\dots, n\}$ such that $c_i \in \{a_i, b_i\}$, then $c_j \in \{a_j, b_j\}$ for every $j \in \{1,\dots,n\}$. 
    \end{itemize}
\end{lemma}

\begin{proof}
    Take such an $R$ and suppose that there exists $(c_1, \dots, c_n)$ which fails the implication. We will show that $R$ cannot be functional, which is a contradiction. By reordering the coordinates for $R$, we may assume that there is $s \in \{1,\dots,n-1\}$ so that $c_i \notin \{a_i, b_i \}$ for $i \in \{1,\dots,s\}$ and $c_i \in \{a_i, b_i \}$ for all $i \in \{s+1,\dots,n\}$. 
    Let $L$ denote the restriction of $R$ to $\{a_1, b_1\} \times \dots \times \{a_n, b_n\}$. Let $S \leq R$ be the subalgebra generated by $\{(c_1, \dots, c_n)\} \cup L$. Since $R$ is functional, so is $S$. Since each $\fA_i$ is conservative, it follows that $\pi_i(S)$ has either three or two elements, depending on whether $ i \leq s$ or $s < i$, respectively. In the former case, $\{a_i, b_i\}$ is the nontrivial block of a congruence of $\pi_i(S)$. In the latter case, we of course obtain the two element minority algebra with underlying set $\{a_i, b_i\}$. Therefore, by relabeling elements, we may assume that we are working with an $S \leq \underbrace{ \mathbf{B} \times \dots \times \mathbf{B}}_{s} \times \underbrace{ \mathbf{C} \times \dots \times  \mathbf{C} }_{n-s}$, where 
    \begin{itemize}
        \item $S$ is functional, 
        \item $\mathbf{B}$ has underlying set $B = \{0,1,c\}$, and $\{0,1\}$ is a block of a congruence,  
        \item $\mathbf{C}$ has underlying set $\{0,1\}$, 
        \item $S$ contains the solution set to the $\mathbb{Z}_2$-linear equation $x_1 + \dots + x_n =1$, and 
        \item $S$ contains the tuple $( \underbrace{ c, \dots, c }_s , z_{s+1}, \dots, z_n)$, for some $z_{s+1}, \dots, z_n \in \{0,1\}$.
    \end{itemize}

    Let $S'$ be the representation of $S$ obtained by interpreting tuples $(x_1, \dots, x_n) \in S$ as pairs $((x_1, \dots, x_s), (x_{s+1}, \dots, x_n)) \in T\times U$, where $T = \pi_{1, \dots, s}(R)$ and $U = \pi_{s+1, \dots, n}(R)$. Notice that $\{0,1\}^{n-s} \subseteq U$, since any tuple of $0$'s and $1$'s of length strictly less than $n$ can be extended to a solution of $x_1 + \dots + x_n = 1$. We now argue that for any $(y_{s+1}, \dots, y_{n}) \in \{0,1 \}^{n-s}$
    $$((\underbrace{c, \dots, c}_s) , (y_{s+1}, \dots, y_n) )\in S'.$$

    Indeed, we may extend $(y_{s+1}, \dots, y_{n})$ to a solution $((y_1, \dots, y_s), (y_{s+1}, \dots, y_{n})) \in S'$ of the linear equation $x_1 + \dots + x_n$. We do the same with the tuple $(z_{s+1}, \dots, z_n)$ to obtain the tuple $((z_1, \dots, z_s), (z_{s+1}, \dots, z_n))$. Since the conservative minority operation $m$ for the algebra $\fA$ preserves $S'$, we obtain that
    \begin{align*}
    ( (m(y_1, z_1, c), \dots, m(y_s, z_s, c)), (m(y_{s+1}, z_{s+1}, z_{s+1}), \dots, m(y_{n}, z_n, z_n)) &= \\
    ( (c, \dots, c), (y_{s+1}, \dots, y_n))  &\in S',
    \end{align*}
    where we are using that $\{a_i, b_i \}$ is collapsed by a congruence of $\mathbf{B}$ to see that $m(y_i, z_i, c) = c$ for each $1 \leq i \leq s$.
    It follows that $S'$ is not functional, since for example, both $((c, \dots, c), (0, \dots, 0, 0)) \in S'$ and $((c, \dots, c), (0, \dots, 0, 1)) \in S'$.
\end{proof}

\begin{proposition}\label{prop:characterizesdfunctmscritical}
    Let $\fA$ be a conservative minority algebra, let $\mathbf{A}_1, \dots, \mathbf{A}_n \leq \fA$ be subdirectly irreducible subalgebras of $\fA$ with respective monoliths $\mu_1, \dots, \mu_n$ and nontrivial blocks $B_1, \dots, B_n$. Let $R \leq \mathbf{A}_1 \times \dots \times \mathbf{A}_n$, for $n \geq 3$, be a subdirect functional multisorted critical relation. Let $R_1$ be the restriction of $R$ to $B_1 \times \dots \times B_n$ and let $R_2$ be the restriction of $R$ to $(A_1 \setminus B_1) \times \dots \times (A_n \setminus B_n)$. The following hold.
    \begin{enumerate}
        \item The subalgebra $\fB_i \leq \fA_i$ with domain is $B_i$ is a two element minority algebra, for each $1 \leq i \leq n$.
        \item For some choice of a labeling $B_i = \{0_i, 1_i\}$
        for each $i \in \{1,\dots, n\}$, the relation $R_1$ is equal to the set of solutions $L$ to the $\mathbb{Z}_2$-linear equation $x_1 + \dots + x_n = 1$.
        \item $R$ is the disjoint union of $R_1$ and $R_2$.
        \item The projection of $R_2$ onto any pair of distinct coordinates is a bijection graph.
        
        \item $\fA_i$ and $\fA_j$ are isomorphic, for all $i,j \in \{1,\dots,n\}$.
    \end{enumerate}
\end{proposition}

\begin{proof}
    The first item \emph{1.} is just Lemma \ref{lem:criticalmonolithblockssizetwo}. In the proof of Lemma \ref{lem:criticalmonolithblockssizetwo}, we showed that $L \subseteq R_1$ for some choice of labeling each of the blocks $B_1, \dots, B_n$ with the set $\{0,1\}$, where the critical tuple for $R$ becomes the tuple $(0, \dots, 0)$. Since $L$ is a maximal affine $\mathbb{Z}_2$-subspace of $\mathbb{Z}_2^n$, it must be that $L = R_1$, as otherwise the critical tuple $(0,\dots, 0) \in R$. Thus, \emph{2.} holds.

    It follows from Lemma~\ref{lem:tuplesdontcrossminimalboundary} that \emph{3.}\ holds.
    To prove that \emph{4.}\ holds, suppose that there are $i,j \in \{1,\dots,n\}$ with $i<j$ so that $\pi_{i,j}(R_2)$ is not a bijection graph. Since $R$ is subdirect, it follows from \emph{3.} that $\pi_{i,j}(R_2) \subseteq (A_i\setminus B_i) \times (A_j \setminus B_j)$ is subdirect.
    Since $\pi_{i,j}(R_2)$ is not a bijection graph, we may suppose that there exist $(a,b), (a,c) \in \pi_{i,j}(R_2)$ with $b \neq c$, where $a \notin B_i$ and $b,c \notin B_j$ (otherwise, change the order of the arguments of $R$). Since $(a,b), (a,c) \in \pi_{i,j} (R)$ as well, this implies that the congruence $\alpha = \pi_{i,j}(R) \circ (\pi_{i,j}(R))^{-1}$ is not the equality relation. Since $b,c \notin B_j$, it follows that $\mu_j$ is properly contained in $\alpha$. We deduce that $(0, b) \in \alpha$. It follows that there exist tuples $$(d_1, \dots, d_{j-1}, 0, d_{j+1}, \dots, d_n), (e_1, \dots, e_{j-1}, b, e_{j+1}, \dots, e_{n}) \in R$$ such that $d_i = e_i$. It follows from Lemma~\ref{lem:tuplesdontcrossminimalboundary} that $d_i \in B_i$ and $e_i \in A_i \setminus B_i$, which is a contradiction. 


    Finally, to see \emph{5.}, let $\mu = \mu_1 \times \dots \times \mu_n$ be the product of the monoliths of $\fA_1, \dots, \fA_n$. It follows from \emph{3.}\ and \emph{4.}\ that $\pi_{i,j}(\eta_\mu(R) )$
    is a bijection graph, for any distinct $1\leq i, j \leq n$. Hence, $\fA_i/ \mu_i $ is isomorphic to $\fA_j / \mu_j$ for any $1 \leq i, j \leq n$. From this we can deduce that $\fA_i$ is isomorphic to $\fA_j$, because $\mu_i$ and $\mu_j$ are each block congruences whose nontrivial block has exactly two elements. Since there is only one two element minority algebra, there is only one possible isomorphism type for $\fA_i$ whose quotient modulo $\mu_i$ produces the isomorphism type of $\fA/ \mu_i$. More formally, since $\mu_i$ is a block congruence with nontrivial class $B_i = \{0_i, 1_i\}$, the domain of $\fA_i / \mu_i$ is equal to $D_i = \{ B_i \} \cup \{ \{a\} : a \in A_i \setminus B_i \}$. Let $\mathcal{T}_i$ be the height 2 conservative minority  tree over the set $D_i \cup B_i$ with vertices 
    \[
    T_i = \{(B_i, 0_i), (B_i, 1_i) \} \cup \{ (\{a \}) : a \in A_i \setminus B_i \}.
    \]
    and local algebras $C_\epsilon = \fA / \mu_i$ and $C_{(B_i)}$ is the two element minority algebra. It is straightforward to check that $\fA_{\mathcal{T}_i}$ is isomorphic to $\fA_i$ and that $\mathcal{T}_i$ and $\mathcal{T}_j$ are isomorphic conservative minority trees, hence $\fA_{\mathcal{T}_i}$ is isomorphic to $\fA_{\mathcal{T}_j}$. The result now follows by composing isomorphisms. 
\end{proof}

\begin{theorem}\label{thm:consminorityrelbasis}
Let $\fA = (A;f)$ be a finite conservative minority algebra. Then for any $R \subseteq A^k$ the following are equivalent.
\begin{enumerate}
    \item $R$ is preserved by $f$. 
    
    \item $R$ has a primitive positive definition from the following finite set of at most ternary relations: 
    \begin{primenumerate}
    \item All unary relations $X \subseteq A$, 
            
    \item All subdirect transversal endomorphism graphs $S \leq \fB \times \mathbf{C}$ for $\mathbf{C} \leq \mathbf{B} \leq \mathbf{A}$ and $\mathbf{C}$ subdirectly irreducible,
            
     \item All isomorphism graphs $G \leq \fA_1 \times \fA_2$ where $\fA_1, \fA_2 \leq \fA$ are subdirectly irreducible,
    \item For each subdirectly irreducible $\fC \leq \fA$ with a monolith whose nontrivial class has exactly two elements, the ternary relation
\begin{align*}
\Lin_{\mathbf{C}}  := \big \{ & (x,y,z) \in \{0,1\}^3 \mid  \text{$(x,y,z)$ is a solution to the $\mathbb{Z}_2$-linear equation $x+y+z = 1$}
\big \} \\
& \cup \big \{(c,c,c): c\in C \setminus \{0,1\} \big \},
\end{align*} 
where $\{0,1\}$ is some labeling of the nontrivial class of the monolith of $\mathbf{C}$.
    \end{primenumerate}
\end{enumerate}
\end{theorem}

\begin{proof}
This theorem is a refinement of Lemma~\ref{lem:preliminarybasis}. The argument that \emph{2.}\ implies \emph{1.}\ is similarly easy as the corresponding implication in Lemma~\ref{lem:preliminarybasis}. To show that \emph{1.}\ implies \emph{2.}, we use our characterization of the subdirect functional multisorted critical relations (the relations listed in \emph{(c)} of Lemma~\ref{lem:preliminarybasis}). Notice that the relations listed in \emph{(a)} and \emph{(b)} of Lemma~\ref{lem:preliminarybasis} are the same as the relations listed in \emph{(a$'$)} and \emph{(b$'$)} here. Therefore, we need only show that the relations listed here can primitively positively define the relations listed in \emph{(c)} of Lemma~\ref{lem:preliminarybasis}. 

So, let $R \leq \fA_1 \times \dots \times \fA_n$ be a subdirect functional multisorted critical relation for subdirectly irreducible $\fA_1, \dots, \fA_n \leq \fA$ and $n \geq 2$. We treat two cases. Suppose that $n = 2$. Since $R$ is subdirect and functional, it is clear that $R$ is a bijection graph, i.e., the graph of an isomorphism between $\fA_1$ and $\fA_2$. Hence, $R$ is definable from the set of listed relations, because it belongs to the relations listed in \emph{(c$'$)}.

Suppose that $n \geq 3$. Item \emph{5.}\ of Lemma~\ref{prop:characterizesdfunctmscritical} indicates that $\fA_1, \dots, \fA_n$ are isomorphic. Let $\mathbf{C} = \fA_1$ and let $G_i \leq \fA_i \times \mathbf{C}$ be an isomorphism graph for each $1 \leq i \leq n$. Define $S \leq \mathbf{C}^n$ by the formula 
\[
\exists y_1, \dots, y_n 
\bigg(
R(y_1, \dots, y_n) \wedge \bigwedge_{1\leq i \leq n} G_i(y_i, x_i) 
\bigg) .
\]
It is easy to see that $\fA_1 \times \dots \times \fA_n$ and $\mathbf{C}^n$ are isomorphic, that $R$ and $S$ are isomorphic (as algebras), and that $R$ can be defined from $S$ by a similar formula. Since $G_i$ is among the relations listed in \emph{(c$'$)} for all $1 \leq i \leq n$, we have reduced the relations listed in \emph{(c)} of Lemma~\ref{prop:characterizesdfunctmscritical} to the subdirect functional multisorted critical $S \leq \mathbf{C}^n$ for subdirectly irreducible $C \leq \fA$.

We further restrict the relations we consider by analyzing the binary projections of such $S$. Let $\mu$ be the monolith of $\mathbf{C}$ and let the nontrivial class of $\mu$ be labeled $B = \{0,1\}$. It follows from the characterization of $S$ given in Proposition~\ref{prop:characterizesdfunctmscritical} that $\pi_{1,i}(S)$ is the disjoint union of the graph of a bijection $Z= \pi_{1,2}(S) \setminus B^2$ and the full relation on $\{0,1\}$. Since $B$ is the nontrivial class of a congruence on $\mathbf{C}$, it is easy to see that the bijection $T = Z \cup \{ (0,0), (1,1) \}$ is an automorphism of $\mathbf{C}$. Hence, we define $S'$ by the formula $\exists y_2  \big ( S(x_1, y_2, \dots, x_n) \wedge T(x_2, y_2) \big)$ to obtain another subdirect functional multisorted critical relation, but now the first two coordinates of any tuple with values outside of $B$ are equal. Iterating over the other coordinates produces a subdirect functional multisorted critical relation 
\begin{align}
U  := \big \{ & (x_1, \dots, x_n) \in \{0,1\}^3 \mid  \text{$(x,y,z)$ is a solution to the $\mathbb{Z}_2$-linear equation $x_1 + \dots + x_n = 1$}
\big \} \nonumber \\
& \cup \big \{(c, \dots, c): c\in C \setminus \{0,1\} \big \}. \label{eq:U}
\end{align} 
It is easy to see that this relation can be defined from those where $n=3$, which are exactly the relations listed in \emph{(d$'$)} of the theorem. This finishes the proof.
\end{proof}

\begin{definition}\label{def:relationalbasisStructureNamed}
    Let $\mathbf{A} = (A;f)$ be a conservative minority algebra. We define $\mathfrak{S}_\mathbf{A}$ to be the structure with domain $A$ whose relations are exactly those listed for $\mathbf{A}$ in Theorem~\ref{thm:consminorityrelbasis}.
\end{definition}

\begin{remark}\label{rem:polclo}
    Note that $\Pol(\mathfrak{S}_\mathbf{A})=\Clo(\mathbf{A})$. 
\end{remark}

\begin{corollary}
    Let $\mathfrak{A}$ be a finite structure with a conservative Maltsev polymorphism. There exists a conservative minority algebra $\mathbf{B} = (A;f)$ such that $\mathfrak{S}_\mathbf{B}$ pp-defines $\mathfrak{A}$.
\end{corollary}

\begin{proof}
    By Lemma~\ref{lem:Carbonnel}, $\bA$ also has a conservative minority polymorphism $f$. Each relation $R$ of $\bA$ is preserved by $f$, and by Theorem~\ref{thm:consminorityrelbasis} $R$ is pp-definable from the relations of $\mathfrak{S}_{\mathbf{B}}$.
\end{proof}

\subsection{Minimal Taylor conservative minority algebras}
\label{sect:minimal-taylor}

To simplify the presentation going forward, it is convenient to work exclusively with minimal Taylor algebras. We now provide some justification of this and then characterize such algebras by specifying their three-element subalgebras. We point out that such a characterization of \emph{all} conservative minimal Taylor algebras will appear in~\cite{Barto3Elements}, and we include this special case for the sake of completeness.

Let us first discuss the general idea. Suppose that $\mathbf{A} = (A;f)$ is a conservative minority algebra. Since $f$ is a ternary operation, it is determined by its behavior on the three-element subsets of $A$, and since $f$ is conservative, it must actually determine a subalgebra on each of the three-element subsets of $A$. Hence, any conservative minority algebra $\mathbf{A}$ is specified by a collection of three-element conservative minority algebras, one for each three-element subset of $A$. Proposition~\ref{prop:MinTayConsMalChar} provides a characterization of which among these operations are minimal Taylor. We first characterize the minimal Taylor conservative minority algebras on a three-element domain and then lift this to a characterization of any conservative minority algebra. 

\begin{lemma}\label{lem:minimalTaylorminorityisminority}
    Let $g$ be a minimal Taylor operation which is generated by a conservative minority operation $f$ on a set $A$. Then $g$ is also a conservative minority operation. 
\end{lemma}

\begin{proof}
    Since $g$ is generated by $f$, it is conservative. By Lemma~\ref{lem:minTaylorSubAlg}, it follows that $g$ restricted to any two-element subset $X \subseteq A$ is minimal Taylor. Since this restriction of $g$ to $X$ is generated by the restriction of $f$ to $X$, it follows that $g$ restricted to $X$ is a minimal Taylor operation on a two-element subset that is generated by the minority operation on two elements. Since the only minimal Taylor operation which can be generated by the minority on two elements is the minority on two elements, it follows that $g$  equals the minority operation on every two-element subset $X \subseteq A$, which is the same thing as saying that $g$ is a minority operation. 
\end{proof}

\begin{corollary}\label{cor:minimalTconstructsCM}
     Let $\mathfrak{A}$ be a finite structure with a conservative Maltsev polymorphism. There exists a minimal Taylor conservative minority algebra $\mathbf{B} = (B;f)$ such that $\mathfrak{S}_\mathbf{B}$ pp-defines $\mathfrak{A}$.
\end{corollary}

We now define an important class of conservative minority algebras whose members will be used later as building blocks for a class of algebras whose corresponding structures pp-construct all finite conservative Maltsev structures (Theorem~\ref{thm:FinalPPConstructionTheorem}). 
 Let $\mathbf{P}^{(2)}_n$ be the algebra 
with the 
domain $P_n := \{0,\dots,n-1\}$
and a single operation 
$p^{(2)}_n \colon P_n^3 \to P_n$ which is defined to be the 
minority operation that returns the second argument on injective triples. For example, $p^{(2)}_4(1,0,0) = 1$ and $p^{(2)}_4(1,2,3) = 2$.
Notice that we could have just as well defined $p_n$ to return the first or third argument on an injective triple, and we call these two other algebras the \emph{$n$-element first} and the \emph{$n$-element third projection minority algebras} and denote them by $\mathbf{P}_n^{(1)}$ and $\mathbf{P}_n^{(3)}$, respectively. We emphasize that for the bulk of the paper, we will only consider the $n$-element second projection minority algebra and use the notation $\mathbf{P}_n$, but here we distinguish all three with superscripts. 

\begin{lemma}\label{lem:brady1ishsimple}
    Let $n \geq 1$ and $1 \leq i \leq 3$ The following statements about $\mathbf{P}^{(i)}_n$ hold.
    \begin{enumerate}
    \item The following relations belong to $\Inv(\mathbf{P}^{(i)}_n)$:
    \begin{enumerate}
        \item all unary relations $X \subseteq \{0, \dots, n-1\}$,
        \item all bijection graphs $G \subseteq n^2$, and 
        \item $L = \{ (x,y,z) \in \{0,1\}^3 : \text{$(x,y,z)$ is a solution to the $\mathbb{Z}_2$-linear equation $x+ y+ z = 1$}\}$.
    \end{enumerate}
    \item $\mathbf{P}_n$ is \emph{hereditarily simple}, i.e., each subalgebra of $\mathbf{P}_n$ has exactly two congruences: the equality relation and the full relation. 
    \end{enumerate}
    
\end{lemma}
\begin{proof}
    Item \emph{1.} can be easily checked by the reader. To see \emph{2.}, let $\mathbf{B} \leq \mathbf{P}_n$. 
    Suppose that $\theta$ is a congruence of $\mathbf{P}_n$ and there is $(a,b) \in \theta$ with $a \neq b$. It follows that $(c,b) \in \theta$ for any $c \in B \setminus \{a,b\}$, because $(a,a), (a,b), (c,c) \in \theta$, so $(p_n(a,a,c), p_n(a,b,c)) = (c,b) \in \theta$. Therefore, any congruence which is not the equality relation is the full relation on $B$ and the lemma is proved. 
\end{proof}

We remark that any relation in $\Inv(\mathbf{P}_n)$ has a primitive positive definition in the set of relations listed in Lemma \ref{lem:brady1ishsimple}. This fact is a special case of Proposition~\ref{prop:hsimpletreerecursivebasis}. So, on a three-element set we have defined three algebras $\mathbf{P}_3^{(1)}$, $\mathbf{P}_3^{(2)}$, and $\mathbf{P}_3^{(3)}$. Now, suppose that a conservative minority algebra $\mathbf{A} = ([3]; f)$ on a three-element domain has a nontrivial congruence $\theta$. Then $\theta$ has exactly one nontrivial class with two elements, say $\{1,2\}$. Since the quotient $\mathbf{A}/\theta$ is a minority algebra, this forces 
\[
f(1,2,3)=f(2,1,3)= f(1,3,2) = f(2,3,1) = f(3,1,2) = f(3,2,1) = 3.
\]
It is easy to check that $\theta$ is the only nontrivial congruence of $\mathbf{A}$. We denote this particular $\mathbf{A}$, as we have just defined it, by $\mathbf{C}_3$. 

\begin{lemma}\label{lem:char3elementMinTaylor}
    Let $\mathbf{A} = ([3];f)$ be a conservative minority algebra on a three-element domain. Then $\mathbf{A}$ is a minimal Taylor algebra if and only if $\mathbf{A}$ is isomorphic to one of $\mathbf{P}_3^{(1)}$, $\mathbf{P}_3^{(2)}$, $\mathbf{P}_3^{(3)}$, or $\mathbf{C}_3$.
\end{lemma}

\begin{proof}
    If $\mathbf{A}$ has a nontrivial congruence $\theta$, then we just argued above that $\mathbf{A}$ is isomorphic to $\mathbf{C}_3$. Hence, there is exactly one conservative minority algebra on the three-element domain that is not simple, so it must be minimal Taylor.

    So, we just need to understand the conservative minority algebras $\mathbf{A}$ on three-element domains that are simple. To do this, it is convenient to argue relationally and apply Theorem~\ref{thm:consminorityrelbasis}. For $\mathbf{A}$ a simple conservative minority algebra on the three-element domain $[3]$, denote by $\Gamma_\mathbf{A}$ the relations listed in \emph{(a')}, \emph{(b')}, \emph{(c')}, and \emph{(d')} of Theorem~\ref{thm:consminorityrelbasis}. Now, $\mathbf{A}$ and all of its subalgebras are simple, hence subdirectly irreducible, and the only subdirectly irreducible subalgebras whose monolith has a nontrivial class with exactly two elements are the two-element subalgebras. 
    
    It follows that the only distinction that can occur between $\Gamma_\mathbf{A}$ and $\Gamma_\mathbf{B}$ for two simple conservative minority algebras on the domain $[3]$ is among the relations listed in \emph{(c')} of the theorem, i.e., $\mathbf{A}$ and $\mathbf{B}$ have different automorphism groups. Since every permutation of $\{0,1,2\}$ is an automorphism of $\mathbf{P}_3^{(1)}$, $\mathbf{P}_3^{(2)}$, and $\mathbf{P}_3^{(3)}$, it follows that these are exactly the simple minimal Taylor conservative minority algebras on $[3]$. It is easy to see that one of these projection minority operations can be obtained from the others by a permutation of variables. 
\end{proof}

\begin{proposition}\label{prop:MinTayConsMalChar}
    Let $\mathbf{A} = (A;f)$ be a conservative minority algebra. The following hold. 
    
    \begin{enumerate}
        \item $\mathbf{A}$ is a minimal Taylor algebra if and only if there exists $1\leq i \leq 3$ such that every three-element subalgebra of $\mathbf{A}$ is either isomorphic to $\mathbf{C}_3$ or $\mathbf{P}_3^{(i)}$. 
        \item $\mathbf{A}$ is minimal Taylor and hereditarily simple if and only if it is isomorphic to $\mathbf{P}_n^{(i)}$ for some $1\leq i \leq 3$. 
    \end{enumerate}
\end{proposition}
\begin{proof}
    To prove \emph{1.}, suppose that $\mathbf{A}$ is a minimal Taylor algebra. By Lemma~\ref{lem:minTaylorSubAlg}, every three-element subalgebra of $\mathbf{A}$ is also a minimal Taylor algebra, and so we can apply Lemma~\ref{lem:char3elementMinTaylor} to conclude that each three-element subalgebra of $\mathbf{A}$ is isomorphic to one of $\mathbf{P}_3^{(1)}$, $\mathbf{P}_3^{(2)}$, $\mathbf{P}_3^{(3)}$, or $\mathbf{C}_3$. 
    
    Define $g(x,y,z) := f(f(y,x,z), f(x,y,z), f(x,z,y))$. Clearly, $g \in \Clo(\mathbf{A})$ and it is readily checked that $g$ is a minority operation. Suppose that $\mathbf{B} \leq \mathbf{A}$ is a three-element subalgebra with domain $B$ which is isomorphic to $\mathbf{P}_3^{(1)}$. Let $(a,b,c) \in B^3$ be an injective triple. Then $$g(a,b,c) = f(f(b,a,c), f(a,b,c), f(a,c,b)) = f(b,a,a) = b.$$ 
    This shows that the algebra $(B;g\vert_B)$ is isomorphic to $\mathbf{P}^{(2)}_3$. A similar argument show that $(B;g\vert_B)$ is also isomorphic to $\mathbf{P}^{(2)}_3$ when $\mathbf{B}$ is isomorphic to either $\mathbf{P}^{(2)}_3$ or $\mathbf{P}^{(3)}_3$. It is similarly easy to check that $g\vert_B = f\vert_B$ when $\mathbf{B}$ is isomorphic to $\mathbf{C}_3$. 

    Therefore, $\Clo(\mathbf{A})$ contains an operation $g$ so that a three-element subalgebra $(B;g\vert_B)$ of the algebra $(A;g)$ is isomorphic to $\mathbf{C}_3$ if and only if $(B; f\vert_B)$ is, while a three-element subalgebra $(B;g\vert_B)$ is isomorphic to $\mathbf{P}^{(2)}_3$ if and only if $(B;f\vert_B)$ is isomorphic to one of $\mathbf{P}_3^{(1)}$, $\mathbf{P}_3^{(2)}$, $\mathbf{P}_3^{(3)}$. Since $f$ is a minimal Taylor operation, it must be that $f$ belongs to the clone generated by $g$. This cannot happen if there are three-element subalgebras $\mathbf{B}_1$ and $\mathbf{B}_2$ of $\mathbf{A}$ respectively isomorphic to $\mathbf{P}^{(i)}_3$ and $\mathbf{P}^{(j)}_3$, for $i \neq j$. This is because the composition tree for a function $h(x,y,z)$ built from $g$ makes no reference to an underlying set, so $h\vert_{B_1}$ is the $i$-th projection minority on $B_1$ if and only if $h\vert_{B_2}$ is also the $i$-th projection minority on $B_2$, for any $B_1$ and $B_2$ which are subsets of cardinality three.  

    So, we have proved the forward implication of \emph{1.}\ of the proposition. To prove the reverse implication, we let $\mathbf{A} = (A;f)$ be an algebra with ternary basic operation $f$ defined on a finite set $A$ which satisfies the condition that every three-element subset of $A$ is a subuniverse of a subalgebra isomorphic to either $\mathbf{C}_3$ or $\mathbf{P}^{(i)}_3$, for some $1\leq i \leq 3$. Obviously, such an $\mathbf{A}$ is a conservative minority algebra, hence $f$ generates a minimal Taylor conservative minority operation $g$ by Lemma~\ref{lem:minimalTaylorminorityisminority}. We just argued that if there are distinct three-element $B_1$ and $B_2$ along with distinct $1 \leq i \neq j \leq 3$ so that $g$ restricted to $B_1$ and $B_2$, respectively, produces the $i$-th and $j$-th projection minority, then $g$ cannot be minimal Taylor. Since the projection minority operations $p_3^{(1)}$, $p_3^{(2)}$, and $p_3^{(3)}$ can each be obtained from the others by a permutation of variables, it follows that $\mathbf{A}$ must already be a minimal Taylor algebra. 

    The proof of \emph{2.} is easy. Suppose that $\mathbf{A}$ is a minimal Taylor conservative minority algebra. We apply \emph{1.}\ and note that if $\mathbf{A}$ has a three-element subalgebra isomorphic to $\mathbf{C}_3$, then it is obviously not hereditarily simple. If there is no such three-element subalgebra, then again by \emph{1.}\ we have that there exists $1\leq i \leq 3$ so that $f$ is isomorphic to $\mathbf{P}^{(i)}_n$. 
\end{proof}

\begin{remark}\label{rem:onlyoneTypeOfMinTaylorConsMin}
In view of \emph{1.}\ of Proposition~\ref{prop:MinTayConsMalChar}, we can now speak about the \emph{type} of a minimal Taylor conservative minority algebra $\mathbf{A}$, which is the unique $1\leq i \leq 3$ such that all three-element subalgebras of $\mathbf{A}$ which are isomorphic to one of $\mathbf{P}^{(1)}_2$, $\mathbf{P}^{(2)}_3$, or $\mathbf{P}^{(3)}_3$ are isomorphic to $\mathbf{P}^{(i)}_3$. However, we note that creating this distinction between algebras makes no difference from a clone or relational point of view. Henceforth, we assume that all minimal Taylor conservative minority algebras are of type $i=2$ and use the notation $\mathbf{P}_n$ instead of $\mathbf{P}_n^{(i)}$.
\end{remark}

For the remainder of this section, we specialize some of the tree machinery from earlier to the context of minimal Taylor conservative minority algebras. 
\begin{corollary}\label{cor:representwithtreealgebrasMinTaylor}
Let $\fA = (A;m)$ be a minimal Taylor conservative minority algebra. There exists a simple reduced full conservative minority tree $\mathcal{T}$ over a collection $S$ of pairwise nonisomorphic minimal Taylor conservative minority algebras such that $\mathcal{T}$ represents $\mathbf{A}$. 
\end{corollary}
\begin{proof}
   This is a straightforward specialization of the proof of Theorem~\ref{thm:representwithtreealgebras}. If $\mathbf{A}$ is minimal Taylor, then every block congruence and quotient used to construct the representing tree $\mathcal{T}$ will also be minimal Taylor. This follows from Lemma~\ref{lem:minTaylorSubAlg} and our earlier observation that quotients are already represented as subalgebras in the conservative setting (cf. Lemma~\ref{lem:quotientsaresubalgebras}). 
\end{proof}

\begin{corollary}\label{cor:minTaylorConsMinTreeConstOthers}
    Let $\mathfrak{A}$ be a finite structure with a conservative Maltsev polymorphism. There exists a simple reduced full conservative minority tree  $\mathcal{T} = (T, \leq, S)$, where $S$ is a collection of pairwise nonisomorphic minimal Taylor conservative minority algebras, such that $\mathfrak{S}_{\mathbf{A}_\mathcal{T}}$ pp-constructs $\mathfrak{A}$.
\end{corollary}

\begin{proof}
    This is an immediate consequence of Corollaries~\ref{cor:minimalTconstructsCM}and~\ref{cor:representwithtreealgebrasMinTaylor}.
\end{proof}

Actually, we can sharpen the pp-construction described by Corollary~\ref{cor:minTaylorConsMinTreeConstOthers} so that the conservative minority tree is a sapling. We formulate this statement separately, because it serves as an important link in our final recursive pp-construction and there it is important to keep track of the local algebras in use. 

\begin{lemma}\label{lem:SaplingTreesConstructGeneralOnes}
    Let $\mathcal{T}$ be a conservative minority tree over $S = \{\mathbf{A}_1, \dots, \mathbf{A}_s \}$. There exists a sapling conservative minority tree $\mathcal{T}'$ over the same $S$ such that $\mathfrak{S}_{\mathbf{A}_{\mathcal{T}'}}$ pp-constructs $\mathfrak{S}_{\mathbf{A}_\mathcal{T}}$. If $\mathcal{T}$ is full, then so is $\mathcal{T}'$, and if $S$ is also a collection of simple algebras, then $\mathbf{A}_{\mathcal{T}'}$ is a subdirectly irreducible algebra. 
\end{lemma}

\begin{proof}
    It suffices to show that $\mathbf{A}_\mathcal{T}$ belongs to the pseudo-variety generated by $\{\mathbf{A}_{\mathcal{T}'}\}$,  because then there is a minion homomorphism from $\Pol(\mathfrak{S}_{\mathbf{A}_\mathcal{T'}}) = \Clo(\mathbf{A}_{\mathcal{T}'})$ 
    to $\Pol(\mathfrak{S}_{\mathbf{A}_{\mathcal{T}}}) = \Clo(\mathbf{A}_\mathcal{T})$ (see Remark~\ref{rem:polclo} and Example~\ref{expl:hsp})
    and hence 
    the statement follows from Theorem~\ref{thm:pp-constr}. We use the following two key facts.
    \begin{itemize}
     \item $\mathbf{A}_\mathcal{T}$ is isomorphic to a subdirect product of its subdirectly irreducible homomorphic images (see~\cite[Theorem 8.6]{BS}).
     \item Any homomorphic image of a conservative minority algebra $\mathbf{A}$ is isomorphic to a subalgebra of $\mathbf{A}$ (see Lemma~\ref{lem:quotientsaresubalgebras}).
    \end{itemize}
    Therefore, the appropriate $\mathcal{T}'$ is one for which $\mathbf{A}_{\mathcal{T}'}$ contains as a subalgebra each subdirectly irreducible subalgebra of $\mathbf{A}_\mathcal{T}$. We know, for every subdirectly irreducible $\mathbf{D} \leq \mathbf{A}_\mathcal{T}$ be subdirectly irreducible, that $\mathcal{T}_D$ is a sapling. It is easy to see that such $\mathcal{T}_D$ can be extended to a maximal sapling $\mathcal{T}_D^M$ (see Section~\ref{sec:saplings}) over $S$ such that $\mathbf{A}_{\mathcal{T}_D}$ is isomorphic to a subalgebra of $\mathbf{A}_{\mathcal{T}_D^M}$. Therefore, we let $\mathcal{T}''$ be the conservative minority tree obtained by arranging the maximal saplings of $\mathcal{T}$ in some sequence, attaching the root of one to the vertex of another so that a sapling is produced. There are many ways to do this, but the details are not important, since the obtained $\mathcal{T}'$ will contain all saplings which represent subdirectly irreducible algebras isomorphic to subdirectly irreducible subalgebras of $\mathbf{A}_\mathcal{T}$. To finish, we let $\mathcal{T}'$ be the  conservative minority tree obtained from $\mathcal{T}''$ by contracting all paths which have no branching (see Section~\ref{sec:treetransformations} for details on this transformation). The remaining statements are straightforward consequences of the definitions and Lemma~\ref{lem:saplingsrepresentsubdirectirreducible}.
\end{proof}

\begin{figure}
    \centering
    \includegraphics[width=0.9\linewidth]{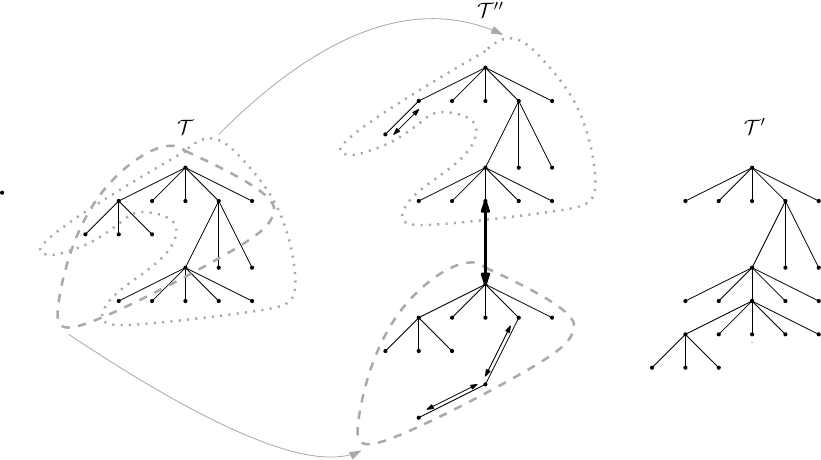}
    \caption{Concatenation of saplings}
    \label{fig:concatenateSaplings}
\end{figure}

We include Figure~\ref{fig:concatenateSaplings} to illustrate the procedure outlined in the proof of Lemma~\ref{lem:SaplingTreesConstructGeneralOnes}. On the left side of the figure we show some the underlying tree structure of some conservative minority tree $\mathcal{T}$. This particular $\mathcal{T}$ has two maximal saplings which could potentially represent nonisomorphic algebras. These are shown in the middle, with differently patterned outlines indicating from where in $\mathcal{T}$ they originate. Also depicted in the middle of the figure is a double arrow indicating where these two saplings will be attached and some shorter double arrows indicating which vertices are trivial and can be discarded. The sapling $\mathcal{T}'$ is shown on the right.

\subsection{The class of algebras $\mathbf{P}^{n,k}$}\label{sec:algebrasP_nk}
For the remainder of the article we will focus on subdirectly irreducible conservative minority algebras which are represented by projection conservative minority trees. In this subsection we define the algebras $\mathbf{P}^{n,k}$ for $n \geq 2$ and $k \geq 0$. In the next subsection we will define structures $\mathfrak{P}^{n,k}$, so that $\Inv(\mathbf{P}^{n,k}) = \Inv(\Pol(\mathfrak{P}^{n,k}))$. The collection of these structures is important for two reasons. The first is the main result of this paper: for all $n \geq 2$ and $k \geq 0$, there exists a $\oplus$L algorithm for $\Csp(\mathfrak{P}^{n,k})$ (Corollary~\ref{cor:OplusAlgforPnk}). The second reason this class of structures is interesting is that it primitively positively constructs the class of all structures which admit a conservative Maltsev polymorphism (Theorem~\ref{thm:DL}). Taken together, these two facts imply that there exists a $\oplus$L algorithm for the CSP of any structure which admits a conservative Maltsev polymorphism.

\begin{definition}\label{def:projminTree}
    Let $n \geq 2$ and let $k \geq 1$. We let $\mathcal{P}_{n,k}$ be the conservative minority tree over $S= \{\mathbf{P}_n\}$ whose underlying tree vertices $P_{n,k}$ consists of the following sequences:
    \[
      P_{n,k}:= \{ (0^{j_1})^\frown (l^{j_2}) : l \in \{0, \dots, n-1 \}, 0 \leq j_1 \leq k-1, \text{ and } 0\leq j_2 \leq 1\}.
    \]
    We then define the algebra $\mathbf{P}^{n,k}$ to be $\mathbf{A}_{\mathcal{P}_{n,k}}$. 
    
    For technical reasons, we also define a trivial conservative minority tree and the algebra it represents. Set $\mathcal{P}_{n,0}$ to be the height zero tree with single vertex $\epsilon$ and $\mathbf{P}^{n,0} = \fA_{\mathcal{P}_{n,0}}$ to be the trivial one element algebra.
\end{definition}

\begin{remark}
Note that under this labeling convention, the projection minority algebra $\mathbf{P}_{n}$ is isomorphic to $\mathbf{P}^{n,1}$, but the two algebras are formally not equal because $\mathbf{P}^{n,1}$ has a domain which consists of length one sequences of elements of the domain of $\mathbf{P}_{n}$. We also remark that the distinction between subscripts and superscripts is intentional and meant to troubleshoot the following annoyance. Formally, we are dealing with algebras represented by conservative minority trees, i.e., algebras $\fA_{\mathcal{T}}$ for a conservative minority tree $\mathcal{T}$. It is customary in the literature to label the domain of an algebra $\fA$ as $A$. If we defined $\mathbf{P}_{n,k} = \fA_{\mathcal{P}_{n,k}}$, then the domain would naturally be referred to as $P_{n,k}$, but this label is already implicitly taken by the underlying set of the conservative minority tree $\mathcal{P}_{n,k}$. Hence, for the algebras we put this information in a superscript, so that $P^{n,k} = A_{\mathcal{P}_{n,k}}$ and there is no conflict. 
\end{remark}
The first few of the trees and algebras described in Definition~\ref{def:projminTree} for the value $n=3$ are depicted in Figure \ref{fig:hereditarilysimpletrees}. 
\begin{figure}
    \centering
    \includegraphics[width=0.75\linewidth]{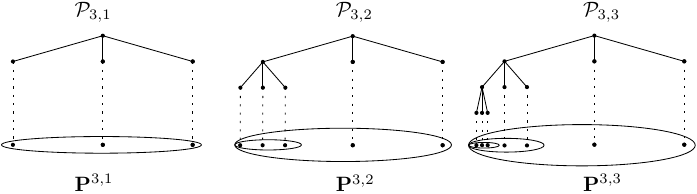}
    \caption{The first few $\mathcal{P}_{n,k}$ and $\mathbf{P}^{n,k}$ for $n=3$.}
    \label{fig:hereditarilysimpletrees}
\end{figure}
In view of the relations listed in Theorem \ref{thm:consminorityrelbasis}, it should come as no surprise that our next step is to analyze the subdirectly irreducible subalgebras of $\mathbf{P}_{n,k}$, the isomorphisms between them, and their transversal endomorphism graphs. It is to this task that we devote the remainder of this subsection. 

\begin{remark}\label{rem:subalgebrasareSI}
Notice that $\mathbf{B}$ is subdirectly irreducible, for every $\mathbf{B} \leq \mathbf{P}^{n,k}$. Indeed, since $\mathcal{P}_{n,k}$ is a sapling, it follows that the subtree $(\mathcal{P}^{n,k})_B$ representing $\mathbf{B}$ (see Lemma~\ref{lem:subalgebrarepresentingtree}) is also a sapling, so we can apply the third item of Lemma~\ref{lem:saplingsrepresentsubdirectirreducible} to deduce that $\mathbf{B}$ is subdirectly irreducible. 
Even though it is redundant, we still emphasize that a subalgebra $\mathbf{B} \leq \mathbf{P}^{n,k}$ is subdirectly irreducible in what follows.
\end{remark}

Of particular importance are those (subdirectly irreducible) subalgebras of $\mathbf{P}^{n,k}$  which are full, i.e., those subalgebras represented by full saplings of $\mathcal{P}^{n,k}$ (see Lemma \ref{lem:hsimplefullsaplingssuffice}). Therefore, we develop some notation for their specification. Such subalgebras are almost completely determined by their trunks, 
which are easily seen to be subsets of 
the trunk of $\mathbf{P}^{n,k}$,
\begin{align*}
\Trunk(\mathbf{P}^{n,k}) & = \{\epsilon, (0), \dots, (\underbrace{0, \dots, 0}_{k-1}) \}\\
                         &= \{\epsilon, 0^1, \dots, 0^{k-1} \} \text{ (here we remind the reader of this shorter notation).}
\end{align*}
Given some $X \subseteq \Trunk(\mathbf{P}^{n,k})$, there is one additional piece of data needed to determine a (subdirectly irreducible) subalgebra $ \fA \leq \mathbf{P}^{n,k}$ that is represented by a full sapling, because the minimal element of $X$ may have a child $v \in P_{n,k}$ which is not a leaf.  In this case, we must choose among the elements of $P^{n,k} \cap \{ w: w \leq v \}$ a particular leaf $w$ to include in the domain of $\fA$. Therefore, we denote by $\mathcal{F}^{n,k}_{(X,w)}$ the full sapling of $\mathcal{P}_{n,k}$ determined by the data $X$ and $w$. For convenience, we denote by $\mathbf{P}^{n,k}_{(X,w)}$ the algebra $\mathbf{A}_{\mathcal{F}^{n,k}_{(X,w)}}$ which is represented by $\mathcal{F}^{n,k}_{(X,w)}$ and refer to its domain as $P^{n,k}_{(X,w)}$. The reader can consult Figure \ref{fig:fullsubalgebra} for an example.

\begin{figure}
    \centering
    \includegraphics[width=0.65\linewidth]{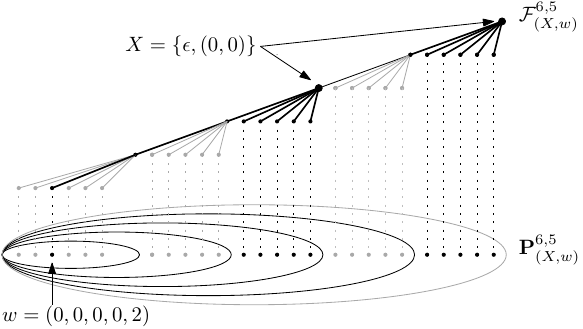}
    \caption{A full subalgebra $\mathbf{P}^{6,5}_{(X,w)} \leq \mathbf{P}^{6,5}$. }
    \label{fig:fullsubalgebra}
\end{figure}

\begin{remark}\label{rem:subalgebrasofP_{n,k}} Notice that, in general, there will be vertices $v$ of the sapling $\mathcal{F}^{n,k}_{(X,w)}$ such that $|C_v| = 1$. In subsection \ref{sec:treetransformations}, we described how the identification of such a vertex with its only child does not change the represented algebra. Performing this procedure iteratively on $\mathcal{F}^{n,k}_{X,w}$, we obtain a simple reduced conservative minority tree which is easily seen to be isomorphic to $\mathcal{P}_{n, |X|}$. It follows that $\mathbf{P}^{n,k}_{(X,w)}$ is isomorphic to $\mathbf{P}^{n,|X|}$.
\end{remark}

We name the embedding $\psi^{n,k}_{(X,w)} \colon \mathbf{P}^{n,|X|} \to \mathbf{P}^{n,k}$ whose image is equal to $\mathbf{P}^{n,k}_{(X,w)}$. Intuitively, this embedding acts like the identity mapping when restricted to $\mathbf{C}_v$ for each $v\in \Trunk(\mathbf{P}^{n, |X|})$, but this intuition is imprecise regarding the leaf $w$. Formally, we set 
\[
\psi^{n,k}_{(X,w)}( (0, \dots, 0)) = w
\]
and require that for every $v \in P^{n,|X|} \setminus \{(0, \dots, 0)\}$ of the form 
$(0^i, c)$ for some $i \in \{0,\dots,|X|-1\}$, 
there exists $j$ so that
\[
\psi^{n, k}_{(X,w)} ( v) = (0^j, c ).
\]
We now define $T^{n,k}_{(X,w)} \leq \mathbf{P}^{n,|X|} \times \mathbf{P}^{n,k}$ to be the graph of $\psi^{n,k}_{(X,w)}$.
Given isomorphic $\mathbf{P}^{n,k}_{(X,w)}$ and $\mathbf{P}^{n,k}_{(X',w')}$, we define the isomorphism graph $T^{n,k}_{(X,w),(X',w')} \leq (\mathbf{P}^{n,k})^2$ between $\mathbf{P}^{n,k}_{(X,w)}$ and $\mathbf{P}^{n,k}_{(X',w')}$ as $(T^{n,k}_{(X,w)})^{-1} \circ T^{n,k}_{(X',w')}$. Hence, the domains $P^{n,k}_{(X,w)}$ and $P^{n,k}_{(X',w')}$ considered as unary relations are defined by the formulas $\exists y (T^{n,k}_{(X,w),(X',w')}(x,y))$ and $\exists x (T^{n,k}_{(X,w),(X',w')}(x,y))$, respectively.

We now indicate some ways to abbreviate the notation we just introduced.
If all children of the minimal element of $X$
are leaves 
or if $w= (0, \dots, 0)$, we simply write $\mathcal{F}_X^{n,k}$ instead of $\mathcal{F}_{(X,w)}^{n,k}$. This abbreviation propagates through all other objects. For example, we could write $T^{n,k}_{X,X'}$ instead of $T^{n,k}_{(X,w),(X',w')}$ whenever both $w = w' = 0^k$. In the special case when $|X| = k-1$, we allow ourselves an even greater abuse of notation and use the unique $0 \leq i \leq k-1$ so that $0^i$ is excluded from $X$ to identify the above objects. That is, we write 
\begin{itemize}
    \item
$\mathbf{P}^{n,k}_{i}$ instead of $\mathbf{P}^{n,k}_{\{\epsilon, 0^1, \dots, 0^{k-1} \} \setminus \{ 0^i \} }$ and $P^{n,k}_i$ for the domain of $\mathbf{P}^{n,k}_i$, 
\item $\psi^{n,k}_{i}$ instead of $\psi^{n,k}_{\{\epsilon, 0^1, \dots, 0^{k-1} \} \setminus \{ 0^i \}}$, and 
\item
$T^{n,k}_{i,j}$ to denote the canonical isomorphism graph between $\mathbf{P}^{n,k}_{i}$ and $\mathbf{P}^{n,k}_{j}$. 
\end{itemize}
We distinguish $T^{n,k}_{0,k-1}$ and call it the \emph{canonical transfer} relation.  The meaning of these definitions is more easily apprehended with a picture and the reader can consult Figure \ref{fig:transferrelation} for some examples.

\begin{remark}
    We emphasize how some of the objects behave in the case $k=1$. Notice that $\mathcal{F}^{n,1}_0$ is, according to our notation, the sapling of $\mathcal{P}_{n,1}$ whose trunk is equal to $\Trunk(\mathcal{P}_{n,1}) \setminus \{ 0^0\}$. Since $\Trunk(\mathcal{P}_{n,1}) = \{\epsilon\}$ and $0^0 = \epsilon$, this means that $\mathcal{F}^{n,1}_0$ has empty trunk, which in view of Remark~\ref{rem:emptytrunk} means that it represents a trivial one element algebra. Since no information about a leaf $w$ is supplied, we apply the convention that $w= (0)$. Hence, $\mathcal{F}^{n,1}$ has an underlying tree consisting of a root $\epsilon$ and a leaf $(0)$. Clearly, $\mathcal{F}^{n,1}_0$ represents the trivial one element algebra, which according to our notation is denoted $\mathbf{P}^{n,1}_0$. Evidently, $\psi^{n,1}_0 \colon \mathbf{P}^{n,0} \to \mathbf{P}^{n,1}$ is an embedding with image $P^{n,1}_0 = \{(0)\}$ and the canonical transfer relation $T^{n,1}_{0,0}$ is the binary relation consisting of the single edge $((0), (0))$.
\end{remark}

\begin{figure}
    \centering
    \includegraphics[width=0.5\linewidth]{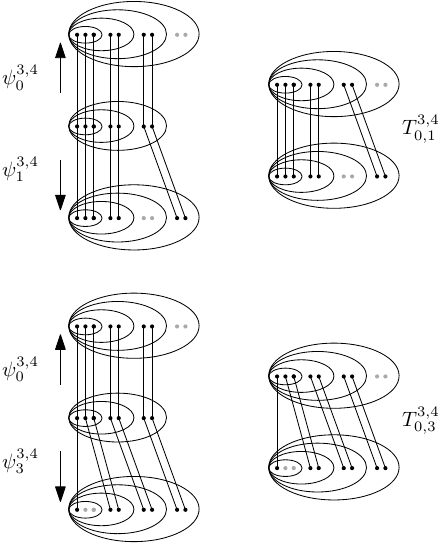}
    \caption{Examples for $n=3, k=4$ of $\psi^{n,k}_i$ and $T^{n,k}_{i,j}$. A gray dot means that the element does not belong to $P^{n,k}_i$.  }
    \label{fig:transferrelation}
\end{figure}

\begin{lemma}\label{lem:hsimplefullsaplingssuffice}
    Let $n \geq 2$ and $k \geq 1$. Let $\mathbf{B}_1, \mathbf{B}_2 \leq \mathbf{P}^{n,k}$ be two isomorphic (subdirectly irreducible) subalgebras and let $S \leq \mathbf{B}_1 \times \mathbf{B}_2$ be the graph of an isomorphism. Then there exist full (subdirectly irreducible) 
    subalgebras $\mathbf{P}^{n,k}_{(X_1,w_1)}, \mathbf{P}^{n,k}_{(X_2, w_2)} \leq \mathbf{P}^{n,k}$ 
    into which $\mathbf{B}_1$ and $\mathbf{B}_2$ respectively embed
     so that 
    \begin{enumerate}
        \item there exists a full isomorphism graph $S' \leq \mathbf{P}^{n,k}_{(X_1,w_1)} \times \mathbf{P}^{n,k}_{(X_2, w_2)}$, and 
        \item $S$ is equal to the restriction of $S'$ to $B_1 \times B_2$.
    \end{enumerate}
\end{lemma}
\begin{proof}
We first establish the following claim.

\begin{claim}\label{SIclaim:specialcase}
Lemma~\ref{lem:hsimplefullsaplingssuffice} holds for the special case when $\Trunk(\mathbf{B}_1) = \Trunk(\mathbf{B_2}) = \Trunk(\mathbf{P}^{n,k})$. 
\end{claim}

\begin{claimproof}
    Suppose we are given the data in the assumptions of Lemma~\ref{lem:hsimplefullsaplingssuffice} along with the assumption of the claim, which is that $\Trunk(\mathbf{B}_1) = \Trunk( \mathbf{B}_2) = \Trunk(\mathbf{P}^{n,k})$ are all equal, evidently to the set 
    $
    \{\epsilon, 0^1, \dots, 0^{k-1} \}.
    $
    Let $\zeta \colon \mathbf{B}_1 \to \mathbf{B}_2$ be an isomorphism and let $S \leq \mathbf{B}_1 \times \mathbf{B}_2$ be the graph of $\zeta$. First, we notice that $\zeta$ preserves the height of all elements of $\mathbf{B}_1$ (the height of an element is equal to its tuple length, or equivalently, its height in $\mathcal{P}_{n,k}$). Specifically, we mean that for every $ 0^{i\frown} (a) \in B_1 $, there exists $b \in \{0, \dots, n \}$ so that 
    $
    \zeta (0^i, a) = (0^i, b)
    $. 
    Indeed, the algebra $\mathbf{P}^{n,k}$ is subdirectly irreducible, and its congruences are the block congruences 
    $\sim_\epsilon, \sim_{(0)}, \dots, \sim_{0^{k-1}}$, 
    which are linearly ordered by inclusion. Observe that the height of an element of $\mathbf{P}^{n,k}$ is equal to the number of nontrivial classes of block congruences to which it belongs, i.e., leaves that belong to $C_\epsilon$ are only collapsed by $\sim_{\epsilon}$, leaves that belong to $C_{(0)}$ are collapsed by $\sim_{\epsilon}$ and $\sim_{(0)}$, and so on. The current assumption that $\Trunk(\mathbf{B}_1) = \Trunk(\mathbf{B_2}) = \Trunk(\mathbf{P}^{n,k})$ implies that each has a congruence lattice isomorphic to that of $\mathbf{P}^{n,k}$. Since isomorphisms preserve the congruence lattice of an algebra, we conclude that $\zeta$ preserves the height of all elements of $\mathbf{B}_1$. 

    Therefore, we define $\zeta_i$ to be the restriction of $\zeta$ to $C_{0^i}$, for each $0 \leq i \leq k-1$. Then each $\zeta_i$ is a partially defined injective mapping on $C_{0^i}$. Since each every local algebra of $\mathcal{T}_{n,k}$ is isomorphic to $\mathbf{P}_n$, it follows that each $\zeta_i$ can be extended to some automorphism $\zeta_i'$ of $\mathbf{C}_{0^i}$. Furthermore, $\zeta_i'$ can be chosen so that it fixes $0^{i+1}$ whenever $0 \leq i < k-1$, because these elements are not leaves of $\mathcal{P}_{n,k}$ and therefore belong to neither $B_1$ nor $B_2$. For each $1 \leq i < k-1$, let $\zeta_i^*$ be the restriction of $\zeta_i'$ to $C_{0^i} \setminus \{0^{i+1}\}$. We now define $\zeta' = \zeta_{k-1}' \cup \bigcup_{0 \leq i <k-1} \zeta_i^*$ and $S' \leq \mathbf{P}^{n,k} \times \mathbf{P}^{n,k}$ to be the graph of $\zeta'$. It is straightforward to check that $S'$ has the desired properties. 
\end{claimproof}

Now we will prove the lemma. Take $\mathbf{B}_1$, $\mathbf{B}_2$, and $S$ as given in the statement and let $X_1 = \Trunk(\mathbf{B}_1)$ and $X_2 = \Trunk(\mathbf{B}_2)$.  There exist $w_1, w_2 \in \mathcal{P}_{n,k}$ such that $\mathbf{B}_1 \leq \mathbf{P}^{n,k}_{(X_1, w_1)}$ and $\mathbf{B}_2 \leq \mathbf{P}^{n,k}_{(X_2, w_2)}$. Let $j = |X_1| = |X_2|$. From our earlier Remark \ref{rem:subalgebrasofP_{n,k}}, we know that both $\mathbf{P}^{n,k}_{(X_1, w_1)}$ and $\mathbf{P}^{n,k}_{(X_2, w_2)}$  are isomorphic to $\mathbf{P}^{n,j}$. Let $\psi^{n,k}_{(X_1,w_1)}$ and $\psi^{n,k}_{(X_2, w_2)}$ be the canonical embeddings that we introduced before the lemma statement. Then the graph $\tau = \psi^{n,k}_{(X_1,w_1)} \circ S \circ (\psi^{n,k}_{(X_2,w_2)})^{-1} \leq \mathbf{P}^{n,j} \times \mathbf{P}^{n,j}$ is the graph of an isomorphism between (subdirectly irreducible) subalgebras of $\mathbf{P}^{n,j}$ whose trunks are equal to $\Trunk(\mathbf{P}^{n,j})$. Applying Claim \ref{SIclaim:specialcase}, there exists an automorphism graph $\tau' \leq \mathbf{P}^{n,j} \times \mathbf{P}^{n,j}$ extending $\tau$. Now define $S' = (\psi^{n,k}_{(X_1,w_1)})^{-1} \circ \tau' \circ \psi^{n,k}_{(X_2,w_2)} \leq \mathbf{P}^{n,k}$. It is straightforward to check that $S'$ satisfies the properties stated in the lemma.
\end{proof}

\begin{figure}
    \centering
    \includegraphics[width=0.80\linewidth]{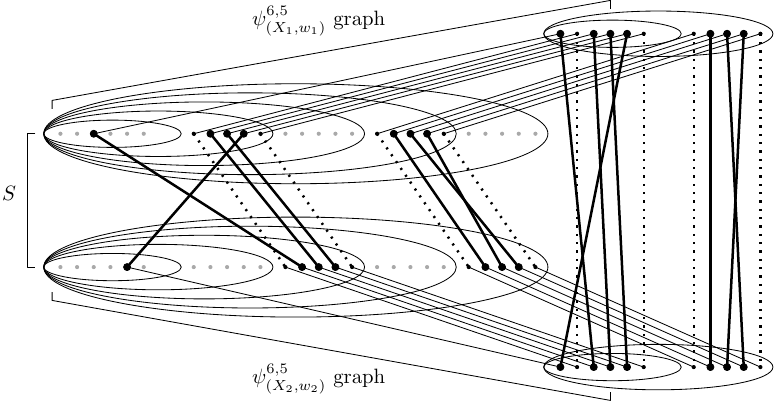}
    \caption{Extending an isomorphism graph to a full isomorphism graph between subalgebras of ${\mathbf P}^{6,5}$.}
    \label{fig:extendlocalisomorphism}
\end{figure}

The reader can consult Figure \ref{fig:extendlocalisomorphism} for an example which illustrates some of the main ideas in the proof of Lemma \ref{lem:hsimplefullsaplingssuffice}. Pictured there are four potato diagrams and some binary relations between them. The two larger potatoes on the left each represent the algebra $\mathbf{P}^{6,5}$, where dots represent elements and closed regions represent nontrivial classes of block congruences. The large black dots represent the elements of (subdirectly irreducible) subalgebras $\mathbf{B}_1, \mathbf{B}_2 \leq \mathbf{P}^{6,5}$. The smaller black dots represent the additional elements of the (subdirectly irreducible) subalgebras $\mathbf{P}^{6,5}_{(X_1, w_1)}, \mathbf{P}^{6,5}_{(X_2, w_2)}$. All other elements are depicted in gray. In this example, both $\mathbf{B}_1$ and $\mathbf{B}_2$ are isomorphic to $\mathbf{P}^{6,2}$, two copies of which are depicted as the smaller potatoes on the right. The graphs of the canonical embeddings $\psi^{6,5}_{(X_1,w_1)}$ and $\psi^{6,5}_{(X_2,w_2)}$ are depicted with ordinary lines and the graph of $S$ is depicted with bold lines. The bold vertical edges on the right of the figure result from transferring the relation $S$ to a relation on $\mathbf{P}^{6,5}$ by means of the canonical embeddings. The dotted edges on the right of the figure then follow from Claim \ref{SIclaim:specialcase} and the canonical embeddings are again used to define the dotted edges on the left, which along with the original bold edges comprise $S'$.

Finally, we prove an analogue of Lemma \ref{lem:hsimplefullsaplingssuffice} for transversal endomorphism graphs instead of isomorphism graphs between subdirectly irreducible subalgebras. 

\begin{lemma}\label{lem:fullcongruencessuffice}
    Let $n \geq 2$ and $k \geq 0$. Let $S \leq \mathbf{B} \times \mathbf{C}$ be a transversal endomorphism graph, for (subdirectly irreducible) $\mathbf{C} \leq \mathbf{B} \leq \mathbf{A}$. Then there exists $\mathbf{C}' \leq \mathbf{P}^{n,k}$ with $\mathbf{C} \leq \mathbf{C}'$ and a transversal endomorphism graph $S' \leq \mathbf{P}^{n,k} \times \mathbf{C}'$ such that $S$ equals the restriction of $S'$ to the set $B \times C$.
\end{lemma}
\begin{proof}
    Take an $S$ as in the lemma statement. Let $(\mathcal{P}_{n,k})_B$ be the sapling which represents $\mathbf{B}$. It follows from Proposition \ref{prop:subdirectlyirreduciblearelinearchains} that the kernel of the endomorphism $\lambda$ encoded by $S$ is a block congruence. By Lemma \ref{lemma:congruencesoftreealgebras} there exists $v \in (\mathcal{P}_{n,k})_B$ so that the kernel of $\lambda$ is equal to $\{(a,b) \in B: \text{ $a=b$ or $a,b \leq v$} \}$. This kernel is obviously a subset of $\sim_v = \{(a,b) \in (P^{n,k})^2: \text{ $a=b$ or $a,b \leq v$} \}$. Let $w \in C$ be the representative of the nontrivial class of the kernel of $\lambda$. Define $C' = \{w\} \cup \{a \in (P^{n,k})^2 : a \nleq v\}$ and let $\mathbf{C}'$ be the subalgebra of $\mathbf{P}^{n,k}$ with domain $C'$. Finally, let $S'$ be the transversal endomorphism graph obtained by restricting the right coordinate of $\sim_v$ to $C'$. 
    \end{proof}
    
The reader can consult Figure \ref{fig:extendtransversalendomorphism} for a picture of an extension of a particular transversal endomorphism graph of a subalgebra of $\mathbf{P}^{6,5}$. The relation $S \leq \mathbf{B} \times \mathbf{C}$ is drawn with bold black lines and the domains of $\mathbf{B}$ and $\mathbf{C}$ are drawn with bigger black dots. The extension $S'$ is drawn with dotted lines. 
\begin{figure}
    \centering
    \includegraphics[width=0.5\linewidth]{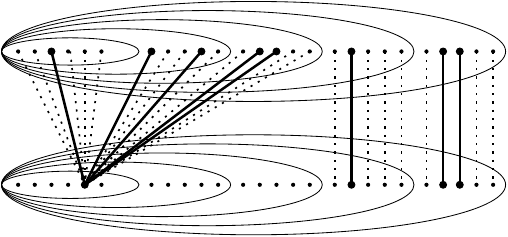}
    \caption{Extending a transversal endomorphism graph of a subalgebra of $\mathbf{P}^{6,5}$.}
\label{fig:extendtransversalendomorphism}
\end{figure}

\subsection{The class of structures $\mathfrak{P}^{n,k}$ }\label{sect:structureP_nk}
Again, we fix some $n \geq 2$. Note that $\Inv(\mathbf{P}^{n,k})$ equals the set of relations with a primitive positive 
definition over the at most ternary relations 
given by Theorem~\ref{thm:consminorityrelbasis}, for every $k \geq 0$. In this section we refine these sets of relations and label these refinements with explicit signatures, producing a sequence of structures $\mathfrak{P}^{n,0}, \mathfrak{P}^{n,1}, \dots, \mathfrak{P}^{n,k}, \dots$ which is amenable to recursive analysis.

We first explain the main concept underlying the definitions that follow. Let $k \geq 1$. Recall that $\sim_{(0)}$ is the congruence of $\mathbf{P}^{n,k}$ which collapses the elements of $P^{n,k}$ which are less than or equal to $(0)$ in $\mathcal{P}_{n,k}$, which in this case is equal to the set $P^{n,k}_0$. Notice that the complement of $P^{n,k}_0$ in $P^{n,k}$ is equal to the set $P^{n,1} \setminus \{(0)\} = \{(1), \dots, (n-1) \}$, which we will henceforth denote by $O_{n}$. We denote by $\mathbf{O}_n \leq \mathbf{P}^{n,k}$ the subalgebra with domain $O_n$ (note that this definition is independent of $k \geq 1$). We also define a structure with domain $O_n$, by first specifying the signature
\[
\tau_{O_n} = \{ R_Y: Y \subseteq O_n \} \cup \{R_\sigma: \text{ $\sigma$ is a permutation of $O_n$ } \} \cup \{\Eq_3 \}
\]
and then setting $\mathfrak{O}_n$ to be the $\tau_{O_n}$-structure with
\begin{itemize}
    \item $R_Y^\mathfrak{O} = Y$,
    \item $R_\sigma^{\mathfrak{O}} = G_\sigma$, where $G_\sigma$ is the graph of $\sigma$, and
    \item $\Eq_3^\mathfrak{O}$ is the ternary equality relation on $O_n$.
\end{itemize}

We say that a relation $R \leq X^l$ is \emph{strongly functional} if $R$ is subdirect and each coordinate of $R$ determines all other coordinates in the sense that 
for every $i \in \{1,\dots,l\}$ and every $v \in \pi_i(R)$ we have 
$|\{(t_1,\dots,t_l) \in R \mid t_i = v\}| = 1$.
Note that every unary relation is trivially strongly functional.

\begin{definition}\label{def:0uniform}
    Let $n\geq 2$ and $ k \geq 0$. We say that $R \in \Inv(\mathbf{P}^{n,k+1})$ is \emph{uniform} if 
    \begin{enumerate}
    \item
    $R$ is the disjoint union of $R$ restricted to $P^{n,k+1}_0$ and $R$ restricted to $O_n$, and
    \item $R$ restricted to $O_n$ is strongly functional. 
    \end{enumerate}
    Any $R \in \Inv(\mathbf{P}^{n,k+1})$ that is not uniform is called a \emph{nonuniform} relation. For $S \in \Inv(\mathbf{P}^{n,k})$ and $Y$ a strongly functional relation on $O_n$ of the same arity as $S$, we denote by $S \oplus Y \subseteq P^{n,k+1}$ the union of $\psi^{n,k+1}_0(S)$ and $Y$. 
\end{definition}

\begin{lemma}\label{lem:0uniformcharacterization}
    Let $n \geq 2$ and let $k \geq 0$. Let $R \subseteq (P^{n,k+1})^l$ for some $l \geq 1$. Then $R$ is a uniform invariant relation of $\mathbf{P}^{n,k+1}$ if and only if there exist  relations $S \in \Inv(\mathbf{P}^{n,k})$ and $Y$ a strongly functional relation on $O_{n}$, each of arity $l$, such that $R = S \oplus Y$.
\end{lemma}

\begin{proof}
We first prove the forward implication. Let $R \subseteq (P^{n,k+1})^l$ be a uniform invariant relation of $\mathbf{P}^{n,k+1}$. It follows immediately from the definition of a uniform relation that $R$ is the disjoint union of its restriction to $P^{n,k+1}_0$ (which we denote by $S'$) and its restriction to $O^{n,1}$ (which we denote by $Y$). It is immediate from the definition of a uniform relation that $Y$ is strongly functional. Since $\mathbf{P}^{n,k+1}$ is conservative, it follows that $S'$ is also an invariant relation of $\mathbf{P}^{n,k+1}$. Therefore, we define $S= (\psi^{n,k+1}_0)^{-1}(S')$ to obtain the desired invariant relation of $\mathbf{P}^{n,k}$.

For the backwards implication, take some arity $l$ relation $S \in \Inv(\mathbf{P}^{n,k})$ and arity $l$ strongly functional $Y$ on $O^{n,1}$, and let $R$ be the union of $\psi^{n,k+1}_0(S)$ and $Y$. We argue that $R \in \Inv(\mathbf{P}^{n,k+1})$. Recall that $m_{\mathcal{P}_{n,k+1}}$ is the conservative minority operation for $\mathbf{P}^{n,k+1}$. Consider three tuples $(a_1, \dots, a_l), (b_1, \dots, b_1), (c_1, \dots, c_l) \in R$. Set 
\[
(d_1, \dots, d_l) = (m_{\mathcal{P}_{n,k+1}}(a_1, b_1, c_1), \dots, m_{\mathcal{P}_{n,k+1}}(a_l, b_l, c_l).
\]
We argue that $(d_1, \dots, d_l) \in R$. First, we may assume that the three input tuples are pairwise distinct, because if two input tuples are equal then $(d_1, \dots, d_l)$ is equal to the remaining third input tuple. 
It follows from the definition of $R$ that each of the input tuples has entries entirely in $P^{n,k+1}_0$ or $O^{n,1}$. If all three tuples satisfy either the former or latter condition, then  $(d_1, \dots, d_l) \in R$, since clearly $\psi_0^{n,k+1}(S)$ and $Y$ are invariant relations of $\mathbf{P}^{n,k+1}$. Therefore, we are left to consider the cases where the input tuples are pairwise distinct and two have entries belonging to either $P^{n,k+1}_0$ or $O^{n,1}$, while the remaining tuple has entries belonging to $O^{n,1}$ or $P^{n,k+1}_0$, respectively. We treat two typical cases. 
\begin{itemize}
    \item Suppose $a_1, \dots, a_l, b_1, \dots, b_l \in P^{n,k+1}_0$, while $c_1, \dots, c_l \in O^{n,1}$. If $k\geq 1$, then $P^{n,k+1}_0$ is the nontrivial block of the congruence $\sim_{(0)}$, so it follows that $(d_1, \dots, d_l) = (c_1, \dots, c_l)$. If $k=0$, then $a_i = b_i$ for all $1\leq i \leq l$ and we have already argued that in this case $(d_1, \dots, d_l) = (c_1, \dots, c_l)$.
    \item Suppose $a_1, \dots, a_l, b_1, \dots, b_l \in O^{n,1} $, while $c_1, \dots, c_l \in P^{n,k+1}_0$. Since $Y$ is a strongly functional relation, it follows that the elements $a_1, \dots, a_l, b_1, \dots, b_l, c_1, \dots, c_l$ are all pairwise distinct. In this case, it is easy to see that $\{a_i, b_i, c_i \}$ is the domain of a subalgebra of $\mathbf{P}^{n,k+1}$ that is isomorphic to $\mathbf{P}^{3,1}$, hence $(d_1, \dots, d_l) = (b_1, \dots, b_l)$ (recall that the minority operation for $\mathbf{P}^{3,1}$ outputs the second entry on an input triple with pairwise distinct entries).
\end{itemize}
This finishes the proof, since it is now obvious from the definition of $R$ that it is a uniform relation. 
\end{proof}

 Now we define the structures $\mathfrak{P}^{n,k}$ along with their signatures $\tau_{n,k}$. We first specify $\mathfrak{P}^{n,0}$ and its signature. Let $\tau_{n,0} = \tau_{n,0}^U := \{  R_\emptyset, R_{\{\epsilon\}}, \Eq_2, \Eq_3 \}$ and $\mathfrak{P}^{n,0}$ be the $\tau_{n,0}$-structure with domain $\{\epsilon\}$ and relations
\begin{itemize}
    \item $R_\emptyset^{\mathfrak{P}^{n,0}} := \emptyset$ and $R_{\{\epsilon\}}^{\mathfrak{P}^{n,0}} := \{\epsilon \}$, and
    \item $\Eq_2^{\mathfrak{P}^{n,0}} $ and $\Eq_3^{\mathfrak{P}^{n,0}}$ are the binary and ternary equality relations.

\end{itemize}
Obviously, $\Inv(\mathbf{P}^{n,0}) = \Inv(\Pol(\mathfrak{P}^{n,0}))$.

Having defined the trivial structure, we proceed with a recursive definition. For the basis we take $k=1$ and set
$\tau_{n,1} = \tau_{n,1}^U \cup \tau_{n,1}^N$, where $\tau_{n,1}^{U} = \tau_{n,1}^{U,1} \cup \tau_{n,1}^{U,2} \cup \tau_{n,1}^{U,3}$ and $\tau_{n,1}^N = \tau_{n,1}^{N,2} \cup \tau_{n,1}^{N,3}$ are the sets of unary, binary, and ternary relation symbols defined as follows.

\begin{itemize}
\item $\tau_{n,1}^{U,1} := \{ (\rho_1, \rho_2): \text{ for unary $\rho_1 \in \tau_{n,0} $ and  $\rho_2 \in \tau_{O_n}$} \}$
\item $\tau_{n,1}^{U,2 } := \{ (\rho_1, \rho_2)  : \text{for binary $\rho_1 \in \tau_{n,0}$ and $\rho_2 \in \tau_{O_n}$}\}$
\item $\tau_{n,1}^{U,3} := \{ (\rho_1, \rho_2)  : \text{for ternary $\rho_1 \in \tau_{n,0}$ and $\rho_2 \in \tau_{O_n}$}\}$
\item $\tau_{n,1}^{N,2} := \{ \Full, T^{n,1}_{0,0}, R_{((0)(1))} \}$
\item $\tau_{n,1}^{N,3} := \{ \Lin \}$
\end{itemize}
Now we set $\mathfrak{P}^{n,1}$ to be the $\tau_{n,1}$-structure whose domain is $P^{n,1}$ with
\begin{itemize}
\item $(\rho_1, \rho_2)^{\mathfrak{P}^{n,1}} := \rho_1^{\mathfrak{P}^{n,0}} \oplus \rho_2^{\mathfrak{O}_n}$,
\item $\Full^{\mathfrak{P}^{n,1}} := P^{n,1} \times P^{n,1}$ is the full relation, 
\item $(T^{n,1}_{0,0})^{\mathfrak{P}^{n,1}} := T^{n,1}_{0,0} $ is the canonical transfer relation
of $\mathbf{P}^{n,1}$, 
\item $R_{((0)(1))}^{\mathfrak{P}^{n,1}}$ is the graph of the permutation of $P^{n,1}$ which switches $(0)$ and $(1)$, and
\item $\Lin^{\mathfrak{P}^{n,1}} := \{ \left((1),(1),(1)\right), \left((1),(0),(0)\right), \left((0),(1),(0)\right), \left((0),(0),(1)\right) \}$ is the set of solutions to the $\mathbb{Z}_2$-linear equation $x_1+x_2+x_3= 1$.
\end{itemize}

For the recursion, suppose that $\tau_{n,k}$ and $\mathfrak{P}^{n,k}$ have been defined for $k \geq 1$. We set $\tau_{n,k+1} = \tau_{n,k+1}^U \cup \tau_{n,k+1}^N $, where where $\tau_{n,k+1}^{U} = \tau_{n,k+11}^{U,1} \cup \tau_{n,k+1}^{U,2} \cup \tau_{n,k+1}^{U,3}$ and $\tau_{n,k+1}^N = \tau_{n,k+1}^{N,2} $ are the sets of unary, binary, and ternary relation symbols defined as follows.
\begin{itemize}
    \item $\tau_{n,1}^{U,1} := \{ (\rho_1, \rho_2): \text{ for unary $\rho_1 \in \tau_{n,0} $ and  $\rho_2 \in \tau_{O}$} \}$ 
\item $\tau_{n,1}^{U,2 } := \{ (\rho_1, \rho_2)  : \text{for binary $\rho_1 \in \tau_{n,0}$ and $\rho_2 \in \tau_{O_n}$}\}$
\item $\tau_{n,1}^{U,3} := \{ (\rho_1, \rho_2)  : \text{for ternary $\rho_1 \in \tau_{n,0}$ and $\rho_2 \in \tau_{O_n}$}\}$
    \item $\tau_{n,k+1}^{N,2} = \{ \Full, T^{n,k+1}_{ 0, k}\}$
\end{itemize}
Now we define $\mathfrak{P}^{n,k+1}$ to be the $\tau_{n,k+1}$ structure with domain $P^{n,k+1}$, where 
\begin{itemize}
\item $(\rho_1, \rho_2)^{\mathfrak{P}^{n,k+1}} := \rho_1^{\mathfrak{P}^{n,k}} \oplus \rho_2^{\mathfrak{O}_n}$,
\item $\Full^{\mathfrak{P}^{n,k+1}} := P^{n,k+1} \times P^{n,k+1}$ is the full relation, and
\item $(T^{n,k+1}_{0,k})^{\mathfrak{P}^{n,k+1}} := T^{n,k+1}_{0,k}$ is the canonical transfer relation of $\mathbf{P}^{n,k}$.
\end{itemize}
 
The signature $\tau_{n,k+1}$ consists of unary, binary, and ternary symbols which name either uniform or nonuniform $\mathbf{P}^{n,k+1}$ relations. By Definition~\ref{def:0uniform} and Lemma~\ref{lem:0uniformcharacterization}, every uniform relation $R$ is equal to $S \oplus Y$, for some matching arity relations $S \in \Inv(\mathbf{P}^{n,k})$ and strongly functional $Y$ on $O_n$. This is why we have chosen to name such relations with pairs, e.g.\ $( \Eq_2, R_{((1)(2))} ) \in \tau_{4,1}$. The collection of uniform relations is not adequate to primitively positively define all of $\Inv(\mathbf{P}^{n,k+1})$. Hence, at each stage of the recursion we include some extra nonuniform relation symbols. The kinds of added nonuniform relations needed are different for $\tau_{n,1}$ and $\tau_{n,k}$ for $k > 1$, and therefore  our recursive definition has three parts. 

Upon unpacking the definition of $\mathfrak{P}^{n,1}$, the reader will discover a structure whose relations are practically identical to the relations for $\mathbf{P}_n$ identified in~\ref{lem:brady1ishsimple}, which may lead them to wonder why the apparently complicated signature $\tau_{n,1}$ is needed. Indeed, it would be possible to choose a simpler signature for a structure which primitively positively defines $\Inv(\mathbf{P}^{n,1})$, but in our opinion this would involve a disproportionate sacrifice of clarity later on when we present our reduction scheme and prove that it is correct. Hence, our definitions are chosen so that the simplicity of the recursive reduction scheme is prioritized over the simplicity of specifying the basis structure for the recursion. The reader can consult Figure~\ref{fig:manypotatoes} for a picture of the relations for the structures $\mathfrak{P}^{3.k}$ for $0 \leq k \leq 3$. In the figure, an edge between two potato diagrams means that the top relation is  labeled by some symbol $\rho_1$ and the bottom relation is labeled by $(\rho_1, \rho_2)$ for some $\rho_2 \in \tau_{O_n}$.

\begin{figure}
    \centering
    \includegraphics[width=1\linewidth]{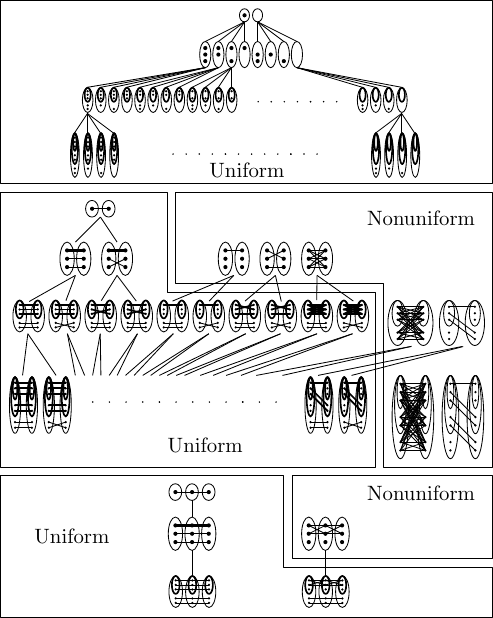}
    \caption{The relations of $\mathfrak{P}^{3,k}$ for small $k$. The unary, binary, and ternary relations are shown in the top, middle, and bottom thirds of the figure, respectively.}
    \label{fig:manypotatoes}
\end{figure}


\begin{lemma}\label{lem:refinedbasiseasydirection}
    For every $n \geq 2$ and $k \geq 0$, $\Inv(\Pol(\mathfrak{P}^{n,k}) \subseteq \Inv(\mathbf{P}^{n,k})$.
\end{lemma}

\begin{proof}
Since $\Inv(\Pol(\mathfrak{P}^{n,k}))$ equals
the set of relations that are primitively positively definable in $\mathfrak{P}^{n,k}$ (Theorem~\ref{thm:inv-pol}), it suffices to show that every relation of $\mathfrak{P}^{n,k}$ belongs to $\Inv(\mathbf{P}^{n,k})$. The proof of this proceeds inductively. We identified in Lemma~\ref{lem:brady1ishsimple} a collection of (all subsets, bijection graphs, and a ternary `linear' relation) and noted that it is easy to prove that each of these relations belongs to $\Inv(\mathbf{P}_n)$. Now, $\mathbf{P}^{n,1}$ is trivially isomorphic to $\mathbf{P}_n$, and it is similarly easy to verify that the relations of $\mathfrak{P}^{n,1}$ all belong to $\Inv(\mathbf{P}^{n,1})$. In order to gain some familiarity with the signature $\tau_{n,1}$, we will do this carefully. It is obvious that every $\tau_{n,0}$ relation of $\mathfrak{P}^{n,0}$ is an invariant relation of $\mathbf{P}^{n,0}$. For $k=1$, we consider each type of relation symbol. 

\begin{itemize}
    \item Let $(\rho_1, \rho_2) \in \tau_{n,1}^{U,1}$. By definition, $\rho_1 = R_{\emptyset} $ or $\rho_1 = R_{\{\epsilon \}}$ and $\rho_2= R_X$ for some $X \subseteq O_n$. Then $(\rho_1, \rho_2)^{\mathfrak{P}^{n,1}} = X$ or $(\rho_1, \rho_2)^{\mathfrak{P}^{n,1}} = \{(0)\} \cup X$, depending on whether $\rho_1= \emptyset$ or $\rho_1= \{\epsilon\}$, respectively. 
    \item Let $(\rho_1, \rho_2) \in \tau_{n,1}^{U,2}$. By definition, $\rho_1 = \Eq_2$ and $\rho_2 = R_\sigma$ for $\sigma$ a permutation of $O_n$, so here in fact $(\rho_1, \rho_2) = (\Eq_2, R_\sigma)$. By definition, $(\Eq_2, R_\sigma)^{\mathfrak{P}^{n,1}}$ is the graph of a permutation of $P^{n,1}$ which fixes $(0)$ and is equal to $\sigma$ on $O_n$. Hence, all $\tau_{n,1}^{U,2}$ relations of $\mathfrak{P}^{n,1}$ belong to $\Inv(\mathbf{P}^{n,1})$. 
    \item Let $(\rho_1, \rho_2) \in \tau_{n,1}^{U,3}$. By definition, $\rho_1 = \Eq_3$ and $\rho_2= \Eq_3$, so in this case $(\Eq_3, \Eq_3)^{\mathfrak{P}^{n,1}}$ is the ternary equality relation on $P^{n,1}$ and is therefore obviously among the invariant relations of $\mathbf{P}^{n,1}$. 
    \item Each of the symbols $\Full, T^{n,1}_{0,0}, R_{((0)(1))}$ and $\Lin$ is easily seen to name an invariant relation of $\mathbf{P}^{n,1}$. 

\end{itemize}
We remark that, since every relation of $\mathfrak{O}_n$ is strongly functional, we could have deduced that all relations of $\mathfrak{P}^{n,1}$ named by $\tau_{n,1}^U$ symbols belong to $\Inv(\mathbf{P}^{n,1})$ by applying Lemma~\ref{lem:0uniformcharacterization}. 

Suppose now for $k \geq 1$ that every $\tau_{n,k}$ relation of $\mathfrak{P}^{n,k}$ belongs to $\Inv(\mathbf{P}_{n,k})$. Applying Lemma~\ref{lem:0uniformcharacterization}, we conclude that all $\tau_{n,k+1}^{U}$ relations of $\mathfrak{P}^{n,k+1}$ belong to $\Inv(\mathfrak{P}^{n,k+1})$. The full relation is obviously invariant, and the canonical transfer relation $T^{n,k+1}_{0,k}$ is by definition a full isomorphism graph between (subdirectly irreducible) subalgebras $\mathbf{P}^{n,k+1}_0$ and $\mathbf{P}^{n,k+1}_k$, hence is also invariant. 
\end{proof}

To finish the section, we prove that the relations of $\mathfrak{P}^{n,k}$ primitively positively define all the relations in $\Inv(\mathbf{P}^{n,k})$. To do this, we will argue that the $\tau_{n,k}$ relations of $\mathfrak{P}^{n,k}$ primitively positively define the relations belonging to $\Inv(\mathbf{P}^{n,k})$ listed in Theorem \ref{thm:consminorityrelbasis}. We first record the main technical parts of the proof as lemmas before collecting the final results in Proposition \ref{prop:hsimpletreerecursivebasis}.

\begin{lemma}\label{lem:blockcongruencesarenamed}
Let $n \geq 2$ and $k \geq 0$. Each block congruence $\sim_v \leq (\mathbf{P}^{n,k})^2$ for $v \in \mathcal{P}_{n,k}$ is named by a basic $\tau_{n,k}$ relation symbol.  
\end{lemma}

\begin{proof}
    As expected, the case $k=0$ is trivial. We therefore proceed by induction on $k \geq 1$. When $k=1$ there are exactly two block congruences: either $\sim_{(0)}$, which is the binary equality relation, or $\sim_\epsilon$, which is the full relation. The full relation is named by the $\tau_{n,1}^N$ symbol $\Full$. The binary equality relation is named by the $\tau_{n,1}^U$ symbol $(\Eq_2, R_{\id})$, where $R_{\id}$ names the graph of the identity permutation of $O_n$ and $\Eq_2$ names the equality relation on the one element algebra $\mathbf{P}^{n,0}$. 

    For the induction step, suppose that the lemma holds for $k \geq 1$. We prove it holds also for $k+1$. Consider without loss of generality the block congruence $\sim_{0^i}$ for $0 \leq i \leq k$. If $i=0$, then $\sim_{0^i}$ is the full relation, which is named by the $\tau_{n,k+1}^N$ relation symbol $\Full$. If $i > 0$, then $\sim_{0^i}$ is a uniform relation and is named by $(\rho, R_{\id})$, where $\rho \in \tau_{n,k} $ names $ \sim_{0^{i-1}}$ and $R_{\id}$ is the graph of the identity permutation of $O_n$. 
\end{proof}

\begin{lemma}\label{lem:fullisomorphismgraphsdefinable}
    Let $n \geq 2$ and $k \geq 1$. Let $\mathbf{A}_1 = \mathbf{P}^{n,k}_{X_1, w_1}$ and $\mathbf{A}_2= \mathbf{P}^{n,k}_{X_2,w_2}$ be full (subdirectly irreducible) subalgebras of  $\mathbf{P}^{n,k}$ and let $S \leq \mathbf{A}_1 \times \mathbf{A}_2$ be an isomorphism graph. Then $S$ has a primitive positive  
    definition with the free variables $(x,y)$ which is of the form   
    \begin{align*}
    \exists z_1 \dots z_l \big( \left(\rho_1(x,z_1) \wedge \rho_2(z_1, z_2) \dots \wedge \rho_l (z_{l-1}, z_l) \wedge \rho_{l+1} (z_l, y)\right) \\
    \wedge  \left( \lambda_1(x_1) \wedge \lambda_2(z_1) \wedge \dots \wedge \lambda_{l+1}(z_l) \wedge \lambda_{l+2}(y)
    \right) \big)
    \end{align*}
    where $\rho_i \in \tau_{n,k}^{U,2} \cup \tau_{n,k}^{N,2}$ for all $1 \leq i \leq l$ and $\lambda_j \in \tau_{n,k}^{U,1}$ for all $1 \leq j \leq l+2$.
\end{lemma}

\begin{proof}
    The proof proceeds by induction on $k \geq 1$. For the basis, let $k=1$ and take $S$ as in the lemma statement. Since the permutations $((1) (2) \dots (n-1)) $ and $((0)(1))$ generate the symmetric group on $P^{n,1}$, it follows that $S$ can be defined by a formula of the kind in the lemma statement, where the unary symbols $\lambda_i$ are all equal to $(R_{\{\epsilon \}}, R_{\{(1), \dots, (n-1)\}} )$  (i.e.\ the symbol which names $P^{n,1}$) and the binary relations $\rho_i$ are either $(\Eq_2,((1) (2) \dots (n-1))) $ or $R_{((0)(1))}$ (i.e.\ the symbols which name the graphs of $((0)(1))$ and $((1)\dots(n-1))$).

    For the induction step, suppose that the lemma holds for $k \geq 1$. We will deduce that it holds for $k+1$ as well. Take $S \leq \fA_1 \times \fA_2$ for $\fA_1, \fA_2 \leq \mathbf{P}^{n,k+1}$ as in the lemma statement. We distinguish cases depending on $X_1= \Trunk(\fA_1)$ and $X_2= \Trunk(\fA_2)$. 
    \begin{itemize}
    \item
    If $\epsilon \in X_1$ and $\epsilon \in X_2$, then it is easy to see that in this case $S$ is a uniform relation. It follows from Lemma~\ref{lem:0uniformcharacterization} that $S = S' \oplus Y$, where $S'$ is a binary invariant relation of $\mathbf{P}^{n,k}$ and $Y$ is a strongly functional binary relation on $O_{n}$. Since $S$ is the graph of a one-to-one function, it follows that there exist full (subdirectly irreducible) subalgebras $\fA_1', \fA_2' \leq \mathbf{P}^{n,k}$ with $S' \leq \fA_1' \times \fA_2'$ an isomorphism graph. By the inductive hypothesis, there is a formula with the free variables $(x,w)$ of the form 
    \begin{align*}
    \exists z_1 \dots z_l \big( \left(\rho'_1(x,z_1) \wedge \rho'_2(z_1, z_2) \dots \wedge \rho'_l (z_{l-1}, z_l) \wedge \rho'_{l+1} (z_l, w)\right) \\
    \wedge  \left( \lambda_1'(x_1) \wedge \lambda_2'(z_1) \wedge \dots \wedge \lambda_{l+1}'(z_l) \wedge \lambda_{l+2}'(w)
    \right) \big)
    \end{align*}
    which defines the relation $S'$. For each $1 \leq i \leq l+1$, set $\rho_i = (\rho_i', R_{\id})$, where $\id$ is the identity permutation on $O_n$. For each $1 \leq j \leq l+2$, set $\lambda_j= (\lambda_j', R_{O_{n}})$. 
    Let $S''$ be the relation defined by
    \begin{align*}
    \chi_1(x,w) := \exists z_1 \dots z_l \big( \left(\rho_1(x,z_1) \wedge \rho_2(z_1, z_2) \dots \wedge \rho_l (z_{l-1}, z_l) \wedge \rho_{l+1} (z_l, w)\right) \\
    \wedge  \left( \lambda_{1}(x_1) \wedge \lambda_{2}(z_1) \wedge \dots \wedge \lambda_{l+1}(z_l) \wedge \lambda_{l+2}(w)
    \right) \big).
    \end{align*}
   
    Notice that $S'' = S' \oplus \Eq_2$, where $\Eq_2$ is the equality bijection graph on $O_{n}$. To finish, we simply need to apply $\sigma$ to the elements of $O_{n}$ without disturbing the other elements of $P^{n,k+1}$, where $\sigma$ is the permutation whose graph is equal to the strongly functional $Y$ from earlier. This is accomplished using the atomic formula
    \[
    \chi_2(w,y):= ( ((\dots (\Eq_2, R_{\id}), \dots R_{\id}), R_{\id}), R_\sigma)(w,y),
    \] 
    so $S$ is defined by the formula 
    \[
    \exists w  \left( \chi_1(x,w) \wedge \chi_2(w,y) \right).
    \]
    Clearly, this formula can be rewritten to have the form specified in the lemma statement. 

    \item Suppose that $\epsilon \notin X_1$ and $ \epsilon \notin X_2$. In this case we define the uniform relation $S' = (\psi^{n,k+1}_0)^{-1}(S) \oplus R_{\id}$, use the formula from the previous case for $S'$, and then update each unary relation by intersecting with $P^{n,k+1}_0$. We leave it to the reader to check these details. 

    \item Suppose that $\epsilon \notin X_1$ and $\epsilon \in X_2$, i.e., that $S$ is nonuniform. The intuition is that, along with uniform relations, only one nonuniform relation is needed to define $S$ and any maximal size nonuniform isomorphism graph could be used. We always use what we have called the canonical transfer relation $T^{n,k+1}_{0,k}$.

   So, let $X' = \Trunk(\mathbf{P}^{n,k+1}_{k}) = \{\epsilon, (0), \dots, 0, \dots,0^{k-1} \}$ and let $T^{n,k+1}_{X',(X_2,w_2)}$ be the canonical isomorphism graph between $\mathbf{P}^{n,k+1}_{k}$ and $\mathbf{P}^{n,k+1}_{(X_2,w_2)}$. Since $\epsilon \in X'$ and $\epsilon \in X_2$, we have already shown there exists a formula $\chi_1(z_2,y)$ defining $T^{n,k+1}_{X', (X_2,w_2)}$ of the kind asserted in the lemma statement. If we then set $Y'= (T^{n,k+1}_{X',(X_2,w_2)})^{-1} \circ (T^{n,k+1}_{0,k-1})^{-1} \leq \mathbf{P}_{0}^{n,k+1} \times \mathbf{P}^{n,k+1}_{(X_2,w_2)}$, it is clear that $\chi_1$ can be modified to produce a formula $\chi_2(z_1,y)$ of the kind asserted in the lemma statement which defines $Y'$. Now consider the relation $Y = S \circ (T^{n,k+1}_{X',(X_2,w_2)})^{-1} \circ (T^{n,k+1}_{0,k-1})^{-1} \leq \mathbf{P}_{(X_1,w_1)} \times \mathbf{P}_0$. Since $\epsilon \notin X_1$ and $\epsilon \notin \Trunk(\mathbf{P}_0)$, we also obtain a formula $\chi(x,z_1)$ of the kind asserted in the lemma statement which defines $Y$. Since $S = Y \circ Y'$, it follows that $\exists z_1 (\chi(x,z_1) \wedge \chi_2(z_1, y))$ defines $S$. This finishes the proof, since $\exists z_1 (\chi(x,z_1) \wedge \chi_2(z_1, y))$ can be rewritten to the form asserted in the lemma statement. \qedhere 
    \end{itemize}
\end{proof}

\begin{proposition}\label{prop:hsimpletreerecursivebasis}

For every $n \geq 2$ and $k \geq 0$, $\Inv(\mathbf{P}^{n,k}) = \Inv(\Pol(\mathfrak{P}^{n,k}))$.
\end{proposition}

\begin{proof}

The proof proceeds by proving that each set of relations is contained in the other. The case when $k=0$ is trivial, so we suppose that $k \geq 1$. We already showed that $\Inv(\Pol(\mathfrak{P}^{n,k})) \subseteq \Inv(\mathbf{P}^{n,k}) $ in Lemma~\ref{lem:refinedbasiseasydirection}. So, we prove that $\Inv(\mathbf{P}^{n,k}) \subseteq \Inv(\Pol(\mathfrak{P}^{n,k}))$ for all $k \geq 1$. Our strategy is to show that for each $k \geq 1$, the $\tau_{n,k}$ relations of $\mathfrak{P}^{n,k}$ primitively positively define the relations for $\mathbf{P}^{n,k}$ listed in Theorem \ref{thm:consminorityrelbasis}. 

\begin{itemize}
    \item As usual, it is obvious that the unary relations for $\mathbf{P}^{n,k+1}$ listed in \emph{(a$'$)} of Theorem \ref{thm:consminorityrelbasis} are definable from the unary relations of $\mathfrak{P}^{n,k+1}$.
     \item For the relations of $\mathbf{P}^{n,k+1}$ listed in \emph{(b$'$)} of Theorem \ref{thm:consminorityrelbasis}, it suffices by 
     Lemma~\ref{lem:fullcongruencessuffice} to consider only subdirect transversal endomorphism graphs $S \leq \mathbf{P}^{n,k} \times \mathbf{C}$, since all others can be obtained by restricting these to a smaller domain. However, each of these is definable in $\mathfrak{P}^{n,k}$ by 
     Lemma~\ref{lem:blockcongruencesarenamed}, since any transversal endomorphism graph can be obtained by restricting a congruence to a transversal in the second coordinate. 
    \item To show that all isomorphism graphs $G \leq \fA_1 \times \fA_2 $ for (subdirectly irreducible) $\fA_1, \fA_2 \leq \mathbf{P}^{n,k}$ (the relations for $\mathbf{P}^{n,k}$ listed in \emph{(c$'$)} of Theorem \ref{thm:consminorityrelbasis}) can be defined from the relations of $\mathfrak{P}^{n,k}$, note first that by Lemma~\ref{lem:hsimplefullsaplingssuffice} we only need to consider isomorphisms between full (subdirectly irreducible) subalgebras. We then apply Lemma~\ref{lem:fullisomorphismgraphsdefinable} to obtain the result. 
    \item Let us argue that all ternary `linear' relations (the relations listed in \emph{(d$'$)} of Theorem \ref{thm:consminorityrelbasis}) are definable in $\mathfrak{P}^{n,k}$. Let $\Lin_\mathbf{C}$ be such a relation, where $\mathbf{C} \leq \mathbf{P}^{n,k}$ is a (subdirectly irreducible) subalgebra whose monolith has nontrivial class $\{a,b\}$. Suppose without loss of generality that $\Lin_\mathbf{C}$ is such that $a$ is labeled by $0$ and $b$ is labeled by $1$. It is easy to see that the function with domain $\mathbf{C}$ which maps $a$ to $0^{k-1\frown}(0)$ and $b$ to $0^{k-1\frown}(1)$ and is the identity elsewhere is an isomorphism. Let us call the image of this isomorphism $\mathbf{B}$ and let us denote the graph of this isomorphism by $S \leq \mathbf{C} \times \mathbf{B}$. Clearly, $\Lin_\mathbf{C}$ can be defined by restricting $P^{n,k}$ to the domain of $\mathbf{B}$ and then transporting this restriction to a relation on $\mathbf{C}^3$ using the relation $S$. \qedhere 
\end{itemize}
\end{proof}

\section{Main primitive positive-construction}\label{sect:MainPP}

In this section we present the main pp-construction on which our main theorem crucially relies. To keep the presentation as simple as possible, the definitions are only given in the context of subdirectly irreducible minimal Taylor conservative minority algebras, although we remark that this machinery could be extended to minimal Taylor conservative minority algebras. In Section~\ref{sec:mutations}, we develop some technical machinery for eliminating simple but not hereditarily simple algebras from a conservative minority tree, which will culminate in a pp-construction. In Section~\ref{sec:recursiveconstruction}, we recursively apply this pp-construction until all problematic local algebras are eliminated. 

\subsection{Mutations}\label{sec:mutations}

Let $\mathcal{T}$ be a sapling full reduced conservative minority tree over a collection of simple minimal Taylor conservative minority algebras $S = \{\mathbf{A}_1, \dots, \mathbf{A}_s \} \cup \{\mathbf{B}\}$. We distinguish the local algebra $\mathbf{B}$, since our aim is to produce a conservative minority tree $\mathcal{T}^{\mathbf{B} \downarrow}$ in which $\mathbf{B}$ does not appear as a local algebra, such that the corresponding structure $\mathfrak{S}_{\mathbf{A}_{\mathcal{T}^{\mathbf{B} \downarrow}}}$ pp-constructs $\mathfrak{S}_{\mathbf{A}_\mathcal{T}}$. Our first task is to define this tree, which we call the \emph{$\mathbf{B}$-unpacking transformation} of $\mathcal{T}$. We remark that this definition does not always produce a pp-construction, but we will later prove that it works provided that $\mathbf{B}$ is not hereditarily simple and has maximal cardinality among the local algebras of $\mathcal{T}$ which are not hereditarily simple. Suppose that the domain of $\mathbf{B}$ is $B = \{b_0, \dots, b_{n-1}\}$ and that 
$
| \{ v \in \mathcal{T}: \mathbf{v}^+ = \mathbf{B}\}| = m.
$ 
We call the algebra $(\mathbf{P}_{n+1})^m$ the \emph{algebra of genes} for $\mathcal{T}$ and $ \mathbf{B}$. A particular tuple $i=(i_1, \dots, i_{m}) \in \{0, \dots, n\}^m$ is called a \emph{gene}. 


\begin{definition}\label{def:splicing}
Let $\mathcal{T}$ be a sapling reduced conservative minority tree over a collection of simple minimal Taylor conservative minority algebras $S = \{\mathbf{A}_1, \dots, \mathbf{A}_s \} \cup \{\mathbf{B}\}$. Suppose that the domain of $\mathbf{B}$ is $B = \{b_0, \dots, b_{n-1}\}$ and 
$
| \{ v \in \mathcal{T}: \mathbf{v}^+ = \mathbf{B}\}| = m.
$  Let $\mathbf{B}_j \leq \mathbf{B}$ be the subalgebra of $\mathbf{B}$ whose domain is $B \setminus \{b_j\}$, for each $0 \leq j \leq n-1$. Suppose that $w= (w_1, \dots, w_p) \in \mathcal{T}$ is vertex $\mathcal{T}$ and $i \in (\mathbf{P}_{n+1})^m$ is a gene. The \emph{$(w,i)$-splice} is the tuple $w^i$ obtained by inserting in order the entries in $i$ in front of each occurrence of an element $b_j$ belonging to $B$ and changing $b_j$ to $j$ in the event that $i_k = n$. More specifically, we decompose $w$ as a concatenation of tuples 
 \[
    w = u_1^\frown(b_{j_1})^\frown
u_2^\frown(b_{j_2})^\frown 
\dots ^\frown (b_{j_l})^\frown
u_{l +1},
    \] 
where $1 \leq l \leq m$, $b_{j_k} \in B$ for each $1 \leq k \leq l$, and $u_k \in (A_1 \cup \dots \cup A_s)^{< \omega}$ for each $1\leq k \leq l$ (note that some of the $u_k$ could be the empty tuple). The $(w,i)$-splice is then defined to be the tuple 
\[
w^i:= u_1^\frown(i_1)^\frown(b'_{j_1})^\frown
u_2^\frown(i_2)^\frown(b'_{j_2})^\frown 
\dots ^\frown (i_l)^\frown(b'_{j_l})^\frown
u_{l +1},
\]
such that for each $1 \leq k \leq l$,
\[
b'_{j_k} = 
\begin{cases}
    b_{j_k} &\text{ if $0\leq i_k \leq n-1$}\\
    j_k     &\text{ if $i_k = n$}.
\end{cases}
\]
Given the definition of a splice, we further define the following terminology and objects. 
\begin{enumerate}
    \item If $i_j \neq j_k$ for all $1 \leq k \leq l$, we say that the gene $i$ is \emph{valid} for $w$ and call $w^i$ a \emph{mutation} of $w$. For each $w \in \mathcal{T}$, we define the set of \emph{$\mathbf{B}$-mutations} of $w$ as
    \[
    \Mut(w):= \{w^i: i \in (\mathbf{P}_{n+1})^m \text{ is valid for $w$} \}.
    \]
    We then define the \emph{$\mathbf{B}$-unpacking} of $\mathcal{T}$ to be the unique conservative minority tree $\mathcal{T}^{\mathbf{B} \downarrow}$ over $S^{\mathbf{B} \downarrow}: = \{\mathbf{A}_1, \dots, \mathbf{A}_s \} \cup \{\mathbf{B}_1, \dots, \mathbf{B}_n\} \cup \{\mathbf{P}_{n}, \mathbf{P}_{n+1}\}$ whose leaves are 
    \[
    A_{\mathcal{T}^{\mathbf{B}\downarrow}} := \bigcup_{w \in A_\mathcal{T}} \Mut(w)
    \]
    \item We say that a gene $i$ is \emph{valid} for a subset $D \subseteq A_\mathcal{T}$ if $i$ is valid for every $w \in D$. If $i$ is valid for $D$, then we define the \emph{$\mathbf{B}$-mutation} of $D$ to be the set $D^i:= \{w^i: w\in D\}$. The \emph{minimal gene} for $D$ is the lexicographically least gene $i$ among the valid genes for $D$. If $i$ is a minimal gene for $D$, then we say that $D^i$ is a \emph{minimal $\mathbf{B}$-mutation} of $D$. 
    
    When it is useful to emphasize that these sets are the domains of a subalgebras, we may write $\mathbf{D}$ and $\mathbf{D}^i$ instead of $D$ and $D^i$.

    \item We say that an $M$-tuple of genes $(i_1, \dots, i_M)$ is \emph{valid} for an $M$-ary relation $R \subseteq (A_\mathcal{T})^M$ if $i_j$ is valid for $\pi_j(R)$, for each $1 \leq j \leq M$. If $(i_1, \dots, i_M)$ is valid for $R$, then we define a \emph{$\mathbf{B}$-mutation} of $R$ to be the relation 
    $R^{(i_1, \dots, i_m)}:= \{(w_1^{i_1}, \dots, w_M^{i_M}): (w_1, \dots, w_M) \in R \}$. The tuple $(i_1, \dots, i_M)$ which is lexicographically minimal in each coordinate and is valid for $R$ is called the \emph{minimal gene sequence}  for $R$.

\end{enumerate}

\end{definition}

\begin{figure}
    \centering
    \includegraphics[width=.98\linewidth]{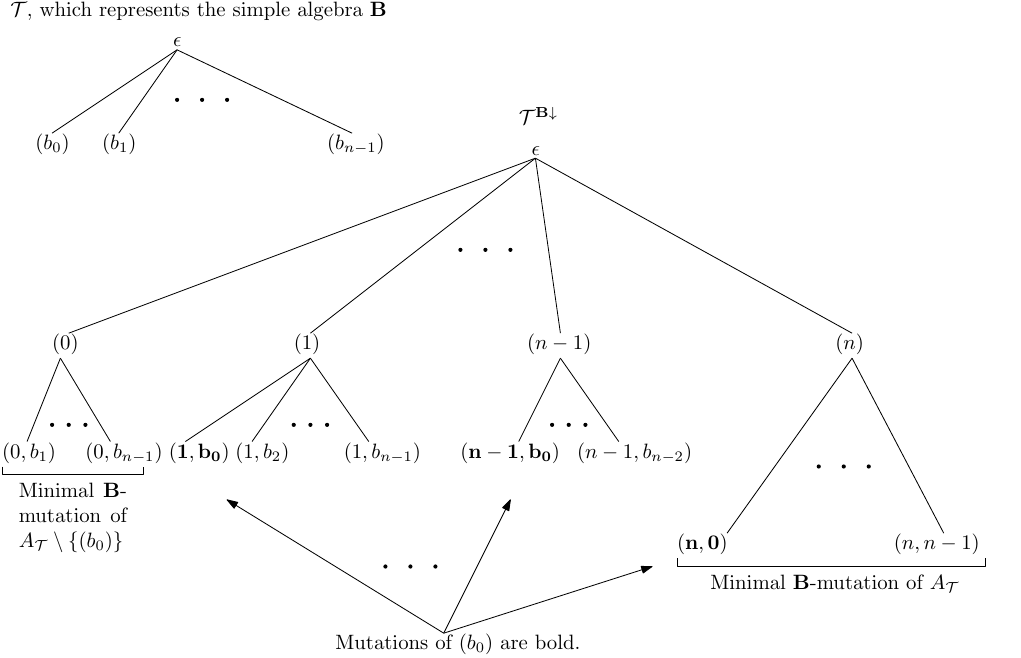}
    \caption{Unpacking a tree representing a simple algebra $\mathbf{B}$}
    \label{fig:unpacking1}
\end{figure}

We discuss Definition~\ref{def:splicing}, first with the aid of Figure~\ref{fig:unpacking1}, where the reduced conservative minority tree $\mathcal{T}$ represents a simple algebra $\mathbf{B}$. Hence, $\mathcal{T}$ is a trivial height one tree whose leaves consist of length one tuples of elements coming from $B = \{b_0, \dots, b_{n-1}\}$. Not depicted in the figure is the algebra of genes, which is in this case $\mathbf{P}_{n+1}$, since $\mathbf{B}$ occurs exactly once as a successor algebra of $\mathcal{T}$. The $\mathbf{B}$-unpacking $\mathcal{T}^{\mathbf{B}\downarrow}$ is depicted below $\mathcal{T}$. There, the leaves all of the valid splices of the leaves of $\mathcal{T}$, which we are calling $\mathbf{B}$-mutations. For example, the mutations of $(b_0)$ are depicted in bold. The successor algebra $\epsilon^{+_{\mathcal{T}^{\mathbf{B} \downarrow}}}$ is equal to the projection minority algebra $\mathbf{P}_{n+1}$, the successor algebra $\mathbf{(j)}^{+_{\mathcal{T}^{\mathbf{B} \downarrow}}}$ is equal to $\mathbf{B}_j$ for each $0 \leq j \leq n-1$, and the successor algebra $\mathbf{(n)}^{+_{\mathcal{T}^{\mathbf{B} \downarrow}}}$ is equal to $\mathbf{P}_{n}$. Hence, the unique maximal congruence $\lambda$ of the algebra $\mathbf{A}_{\mathcal{T}^{\mathbf{B} \downarrow}}$ has $(n+1)$-many classes, where each of the first $n$ of the $\lambda$-classes encloses one of the maximal proper subalgebra $\mathbf{B}_j$ of $\mathbf{B}$, and the other remaining class encloses $\mathbf{P}_n$. 

\begin{figure}
    \centering
    \includegraphics[width=0.85\linewidth]{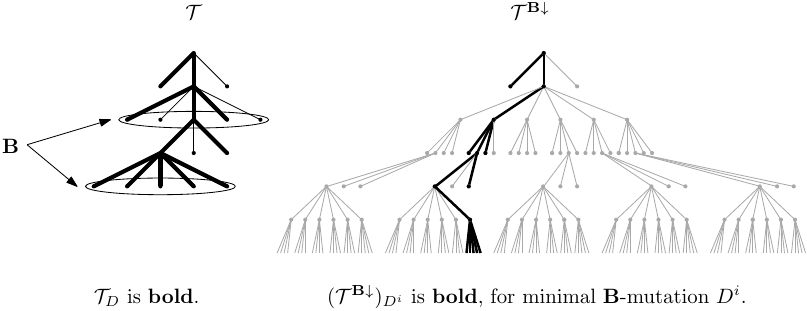}
    \caption{Unpacking a tree with two successor algebras equal to $\mathbf{B}$}
    \label{fig:unpacking2}
\end{figure}

The picture of $\mathcal{T}^{\mathbf{B} \downarrow}$ is more complicated when the number of successor algebras in $\mathcal{T}$ is greater than one, since the number of genes grows exponentially. Nevertheless, we provide one in Figure~\ref{fig:unpacking2}. The reader will find on the left of the figure a picture of a conservative minority tree with two successor algebras equal to $\mathbf{B}$. On the right is a complete picture of the $\mathbf{B}$-unpacking transformation of $\mathcal{T}$. A subtree $\mathcal{T}_D$ of $\mathcal{T}$ determined by some $\mathbf{D} \leq \mathbf{A}_\mathcal{T}$ is indicated with bold lines on the left of the figure. Also in bold, we indicate in the $\mathbf{B}$-unpacking of $\mathcal{T}$ the subtree determined by the minimal $\mathbf{B}$-mutation of $\mathbf{D}$. 

The following sequence of lemmas develops some important properties of the $\mathbf{B}$-unpacking transformation $\mathcal{T}^{\mathbf{B}\downarrow}$ and its connection to the invariant relations of $\mathbf{A}_\mathcal{T}$.

\begin{lemma}\label{lem:B-fullnessispreservedbyIsomorphism}
     Let $\mathcal{T}$ be a sapling full reduced conservative minority tree over a pairwise nonisomorphic collection of simple minimal Taylor conservative minority algebras $S = \{\mathbf{A}_1, \dots, \mathbf{A}_s \} \cup \{\mathbf{B}\}$. Suppose that domain of $\mathbf{B}$ is $B = \{b_0, \dots, b_{n-1}\}$, that $\mathbf{B}$ is not hereditarily simple, and that $|B|$ is an upper bound to the set of cardinalities
     \[
     \{ |A_i| : 1 \leq i \leq s \text{ and $\mathbf{A}_i$ is not hereditarily simple }\}.
     \] Let $\mathbf{U}_1$ and $\mathbf{U}_2$ be isomorphic subalgebras of $\mathbf{A}_\mathcal{T}$ have domains $U_1$ and $U_2$, respectively. Let $\psi: \mathbf{U}_1 \to \mathbf{U}_2$ be an isomorphism. Let $u,v,w \in U_1$. Suppose that $d_1 = \bigvee \{u,v,w \}$ is such that $\mathbf{d}_1^{+_{\mathcal{T}_{U_1}}}$ is isomorphic to $\mathbf{B}$. The following hold. 

     \begin{enumerate}
         \item The vertex $d_1 \in \mathcal{T}_{U_1}$ is full in $\mathcal{T}$ (meaning $\mathbf{d}_1^{+_{\mathcal{T}_{U_1}}}=\mathbf{d}_1^{+_{\mathcal{T}}}$) and $\mathbf{d}_1^{+_{\mathcal{T}_{U_1}}} = \mathbf{B}$.
         \item The above property is preserved by the isomorphism $\phi$. Specifically, we mean that the vertex $d_2= \bigvee \{\psi(u), \psi(v), \psi(w) \}$ is also full in $\mathcal{T}$ and $\mathbf{d}_2^{+_{\mathcal{T}_{U_2}}} = \mathbf{B}$.
     \end{enumerate}
     
\end{lemma}

\begin{proof}
   The proof of \emph{1.}\ is more or less immediate. If $\mathbf{d}_1^{+_{\mathcal{T}_{U_1}}}$ is not equal to $\mathbf{B}$, then it must be a subalgebra of one of the $\mathbf{A}_1, \dots, \mathbf{A}_s$. Now, $\mathbf{B}$ is not hereditarily simple, so $\mathbf{d}_1^{+_{\mathcal{T}_{U_1}}}$ cannot be a subalgebra of a hereditarily simple local algebra. Since $\mathbf{B}$ has maximal cardinality among the local algebras of $\mathcal{T}$ which are not hereditarily simple, it would follow that some $\mathbf{A}_k$  is isomorphic to $ \mathbf{B}$, which is impossible since we assume that the local algebras of $\mathcal{T}$ are pairwise nonisomorphic. Therefore, the successor algebra $\mathbf{d}_1^{+_{\mathcal{T}_{U_1}}}$ is equal to $\mathbf{B}$, and again the assumption that $\mathbf{B}$ has maximum size among the local algebras of $\mathcal{T}$ forces $d_1$ to be full in $\mathcal{T}$. 

   To prove \emph{2.}, we first observe that the definition of $d_1$ implies that the least block congruence $\theta_1$ of $\mathbf{U}_1$ whose nontrivial class contains $\{u,v,w\}$ is equal to $\sim_{}d_1$ restricted to $U_1$. Indeed, this follows from Lemma~\ref{lem:CongruencesofSubalgebras} along with the fact that $\mathbf{B}$ is simple. Now, let $D_1$ be the nontrivial class of $\theta_1$ and let $\mathbf{D}_1 \leq \mathbf{U}_1$ be the corresponding subalgebra. It follows from Lemma~\ref{lem:subalgebrasuccessor} and the simplicity of $\mathbf{B}$ that $\sim_{C_{d_1}}$ restricted to $\mathbf{D}_1$ is the unique maximal congruence of $\mathbf{D}_1$. Putting this all together, we see that $\{u,v,w\}$ satisfy the following property:
   \begin{itemize}
       \item Let $D_1$ be the nontrivial class of the least block congruence $\theta_1$ collapsing $\{u,v,w\}$. The quotient of the corresponding subalgebra $\mathbf{D}_1$ by its unique maximal congruence is isomorphic to $\mathbf{B}$. 
   \end{itemize}
   The above property is obviously preserved by isomorphisms. So, let $\theta_2$ be the least block congruence of $\mathbf{U}_2$ which collapses $\{\psi(u), \psi(v), \psi(w) \}$, let $d_2 = \bigvee \{\psi(u), \psi(v), \psi(w) \}$, and let $D_2$ be the nontrivial $\theta_2$-class. By Lemma~\ref{lem:CongruencesofSubalgebras}, it follows that $d_2^{+_{\mathcal{T}_{D_2}}}$ is the nontrivial class of $\overline{\theta_2}$ of the subalgebra $\mathbf{d}_2^{+_{\mathcal{T}_{U_2}}} \leq \mathbf{C}$ for some $\mathbf{C}$ among the $\mathbf{A}_1, \dots, \mathbf{A}_s, \mathbf{B}$. 
   
       We will show that $\mathbf{C} = \mathbf{B}$. 
       To see this, we will first argue that there is a subalgebra of $\mathbf{d}_2^{+_{\mathcal{T}_{U_2}}} \leq \mathbf{C}$ which is isomorphic to $\mathbf{B}$. Indeed, let $U_2^{<d_2} = \{ u \in U_2: u \leq_\mathcal{T} d_2\}$. By Lemma~\ref{lem:subalgebrasuccessor}, the subalgebra prefix successor map $\phi$ maps $\mathbf{U}_2$ onto $\mathbf{d}_2^{+_{\mathcal{T}_{U_2}}}$ and the kernel of $\phi$ is $\sim_{C_{d_2}}$. Then $\phi(D_2)$ is a subalgebra of $\mathbf{d}_2^{+_{\mathcal{T}_{U_2}}}$ which is isomorphic to $D_2/ \mu$, where $\mu$ is the restriction of $\sim_{C_{d_2}}$ to $D_2$. 
       
       Let $\lambda$ be the unique maximal congruence of $D_2$. We know that $\mathbf{D}_2/\lambda$ is isomorphic to $\mathbf{B}$, and since $\lambda$ is maximal, it follows that $\mu \subseteq \lambda$. It follows from the correspondence theorem that $\phi(\lambda)$ is a congruence of $\phi(\mathbf{D}_2)$ and that $\phi(\mathbf{D}_2)/\phi(\lambda)$ is isomorphic to $\mathbf{B}$. Since $\phi(\mathbf{D}_2)$ is conservative, any set of representatives of the $\phi(\lambda)$-classes is a subalgebra of $\phi(\mathbf{D}_2)$, which is evidently isomorphic to $\mathbf{B}$. Hence, we have argued that there is a subalgebra of $\mathbf{d}_2^{+_{\mathcal{T}_{U_2}}} \leq \mathbf{C}$ which is isomorphic to $\mathbf{B}$. 

       Now we can finish the proof. If $\mathbf{C}$ is a local algebra of $\mathcal{T}$ which is not hereditarily simple, then it must be equal to $\mathbf{B}$. Indeed, suppose that $\mathbf{C}$ is hereditarily simple. We assume that all such local algebras have size at most $\mathbf{B}$, so the fact that $\mathbf{d}_2^{+_{\mathcal{T}_{U_2}}} $ has a subalgebra isomorphic to $\mathbf{B}$ forces $\mathbf{d}_2^{+_{\mathcal{T}_{U_2}}} = \mathbf{C}$, and this is isomorphic to $\mathbf{B}$. Since we assume the local algebras of $\mathcal{T}$ are pairwise nonisomorphic, this case forces $\mathbf{B} = \mathbf{C}$. The other possibility is that $\mathbf{C}$ is a hereditarily simple local algebra of $\mathcal{T}$, but this is impossible, since $\mathbf{B}$ is not hereditarily simple. 
\end{proof}

\begin{lemma}\label{lem:subalgebraIsomorphismMutationLemma}
    Let $\mathcal{T}$ be a sapling full reduced conservative minority tree over a pairwise nonisomorphic collection of simple minimal Taylor conservative minority algebras $S = \{\mathbf{A}_1, \dots, \mathbf{A}_s \} \cup \{\mathbf{B}\}$. Suppose that domain of $\mathbf{B}$ is $B = \{b_0, \dots, b_{n-1}\}$, that $\mathbf{B}$ is not hereditarily simple, and that $|B|$ is an upper bound to the set of cardinalities
     \[
     \{ |A_i| : 1 \leq i \leq s \text{ and $\mathbf{A}_i$ is not hereditarily simple}\}.
     \]  Let $\mathbf{U}_1$ and $\mathbf{U}_2$ be isomorphic subalgebras of $\mathbf{A}_\mathcal{T}$, let $i_1$ and $i_2$ be their respective minimal genes, and let $G \leq \mathbf{U}_1 \times \mathbf{U}_2$ be a subdirect relation which is the graph of an isomorphism. Then $G^{(i_1, i_2)}$ is an invariant relation of $\mathbf{A}_{\mathcal{T}^{\mathbf{B}\downarrow}}$. 

\end{lemma}

\begin{proof}
    We call the isomorphism of which $G$ is a graph $\psi$ and we call the mapping of which $G^{(i_1,i_2)}$ is a graph $\psi^{(i_1,i_2)}$. We will be finished if we can show that $\psi^{(i_1, i_2)}$ is an isomorphism between the minimal $\mathbf{B}$-mutated subalgebras $\mathbf{U}_1^{i_1} $ and $\mathbf{U}_2^{i_2}$. First, observe that $\psi^{(i_1,i_2)}$ is a bijection. Indeed, $\psi$ is a bijection and it follows easily from the definition of a $(w,i)$-splice (see Definition~\ref{def:splicing}) that the mappings which send an element of $u_1 \in \mathbf{U}_1$ to its mutation $u_1^{i_1}$ and an element $u_2 \in \mathbf{U}_1$ to its mutation $u_2^{i_2}$ are bijections. Hence, $\psi^{(i_1,i_2)}$ is a bijection. 

    So, we just need to show that $\psi^{(i_1,i_2)}$ satisfies the homomorphism property. Take three-elements $u^{i_1},v^{i_1},w^{i_1} \in D_1^{i_1}$, so $u,v,w \in D_1$. Let $d_1 = \bigvee \{u,v,w\}$ and $d_2 = \bigvee \{\psi(u), \psi(v), \psi(w) \}$. We treat two cases.
    \begin{itemize}
        \item Suppose that $\mathbf{d}_1^{+_{\mathcal{T}_{U_1}}} = \mathbf{B}$. By Lemma~\ref{lem:B-fullnessispreservedbyIsomorphism}, it follows that $\mathbf{d}_2^{+_{\mathcal{T}_{U_2}}} = \mathbf{B}$. So, $\mathcal{T}_{U_1}$ is full at $d_1$ and $\mathcal{T}_{U_2}$ is full at $d_2$, and their respective successor algebras are equal to $\mathbf{B}$. First, this means there exist $j_1, j_2, j_3,  \in \{0, \dots, n-1 \}$ and $j'_1, j'_2, j'_3,  \in \{0, \dots, n-1\}$ such that 
        \begin{align*}
            u^{d_1^+}&=  b_{j_1} &&\psi(u)^{d_2^+} =  b_{j'_1}\\
            v^{d_1^+}&=  b_{j_2} &&\psi(v)^{d_2^+} =  b_{j'_2}\\
            w^{d_1^+}&=  b_{j_3} &&\psi(w)^{d_2^+} =  b_{j'_3}.
        \end{align*}
    By the definition of a valid gene, it follows that the entries of $i_1$ and $i_2$ that are inserted after $d_1$ and $d_2$ respectively must each be $n$. Hence, we observe that
    \begin{align*}
            (u^{i_1})^{+_{d_1^*}}&= j_1 &&(\psi(u)^{i_2})^{+_{d_2^*}} = j'_1\\
            (v^{i_1})^{+_{d_1^*}}&= j_2 &&(\psi(v)^{i_2})^{+_{d_2^*}} = j'_2\\
            (w^{i_1})^{+_{d_1^*}}&= j_3 &&(\psi(w)^{i_2})^{+_{d_2^*}} = j'_3, 
        \end{align*}
    where $d_1^*:=\bigvee \{u^{i_1},v^{i_1},w^{i_1}\} = d_1^\frown (n)$ and  $d_2^*:= \bigvee \{\psi(u)^{i_2}, \psi(v)^{i_2}, \psi(w)^{i_2} \} = d_2^\frown(n)$. 

    We want to see that 
    \[
    \psi^{(i_1, i_2)}(m_{\mathcal{T}^{\mathbf{B} \downarrow}} ( u^{i_1},v^{i_1},w^{i_1})) = m_{\mathcal{T}^{\mathbf{B} \downarrow}}(  \psi^{(i_1, i_2)}(u^{i_1}), \psi^{(i_1, i_2)}(v^{i_1}), \psi^{(i_1, i_2)}(w^{i_1})),
    \]
    but this is now clear, since the of $m_{\mathcal{T}^{\mathbf{B} \downarrow}}$ involves consulting the $d_1^*$ or $d_2^*$ successor algebras. In the present case, each successor algebra is equal to the projection minority algebra $\mathbf{P}_n$ and all bijection graphs on $\mathbf{P}_n$ are invariant, so the homomorphism property for $\psi^{(i_1,i_2)}$ follows.
    \item Suppose that $\mathbf{d}_1^{+_{\mathcal{T}_{U_1}}} \neq \mathbf{B}$. By Lemma~\ref{lem:B-fullnessispreservedbyIsomorphism}, it follows that $\mathbf{d}_2^{+_{\mathcal{T}_{U_2}}} \neq \mathbf{B}$ also. It follows from the definition of a minimal gene (Definition~\ref{def:splicing}) that the successor algebras used to compute $m_{\mathcal{T}^{\mathbf{B} \downarrow}} ( u^{i_1},v^{i_1},w^{i_1})$ and $m_{\mathcal{T}^{\mathbf{B} \downarrow}}(  \psi^{(i_1, i_2)}(u^{i_1}), \psi^{(i_1, i_2)}(v^{i_1}), \psi^{(i_1, i_2)}(w^{i_1}))$ are respectively equal to the successor algebras used to compute $m_\mathcal{T}(u,v,w)$ and $m_\mathcal{T}(\psi(u), \psi(v), \psi(w))$, so the homomorphism property for $\psi^{(i_1,i_2)}$ follows in this case also. 
    \qedhere 
    \end{itemize} 
\end{proof}

\begin{lemma}\label{lem:MutationsofBlockCongruencesAreInvariant}
    Let $\mathcal{T}$ be a sapling full reduced conservative minority tree over a collection of simple minimal Taylor conservative minority algebras $S = \{\mathbf{A}_1, \dots, \mathbf{A}_s \} \cup \{\mathbf{B}\}$, where $\mathbf{B}$ has domain $B = \{b_0, \dots, b_{n-1}\}$. Let $\mathcal{T}^{\mathbf{B}\downarrow}$ be the $\mathbf{B}$-unpacking of $\mathcal{T}$. Suppose that $\mathbf{U} \leq \mathbf{A}_\mathcal{T}$ has a block congruence $\theta$ whose nontrivial class is $D$. If $i$ is the minimal gene for $\mathbf{U}$, then $\theta^{(i,i)}$ is a block congruence of $\mathbf{U}^i$ whose nontrivial class is $D^i$. 
\end{lemma}

\begin{proof}
    Let the domain of $\mathbf{U}$ be $U = D \cup V$, where $D = \{d_1, \dots, d_k\}$ is the unique nontrivial $\theta$-block, and $V = \{v_1, \dots, v_l \}$. Let $d = \bigvee D$, i.e.\ $d$ is the maximal prefix common of every element of $D$. Applying the forward direction of Lemma~\ref{lem:CongruencesofSubalgebras}, we obtain that $d^{+_{\mathcal{T}_D}}$ is the nontrivial block of a block congruence $\overline{\theta}$ of the subalgebra $\mathbf{d}^{+_{\mathcal{T}_U}}$. On the other hand, consider the equivalence relation $\theta^{(i,i)}$ over $U^i = D^i \cup V^i$. Observe that $\theta^{(i,i)}$ has exactly one nontrivial class $D^i$. We will apply the backwards direction of Lemma~\ref{lem:CongruencesofSubalgebras} to conclude that $\theta^{(i,i)}$ is a block congruence of $\mathbf{U}^i$. Let $d^*$ be the least upper bound of $D^i$ in $\mathcal{T}^{\mathbf{B} \downarrow}$. We treat two cases.

    \begin{itemize}
        \item Suppose that $\mathbf{d}^{+_\mathcal{T}} \in \{\mathbf{A}_1, \dots, \mathbf{A}_s \}$, i.e.\, that $\mathbf{d}^{+_\mathcal{T}}$ is not equal to $\mathbf{B}$. In this case, it is easy to see that the maximal common prefix among the elements of $D^i$ is exactly $d^i$, i.e.\ that 
    \[
    d^* = \bigvee D^i = \left(\bigvee D \right)^i = d^i.
    \]
    Furthermore, although the $\mathbf{B}$-mutation transformation may change the isomorphism type of $\mathbf{D}$, it does not perturb the local structure which makes $\theta$ a block congruence. Specifically, we observe that 
    \begin{align*}
    (\mathbf{d^*})^{+_{\mathcal{T}^\mathbf{B} \downarrow}} &= \mathbf{d}^{+_\mathcal{T}},\\
    (\mathbf{d}^*)^{+_{(\mathcal{T}^{\mathbf{B} \downarrow})_{U^i}}}&= \mathbf{d}^{+_{\mathcal{T}_U}} , \text{ and}\\
     (\mathbf{d}^*)^{+_{(\mathcal{T}^{\mathbf{B} \downarrow})_{D^i}}}&= \mathbf{d}^{+_{\mathcal{T}_D}} .
    \end{align*}
    We deduced from the forward direction of Lemma~\ref{lem:CongruencesofSubalgebras} that $\mathbf{d}^{+_{\mathcal{T}_D}} $ is the nontrivial class of a block congruence $\overline{\theta}$ of $\mathbf{d}^{+_{\mathcal{T}_U}} $, so in view of the above equalities, we can apply the backwards direction of the lemma to conclude that the equivalence relation $\theta^i$ is a congruence of $\mathbf{D}^i$. 

    \item Suppose that $\mathbf{d}^{+_\mathcal{T}} = \mathbf{B}$. In this case, we have  $d^{+_{\mathcal{T}_U}} \subseteq B$, and it is straightforward to see that the maximal common prefix among elements of $D^i$ is
    \[
     d^*:=\bigvee U^i  = 
     \begin{cases}
         (d^i)^\frown(j) &\text{if $d^{+_{\mathcal{T}_U}} \subsetneq B$, where $0\leq j \leq n-1$ is minimal such that $b_j \notin U^{C_u}$} \\
         (d^i)^\frown(n) &\text{if $d^{+_{\mathcal{T}_U}} = B$ (i.e.\ $\mathcal{T}_U$ is full in $\mathcal{T}$ at $d$}).
     \end{cases}
    \]
    In either case, the backwards direction of Lemma~\ref{lem:CongruencesofSubalgebras} can again be applied to conclude that $\theta^i$ is invariant. Here are the details of each case.

    \begin{itemize}
        \item If $d^* = (d^i)^\frown(j)$ for some $0 \leq j \leq n-1$, then it again follows from the definition of the local structure of $\mathcal{T}^{\mathbf{B}\downarrow}$ that 
         \begin{align*}
    (\mathbf{d^*})^{+_{\mathcal{T}^\mathbf{B} \downarrow}} &= \mathbf{d}^{+_\mathcal{T}},\\
    (\mathbf{d}^*)^{+_{(\mathcal{T}^{\mathbf{B} \downarrow})_{U_i}}}&= \mathbf{d}^{+_{\mathcal{T}_U}} , \text{ and}\\
     (\mathbf{d}^*)^{+_{(\mathcal{T}^{\mathbf{B} \downarrow})_{D^i}}}&= \mathbf{d}^{+_{\mathcal{T}_D}} .
    \end{align*}
    so we conclude that $\theta^i$ is a congruence of $\mathbf{D}^i$ like we did above. 
    \item On the other hand, if $d^* = (d^i)^\frown (n)$, then $(\mathbf{d}^*)^{+_{(\mathcal{T}^{\mathbf{B} \downarrow})_{U_i}}} = \mathbf{P}_n$. This is the only situation where something could go wrong, since the relevant local algebra has been changed from $\mathbf{B}$ to $\mathbf{P}_n$. However, this is not a problem, because we know that $\mathbf{d}^{+_{\mathcal{T}_D}}$ is the nontrivial class of a congruence $\overline{\theta}$ on $\mathbf{d}^{+_{\mathcal{T}_U}}  = \mathbf{B}$, which is simple by assumption. Therefore, $\mathbf{d}^{+_{\mathcal{T}_D}} = \mathbf{d}^{+_{\mathcal{T}_U}} = \mathbf{d}^{+_\mathcal{T}}$, and it follows that $\theta$ is equal to $\sim_{u}$ restricted to $D$. We conclude that $\theta^i$ is equal to $\sim_{u^*}$ restricted to $D^i$.  \qedhere 
    \end{itemize}
    \end{itemize}
\end{proof}

\begin{corollary}\label{cor:importantBlockCongruenceMutationCorollary}
     Let $\mathcal{T}$ be a sapling reduced full conservative minority tree over a collection of simple minimal Taylor conservative minority algebras $S = \{\mathbf{A}_1, \dots, \mathbf{A}_s \} \cup \{\mathbf{B}\}$, where $\mathbf{B}$ has domain $B = \{b_0, \dots, b_{n-1}\}$. Let $\mathcal{T}^{\mathbf{B}\downarrow}$ be the $\mathbf{B}$-unpacking of $\mathcal{T}$. Suppose that $\theta$ is a block congruence of a subalgebra $\mathbf{U} \leq \mathbf{A}_\mathcal{T}$ whose domain $U = D \cup V$, where $D$ is the nontrivial $\theta$-class. Let $D_1 \subseteq D$, $D_2 \subseteq D$, and $D_3 \subseteq U$ be nonempty, and let $i_1$, $i_2$, and $i_3$ be the minimal genes for $U_1 = D_1 \cup V$, $U_2 = D_2 \cup V$, and $U_3 = D_3 \cup V$ respectively. The following hold.
     \begin{enumerate}
         \item The relation 
        \[
        R = D_1 \times D_2\times D_3 \cup \{ (v,v,v): v \in V\} \subseteq U_1 \times U_2 \times U_3
        \]
        is an invariant relation of $\mathbf{A}_{\mathcal{T}}$.
         \item The relation $R^{(i_1,i_2,i_3)}$, which equals
         \[
     (D_1)^{i_1} \times (D_2)^{i_2}\times (D_3)^{i_3} \cup \{ (v^{i_1},v^{i_2},v^{i_3}): v \in V\} \subseteq (U_1)^{i_1} \times (U_2)^{i_2} \times (U_3)^{i_3}
     \]
     is an invariant relation of $\mathbf{A}_{\mathcal{T}^{\mathbf{B}\downarrow}}$.
     \end{enumerate}
     
\end{corollary}

\begin{proof}
    We prove \emph{1.} first. We let $\theta_1$ be the block congruence of the subalgebra $\mathbf{U}_1$ whose nontrivial class is $D_1$ and define $\theta_2$ and $\theta_3$ similarly. Since $D_1, D_2, D_3$ are all subsets of $D$, which is the nontrivial $\theta$-class, it follows that $\mathbf{U} / \theta$, $\mathbf{U}_1 / \theta_1$,$\mathbf{U}_2 / \theta_2$, and $\mathbf{U}_3 / \theta_3$ are pairwise isomorphic, via the mapping which identifies the nontrivial class of each algebra and acts as the identity otherwise. This means that 
    \[
    \overline{\mathbf{R}}: = \{(D_1, D_2, D_3) \} \cup \{ (v,v,v) : v\in V\} \leq (\mathbf{U}_1 / \theta_1) \times (\mathbf{U}_2 / \theta_2) \times (\mathbf{U}_3 / \theta_3)
    \]
    is a subalgebra of the product of the three isomorphic quotients. Let $\eta_1, \eta_2$, and $\eta_3$ be the respective natural maps from $\mathbf{U}$ onto $\mathbf{U}_1 / \theta_1$,$\mathbf{U}_2 / \theta_2$, and $\mathbf{U}_3 / \theta_3$. Let $\eta: = \eta_1 \times \eta_2 \times \eta_3$. Evidently, $R = \eta^{-1}(\overline{R})$, hence $R$ is a subalgebra of $\mathbf{U}^3$. 

    To prove \emph{2.}, we apply the same reasoning. By Lemma~\ref{lem:MutationsofBlockCongruencesAreInvariant}, we know that $\theta_j^{(i_j,i_j)}$ is a block congruence of $\mathbf{U}_j^{i_j}$ whose nontrivial class is $D^{i_j}$, for each $1\leq j \leq 3$. Let $i$ be the minimal gene for $U$. To make the argument given for \emph{1.} work here, we just need to see that $\mathbf{U}^i / \theta^{(i,i)}$, $\mathbf{U}_1^{i_1} / \theta_1^{(i_1,i_1)}$,$\mathbf{U}_2^{i_2} / \theta_2^{(i_2,i_2)}$, and $\mathbf{U}_3^{i_3} / \theta_3^{(i_3,i_3)}$ are pairwise isomorphic. This follows from the fact that, while the respective minimal genes $i$, $i_1$, $i_2$, and $i_3$ can be different, they do not differ in coordinates which are relevant for splicing which occurs above $d = \bigvee D$ in $\mathcal{T}$, and this is the only information which is relevant for the quotients. 
\end{proof}

Next, we define a mapping from the clone of the algebra represented by the $\mathbf{B}$-unpacking transformation of $\mathcal{T}$ into the clone of all operations on $A_\mathcal{T}$. While this mapping will preserve arity and commute with taking minors, it does not necessarily map into $\Clo(\mathbf{A}_\mathcal{T})$. However, we will later show that this is the case for a suitable choice of $\mathbf{B}$ among the local algebras of $\mathcal{T}$.

\begin{definition}\label{def:minorpreservingmapping}
    Let $\mathcal{T}$ be a sapling reduced conservative minority tree over a collection of simple minimal Taylor conservative minority algebras $S = \{\mathbf{A}_1, \dots, \mathbf{A}_s \} \cup \{\mathbf{B}\}$ and let $\mathcal{T}^{\mathbf{B}\downarrow}$ be the $\mathbf{B}$-unpacking of $\mathcal{T}$. We define a function 
    \[
    \Delta_{\mathcal{T}^{\mathbf{B} \downarrow}} \colon \Clo(\mathbf{A}_{\mathcal{T}^{\mathbf{B}\downarrow}}) \to \bigcup_{N \geq 1} A_\mathcal{T}^{(A_\mathcal{T}^N)}
    \]
    in the following manner. Let $f \in \Clo(\mathbf{A}_{\mathcal{T}^{\mathbf{B}\downarrow}})$ be of arity $N$, let $(w_1, \dots, w_N)$ be an $N$-tuple of leaves of $\mathcal{T}$, and let $i$ be the minimal gene for $C = \{w_1, \dots, w_N \}$. We then set 
    \[
    \Delta_{\mathcal{T}^{\mathbf{B} \downarrow}}(f)(w_1, \dots, w_N) := w_z,
    \]
    where $1\leq z \leq N$ is such that $f(w_1^i, \dots, w_N^i) = w^i_z$. 
\end{definition}

\begin{lemma}\label{lem:DeltaIsMinorPreserving}
    Let $\mathcal{T}$ be a sapling reduced conservative minority tree over a collection of simple minimal Taylor conservative minority algebras $S = \{\mathbf{A}_1, \dots, \mathbf{A}_s \} \cup \{\mathbf{B}\}$ and let $\mathcal{T}^{\mathbf{B}\downarrow}$ be the $\mathbf{B}$-unpacking of $\mathcal{T}$. The function $\Delta_{\mathcal{T}^{\mathbf{B}\downarrow}}$ defined in Definition~\ref{def:minorpreservingmapping} preserves arities and commutes with taking minors.
\end{lemma}

\begin{proof}
    To make the notation less cluttered, we write $\Delta$ instead of $\Delta_{\mathcal{T}^{\mathbf{B}\downarrow}}$. It is immediate from the definition that $\Delta$ preserves the arity of functions. Showing that $\Delta$ commutes with taking minors is also straightforward, but we include the argument for completeness. So, let $f \in \Clo(\mathbf{A}_\mathcal{T})$ be an $N$-ary term function, let $\sigma: [M] \to [N]$, and let $f^\sigma$ be the $M$-ary minor of $f$ defined by $f^\sigma(x_1, \dots, x_M) = f(x_{\sigma(1)}, \dots, x_{\sigma(M)})$. To show that 
    \[
    \Delta(f)^\sigma = \Delta(f^\sigma),
    \]
    we will directly show that the two functions have the same evaluation on every input tuple. Let $(w_1, \dots, w_M)$ be an $M$-tuple of leaves of $\mathcal{T}$ and let $i$ be the minimal gene of $\{w_{1}, \dots, w_{M}\}$. We then have by definition of $\Delta$ that
    $
     \Delta(f)^\sigma((w_1, \dots, w_M)) = \Delta(f)(w_{\sigma(1)}, \dots, w_{\sigma(M)}) = w_{\sigma(z)},
    $
    where $1\leq z \leq M$ some (not necessarily unique) index such that $f(w^i_{\sigma(1)}, \dots, w^i_{\sigma(M)})= w_{\sigma(z)}^i$. On the other hand, 
    $
    f(w^i_{\sigma(1)}, \dots, w^i_{\sigma(M)}) = f^\sigma(w^i_1, \dots, w^i_M),
    $
    so $\Delta(f^\sigma)(w_1, \dots, w_M) = w_{\sigma(z)}$ also.
\end{proof}

\begin{theorem}\label{thm:DeltaisMinionHom}
    Let $\mathcal{T}$ be a sapling full reduced conservative minority tree over a collection of simple minimal Taylor conservative minority algebras $S = \{\mathbf{A}_1, \dots, \mathbf{A}_s \} \cup \{\mathbf{B}\}$ and let $\mathcal{T}^{\mathbf{B}\downarrow}$ be the $\mathbf{B}$-unpacking of $\mathcal{T}$. Suppose that domain of $\mathbf{B}$ is $B = \{b_0, \dots, b_{n-1}\}$, that $\mathbf{B}$ is not hereditarily simple, and that $|B|$ is an upper bound to the set of cardinalities
     \[
     \{ |A_i| : 1 \leq i \leq s \text{ and $\mathbf{A}_i$ is not hereditarily simple}\}.
     \] 
     Then the function $\Delta_{\mathcal{T}^{\mathbf{B}\downarrow}}$ defined in Definition~\ref{def:minorpreservingmapping} is a minion homomorphism from $\Clo(\mathbf{A}_{\mathcal{T}^{\mathbf{B}\downarrow}})$ into $\Clo(\mathbf{A}_{\mathcal{T}})$.
\end{theorem}

\begin{proof}

    As before, we write $\Delta$ instead of $\Delta_{\mathcal{T}^{\mathbf{B}\downarrow}}$
    In view of Lemma~\ref{lem:DeltaIsMinorPreserving}, we need to show that the image of $\Delta$ is contained in $\Clo(\mathbf{A}_\mathcal{T})$. 
    So, take $f$ an $N$-ary term function of $\mathbf{A}_{\mathcal{T}^{\mathbf{B} \downarrow}}$. We want to see that $\Delta_{\mathcal{T}^{\mathbf{B} \downarrow}}(f)$ is also a term function of $\mathbf{\mathbf{A}_\mathcal{T}}$, which is equivalent to showing that it preserves all of the relations for $\mathfrak{S}_{\mathbf{A}_\mathcal{T}}$ (recall this is the structure which has all of the relations listed in Theorem~\ref{thm:consminorityrelbasis} for the algebra $\mathbf{A}_\mathcal{T}$). 

    \begin{itemize}
    \item[\emph{(a')}] Let $(w_1, \dots, w_N) \in A_\mathcal{T}^N$ and let $W=\{w_1, \dots, w_N\} \subsetneq A_\mathcal{T}$. Let $i$ be the minimal gene for $W$. Since $\mathbf{A}_{\mathcal{T}^{\mathbf{B} \downarrow }}$ is conservative, it follows that $f(w_1^{i}, \dots, w_N^{i}) = w_j^{i}$ for some $1 \leq j \leq N$. By definition, we have that $\Delta_{\mathcal{T}^{\mathbf{B} \downarrow}}(f)(w_1, \dots, w_N) = w_j$, hence all the unary relations of $\mathfrak{S}_{\mathcal{T}}$ are preserved by $\Delta_{\mathcal{T}^{\mathbf{B} \downarrow}}(f) $.

    \item[\emph{(b')}] Since we have already shown that unary relations are preserved above, it is enough to show that congruences of subalgebras are preserved, since every subdirect transversal endomorphism graph can be pp-defined from a congruence by restricting one factor to the transversal. We can further restrict to just showing that block congruences of subalgebras are preserved, since the other congruences are joins of block congruences, and congruence joins are pp-definable in Maltsev algebras by congruence permutability. 

    So, let $\mathbf{U} \leq \mathbf{A}_\mathcal{T}$ be a subalgebra with underlying set $U = D \cup V$, where $D = \{d_1, \dots, d_k \} $ is the nontrivial class of a block congruence $\theta$ and $V = \{v_1, \dots, v_l\}$. Take an $N$-many $\theta$-pairs $(a_1, b_1), \dots, (a_N, b_N) \in \theta$. We may assume without loss of generality that every pair of the form $(v,v)$ for $v \in V$ occurs in this list, since otherwise we can restrict the domain to obtain a subalgebra $\mathbf{U}$ and a congruence $\theta$ for which this is the case. Now set $D_1 = D \cap \{a_1, \dots, a_N\}$, $D_2 = D \cap \{b_1, \dots, b_N\}$, $U_1 = D_1 \cup V$, and $U_2 = D_2 \cup V$. Let $i_1$ be the minimal gene for $U_1$ and let $i_2$ be the minimal gene for $U_2$. Observe that 
    \[
    (a_1, b_1), \dots, (a_N, b_N) \in R:= D_1 \times D_2 \cup \{(v,v): v\in V\},
    \]
    and by Corollary~\ref{cor:importantBlockCongruenceMutationCorollary}, we know that $R^{(i_1,i_2)} \in \Inv(\mathbf{A}_{\mathcal{T}^{\mathbf{B} \downarrow}})$ (the corollary applies to three factors, but clearly holds also for two factors). In particular, we deduce that
    \[
    \left(
    f(a_1^{i_1}, \dots, a_N^{i_1}), f(b_1^{i_2}, \dots, b_N^{i_2})
    \right)  \in R^{(i_1,i_2)},
    \]
    and so it follows that 
    \[
    \left(
    \Delta_{\mathcal{T}^{\mathbf{B} \downarrow}}(f)(a_1, \dots, a_N), \Delta_{\mathcal{T}^{\mathbf{B} \downarrow}}(f)(b_1, \dots, b_N) 
    \right) \in R.
    \]
    Since $R \subseteq \theta$, the proof of this case is finished.

    \item[\emph{(c')}] The proof that isomorphism graphs between subdirectly irreducible subalgebras are preserved by $\Delta_{\mathcal{T}^{\mathbf{B} \downarrow}}(f)$ is similar to the proof given above for \emph{(b')}, except we use Lemma~\ref{lem:subalgebraIsomorphismMutationLemma} instead of Corollary~\ref{cor:importantBlockCongruenceMutationCorollary}. We leave these details to the reader. 

    \item[\emph{(d')}] Let $\mathbf{U} \leq \mathbf{A}_\mathcal{T}$ be a subalgebra with underlying set $U = D \cup V $, where $D = \{d_1, d_2\}$ is the nontrivial two-element block of a minimal block congruence $\theta$ on $\mathbf{U}$. We point out that we are proving something more general, since we do not assume here that $\mathbf{U}$ is subdirectly irreducible. We want to show that $ \Delta_{\mathcal{T}^{\mathbf{B} \downarrow}}(f)$ preserves 
    \[
    \Lin_\mathbf{U} := \{ (d_1, d_2, d_2), (d_2, d_1, d_2), (d_2, d_2, d_1), (d_1, d_1, d_1) \} \cup \{(v,v,v): v\in V\}
    \]
    Let us name the first above set of triples by $L$ and the second by $E$. Consider a list of tuples 
    \[
    (a_1, b_1, c_1) \dots, (a_N, b_N, c_N) \in  \Lin_\mathbf{U}.
    \]
    Without loss of generality, we may assume that every tuple from $E$ occurs in the above list, as otherwise we can restrict $\mathbf{U}$ to obtain $\mathbf{U'}$ for which this is the case. We now argue based on which tuples from $L$ occur in the above list. We now argue by case analysis
    \begin{itemize}
        \item Suppose that every tuple from $\Lin_\mathbf{U}$ occurs among the $(a_1, b_1, c_1) \dots, (a_N, b_N, c_N)$. In this case $\{a_1, \dots, a_N \} = \{b_1, \dots, b_N\} = \{c_1, \dots, c_N\} = U$, so the minimal gene for each of the sets is the same and we call this minimal gene $i$. It follows from Lemma~\ref{lem:MutationsofBlockCongruencesAreInvariant} that $\theta^{(i,i)}$ is block congruence of $\mathbf{U}^i$ with nontrivial class $D^i = \{ d_1^i, d_2^i\}$. We conclude that $\Lin_\mathbf{U}^{(i,i,i)}$ is an invariant relation of $\mathbf{A}_{\mathcal{T}^{\mathbf{B} \downarrow}}$. 

        Now the argument is similar to the previous ones given. We know that 
         \[
        \left(
        f(a_1^{i}, \dots, a_N^{i}), f(b_1^{i}, \dots, b_N^{i}), f(c_1^{i}, \dots, c_N^{i})
        \right)  \in \Lin_\mathbf{U}^{(i,i,i)},
        \]
        and so it follows that 
        \[
        \left(
        \Delta_{\mathcal{T}^{\mathbf{B} \downarrow}}(f)(a_1, \dots, a_N), \Delta_{\mathcal{T}^{\mathbf{B} \downarrow}}(f)(b_1, \dots, b_N) , \Delta_{\mathcal{T}^{\mathbf{B} \downarrow}}(f)(c_1, \dots, c_N)
        \right) \in R.
        \]
        \item Suppose that three tuples from $L$ occur among the $(a_1, b_1, c_1) \dots, (a_N, b_N, c_N)$. This case is identical to the previous one, since it must also hold that $\{a_1, \dots, a_N \} = \{b_1, \dots, b_N\} = \{c_1, \dots, c_N\} = U$. The rest of the argument is identical to the previous case. 
        \item Suppose that two tuples from $L$ occur among the $(a_1, b_1, c_1) \dots, (a_N, b_N, c_N)$. A typical example of this situation is that the tuples $(d_2, d_1, d_2)$ and $(d_2, d_2, d_1)$ are listed and the remaining listed tuples are of the form $(v,v,v)$ for $v \in V$. Hence, $\{a_1, \dots, a_N\} = \{d_2 \} \cup V$, while $\{b_1, \dots, b_N\} = \{c_1, \dots, c_N\} = U$. Let $U_1 = \{d_2\}  \cup V$, $U_2 = U$, and $U_3 = U$ and let $i_1$, $i_2$, and $i_3$ be their respective minimal genes. If we let $R$ be the relation
        \[
         R:= \{d_2\} \times \{d_1, d_2\}\times \{d_1, d_2 \} \cup \{ (v,v,v): v \in V\} \subseteq (U_1) \times (U_2) \times (U_3),
        \]
        then it follows from Corollary~\ref{cor:importantBlockCongruenceMutationCorollary} that $R^{(i_1, i_2,i_3)}$, which is the following set of tuples
         \[
      \{d_2\}^{i_1} \times \{d_1, d_2\}^{i_2}\times \{d_1, d_2 \}^{i_3} \cup \{ (v^{i_1},v^{i_2},v^{i_3}): v \in V\} \subseteq (U_1)^{i_1} \times (U_2)^{i_2} \times (U_3)^{i_3},
     \]
     is an invariant relation of $\mathbf{A}_{\mathcal{T}^{\mathbf{B}\downarrow}}$. Furthermore, the function $\psi \colon U_2 \to U_3$ which switches $d_1$ and $d_2$ is an isomorphism. If we denote the graph of $\psi$ by $G$, it follows from Lemma~\ref{lem:subalgebraIsomorphismMutationLemma} that $G^{(i_2,i_3)} \subseteq (U_2)^{i_2} \times (U_3)^{i_3} $ is also an invariant relation. 

     Now define the ternary relation 
     \[
     S = \{(x,y,z) : (x,y,z) \in R^{(i_1, i_2,i_3)} \text{ and } (y,z) \in G^{(i_2,i_3)}.
     \]
     Since $S$ is pp-definable from invariant relations of $\mathbf{A}_{\mathcal{T}^{\mathbf{B} \downarrow}}$, it is also invariant. Evidently, we also have that 
     \[
     S = \{ (a_1^{i_1}, b_1^{i_2}, c_1^{i_3}), \dots, (a_N^{i_1}, b_N^{i_2}, c_N^{i_3}) \},
     \]
     so once again finish the argument like before. We know that
         \[
        \left(
        f(a_1^{i}, \dots, a_N^{i}), f(b_1^{i}, \dots, b_N^{i}), f(c_1^{i}, \dots, c_N^{i})
        \right)  \in \Lin_\mathbf{U}^{(i,i,i)},
        \]
        and so it follows that 
        \[
        \left(
        \Delta_{\mathcal{T}^{\mathbf{B} \downarrow}}(f)(a_1, \dots, a_N), \Delta_{\mathcal{T}^{\mathbf{B} \downarrow}}(f)(b_1, \dots, b_N) , \Delta_{\mathcal{T}^{\mathbf{B} \downarrow}}(f)(c_1, \dots, c_N)
        \right) \in R.
        \]
    \item If only one tuple from $L$ occurs among the $(a_1, b_1, c_1),  \dots, (a_N, b_N, c_N)$, then the argument is easier, since now the tuples all belong to a relation which is pp-definable with isomorphism graphs and so Lemma~\ref{lem:subalgebraIsomorphismMutationLemma} applies. We leave the details of this case to the reader.  \qedhere 
    \end{itemize}
    \end{itemize}
\end{proof}

\begin{corollary}\label{cor:BUnpackingPPConstructs}
    Let $\mathcal{T}$ be a sapling full reduced conservative minority tree over a collection of simple minimal Taylor conservative minority algebras $S = \{\mathbf{A}_1, \dots, \mathbf{A}_s \} \cup \{\mathbf{B}\}$ and let $\mathcal{T}^{\mathbf{B}\downarrow}$ be the $\mathbf{B}$-unpacking of $\mathcal{T}$. Suppose that $\mathbf{B}$ is not hereditarily simple, and that $|B|$ is an upper bound to the set of cardinalities
     \[
     \{ |A_i| : 1 \leq i \leq s \text{ and $\mathbf{A}_i$ is not hereditarily simple}\}.
     \] 
    The structure $\mathfrak{S}_{\mathcal{T}^{\mathbf{B}\downarrow}}$ pp-constructs the structure $\mathfrak{S}_{\mathcal{T}}$.
\end{corollary}

\begin{proof}
    This is an immediate consequence of Theorem~\ref{thm:DeltaisMinionHom} along with Theorem~\ref{thm:pp-constr}. 
\end{proof}

\subsection{Iterative construction}\label{sec:recursiveconstruction}

In this section we prove one of our main results. The idea is to apply the $\mathbf{B}$-unpacking construction and then extract from the resulting conservative minority tree a sapling conservative minority tree which represents a subdirectly irreducible algebra. This procedure then continues iteratively until no local algebras which are not hereditarily simple remain. So, let $\Simple$ be the class of finite simple minimal Taylor conservative minority algebras. To accurately formulate the induction, we introduce a norm on the class
$\SimpleLocal$, which we define to be the class of all finite subsets of $\Simple$.

We start by fixing some $\Lambda \colon \Simple \to \mathbb{N}$ which satisfies the following properties:
\begin{itemize}
    \item $\Lambda(\mathbf{A}) = \Lambda(\mathbf{B})$ if and only if $\mathbf{A}$ and $\mathbf{B}$ are isomorphic or $\mathbf{A}$ and $\mathbf{B}$ are each hereditarily simple, in which case $\Lambda(\mathbf{A}) = \Lambda(\mathbf{B}) = 0$. 
    \item If the $|A| < |B|$, then $\Lambda(\mathbf{A}) < \Lambda(\mathbf{B})$. 
\end{itemize}
Given $S \in \SimpleLocal$, we then define the \emph{norm} of $S$ as
$
\| S \| := \max( \{ \Lambda(\mathbf{A}): \mathbf{A} \in S\}) 
$. It is straightforward to see that this norm induces a transitive order on $\SimpleLocal$ which has no infinite descending chains. 

\begin{lemma}\label{lem:ppConstructInductionStep}
    Let $\mathcal{T}$ be a full conservative minority tree over a collection of pairwise nonisomorphic minimal Taylor simple conservative minority algebras $S \in \SimpleLocal$. Suppose that $\| S \| > 0$ 
    (i.e.\ there exists $\mathbf{A} \in S$ which is not hereditarily simple.) Then, there exists a full conservative minority tree $\mathcal{T}''$ over $S''$ such that the following hold. 
    \begin{enumerate}
    \item $S''$ is a collection of pairwise nonisomorphic minimal Taylor simple conservative minority algebras.
        \item $\| S'' \| < \| S \|$.
        \item $\mathfrak{S}_{\mathbf{A}_{\mathcal{T}''}}$ pp-constructs $\mathfrak{S}_{\mathbf{A}_{\mathcal{T}}}$.
    \end{enumerate}
\end{lemma}
\begin{proof}
    To begin, we may assume that every algebra in $S$ is used in $\mathcal{T}$, otherwise we can discard these extraneous local algebras and obtain $\mathcal{T}''$ over $\mathcal{S}'$ with $\| S ''\| \leq \| S\|$ and argue instead within this new setting. We first label the local algebras so that $S = \{ \mathbf{A}_1, \dots, \mathbf{A}_s \} \cup \mathbf{B}$, where $\mathbf{B}$ is not hereditarily simple and its domain $B$ has cardinality $|B| = \|S \| = n$.  

    First, we invoke Lemma~\ref{lem:SaplingTreesConstructGeneralOnes} to obtain a sapling full reduced conservative minority tree $\mathcal{T}'$ over $S$ such that $\mathfrak{S}_{\mathbf{A}_{\mathcal{T}'}}$ pp-constructs $\mathfrak{S}_{\mathbf{A}_\mathcal{T}}$. Next, we set (see Definition~\ref{def:splicing})
    \begin{align*}
    \mathcal{T}''' &:= (\mathcal{T}')^{\mathbf{B} \downarrow} \text{, and }\\
     S''' &:= S^{\mathbf{B} \downarrow} = \{\mathbf{A}_1, \dots, \mathbf{A}_s \} \cup \{\mathbf{B}_1, \dots, \mathbf{B}_n\} \cup \{\mathbf{P}_{n}, \mathbf{P}_{n+1}\}.
    \end{align*}
    Now we invoke Corollary~\ref{cor:BUnpackingPPConstructs} to see that $\mathfrak{S}_{\mathbf{A}_{\mathcal{T}'''}}$ pp-constructs $\mathfrak{S}_{\mathbf{A}_{\mathcal{T}}}$. We are not finished yet though, since there could now be redundancies among the members of $S'''$ and, more importantly, the maximal proper subalgebras $\mathbf{B}_1, \dots, \mathbf{B}_n \leq \mathbf{B}$ may not be simple. To finish, we refine $\mathcal{T}'''$ and $S'''$ by replacing any occurrence of a nonsimple successor algebra $\mathbf{B}_j$ in $\mathcal{T}'''$ by a full conservative minority tree over a collection of simple algebras which represents $\mathbf{B}_j$ (for details on this refinement, see Section~\ref{sec:treetransformations}). This last procedure does not change the isomorphism type of the represented algebra and cannot increase the norm of $S'''$ (see Remark~\ref{rem:NormofSdoesNotIncrease}), hence we obtain $\mathcal{T}''$ and $S''$ as asserted in the lemma statement. 
\end{proof}
\begin{theorem}\label{thm:FinalPPConstructionTheorem}
    Let $\mathfrak{A}$ be a finite structure with a conservative Maltsev polymorphism. There exists $n \geq 2$ and $k \geq 1$ such that $\mathfrak{P}^{n,k}$ pp-constructs $\mathfrak{A}$. 
\end{theorem}

\begin{proof}
   We invoke Corollary~\ref{cor:minTaylorConsMinTreeConstOthers} to obtain $\mathcal{T}$ a full conservative minority tree over a collection of pairwise nonisomorphic minimal Taylor simple conservative minority algebras $S \in \SimpleLocal$. We proceed inductively on $\| S\|$.
   
   If $\| S \| = 0$, then all local algebras are hereditarily simple and minimal Taylor, hence by \emph{2.}\ of Proposition~\ref{prop:MinTayConsMalChar}, each local algebra is isomorphic to $\mathbf{P}_n$ for some $ n \geq 2$ (recall our convention that all minimal Taylor conservative minority algebras are of type $i=2$). In this case, there exists $n$ and $k$ so that every subdirectly irreducible subalgebra of $\mathbf{A}_\mathcal{T}$ is embeddable in $\mathbf{P}^{n,k}$, so $\mathbf{A}_\mathcal{T}$ belongs to the pseudo-variety generated by $\{\mathbf{P}^{n,k}\}$ (cf.\ the proof of Lemma~\ref{lem:SaplingTreesConstructGeneralOnes} and \emph{3.}\ of Lemma~\ref{lem:saplingsrepresentsubdirectirreducible}). It follows that $\mathfrak{S}_{\mathbf{A}_{\mathcal{P}_{n,k}}}$ pp-constructs $\mathfrak{S}_{\mathbf{A}_\mathcal{T}}$. Now we invoke Proposition~\ref{prop:hsimpletreerecursivebasis}, which indicates that $\mathfrak{S}_{\mathbf{A}_{\mathcal{P}_{n,k}}}$ is pp-definable in $\mathfrak{P}^{n,k}$. Hence, we have proved the basis of the induction. If $\| S \| > 0$, then we invoke Lemma~\ref{lem:ppConstructInductionStep}, apply the induction hypothesis, and compose the pp-constructions. 
\end{proof}


\section{Solving conservative Maltsev CSPs}
\label{sect:alg} 
\begin{figure} 
\begin{center} 
\begin{tikzcd}
    \Solve_{n,k+1} \arrow[r] \arrow[d] & \Reduce_{n,k+1} \arrow[d] \\
    \Solve_{n,k} \arrow[r] \arrow[d] & \Reduce_{n,k}  \arrow[d] \\
    \vdots  \arrow[d] & \vdots \arrow[d] \\
     \Solve_{n,2} \arrow[r] \arrow[d] & \Reduce_{n,2} \arrow[d] \\
    \Solve_{n,1}  & \arrow[l] \Reduce_{n,1}
\end{tikzcd}
\end{center}
\caption{Dependencies between the subprocedures of the algorithm.}
\label{fig:alg} 
\end{figure} 
In this section we present, for every $n \geq 2$ and $k \geq 1$, an algorithm $\Solve_{n,k}$ that solves $\Csp(\mathfrak{P}^{n,k})$. To this end we present, for every $n \geq 2$ and $k \geq 1$, an algorithm $\Reduce_{n,k}$ that reduces $\Csp(\mathfrak{P}^{n,k})$ to $\Csp(\mathfrak{P}^{n,k-1})$.
 We first specify the procedure $\Solve_{n,1}$, which involves solving a system of linear equations over ${\mathbb Z}_2$. The reduction procedures are then recursively specified. The output of $\Solve_{n,1}$ is consulted on each connected component of an input of $\Reduce_{n,1}$ to produce an instance of $\Csp(\mathfrak{P}^{n,0})$. For $k > 1$, the procedure $\Reduce_{n,k}$ computes an auxiliary instance and then combines the output of $\Reduce_{n,k-1}$ on this instance with information obtained from the original instance to obtain an output. To solve $\Csp(\mathfrak{P}^{n,k})$ for $k >1$, we may then compose reductions until we obtain an instance of $\mathfrak{P}^{n,1}$ and then apply $\Solve_{n,1}$.  The reader can consult Figure \ref{fig:alg} for a picture of the dependencies of these procedures.

\subsection{The procedure Solve } 
\label{sect:solve}
See Algorithm~\ref{alg:solve} for the pseudo-code of the procedure Solve$_{n,1}$. 
As we have already explained, 
if $k \geq 2$ the procedure repeatedly calls $\Reduce$ so that we are left with deciding the satisfiability 
of a given instance $\bI$ of $\Csp(\mathfrak{P}^{n,k})$ for $k=1$. 

In this case, it will be more convenient to work with the following structure $\bS$. 
This is sufficient, because 
up to renaming elements $(i)$ of $\bP^{n,1}$ to $i$, for $i \in \{0,\dots,n-1\}$, every relation of $\bS$ 
is primitively positively definable in $\bP^{n,k}$, and vice versa. 

\begin{definition}\label{def:S}
    Let $\bS$ be the structure 
    with the domain $S := \{0,\dots,n-1\}$ and the signature 
$$\tau := \big \{R_X \mid X \subseteq S \big \} \cup \big \{ G_{\sigma} \mid \sigma \in S_n \big \} \cup \{L\}$$ 
whose symbols denote the relations 
specified in Lemma~\ref{lem:brady1ishsimple}.
\end{definition}


Let $\bI$ be an instance of 
$\Csp(\bS)$. 
A \emph{bijection path in $\bI$ (from $x_0$ to $x_k$)} is a sequence $$x_0,\sigma_1,x_1,\dots,\sigma_k,x_k$$ of elements of $I$ 
such that $(x_{i-1},x_{i}) \in G^{\bI}_{\sigma_i}$ for all $i \in \{1,\dots,k\}$. 

\begin{definition}\label{def:bool-sys}
    The \emph{system of $\mathbb{Z}_2$-linear equations associated to $\bI$} 
    is the $\{L, G_{\id},G_{(01)},R_{\{0\}},R_{\{1\}}\}$-structure
    $\bE$ with domain $I$ and 
    \begin{itemize} 
    \item $L^{\bE} := \Lin^{\bI}$, 
    \item $x_0 \in R^{\bE}_{\{c\}}$, for $c \in \{0,1\}$
    and $x_0 \in I$, 
    if and only if in $\bI$ there exists a bijection 
    path $x_0,\sigma_1,x_1,\dots,\sigma_k,x_k$ 
    such that for every $a \in \{0,\dots,n-1\} \setminus \{c\}$
    there exists $R \in \tau_{n,1}$
    and $t \in R^{\bI}$ 
    with entry $x_k$ such that all tuples in 
    $R^{\bS}$ do not have the value 
     $\sigma_k \circ \cdots \circ \sigma_1(a)$ at the same entry, and 
     \item $(x_0,x_k) \in G^{\bE}_{\id}$, for $x_0,x_1 \in I$, 
     if and only if 
    if there exists a bijection path 
    $x_0,\sigma_1,x_1,\dots,\sigma_k,x_k$ in $\bI$ 
    such that 
    $\sigma_k \circ \cdots \circ \sigma_1(0)=0$. Similarly, $(x_0,x_k) \in G^{\bE}_{(01)}$, for $x_0,x_1 \in I$, 
     if and only if 
    if there exists a bijection path 
    $x_0,\sigma_1,x_1,\dots,\sigma_k,x_k$ in $\bI$ 
    such that 
    $\sigma_k \circ \cdots \circ \sigma_1(0)=1$.
    \end{itemize}
\end{definition}

    Note that $\bE$ from Definition~\ref{def:bool-sys} might  indeed be viewed as a system of $\mathbb{Z}_2$-linear equations.


\begin{definition}\label{def:perm-sys}
    The \emph{permutational system associated to $\bI$} is the instance of
    $\Csp(\bS)$ obtained from
    $\bI$ by replacing $\Lin^{\bI}$ by $\emptyset$. 
\end{definition}

    Note that the permutational system associated to $\bI$ may be viewed as a CSP for the structure obtained from $\bS$ by removing the relation $\Lin$ from the signature; 
    it is easy to see that the resulting structure also has 
    the polymorphism which 
    is the majority operation 
    that returns 
    the first argument if
    the three arguments are distinct. 
    We will see in Section~\ref{sect:symlin} that 
    it can be solved by symmetric linear Datalog. 

\begin{algorithm}
\KwIn{Instance $\bI$ of $\Csp(\mathfrak{P}^{n,k})$}

\While{$k \geq 2$}{
$k \gets k-1$\\
Let $\overline{\mathfrak{I}}$ be the instance of $\Csp(\mathfrak{P}^{n,k})$ 
computed by $\Reduce_{n,k}(\bI)$ \\
$\bI \gets \bar \bI$ 
}

\Make $\bI$ into instance of $\Csp(\bS)$ \tcp*{\parbox[t]{2in}{ see Definition~\ref{def:S}}} 
\Make $\bE :=$ system of $\mathbb{Z}_2$-linear equations for $\bI$ \tcp*{\parbox[t]{2in}{ see Definition~\ref{def:bool-sys}}}
\Make $\bT :=$ permutational system for $\bI$ \tcp*{\parbox[t]{2in}{ see Definition~\ref{def:perm-sys}}}

\KwOut{REJECT if $\bE$ is unsatisfiable 
\tcp*{\parbox[t]{2in}{Gaussian Elimination}} \\
REJECT if $\bT$ is unsatisfiable 
\tcp*{\parbox[t]{2in}{symmetric linear Datalog}} \\
otherwise ACCEPT.
}

\caption{$\Solve_{n,k}$}
\label{alg:solve}



\end{algorithm}

\begin{theorem}\label{thm:solveiscorrect}
    Let $\bI$ be an instance of $\Csp(\bP^{n,k})$. 
    Algorithm~\ref{alg:solve} returns 
    REJECT if $\bI$ has no homomorphism to $\bS$, and returns ACCEPT otherwise. 
\end{theorem}
\begin{proof}
    We proceed by induction on $k \geq 1$.
    For the basis, suppose that $k =1 $ and let $\mathfrak{I}$ be an instance of $\Csp(\mathfrak{P}^{n,1})$. 
    Let $\bE$ be the system  of $\mathbb{Z}_2$-linear equations 
    for $\bI$. 
    If $\bE$, viewed as a system of ${\mathbb Z}_2$-linear equations, is unsatisfiable in 
    ${\mathbb Z}_2$ (which can for instance be checked with Gaussian elimination),  
    then $\bI$ has no homomorphism to 
    $\bS$, because every equation of $\bE$ holds for every homomorphism from $\bI$ to 
    $\bS$. 
    Conversely, if $\bE$ is satisfiable, 
    then we may obtain a homomorphism from $\bI$ to 
    $\bS$
    as follows: 
    partition the elements of $I$ into $X$ and $Y$ where
    $X$ are the elements that are connected via some bijection path in $\bI$ to $x_k$, and $x_k$ appears in a tuple 
    of $\Lin^{\bI}$, 
    and $Y := I \setminus X$ are the remaining variables. 
    Let $s \colon X  \cup Y \to \{0,\dots,n-1\}$ be the map that coincides on $X$ with a solution to $\bE$, and on $Y$ with a homomorphism from $\bT$ to $\bS$. It is straightforward to verify that $s$ satisfies all the constraints.

    Now suppose that $k >1$. We will show in Proposition~\ref{prop:proofofreduction} that the instance $\overline{\mathfrak{I}}$ output by $\Reduce_{n,k}(\mathfrak{I})$ is satisfiability equivalent to $\mathfrak{I}$ (NB Proposition~\ref{prop:proofofreduction} only relies on the case $k=1$ of this theorem, which we have just established) hence the result follows. 
\end{proof}

\begin{remark}
    The inductive proofs of Theorem~\ref{thm:solveiscorrect} and Proposition~\ref{prop:proofofreduction} are superficially entangled, but the reader can consult Figure~\ref{fig:alg} to assure themselves that the inductions are well-founded. The basis is shown in the bottom left of the corner, the induction for Proposition~\ref{prop:proofofreduction} proceeds up the right column, and finally the induction for Theorem~\ref{thm:solveiscorrect} proceeds up the left column.
\end{remark}

We will explain in Section~\ref{sect:descr} how to implement this algorithm in ${\mathbb Z}_2$-Datalog, which then implies that 
$\Csp(\bS)$ and $\Csp(\bP^{n,k})$ are in $\oplus L$.

\subsection{A reduction example}
Before we begin our formal presentation of the procedure $\Reduce$, we describe how the algorithm $\Reduce_{n,k}$ works by providing a toy example of $\Reduce_{3,2}$. For $\mathfrak{I}$ an instance of $\Csp(\mathfrak{P}^{n,k})$, we will often denote by $\overline{\mathfrak{I}}$ the output of $\Reduce_{n,k}$ on input $\mathfrak{I}$. Let $\mathfrak{I}$ be an instance of $\Csp(\mathfrak{P}^{n,2})$. The idea is to first find the components of $\mathfrak{I}$ which are connected by uniform relations, because on these components the values of a solution must either belong entirely to $P^{n,2}_{0}$ or its complement $O_n$. If such a component contains a variable $x$ for which the constraint $(x,y) \in (T^{n,2}_{0,1})^\mathfrak{I}$ holds, then it is impossible on this component for a solution to $\mathfrak{I}$ to take a value outside of $P^{n,2}_0$, hence the instance $\mathfrak{I}$ restricted to this component is really an instance of $\Csp(\mathfrak{P}^{n,1})$. We call such a component a \emph{transfer source} component (see Definition~\ref{def:components}). For any component that is not a transfer source component, we still consider the instance which results from restricting to $P^{n,2}_{0}$, but here we run $\Reduce_{n,1}$, which allows us to collapse $P^{n,2}_0$ to a one element set and then replace all uniform relations on the component with uniform $\tau_{n,1}$ relations. So, for each uniform connected component there are two ways to find an instance of $\Csp(\mathfrak{P}^{n,1})$ which depend on whether the component is a transfer source component or not. The $\tau_{n,2}^N$ constraints of $\mathfrak{I}$ can then be assigned to $\tau_{n,1}$-constraints in a manner so that the obtained instance $\overline{\mathfrak{I}}$ of $\Csp(\mathfrak{P}^{n,1})$ is the reduced instance corresponding to $\mathfrak{I}$.

\begin{figure}
    \centering
    \includegraphics[width=0.70\linewidth]{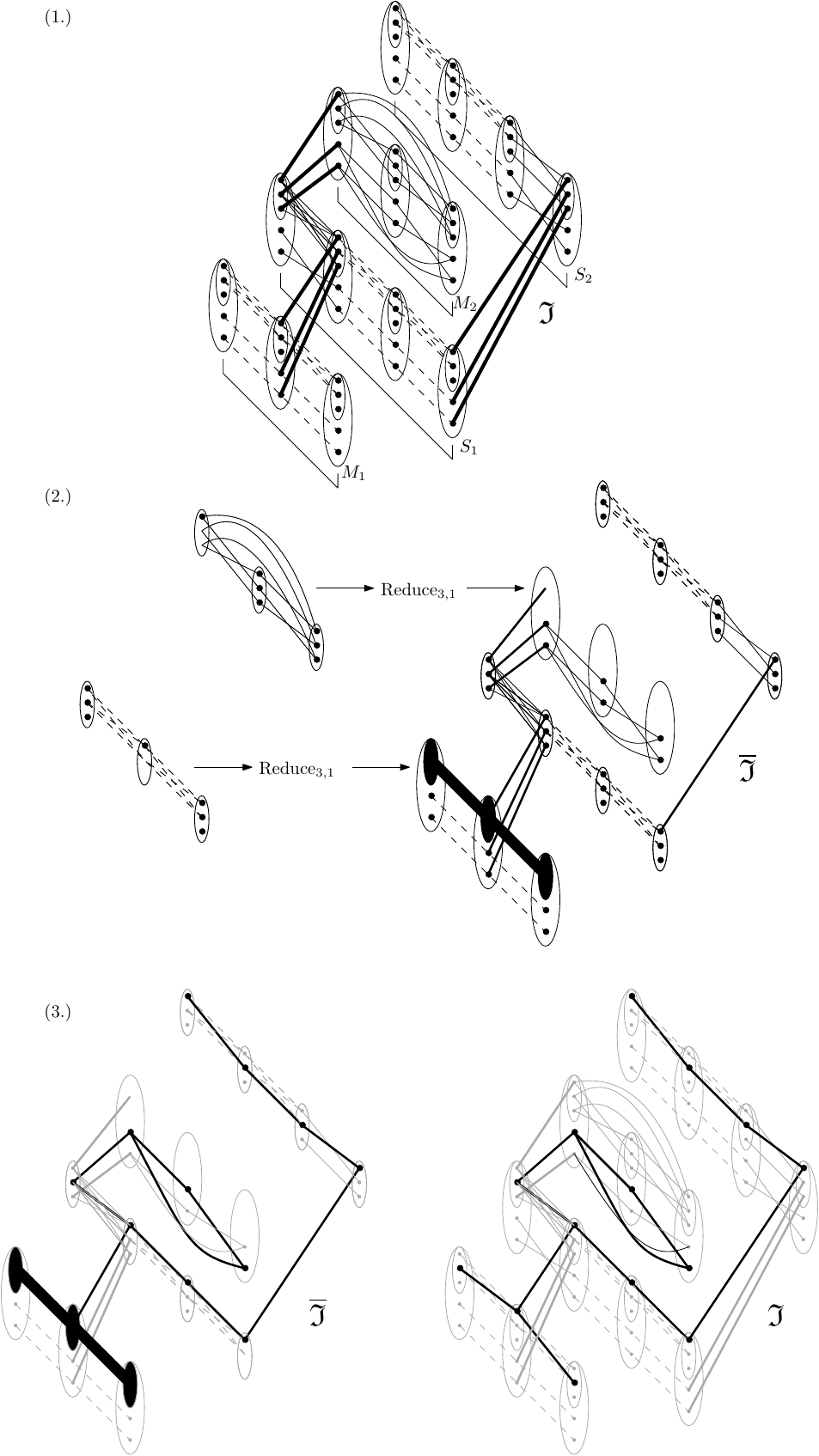}
    \caption{Reducing an instance of $\Csp(\mathfrak{P}^{3,2})$.}
    \label{fig:cspinstance}
\end{figure}


We explain this further by considering the example instance of $\Csp(\mathfrak{P}^{3,2})$ shown in item (1.) of Figure \ref{fig:cspinstance}. Each variable of the instance is depicted as the domain of $\mathbf{P}^{3,2}$ (a potato diagram). The elements of the domain are drawn as solid dots. Each potato is drawn with an inner potato whose inside three elements are elements of $\mathbf{P}^{3,2}_0$, which is of course isomorphic to $\mathbf{P}^{3,1}$. Relations that hold between variables are drawn in three ways. The ternary linear relation is drawn with dashed lines and binary relations are drawn with solid lines, except for the canonical transfer relation $T^{3,2}_{0,1}$, which is drawn with heavier solid lines. The depicted instance has four uniform connected components with respect to the uniform relations (see Definition~\ref{def:components}). These components are arranged on diagonals and labeled $M_1, M_2$ and $S_1, S_2$, where $S_1$ and $S_2$ are transfer source components.

The next thing we do is add some unary constraints to the instance by looking at the canonical transfer relation $T^{3,2}_{0,1}$. Suppose that $x,y \in (T^{3,2}_{0,1})^\mathfrak{I}$ holds in the instance. We call such $y$ a \emph{transfer target} variable and add the unary constraint $y \in ((R_{\{\epsilon\}}, R_{\emptyset}), R_{O_3})^\mathfrak{I}$ for all target variables $y$ (note that $((R_{\{\epsilon\}}, R_{\emptyset}), R_{O_3})^{\mathfrak{P}^{3,2}} = P^{3,2}_1$, which is defined by the formula $\exists x T^{3,2}_{0,1}(x,y)$).

In item (2.) of the figure we show the reduced instance $ \overline{\mathfrak{I}}$ is formed. On each uniform connected component an instance is formed by restricting to the set $P^{3,2}_0$. On the $(0)$-source components $S_1$ and $S_2$, this restriction is already a reduction, since a solution to the original instance $\mathfrak{I}$ must take values in this set. On the other components $M_1$ and $M_2$, the reduced instance is obtained by first calling $\Reduce_{3,1}$ on the restricted instances and then combining the outputs with the strongly functional parts of the uniform $\tau_{n,2}^U$ constraints to obtain uniform $\tau_{3,1}$ constraints. Then, the constraints $T^{3,2}_{0,1}$ are assigned to either $\Eq_2$ or $T^{3,1}_{0,0}$, depending on whether $y$ is a target variable or a source variable, respectively.

In item (3.) of Figure \ref{fig:cspinstance}, we demonstrate the relationship between a solution of $\overline{\mathfrak{I}}$ and a solution of $\mathfrak{I}$. While it is clear that a solution of $\mathfrak{I}$ must correspond to a solution of $\overline{\mathfrak{I}}$, it not immediately obvious that a solution of $\overline{\mathfrak{I}}$ `lifts' to a solution of $\mathfrak{I}$. The potential problem is that since the $\sim_{(0)}$-class $P^{3,2}_0$ has been collapsed in the connected component $M_1$, there could be constraint information that is lost. However, at the beginning of the reduction the additional unary constraints that are consequences of the canonical transfer relation are added and the instance restricted to $P^{3,2}_0$ was found to be satisfiable.

Of course, the general reduction scheme contains subtleties that exceed the scope of this informal explanation. So, we conclude this informal exposition and move on to the detailed exposition. 

\subsection{The reduction from $\Csp(\mathfrak{P}^{n,k+1})$ to $\Csp(\mathfrak{P}^{n,k})$} 
The procedure $\Reduce_{n,k+1}$ computes for a given instance 
$\bI$ of $\Csp(\mathfrak{P}^{n,k+1})$ 
not only a satisfiability-equivalent instance $\overline{\bI}$ 
of $\Csp(\mathfrak{P}^{n,k})$, but also a function $\Pi$ with domain $I$ which assigns to each domain element a number belonging to $\{0, \dots, k\}$ which encodes the relationship between solutions to $\mathfrak{I}$ and the output instance $\overline{I}$. The function $\Pi$, which we call the \emph{reduction atlas} (see Definition~\ref{def:(k-1)reduction}), is then used recursively to reduce constraints named by the canonical transfer relation. 

In this section we present our sequence of reductions, first by presenting the basis algorithm $\Reduce_{n,1}$ (Algorithm~\ref{alg:reducebasis}), and then algorithms $\Reduce_{n,k+1}$, each of which recursively calls its predecessor (Algorithm~\ref{alg:reduce}). The pseudocode for the algorithms refers to several of the definitions given in this section. We then inductively prove that our reduction scheme is correct in Proposition~\ref{prop:proofofreduction}.

\begin{algorithm}
\label{alg:basisreduction}
\caption{$\Reduce_{n,1}$}
\label{alg:reducebasis}
\KwIn{Instance $\mathfrak{I}$ of $\Csp(\mathfrak{P}^{n,1})$}

\Make Connected components $C_1, \dots, C_v$ \tcp*{\parbox[t]{2in}{ see Definition~\ref{def:components}}}
\Make $\mathfrak{C}_1, \dots, \mathfrak{C}_v$ \tcp*{\parbox[t]{2in}{ see (2) of Definition~\ref{def:listofdifferentinstances(k=0also)}}}

\For{$1 \leq j \leq v$}{

\Make Instance $\overline{\mathfrak{C}_j}$ with domain $\overline{C_j}=C_j$, and
$\Eq_2^{\overline{\mathfrak{C_j}}} := \Sup_2(\mathfrak{C_j}) $ and
$\Eq_3^{\overline{\mathfrak{C_j}}} := \Sup_3(\mathfrak{C_j}) $

\If{$\Solve_{n,1}(\mathfrak{C_j})$ \textnormal{returns} \textnormal{ACCEPT}
}{$R_\emptyset^{\overline{\mathfrak{C}_j}} := \emptyset$, 
$R^{\overline{\mathfrak{C}_j}}_{\{\epsilon\}} := C_j$  
}

\If{$\Solve_{n,1}(\mathfrak{I})$ \textnormal{returns} \textnormal{REJECT}
}{$R_\emptyset^{\overline{\mathfrak{C}_j}} :=  C_j$, 
$R^{\overline{\mathfrak{C}_j}}_{\{\epsilon\}} := \emptyset $  
}

}

\Make $\overline{\mathfrak{I}} := \overline{\mathfrak{C}_1} \cup \dots \cup \overline{\mathfrak{C}_j}$\\

\Make Function $\Pi \colon \overline{I} \to \{0\} $\\
\KwOut{Instance $\overline{\mathfrak{I}}$ and function $\Pi$.
}
\end{algorithm}

\begin{algorithm}\label{alg:recursivereduction}
\caption{$\Reduce_{n,k+1}$ for $k \geq 1$}\label{alg:reduce}
\KwIn{Instance $\mathfrak{I}$ of $\Csp(\mathfrak{P}^{n,k+1})$}\

$\mathfrak{I} \gets \mathfrak{I}^*$ \tcp*{\parbox[t]{3in}{\raggedright 
see (3) of Definition \ref{def:innerouterinstances+targetcompatible}}}   \

$\mathfrak{I} \gets \mathfrak{I}^C$ \tcp*{\parbox[t]{3in}{\raggedright 
see (3) of Definition \ref{def:listofdifferentinstances(k=0also)}}}   \

\Make $\mathfrak{I}^U, \mathfrak{I}^N$ \tcp*{\parbox[t]{3in}{
see (4) and (5) of Definition \ref{def:listofdifferentinstances(k=0also)}}} \

\Make Uniform connected components $S_1, \dots, S_r, M_1, \dots, M_v$ \tcp*{\parbox[t]{1in}{ 
see Definition \ref{def:components}}} \

\Make $\mathfrak{S}^U_1, \dots, \mathfrak{S}^U_r, \mathfrak{M}^U_1, \dots, \mathfrak{M}^U_v$ \tcp*{\parbox[t]{3in}{ 
see Definition \ref{def:listofdifferentinstances(k=0also)}}} \

\Make $\Pi$ \tcp*{\parbox[t]{3in}{
see Definition \ref{def:(k-1)reduction}}} \

\hrulefill 

\Make $\mathfrak{S}^U := \mathfrak{S}_1^U \cup \dots \cup \mathfrak{S}_r^U$ \tcp*{\parbox[t]{3in}{see (1) of Definition \ref{def:listofdifferentinstances(k=0also)}}} \

\Set $\Pi(x) := 0 $ for all $x \in S = S_1 \cup \dots \cup S_r$\

\Make $\overline{\mathfrak{S}^U} \coloneq (\mathfrak{S}^U)^{(0)} $
\tcp*{\parbox[t]{3in}{ 
see Definition \ref{def:innerouterinstances+targetcompatible}}}

\hrulefill

\Make $\mathfrak{M}^U := \mathfrak{M}_1^U \cup \dots \cup \mathfrak{M}_v^U$\

\Make $\overline{(\mathfrak{M}^U)^{(0)}} := \text{ instance output by } \Reduce_{n,k}( (\mathfrak{M}^U)^{(0)})$\\
\Make $\Pi_M:= \text{ reduction atlas output by } \Reduce_{n,k}( (\mathfrak{M}^U)^{(0)})$\\
\Set $\Pi(x) := \Pi_M(x)+1$ for all $x \in M = M_1 \cup \dots \cup M_v$ \\

\Make $\overline{\mathfrak{M}^U} \coloneqq  \overline{(\mathfrak{M}^U)^{(0)}} \cdot (\mathfrak{M}^U)^{\lnot(0)}$ 
\tcp*{\parbox[t]{1.5in}{
see Definition~\ref{def:uniformproduct}}}

\hrulefill

\Make $\overline{\mathfrak{I}^N} \coloneq \NReduce_{n,k+1}(\mathfrak{I}^N, \Pi)$ 
\tcp*{\parbox[t]{2.5in}{
see Algorithm~\ref{alg:outerreduce}}} \

\hrulefill

\Make {$\overline{\mathfrak{I}} \coloneq \overline{\mathfrak{S}^U} \cup  \overline{\mathfrak{M}^U} \cup \overline{\mathfrak{I}^N}   $}
\

\KwOut{Instance $\overline{\mathfrak{I}}$ of $\Csp(\mathfrak{P}^{n,k})$ and function $\Pi$} 
\end{algorithm}

\begin{definition}\label{def:supportconstraintssize}
   Let $\tau$ be a relational signature and let $\mathfrak{I}$ be a $\tau$-structure. 
       For $x \in I$, $1\leq k \leq 3$, and $x \in I^k$ we define $\Cons^\mathfrak{I}(x):= \{\rho \in \tau: x \in \rho^\mathfrak{I} \}$.
       We set $\Sup_k^\mathfrak{I} := \{ x \in I^k: |\Cons^\mathfrak{I}(x)| \geq 1 \}$ and refer to $\Sup_k^{\mathfrak{I}}$ as the \emph{($k$-ary) support} of $\mathfrak{I}$.
       We define the \emph{size} of $\mathfrak{I}$ as \[        |\mathfrak{I}| = |I| + \sum_{x \in I} |\Cons^\mathfrak{I}(x)| +         \sum_{(x,y) \in I^2 }|\Cons^\mathfrak{I}(x,y)|+        \sum_{(x,y,z) \in I^3} |\Cons^\mathfrak{I}(x,y,z)| .       \] 
\end{definition}

\begin{definition}\label{def:listofdifferentinstances(k=0also)}
Let $n \geq 2$ and $k \geq 0$.
\begin{enumerate}
    \item Given a collection of instances $\mathfrak{I}_1, \dots, \mathfrak{I}_l$ of $\Csp(\mathfrak{P}^{n,k})$, we denote by $\mathfrak{I}_1 \cup \dots \cup \mathfrak{I}_l$ the instance of $\Csp(\mathfrak{P}^{n,k})$ with domain $I_1 \cup \dots \cup I_l$ and constraints the union of the constraints of the $\mathfrak{I}_1, \dots, \mathfrak{I}_l$. 

    \item Given an instance $\mathfrak{I}$ of $\Csp(\mathfrak{P}^{n,k})$ and a subset $M \subseteq I$, we denote by $\mathfrak{M}$ the \emph{induced subinstance} of $\mathfrak{I}$, which we define as the induced $\tau_{n,k}$-substructure of $\mathfrak{I}$ with domain $M$.

    \item We say an instance $\mathfrak{I}$ of $\Csp(\mathfrak{P}^{n,k})$ is \emph{concise} if for each $x \in I$, there is exactly one unary symbol $\rho \in \tau_{n,1}^{U,1}$ that constrains $x$ (i.e.,  $\Sup_1^\mathfrak{I} = I$ and $|\Cons^\mathfrak{I}(x)| = 1$ for each $x \in I$). Given an instance $\mathfrak{I}$ of $\Csp(\mathfrak{P}^{n,k})$, we denote by $\mathfrak{I}^C$ the instance with domain $I$ so that $\rho^{\mathfrak{I}^C} = \rho^{\mathfrak{I}}$ for every binary or ternary symbol $\rho \in \tau_{n,k}$, and $\Cons^{\mathfrak{I}^C}(x) = \{\lambda\}$, where $\lambda \in \tau_{n,k}$ is such that 
    \[
    \lambda^{\mathfrak{P}^{n,k}} = P^{n,k} \cap \left(\bigcap_{\rho \in \Cons^{\mathfrak{I}}(x)} \rho^{\mathfrak{P}^{n,k}}\right).
    \]

    \item We say an instance $\mathfrak{I}$ of $\Csp(\mathfrak{P}^{n,k})$ is a \emph{uniform instance} if it only contains constraints named by $\tau^U_{n,k}$ symbols. Given an instance $\mathfrak{I}$ of $\Csp(\mathfrak{P}^{n,k})$, we denote by $\mathfrak{I}^U$ the \emph{uniform reduct} of $\mathfrak{I}$, which has domain $I$ but only contains those constraints from $\mathfrak{I}$ that are named by $\tau_{n,k}^U$ symbols. 
\item We say an instance $\mathfrak{I}$ of $\Csp(\mathfrak{P}^{n,k})$ is a \emph{nonuinform instance} if it only contains constraints named by $\tau^N_{n,k}$ symbols. Given an instance $\mathfrak{I}$ of $\Csp(\mathfrak{P}^{n,k})$, we denote by $\mathfrak{I}^N$ the \emph{nonuniform reduct} of $\mathfrak{I}$, which has domain $I$ but only contains those constraints from $\mathfrak{I}$ that are named by $\tau_{n,k}^N$ symbols. 
\end{enumerate}
\end{definition}

 One of the main structural features of a $\Csp(\mathfrak{P}^{n,k+1})$-instance $\mathfrak{I}$ that is used to reduce it to an instance of $\Csp(\mathfrak{P}^{n,k})$ is the decomposition into connected components of $I$ with respect to the uniform relations.

\begin{definition}\label{def:components}
Let $n \geq 2$ and $k \geq 0$. Let $\mathfrak{I}$ be an instance of $\Csp(\mathfrak{P}^{n,k})$. We define the \emph{connectivity graph} of $\mathfrak{I}$ to be the undirected graph with underlying set $I$ and edge set equal to the union the following sets:
\begin{itemize}
    \item $\{ (x,y) : \text{there exists binary $\rho \in \tau_{n,k}$ such that } (x,y) \in \rho^\mathfrak{I} \}$
    \item $\{(x,y) :\text{there exists ternary $\rho \in \tau_{n,k}$ and $z \in I$ such that } (x,y,z), (x,z,y), \text{or } (z,x,y) \in \rho^\mathfrak{I} \} $
\end{itemize}
 A \emph{connected component} of $\mathfrak{I}$ is a connected component of the connectivity graph of $\mathfrak{I}$. A \emph{uniform connected component} of $\mathfrak{I}$ is a connected component of $\mathfrak{I}^U$. Suppose that $(x,y) \in (T^{n,k}_{0,k-1})^\mathfrak{I}$. In this situation we call $x$ a \emph{transfer source variable} and $y$ a \emph{transfer target variable}. We call any uniform connected component containing a transfer source variable a \emph{transfer source} component. 
 \end{definition}

 We adopt the following convention for listing the uniform connected components of a $\Csp(\mathfrak{P}^{n,k})$ instance $\mathfrak{I}$.  The transfer source components and the other uniform connected components will be enumerated separately, with the transfer source components listed first. Specifically, if we say that $S_1, \dots, S_r, M_1, \dots, M_v$ are the uniform connected components of $\mathfrak{I}$, we mean that $S_1, \dots, S_r$ are exactly the transfer source components.

\begin{definition}\label{def:innerouterinstances+targetcompatible}
Let $n \geq 2$ and $k \geq 1$.
\begin{enumerate}

\item Given a uniform instance $\mathfrak{I}$ of $\Csp(\mathfrak{P}^{n,k})$, we denote by $\mathfrak{I}^{(0)}$ the \emph{inner restriction} of $\mathfrak{I}$, which is the following instance of $\Csp(\mathfrak{P}^{n,k-1})$:
\begin{itemize}
    \item The domain of $\mathfrak{I}^{(0)}$ is denoted $I^{(0)}$ and is equal to $I$.
    \item $\rho^{\mathfrak{I}^{(0)}} = \bigcup_{\lambda \in \tau_{O_n}} (\rho, \lambda)^\mathfrak{I} $, for each $\rho \in \tau_{n,k-1}$.
    
\end{itemize}

\item Given a uniform instance $\mathfrak{I}$ of $\Csp(\mathfrak{P}^{n,k})$, we denote by $\mathfrak{I}^{ \lnot (0)}$ the \emph{outer restriction} of $\mathfrak{I}$, which is the following instance of $\Csp(\mathfrak{O}_n)$:
\begin{itemize}
    \item The domain of $\mathfrak{I}^{\lnot (0)}$  is denoted $I^{\lnot(0)}$ and is equal to $I$.
    \item $\lambda^{\mathfrak{I}^{(0)}} = \bigcup_{\rho \in \tau_{n,k-1}} (\rho, \lambda)^\mathfrak{I} $, for each $\lambda \in \tau_{O_n}$.
\end{itemize}

\item We say that an instance $\mathfrak{I}$ of $\Csp(\mathfrak{P}^{n,k})$ is \emph{transfer target compatible} if for every transfer target variable $y$, there is a unary symbol $\rho \in \tau_{n,k}$ with $\rho^{\mathfrak{P}^{n,k}} \subseteq P^{n,k}_{k-1}$ and $y \in \rho^\mathfrak{I}$. Given an instance $\mathfrak{I}$ of $\Csp(\mathfrak{P}^{n,k})$, we denote by $\mathfrak{I}^*$ the instance obtained from $\mathfrak{I}$ by asserting for all transfer target variables $y$ the extra unary constraint $y \in \lambda^\mathfrak{I}$, where 
\[
\lambda = ( (\dots ((R_{\{\epsilon\}}, R_\emptyset), R_{O_n}), \dots, R_{O_n}),R_{O_n}) \in \tau_{n,k}
\]
is the symbol naming $P^{n.k}_{k-1} \subseteq P^{n,k}$.
\end{enumerate}
\end{definition}

\begin{lemma}\label{lem:uniformsolutionssplit}
    Let $n \geq 2$ and $k \geq 1$. Let $\mathfrak{I}$ be a uniform instance of $\Csp(\mathfrak{P}^{n,k})$ with only one uniform connected component.  Then, every solution $s$ of $\mathfrak{I}$ satisfies exactly one of the following:
    \begin{enumerate}
    \item  $s(x) \in P^{n,k}_0$  for all $x \in I$, which is the case if and only if $(\psi^{n,k}_0)^{-1}\circ s$ is a solution of the inner restriction $\mathfrak{I}^{(0)}$.
    \item $s(x) \in  O_n$ for all $x \in I$, which is the case if and only if $s$ is a solution of the outer restriction $\mathfrak{I}^{\lnot (0) }$.
    
\end{enumerate}

\end{lemma}
\begin{proof}
    Let $\mathfrak{I}$ be such an instance and let $s \colon I \to P^{n,k}$ be a solution of $\mathfrak{I}$. We show that either $s(x) \in P^{n,k}_0$ for all $x \in I$, or $s(x) \notin P^{n,k}_0$ for all $x \in I$. Indeed, since $s$ is a solution, one of these cases must hold for some $y \in I$, suppose it holds that $s(y) \in P^{n,k}_0$. Now let $z \in I$ be connected to $y$ by a binary or ternary $\tau^U_{n,k}$ relation, for example, suppose $(y,z)\in (\rho_1, \rho_2)^\mathfrak{I}$. Since there are no edges that connect $P^{n,k}_0$ to $O_n$ in $(\rho_1, \rho_2)^{\mathfrak{P^{n,k}}}$, we conclude that $s(z)$ must also belong to $P^{n,k}_0$. Since $\mathfrak{I}$ is connected by uniform relations, an easy induction on path length proves that $s(x) \in P^{n,k}_0$ for all $x \in I$. Obviously, the argument works the same way for the case when $s(y) \in  O_n$. The remaining statements of the lemma are now immediate consequences of what we just proved and the definitions of $\mathfrak{I}^{(0)}$, $\mathfrak{I}^{\lnot(0)}$, $\mathfrak{P}^{n,k}$, and $\mathfrak{O}_n$.     
\end{proof}

Lemma \ref{lem:uniformsolutionssplit} establishes a close connection between uniform instances $\mathfrak{I}$ of $\Csp(\mathfrak{P}^{n,k})$ with a single uniform connected component and the inner and outer restrictions $\mathfrak{I}^{(0)}$ and $\mathfrak{I}^{\lnot(0)}$. The main idea underlying our recursive reduction scheme is to replace the inner restriction $\mathfrak{I}^{(0)}$ of such instances with its reduced instance and then combine this data with the outer restriction $\mathfrak{I}^{\lnot (0)}$. Hence, our next task is to understand how an instance $\mathfrak{I}_1$ of $\Csp(\mathfrak{P}^{n,k-1})$ may be combined with an instance of $\mathfrak{I}_2$ of $\Csp(\mathfrak{O}_n)$ to obtain a uniform instance $\mathfrak{I}$ of $\Csp(\mathfrak{P}^{n,k})$ with the property that $\mathfrak{I}_1 = \mathfrak{I}^{(0)}$ and $\mathfrak{I}_2 = \mathfrak{I}^{\lnot (0)}$. A moment's reflection reveals that, for fixed $\mathfrak{I}_1$ and $\mathfrak{I}_2$, there are potentially many instances $\mathfrak{I}$ with the above property. This is illustrated in Example~\ref{ex:differentinstancesums}.

\begin{example}\label{ex:differentinstancesums}
Let $I = \{x,y\}$. Let $\mathfrak{I}_1$ be the instance of $\Csp(\mathfrak{P}^{n,1})$ with domain $I$ where $(x,y) \in (T^{n,1}_{0,0})^{\mathfrak{I}_1}$ and $(x,y) \in (\Eq_2, R_{((1)(2))})^{\mathfrak{I}_1}$, while all other relations are empty. Let $\mathfrak{I}_2$ be the instance of $\Csp(\mathfrak{O}_n)$ with domain $I$ where $(x,y) \in (R_{((1)(2))})^{\mathfrak{I}_2}$ and $(x,y) \in (R_{((2)(3))})^{\mathfrak{I}_2}$, while all other relations are empty. Consider the following instances $\mathfrak{I}$, $\mathfrak{I}'$, and $\mathfrak{I}''$ of $\Csp(\mathfrak{P}^{n,2})$, each with domain $I$. 
\begin{itemize}
    \item Let $\mathfrak{I}$ be the instance where $(x,y)$ belongs to 
    $(T^{n,1}_{0,0}, R_{((1)(2))})^\mathfrak{I} $ and $((\Eq_2, R_{((1)(2))}), R_{((2)(3))})^\mathfrak{I}$.
    \item Let $\mathfrak{I}'$ be the instance where $(x,y)$ belongs to 
    $((\Eq_2, R_{((1)(2))}) , R_{((1)(2))})^{\mathfrak{I}'} $ and $(T^{n,1}_{0,0}, R_{((2)(3))})^{\mathfrak{I}'}$.
    \item Let $\mathfrak{I}''$ be the instance where $(x,y)$ belongs to 
    $((\Eq_2, R_{((1)(2))}) , R_{((1)(2))})^{\mathfrak{I}''} $, $(T^{n,1}_{0,0}, R_{((1)(2))})^{\mathfrak{I}''}$, and $(T^{n,1}_{0,0}, R_{((2)(3))})^{\mathfrak{I}''}$.
\end{itemize}
Clearly, $\mathfrak{I}^{(0)} = (\mathfrak{I}')^{(0)} = (\mathfrak{I}'')^{(0)} = \mathfrak{I}_1$ and $\mathfrak{I}^{\lnot (0)} = (\mathfrak{I}')^{\lnot(0)} = (\mathfrak{I}'')^{\lnot(0)} = \mathfrak{I}_2$.
\end{example}
In view of Example~\ref{ex:differentinstancesums}, we make the following definitions.

\begin{definition}\label{def:sumofinstances}
    Let $n \geq 2$ and $k \geq 0$. Let $\mathfrak{I}_1$ and $\mathfrak{I}_2$ be instances of $\Csp(\mathfrak{P}^{n,k})$ and $\Csp(\mathfrak{O}_n)$, respectively. We define
    \[
    \mathfrak{I}_1 \oplus \mathfrak{I}_2  := \{ \text{ Instance  $\mathfrak{I}$ of $\Csp(\mathfrak{P}^{n,k+1})$ } : \mathfrak{I}^{(0)} = \mathfrak{I}_1 \text{ and } \mathfrak{I}^{\lnot(0)} = \mathfrak{I}_2 \}.
    \]
   Notice that $\mathfrak{I}_1 \oplus \mathfrak{I}_2$ may be empty.
\end{definition}

\begin{definition}\label{def:uniformproduct}
    Let $n \geq 0$, $k \geq 0$, and let $\mathfrak{I}_1$ and $\mathfrak{I}_2$ be instances of $\Csp(\mathfrak{P}^{n,k})$ and $\Csp(\mathfrak{O}_n)$, respectively, such that $I_1 = I_2$. We define an instance $\mathfrak{I}_1 \cdot \mathfrak{I}_2$ of $\Csp(\mathfrak{P}^{n,k+1})$ with domain $I = I_1 = I_2$ and constraints defined by $(x_1, \dots,  x_m) \in (\rho_1, \rho_2)^{\mathfrak{I}_1 \cdot \mathfrak{I}_2}$ if and only if $(x_1, \dots, x_m) \in \rho_1^{\mathfrak{I}_1}$ and $(x_1, \dots, x_m) \in \rho_2^{\mathfrak{I}_2}$, for each $1 \leq m \leq 3$. 
\end{definition}

\begin{lemma}\label{lem:nonemptyuniformsum}
    Let $n \geq 0$, $k \geq 0$, and let $\mathfrak{I}_1$ and $\mathfrak{I}_2$ be instances of $\Csp(\mathfrak{P}^{n,k})$ and $\Csp(\mathfrak{O}_n)$, respectively. Then, $\mathfrak{I}_1 \cdot \mathfrak{I}_2 \in \mathfrak{I}_1 \oplus \mathfrak{I}_2$ if and only if $I_1 = I_2$ and $\Sup_j(\mathfrak{I}_1) = \Sup_j(\mathfrak{I}_2)$ for each $1\leq j \leq 3$.  
\end{lemma}
\begin{proof}
    Suppose that $\mathfrak{I}_1 \cdot  \mathfrak{I}_2 \in \mathfrak{I}_1 \oplus \mathfrak{I}_2$. By Definition~\ref{def:sumofinstances}, we have that $(\mathfrak{I}_1 \cdot \mathfrak{I}_2)^{(0)} = \mathfrak{I}_1$ and $(\mathfrak{I}_1 \cdot \mathfrak{I}_2)^{\lnot (0) } = \mathfrak{I}_2$, hence $I_1 = I_2$. It is straightforward to apply \emph{1.}\ and \emph{2.}\ of Definition~\ref{def:innerouterinstances+targetcompatible} to conclude that $\Sup_j(\mathfrak{I}_1) = \Sup_j(\mathfrak{I}_2)$ for each $1\leq j \leq 3$. 

    For the other direction, take $\mathfrak{I}_1$ and $\mathfrak{I}_2$ so that $I_1 = I_2$ and $\Sup_j(\mathfrak{I}_1) = \Sup_j(\mathfrak{I}_2)$ for each $1\leq j \leq 3$.
    It is also similarly straightforward to verify that $\mathfrak{I}_1 \cdot \mathfrak{I}_2 \in \mathfrak{I}_1 \oplus \mathfrak{I}_2$.
\end{proof}






We now define a central property of our reduction from
$\Csp(\mathfrak{P}^{n,k+1})$ to $\Csp(\mathfrak{P}^{n,k})$, namely that it comes with a \emph{reduction atlas} which relates the solutions 
to an instance $\mathfrak{I}$ of $\Csp(\mathfrak{P}^{n,k+1})$ to the solutions of the computed instance of $\Csp(\mathfrak{P}^{n,k})$. 

\begin{definition}\label{def:(k-1)reduction}
Let $k \geq 0$, let $\overline{\mathfrak{I}}$ be an instance of $\Csp(\mathfrak{P}^{n,k+1})$,
and let $\overline{\mathfrak{I}}$ be an instance of $\Csp(\mathfrak{P}^{n,k})$ such that $\overline{I} = I$.
Then a \emph{reduction atlas}
 $\Pi$ 
 for $(\bI,\overline{\mathfrak{I}})$ is 
  a function $\Pi \colon I \to \{0, \dots, k \}$
    which encodes the relationship between solutions of $\mathfrak{I}$ and the solutions of $\overline{\mathfrak{I}}$ in the following way.
    \begin{enumerate}
    \item If $ 0 \leq \Pi(x) < k$, then $s(x) \in P^{n,k+1}_{\Pi(x)}$ for every solution $s \colon I \to P^{n,k+1}$ of $\mathfrak{I}$. 
    \item For every solution $s \colon I \to P^{n,k+1}$ of $\mathfrak{I}$, there is a solution $\overline{s} \colon \overline{I} \to P^{n,k}$ of $\overline{\mathfrak{I}}$  defined as
    \[
    \overline{s}(x) = 
    \begin{cases}
        \phi_{0^{k}}(s(x)) & \text{ if $\Pi(x) = k$, otherwise}\\
        (\psi^{n,k+1}_{\Pi(x)})^{-1}(s(x)) & \text{ if $0 \leq \Pi(x) \leq k-1$.}
    \end{cases}
    \label{eq:R1} \tag{R1}
    \]
    \item For every solution $\overline{s} \colon \overline{I} \to P^{n,k}$, there is a solution $s \colon I \to P^{n,k+1} $ which satisfies 
    \[
    s(x) \in 
    \begin{cases}
        \phi^{-1}_{0^{k}}(\overline{s}(x)) & \text{ if $\Pi(x) = k$, otherwise}\\
        \{\psi^{n,k+1}_{\Pi(x)}(\overline{s}(x))\} & \text{ if $0 \leq \Pi(x) \leq k-1$.}
    \end{cases}
     \label{eq:R2} \tag{R2}
    \]
    \end{enumerate}
    Notice that if $k=0$, then the definitions in 2.\ and 3.\ above prioritize the first case.
\end{definition}

We provide some explanation of the reduction atlas $\Pi$ of Definition~\ref{def:(k-1)reduction} with Figure \ref{fig:reductionexplanation}. Roughly speaking, the partition $\Pi^{-1}(0), \dots, \Pi^{-1}(k)$ of $I = \overline{I}$ provides the information necessary to translate between solutions $s$ of $\mathfrak{I}$ and solutions $\overline{s}$ of $\overline{\mathfrak{I}}$. In the figure we show four variables $x_1, \dots, x_4 \in I$ for some instance $\mathfrak{I}$ of $\Csp(\mathfrak{P}^{3,4})$, where the instance itself is not depicted. Each potato in the top row is a diagram of the algebra $\mathbf{P}^{3,4}$, while each potato in the bottom row is a diagram of the algebra $\mathbf{P}^{3,3}$. If a variable $x \in I$ satisfies $\Pi(x) = i$ for $0 \leq i \leq 2$, then $s(x)$ and $\overline{s}(x)$ are related by the isomorphism $\psi^{3,4}_j$. When $\Pi(x) = 3$, they are related by the graph of the homomorphism $\phi_{(0,0,0)}$, i.e., the homomorphism which collapses monolith of $\mathbf{P}^{3,4}$. 

\begin{figure}
    \centering
    \includegraphics[width=0.8\linewidth]{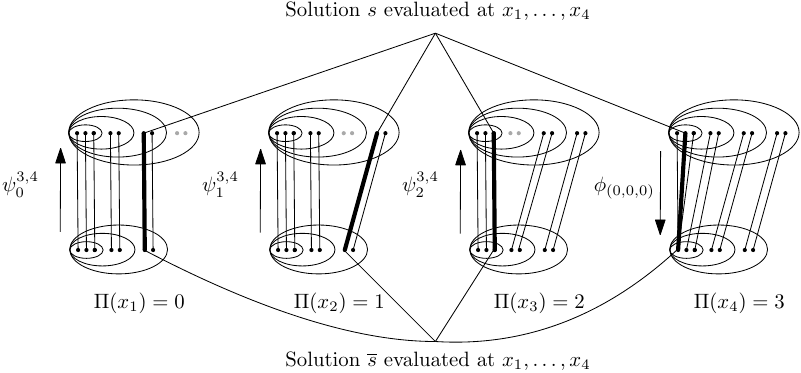}
    \caption{Example for $n=3, k=3$ of the connection between solutions $s$ of $\mathfrak{I}$ and $\overline{s}$ of $\overline{\mathfrak{I}}$.}
    \label{fig:reductionexplanation}
\end{figure}

Now we explain how the information provided by the reduction atlas $\Pi$ is used during the run of $\Reduce_{n,k+1}(\mathfrak{I})$. The idea is straightforward: all uniform relations are reduced using recursive calls to $\Reduce_{n,k}$ and combining the outputs with the original constraints. The reduction atlas $\Pi$ is already produced at this stage of the computation and it is then used to reduce the remaining nonuniform constraints (Algorithm~\ref{alg:outerreduce}). The constraints using the full relation are obviously reduced again to the full relation, so the function $\Pi$ is only used to reduce constraints involving the canonical transfer relation. For each constraint $(x,y)\in (T^{n,k+1}_{0,k})^{\mathfrak{I}}$ in the instance $\mathfrak{I}$, we have $\Pi(x) =0$ and $\Pi(y) = i$ for some $0 \leq i \leq k$. Each possible output value of $\Pi$ indicates a way that the set of solution values $P^{n,k+1}$ for $\mathfrak{I}$ can be interpreted as a set of solution values $P^{n,k}$ for the reduction $\overline{I}$. When $0 \leq i \leq k-1$ solutions must evaluate to an element of $P^{n,k+1}_i$, which is isomorphic to $P^{n,k}$. When $i = k$, we take a quotient by the monolith of $\mathbf{P}^{n,k+1}$. The new constraint in $\overline{\mathfrak{I}^N}$ is obtained by analyzing the interaction between the canonical transfer relation $T^{n,k+1}_{0,k}$ and each of the possibilities. The resulting relations are each named by some $\tau_{n,k}$ symbol. Since $\tau_{n,k}$ is built recursively, we first give these relation symbols alternative names so that they are more easily identified. 

\begin{definition}\label{def:specialsignaturenames}
Let $n \geq 2$ and let $ r \geq 1$. We recursively define

\begin{enumerate}\label{def:reducetransfersymbols}
    \item $S^r_0 := T^{n,r}_{0,y-1} \in \tau_{n,r}$
    \item $S^r_{i+1} := (S^y_i, R_{\id}) \in \tau_{n,r+i+1}$ for $0 \leq i$.
\end{enumerate}
We also define $\Eq_2:= ((\dots(\Eq_2, R_{\id}) \dots, R_{\id}), R_{\id} )$ to be the symbol which interprets as the binary equality relation.
\end{definition}
Informally, the relation $S^y_l$ is the canonical transfer relation for $\mathfrak{P}^{n,y}$, but now with $l$-many additional block congruences enclosing it and additional constant pairs for the extra elements.  Now we explain how to reduce canonical transfer relation constraints. The reader can consult Figure \ref{fig:reductionoftransferinstance} for a complete picture of how the reduction works for the case $n=3$ and $k=5$.
\begin{figure}
    \centering
    \includegraphics[width=0.75\linewidth]{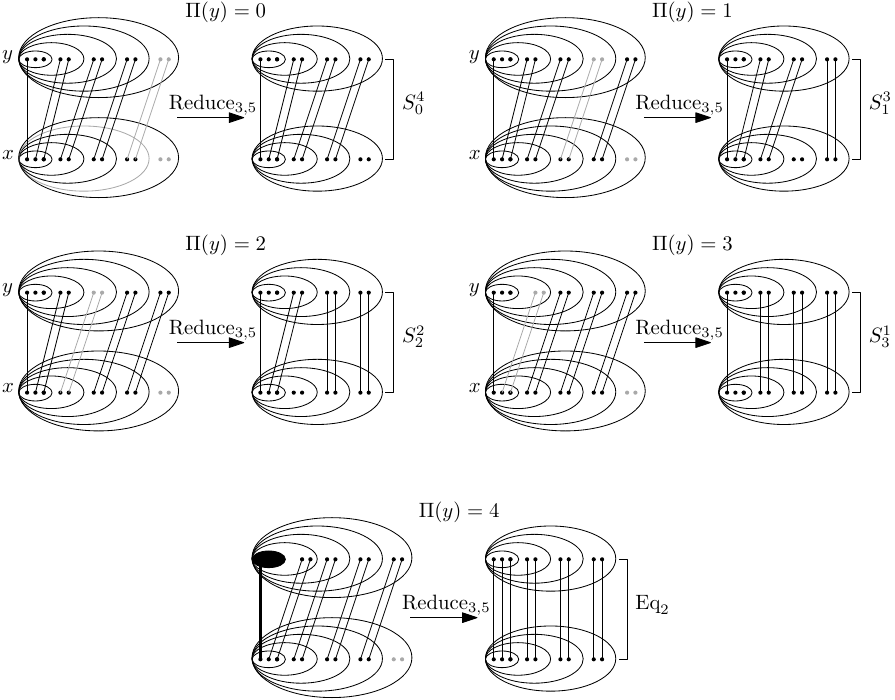}
    \caption{Assignment of constraints to $\overline{\mathfrak{I}^N}$ for $k=4,n=5$.}
    \label{fig:reductionoftransferinstance}
\end{figure}
There, each of the five possible situations that can occur are shown. For example, in the top left part of the figure two variables $x, y$ for which $T^{3,5}_{0,4}(x,y)$ holds are shown along with their respective $\mathbf{P}^{3,5}$ potatoes. We suppose that this particular $y$ is such that $\Pi(y) = 0$, which means that it belongs to some transfer source component and hence the two elements depicted in gray cannot equal $s(y)$ for a solution $s$ of $\overline{I}$. Similarly, the two elements depicted in gray in the potato for $x$ cannot equal $s(x)$ and neither of the two edges depicted in gray can be edges connecting $s(x)$ to $s(y)$. Deleting this information from these two potatoes produces the relation $S_0^4$, which we show further to the right. The explanation for the remaining situations is similar. We formalize this with the procedure $\NReduce_{n,k+1}$, which is called as a subroutine at the end of $\Reduce_{n,k+1}$.

\begin{algorithm}
\label{alg:outerreduce}
\caption{$\NReduce_{n,k+1}$ for $k \geq 1$}
\KwIn{nonuniform instance $\mathfrak{I}$ of $\Csp(\mathfrak{P}^{n,k+1})$ and a function $\Pi: I \to \{0, \dots, k\}$}\

\Make Instance $\overline{\mathfrak{I}}$ of $\Csp(\mathfrak{P}^{n,k-1})$ with domain $\overline{I}= I$ and all relations initially empty.

$\Full^{\overline{\mathfrak{I}}} := \Full^\mathfrak{I}$ \\
\For{ $ 0 \leq l < k $}
{
$(S^{k-l}_l)^{\overline{\mathfrak{I}}} := \{ (x,y) : \Pi(y) = l \} \cap  (T^{n,k+1}_{0,k})^\mathfrak{I}$
}
$\Eq_2^{\overline {\mathfrak{I}}} := \{ (x,y) : \Pi(y) = k \} \cap  (T^{n,k+1}_{0,k})^\mathfrak{I}$\\

\KwOut{$\overline{\mathfrak{I}}$} \
\end{algorithm}

\begin{proposition}\label{prop:proofofreduction}
    Let $\mathfrak{I}$ be an instance of $\Csp(\mathfrak{P}^{n,k})$, for $n \geq 2$ and $k \geq 1$. Let $\overline{\mathfrak{I}}$ be the instance computed by $\Reduce_{n,k}(\mathfrak{I})$ and let $\Pi$ be the function computed by $\Reduce_{n,k}(\mathfrak{I})$. The following hold:
    \begin{enumerate}
        \item $\Pi$ is a reduction atlas for $(\mathfrak{I}, \overline{\mathfrak{I}})$. 
        \item 
    $\Sup_2(\overline{\mathfrak{I}}) = \Sup_2(\mathfrak{I})$ and $\Sup_3(\overline{\mathfrak{I}}) = \Sup_3(\mathfrak{I})$.
    If $\mathfrak{I}$ is concise, then also $\Sup_1(\overline{\mathfrak{I}}) = \Sup_1(\mathfrak{I})$.
   \end{enumerate}
\end{proposition}

\begin{proof}
   The proof proceeds by induction on $k$. For the basis, suppose that $k=1$. Let $\mathfrak{I}$ be an instance of $\Csp(\mathfrak{P}^{n,1})$. We are referring to $\Reduce_{n,1}$ here, see Algorithm~\ref{alg:basisreduction}. Let us show that \emph{1.} of the proposition holds, that is, let us verify that Definition~\ref{def:(k-1)reduction} is satisfied. First, we note that \emph{1.} of Definition~\ref{def:(k-1)reduction} is trivially satisfied, since the function $\Pi$ takes constant value $0$ and in this case $k=0$. Next, let $C_1, \dots, C_v$ be the connected components of $\mathfrak{I}$ and let $\mathfrak{C}_1, \dots, \mathfrak{C}_v$ be the respective induced subinstances of $\mathfrak{I}$. Clearly, $\mathfrak{I}$ has a solution if and only if each $\mathfrak{C}_j$ has a solution, for each $1 \leq j \leq v$. 
   
   Since we established in the basis of the induction argument of Theorem~\ref{thm:solveiscorrect} that $\Solve_{n,1}$ solves $\Csp(\mathfrak{P}^{n,1})$, it follows that $\phi_{0^0} \circ s$ is a solution to $\overline{\mathfrak{I}}$ whenever $s$ is a solution to $\mathfrak{I}$. Conversely, there exists a solution $s\in \phi_{0^0}^{-1} \circ \overline{s}$ of $\mathfrak{I}$ for every solution $\overline{s}$ to $\overline{\mathfrak{I}}$ (recall that $0^0= \epsilon$ and $\phi_{\epsilon}$ is the function which maps every element of $P^{n,1}$ to $\epsilon \in P^{n,0}$). Hence, items \emph{2.} and \emph{3.} of Definition~\ref{def:(k-1)reduction} hold, so \emph{1.} of the proposition is proved. Upon inspection of Algorithm~\ref{alg:basisreduction}, it is clear that \emph{2.} of the proposition also holds, so the basis for the induction is established.

     Now suppose the proposition holds for $k \geq 1$. We show that it also holds for $k+1$, so here we refer to Algorithm \ref{alg:recursivereduction}. Let $\mathfrak{I}$ be an instance of $\Csp(\mathfrak{P}^{n,k+1})$. We will first argue that \emph{1.} of the proposition holds. The first steps of $\Reduce_{n,k+1}$  are to replace $\mathfrak{I}$ with $(\mathfrak{I}^*)^C$. It is easy to see that the solutions sets to $\mathfrak{I}$ and $(\mathfrak{I}^*)^C$ are equal. Henceforth, we assume that $\mathfrak{I}$ is both transfer target compatible and concise. i.e., that $\mathfrak{I} = (\mathfrak{I}^*)^C$. 
     
     We first fix some notation for the remainder of the proof. Let $S_1, \dots, S_r, M_1, \dots, M_v$ be the uniform connected components of $I$, with $S_1, \dots, S_r$ distinguished as the source components (see Definition~\ref{def:components}). Let $\mathfrak{S}_1, \dots, \mathfrak{S}_r, \mathfrak{M}_1, \dots, \mathfrak{M}_v$ be the induced subinstances of $\mathfrak{I}$ with respective domains the uniform connected components listed above and let $\mathfrak{S}_1^U, \dots, \mathfrak{S}_r^U, \mathfrak{M}^U_1, \dots, \mathfrak{M}^U_v$ be their respective uniform reducts. Let $\mathfrak{S}^U = \mathfrak{S}^U_1 \cup \dots \cup \mathfrak{S}^U_r$ and $\mathfrak{M}^U = \mathfrak{M}^U_1 \cup \dots \cup \mathfrak{M}^U_v$. Let $\overline{(\mathfrak{M}_j^U)^{(0)}}$ and $\Pi_j$ be the instance and reduction atlas output by $\Reduce_{n,k}((\mathfrak{M}_j^U)^{(0)})$ for each $1 \leq j \leq v$  and let $\overline{(\mathfrak{M}^U)^{(0)}}$ and $\Pi_M$ be the instance and reduction atlas output by $\Reduce_{n,k}((\mathfrak{M}^U)^{(0)})$. Let $\overline{\mathfrak{M}^U} = \overline{(\mathfrak{M}^U)^{(0)}}  \cdot (\mathfrak{M}^U)^{\lnot(0)}$ and let $\overline{\mathfrak{M}^U_j} = \overline{(\mathfrak{M}_j^U)^{(0)}}  \cdot (\mathfrak{M}_j^U)^{\lnot(0)}$, for each $1 \leq j \leq v$.

     \begin{claim}\label{claim:reductiondistributes}
          The following hold. 
         \begin{enumerate}
         \item $\overline{\mathfrak{M}^U_j} \in \overline{(\mathfrak{M}_j^U)^{(0)}}  \oplus (\mathfrak{M}_j^U)^{\lnot(0)}$, for each $1 \leq j \leq v$.
         \item $\overline{(\mathfrak{M}^U)^{(0)}} = \overline{(\mathfrak{M}_1^U)^{(0)}} \cup \dots \cup \overline{(\mathfrak{M}_v^U)^{(0)}}  $.
         \item $\Pi_M = \Pi_1 \cup \dots \cup \Pi_v$.
         
         \item  $\overline{\mathfrak{M}^U} = \overline{\mathfrak{M}^U_1} \cup \dots \cup \overline{\mathfrak{M}^U_v}$.
        \end{enumerate}
     \end{claim}

     \begin{claimproof}
         We first show item \emph{1.} of the claim. Since $\mathfrak{I}$ is concise, it follows that $\mathfrak{M}_j^U$ is concise, and then also that its inner restriction and outer restrictions $(\mathfrak{M}_j^U)^{(0)}$ and $(\mathfrak{M}_j^U)^{\lnot (0)}$ are concise. It follows from the inductive hypothesis that 
    \begin{align*}
    \Sup_1(\overline{(\mathfrak{M}_j^U)^{(0)}}) &= \Sup_1((\mathfrak{M}_j^U)^{(0)}),\\
    \Sup_2(\overline{(\mathfrak{M}_j^U)^{(0)}}) &= \Sup_2((\mathfrak{M}_j^U)^{(0)}), \text{ and }\\
    \Sup_3(\overline{(\mathfrak{M}_j^U)^{(0)}}) &= \Sup_3((\mathfrak{M}_j^U)^{(0)}).
    \end{align*}
    Moreover, it follows readily from the definitions that 
     \begin{align*}
    \Sup_1((\mathfrak{M}_j^U)^{(0)}) &= \Sup_1((\mathfrak{M}_j^U)^{\lnot (0)}),\\
    \Sup_2((\mathfrak{M}_j^U)^{(0)}) &= \Sup_2((\mathfrak{M}_j^U)^{\lnot (0)}), \text{ and }\\
    \Sup_3((\mathfrak{M}_j^U)^{(0)}) &= \Sup_3((\mathfrak{M}_j^U)^{\lnot (0)}).
    \end{align*}
    Hence, item \emph{1.} follows from Lemma~\ref{lem:nonemptyuniformsum}.

    Items \emph{2.} and \emph{3.} of the claim follow from the observation that $\Reduce_{n,k}$ only uses information obtained by considering connected components, hence both outputs distribute over a union of disconnected instances. Item \emph{4.} of the claim follows from item \emph{2.} of the claim and Definition~\ref{def:uniformproduct}.
     \end{claimproof}

    \begin{claim}\label{claim:(a)holds}
         Item 1. of Definition~\ref{def:(k-1)reduction} holds, that is, $s(x) \in P^{n,k+1}_{\Pi(x)}$ for every solution $s \colon I \to P^{n,k+1}$ of $\mathfrak{I}$, whenever $0 \leq \Pi(x) < k$.
    \end{claim}
    \begin{claimproof}
         Suppose that $s \colon I \to P^{n,k+1}$ is a solution to $\mathfrak{I}$ and that $\Pi(x)=i$ for some $0 \leq i < k$. We want to see that $s(x)$ belongs to $P^{n,k+1}_i$. We analyze two cases.

    \begin{itemize}
         \item Suppose $i=0$. Evidently, this is the case if and only if $x$ belongs to a transfer source component $S_j$ for some $1 \leq j \leq r$. From Lemma~\ref{lem:uniformsolutionssplit}, we know that $s(y) \in P^{n,k+1}_0$ for all $y \in S_j$ or $s(y) \in O_n$ for all $y \in S_j$. Since $S_j$ is a transfer source component, there exists $z \in S_r$ and $w \in I$ so that $T^{n,k+1}_{0,k}(z,w)$ is a constraint of $\mathfrak{I}$. Since the image of $T^{n,k+1}_{0,k}$ under projection onto the first coordinate is $P^{n,k+1}_0$, we conclude that $s(x) \in P^{n,k+1}_0$. 

         \item Suppose that $i \neq 0$. In this case $x$ belongs to $M$, and hence to $M_j$ for some $1\leq j \leq v$. Again applying Lemma~\ref{lem:uniformsolutionssplit}, we deduce that $s(y) \in P^{n,k+1}_0$ for all $y \in M_j$ or $s(y) \in O_n$ for all $y \in M_j$. If the latter holds, then $s(x) \in P^{n,k+1}_i$, since $O_n \subseteq P^{n,k+1}_i$ whenever $1 \leq i \leq k$. If the former holds, then by \emph{1.} of Lemma~\ref{lem:uniformsolutionssplit} the restriction to $M_j$ of $(\psi^{n,k+1}_0)^{-1}\circ s$ is solution to $(\mathfrak{M}_i^U)^{(0)}$. By the definition of $\Pi$ and \emph{3.} of Claim~\ref{claim:reductiondistributes}, we deduce that $\Pi_j(x) =i-1$. By the inductive hypothesis, this implies that $((\psi^{n,k+1}_0)^{-1}\circ s)(x) \in P^{n,k}_{i-1}$. It follows that $s(x) \in \psi_0^{n,k+1}(P^{n,k}_{i-1})$. We now leave it to the reader to verify that $\psi_0^{n,k+1}(P^{n,k}_{i-1}) \subseteq P^{n,k+1}_i$.
    \end{itemize}
    Hence, the claim is established.
              \end{claimproof}

    \begin{claim}\label{claim:(b)holds}
        Item 2. of Definition~\ref{def:(k-1)reduction} holds, that is, given a solution $s \colon I \to P^{n,k+1}$ of $\mathfrak{I}$, the function $\overline{s}$ defined in \ref{eq:R1} is a solution to $\overline{\mathfrak{I}}$. 
    \end{claim}
    \begin{claimproof}
    Let $s \colon I \to P^{n,k+1}$ be a solution to $\mathfrak{I}$ and let $\overline{s} \colon I \to P^{n,k}$ be as defined in \ref{eq:R1} (note that it follows from Claim~\ref{claim:(a)holds} that $\overline{s}$ is well-defined, since $P^{n,k+1}_i$ is the image of $\psi^{n,k+1}_i$). We argue that $\overline{s}$ is indeed a solution to $\overline{\mathfrak{I}}$. Since $\overline{I} = \overline{\mathfrak{S}^U} \cup \overline{\mathfrak{M}^U} \cup \overline{\mathfrak{I}^N}$, this amounts to showing that $\overline{s}$ is a solution to each of these instances.

    \begin{itemize}
        \item It is straightforward to see that $\overline{s}$ restricted to $S$ for $1\leq j \leq r$ is a solution to $\overline{\mathfrak{S}^U}$. Indeed, the domain of $\mathfrak{S}^U$ is $S = S_1 \cup \dots \cup S_r$ is the union of source components and we have already observed that $s(x) \in P^{n,k+1}_0$ for all $1 \leq j \leq r$ and $x \in S_j$. We assume $s$ restricted to $S$ is a solution to $\mathfrak{S}_j^U$, so it follows from the definition of an inner restriction (see Definition \ref{def:innerouterinstances+targetcompatible}) that $\overline{s}$ is a solution to $\overline{\mathfrak{S}^U} = (\mathfrak{S}^U)^{(0)}$.

        \item Now we show that $\overline{s}$ restricted to $M$ is a solution to $\overline{\mathfrak{M}^U}$, which in view of item \emph{4.} of Claim \ref{claim:reductiondistributes}, is true if and only if $\overline{s}$ restricted to $\overline{\mathfrak{M}_j^U}$, for every $1 \leq j \leq v$. By Lemma \ref{lem:uniformsolutionssplit}, we know that either $s(x) \in P^{n,k+1}_0$, for all $x \in M_j$, or $s(x) \in O_n$, for all $x \in M_j$. 
        \begin{itemize}
            
        \item 
        In the former case, it follows that $(\psi_0^{n,k+1})^{-1} \circ s$ is a solution to $(\mathfrak{M}_j^U)^{(0)}$. By the inductive hypothesis, we know that $\Pi_j$ is a reduction atlas for the pair of instances $((M_j^U)^{(0)}, \overline{(\mathfrak{M}_j^U)^{(0)}})$. It follows that 
        \[
        \overline{(\psi_0^{n,k+1})^{-1} \circ s} = 
        \begin{cases}
        (\phi_{(1)^{k-1}} \circ (\psi_0^{n,k+1})^{-1} \circ s)(x) & \text{ if $\Pi_j(x) = k-1$},\\
        ((\psi^{n,k}_{\Pi_j(x)})^{-1} \circ (\psi_0^{n,k+1})^{-1} \circ s)(x) & \text{ if  $0 \leq \Pi_j(x) \leq k-2$.}
        \end{cases}
        \]
        is a solution to $\overline{(\mathfrak{M}_j^U)^{(0)}}$. Since 
        \[
        \overline{\mathfrak{M}_j^U} = \overline{ (\mathfrak{M}_j^U)^{(0)}}  \cdot (\mathfrak{M}_j^U)^{\lnot(0)},
        \]
        it follows (from Definitions~\ref{def:sumofinstances} and~\ref{def:uniformproduct} and Lemma~\ref{lem:uniformsolutionssplit}) that $\psi^{n,k+1}_0(\overline{(\psi_0^{n,k+1})^{-1} \circ s}) $ is a solution to $\overline{\mathfrak{M}_j^U}$. The reader can check that the following composition rules hold for any $c \in P^{n,k+1}_0$:
        \begin{align*}
            (\psi^{n,k+1}_0\circ \phi_{(1)^{k-1}} \circ (\psi_0^{n,k+1})^{-1})(c) &= \phi_{(1)^{k}}(c)\\
            (\psi^{n,k+1}_0\circ (\psi^{n,k}_i)^{-1} \circ (\psi_0^{n,k+1})^{-1})(c) &= (\psi^{n,k+1}_{i+1})^{-1} (c).
        \end{align*}
        Applying these composition rules along the fact that $\Pi(x)= \Pi_j(x) +1$ (this follows from $\Pi(x) = \Pi_M(x) +1$ and item \emph{3.}\ of Claim~\ref{claim:reductiondistributes}), we can transform the definition of $\psi^{n,k+1}_0(\overline{(\psi_0^{n,k+1})^{-1} \circ s}) $ to match the definition of $\overline{s}$.
        \item In the latter case, it follows from Lemma \ref{lem:uniformsolutionssplit} that $s$ is a solution to $(\mathfrak{M}_j^{U})^{\lnot (0)}$. Notice that in this case, $\overline{s}(x) = s(x)$ for all $x \in M_j$, hence $\overline{s}$ is a solution to $(\mathfrak{M}_j^{U})^{\lnot (0)}$, which is equal to  $ (\overline{\mathfrak{M}_j^U)}^{\lnot (0)}$. Therefore, $\overline{s}$ is a solution to $\overline{\mathfrak{M}_j^U}$. 
        \end{itemize}
    \end{itemize}

    Now we argue that $\overline{s}$ is a solution to $\overline{\mathfrak{I}^N} = \NReduce(\mathfrak{I}^N, \Pi)$. This follows more or less immediately from the definitions and Algorithm~\ref{alg:outerreduce}. Suppose that $(x,y) \in (T^{n,k+1}_{0,k})^\mathfrak{I}$. We know that $\Pi(x) = 0$ and $\Pi(y) =i$ for some $0 \leq i \leq k$ (since $x$ is a transfer source variable). If $i \neq k$, it follows from the definition of $\overline{s}$ that the pair $(\overline{s}(x), \overline{s}(y))$ must belong to the relation
    \[
    \psi^{n,k+1}_0 \circ T^{n,k+1}_{0,k} \circ (\psi^{n,k+1}_i)^{-1} = (S_i^{k-i})^{\mathfrak{P}^{n,k}} ,
    \]
    which is exactly the constraint that is assigned in $\NReduce_{n,k+1}$ to such $x,y$ (note the abuse of notation, where the above composition is a relational composition). If $i = k$, the same kind of reasoning applies, except here the pair $(\overline{s}(x), \overline{s}(y))$ must belong to the relation 
    \[
    (\psi^{n,k+1}_0)^{-1} \circ T^{n,k+1}_{0,k} \circ \phi_{0^k} = \Eq_2^{\mathfrak{P}^{n,k}}.
    \]
    Since the constraints named by the full relation are always satisfiable, 
    this finishes the proof of Claim~\ref{claim:(b)holds}. 
    \end{claimproof}

     \begin{claim}\label{claim:(c)holds}
        Item 3. of Definition~\ref{def:(k-1)reduction} holds, that is, given a solution $\overline{s} \colon \overline{I} \to P^{n,k+1}$ of $\mathfrak{I}$, there exists a solution $s$ to the instance $\mathfrak{I}$ which satisfies the property defined in \ref{eq:R2}. 
    \end{claim}

    \begin{claimproof}
    Let $\overline{s} \colon  \overline{I} \to P^{n,k}$ be a solution of $\overline{I}$. We will argue that there is a solution $s \colon I \to P^{n,k+1}$ that satisfies \ref{eq:R2} of Definition~\ref{def:(k-1)reduction}. By Claim~\ref{claim:reductiondistributes}, we are able to do this independently on each uniform connected component. Let $\overline{s}_1, \dots, \overline{s}_r, \overline{t}_1, \dots, \overline{t}_v$ be the respective restrictions of $\overline{s}$ to $S_1, \dots, S_r, M_1, \dots, M_v$. We construct the solution $s$ as a union of solutions $s_1, \dots, s_r, t_1, \dots, t_v$, which are respective solutions of the subinstances $\mathfrak{S}^U_1, \dots, \mathfrak{S}^U_r, \mathfrak{M}^U_1, \dots, \mathfrak{M}^U_v$. 
    \begin{itemize}
    \item 
    The definition for transfer source components is clear, which is for $1\leq j \leq r$ to set $s_j(x) = \psi_0^{n,k+1}(\overline{s}_j(x))$. 

    \item
     Let $1 \leq j \leq v$ and consider the component $M_j$, the instance $\mathfrak{M}_j^U$, and the solution $\overline{t}_j$ to $\overline{\mathfrak{M}_j^U} = \overline{(\mathfrak{M}_j^U)^{(0)}} \cdot (\mathfrak{M}_j^U)^{\lnot (0)}$. By Lemma \ref{lem:uniformsolutionssplit}, we know that either $\overline{w}_j := (\psi^{n,k}_0)^{-1} \circ \overline{t}_j$ is a solution to $\overline{(\mathfrak{M}_j^U)^{(0)}}$ or $\overline{t}_j$ is a solution to $(\mathfrak{M}_j^U)^{\lnot (0)}$.
     \begin{itemize}
     \item
     In the latter case, we define $t_j = \overline{t}_j$. 
     \item
     In the former case, there exists by the inductive hypothesis a solution $w_j$ to $(\mathfrak{M}_j^U)^{(0)}$ which satisfies
    
    \[
    w_j(x) \in 
    \begin{cases}
        \phi^{-1}_{(1)^{k-1}}(\overline{w}_j(x))  & \text{ if $\Pi_j(x) = k-1$},\\
        \{\psi^{n,k}_{\Pi_j(x)}(\overline{w}_j(x))\} & \text{ if  $0 \leq \Pi_j(x) \leq k-2$.}
    \end{cases}
    \]
    In this case we we define $t_j(x) = \psi_0^{n,k+1} \circ w_j$. Now we define $s$ to the union of the solutions $s_1, \dots, s_r, t_1, \dots, t_v$. 
    \end{itemize}
     \end{itemize}

    Let us check that $s$ satisfies the condition given in \ref{eq:R2}. Suppose that $\Pi(x)= 0$. Then, $x$ belongs to a transfer source component $S_j$ for $1 \leq j \leq r$, so \ref{eq:R2} is satisfied by an immediate application of the definition of $s$. Suppose that $1 \leq \Pi(x) \leq k$. Then $x$ belongs to some $M_j$ for $1\leq j \leq v$, hence it follows that $s(x) = t_j(x)$ and $\Pi_j(x) = \Pi(x)-1$. It follows from the definition of $t_j$ that 

     \[
    t_j(x) \in 
    \begin{cases}
        (\psi_0^{n,k+1} \circ \phi^{-1}_{(1)^{k-1}} \circ (\psi^{n,k}_0)^{-1} \circ \overline{t}_j)(x) \cup O^{n,1} & \text{ if $\Pi_j(x)= k-1$},\\
        \{(\psi_0^{n,k+1} \circ \psi^{n,k}_{\Pi_j(x)} \circ (\psi^{n,k}_0)^{-1} \circ \overline{t}_j)(x)\} \cup O^{n,1} & \text{ if  $0 \leq \Pi_j(x) \leq k-2$.}
    \end{cases}
    \]
    We leave it to the reader to check that the above condition can be rewritten to the condition \ref{eq:R2} of Definition \ref{def:(k-1)reduction}.

    It is clear from its definition that $s$ is a solution of $\mathfrak{I}^U= \mathfrak{S}^U \cup \mathfrak{M}^U$. So, it remains only to show that $s$ is a solution to $\mathfrak{I}^N$. Again, it is obvious that constraints involving the full relation are satisfied. So, let us consider some $(x,y) \in (T^{n,k+1}_{0,k})^\mathfrak{I}$. We know that $\Pi(x) = 0$ and that $ 0 \leq \Pi(y) \leq k$. If $ \Pi(y) \neq k$, it follows from the definition of $\overline{\mathfrak{I}^N}$ that $(\overline{s}(x), \overline{s}(y)) \in (S_i^{k-i})^{\mathfrak{P}^{n,k}}$. In this case

    \[ (s(x),s(y)) \in 
    (\psi^{n,k+1}_0)^{-1} \circ (S_i^{k-i})^{\mathfrak{P}^{n,k}} \circ (\psi^{n,k+1}_i) = T_{0,k}^{n,k+1},
    \]
    as desired. 
    On the other hand, if $i =k$, then  $(\overline{s}(x), \overline{s}(y)) \in \Eq_2^{\mathfrak{P}^{n,k}}$. Observe that
     \[ (s(x),s(y)) \in 
    (\psi^{n,k+1}_0)^{-1} \circ \Eq_2^{\mathfrak{P}^{n,k}} \circ (\phi^{-1}_{(1)^k}),
    \]
    where the above composition is a relational composition.
    The above relation is equal to 
    \[
    T^{n,k+1}_{0, k} \cup \{ (0^k,0  ), (0^k, a) ) : 1 \leq a \leq n-1 \},
    \]
    but $(s(x), s(y))$ cannot take values from the second set, because unary relations were added at the very start of running $\Reduce_{n,k+1}$ to $\mathfrak{I}$ to assure transfer target compatibility. This unary constraint information is inherited by the uniform reduct $\mathfrak{I}^U$ and we already argued that $s$ is a solution to $\mathfrak{I}^U$.
    \qedhere
    \end{claimproof}

    It remains to argue that \emph{2.} of the proposition statement holds, but this is easy. Upon inspecting Algorithm~\ref{alg:reduce}, the reader can convince themselves that the binary and ternary supports of $\mathfrak{I}$ and $\overline{\mathfrak{I}}$ are equal. The only possibility for the support to change is if the input instance $\mathfrak{I}$ is not concise, since the first steps of the reduction replace $\mathfrak{I}$ by $(\mathfrak{I}^*)^C$. If $\mathfrak{I}$ is concise to begin with, then it also holds that $\Sup_1(\mathfrak{I}) = \Sup_1(\overline{\mathfrak{I}})$.
\end{proof}   

For $s \in {\mathbb N}$, let 
$g(s)$ denote the complexity of Gaussian elimination for a ${\mathbb Z}_2$-linear system with $s$ constraints. 

\begin{theorem}\label{thm:solvealgorithmcomplexity}
    Let $n \geq 2$ and $k \geq 1$. Let $\mathfrak{I}$ be an instance of $\Csp(\mathfrak{P}^{n,k})$. Then, $\Solve_{n,k}(\mathfrak{I})$ outputs \textnormal{ACCEPT} if $\mathfrak{I}$ has a solution,
    and outputs \textnormal{REJECT} otherwise.
    The worst-case running time is in 
    $O(k n s + g(s))$ where $s$ denotes the  number of constraints in the input. 
\end{theorem}
\begin{proof}
    The correctness of the procedure Solve is shown in Theorem~\ref{thm:solveiscorrect}. 
    It is straightforward to prove that the worst-case complexity of $\Reduce_{n,1}$ is in $O(s)$; the complexity
    of $\Reduce_{n,k}$ for $k \geq 2$ is 
    dominated by the computation of the connected components, which can be done in $O(ns)$, and the recursive call, which 
    is in $O((k-1) n s)$ by induction.
    The computation of $E$ and $\bT$ in 
    $\Solve_{n,1}$ can be done in $O(s)$, whereas testing the satisfiability of $E$ is in $g(s)$. 
    Hence, the overall complexity of $\Solve_{n,k}$ is in $O(k (n s + g(s)))$. 
\end{proof}

\section{Descriptive complexity}
\label{sect:descr}
In this section we introduce $\mathbb{G}$-Datalog programs for cyclic groups $\mathbb{G}$. 
We show that if $\mathbb{G}$ is ${\mathbb Z}_p$, then symmetric linear  $\mathbb{G}$-Datalog programs can be evaluated in the complexity class $\Mod_p L$. 
We will then implement the procedures from the previous section in symmetric linear ${\mathbb Z}_2$-Datalog, which therefore shows that our algorithm is in $\oplus L = \Mod_2 L$. 

\subsection{(Symmetric linear) Datalog}
\label{sect:symlin}
A \emph{Datalog program} $\Pi$ consists of two finite disjoint sets of relation symbols, the EDBs $\tau$ and the IDBs $\sigma$, a distinguished IDB $G$ (the \emph{goal predicate}), and a finite set of \emph{rules}
of the form 
$$P(x_1,\dots,x_n) \; {:}{-} \; \phi(x_1,\dots,x_n,y_1,\dots,y_m)$$ where $P$ is an IDB, and $\phi$ is a conjunction of atomic formulas that can be formed from IDBs and EDBs. We assume that the equality relation $=$ is among the EDBs.   
The atomic formula 
$P(x_1,\dots,x_n)$ is called the \emph{head} of the rule, and $\phi$ is called the \emph{body} of the rule. 
We say that $\Pi$ has \emph{width $(\ell,k)$} if every IDB has arity at most $\ell$, and every rule has at most $k$ variables.

To define the operational semantics of Datalog, 
let $\bA$ and $\bA'$ be $(\tau \cup \sigma)$-structures on the same domain $A$.  
We say that the program $\Pi$ \emph{derives $\bA'$ from $\bA$ in one step}, and write $\bA \vdash \bA'$, if $\Pi$ contains a rule $P (x_1,\dots,x_n) \; {:}{-} \; \phi(x_1,\dots,x_n,y_1,\dots,y_m)$ and a satisfying assignment $h$ of $\phi$ in $\bA$ such that 
\begin{itemize}
    \item $P^{\bA'} = P^{\bA} \cup \{(h(x_1),\dots,h(x_n))\}$; 
    \item $Q^{\bA'} = Q^{\bA}$ for all $Q \in \sigma \cup \tau \setminus \{P\}$. 
\end{itemize}
We write $\bA \vdash^n \bA'$, 
for $n \in {\mathbb N}$, if 
$\bA = \bA'$ and $n=0$,
or $\bA \vdash \bA'$ and $n=1$, 
or, inductively, if $\bA \vdash^{n-1} \bA'$ and $\bA'' \vdash \bA'$ if $n > 1$.  
If $A$ is finite, then 
there exists an $n \in {\mathbb N}$
and a structure $\bA^*$ such that 
\begin{itemize}
\item $\bA\vdash^n \bA^*$, 
\item for all $\bA'$ such that $\bA^* \vdash \bA'$ we have $\bA'= \bA^*$,
\item $\bA^*$ is unique, and can be computed in polynomial time from $\bA$. 
\end{itemize} 
In this case, 
we say that $\Pi$ \emph{computes} $\bA^*$ on $\bA$.  
If $a \in A^k$ and $a \in R^{\bA^*}$, then we say that
$\Pi$ derives $R(a)$ on $\bA$. 
We say that a Datalog program $\Pi$ \emph{solves} $\Csp(\bB)$ if $\Pi$ derives the goal predicate $G$ on an instance $\bA$ of $\Csp(\bB)$ if and only if $\bA$ has no homomorphism to $\bB$.

A rule $P(x_1,\dots,x_n) \; {:}{-} \; \phi$ 
in a Datalog program is called \emph{linear} if the body $\phi$ contains at most one conjunct with an IDB. 
A Datalog program $\Pi$ is called \emph{linear} if all its rules are linear. 
A rule is called \emph{recursion-free} if its body does not contain any IDBs. 
A linear rule 
\begin{align}
    P(x_1,\dots,x_n) \; {:}{-} \; \phi \wedge Q(y_1,\dots,y_m) \label{eq:symrule}
\end{align}
of a Datalog program $\Pi$ where $Q \in \sigma$
is an IDB 
is called \emph{symmetric (in $\Pi$)} if 
$\Pi$ also contains the \emph{symmetric rule} (up to renaming variables) 
$$Q(y_1,\dots,y_m) \; {:}{-} \; \phi \wedge P(x_1,\dots,x_n).$$
A linear Datalog program $\Pi$ is called \emph{symmetric}
if all its rules of the form as in~\ref{eq:symrule} are symmetric. Note that if all rules of $\Pi$ are recursion-free then $\Pi$ is symmetric linear. 
One of the motivations of symmetric linear Datalog is the following result.\footnote{Remarkably, every CSP which is known to be in logarithmic space can be solved by a symmetric linear Datalog program.} 

\begin{theorem}[Egri, Larose, Tesson~\cite{EgriLT08}]
\label{thm:L}
    Symmetric linear Datalog programs can be evaluated in L. 
\end{theorem}

\begin{theorem}[Dalmau and Larose~\cite{DalmauLarose08} (Theorem 1)]\label{thm:DL}
    If $\bB$ is a finite structure with a finite relational signature preserved by 
    a majority and a Maltsev polymorphism. 
    Then $\Csp(\bB)$ can be solved by symmetric linear Datalog. 
\end{theorem}



We may also use this formalism to define \emph{(linear, symmetric linear) Datalog transformations}. 
Instead of a particular goal predicate, we then need a subset of the IDBs, the \emph{output predicates}.
If $\rho$ is the set of output predicates, then on input $\fA$, the output structure is the $\rho$-reduct of $\fA^*$. 

\subsection{Cyclic group Datalog}
\label{sect:RDatalog}
Let $\mathbb{G}$ be a cyclic group, i.e., a group that is isomorphic to 
${\mathbb Z}$ or to ${\mathbb Z}_n$ for some $n \in {\mathbb N}$. We introduce 
$\mathbb{G}$-Datalog as an extension of Datalog designed to express polynomial-time algorithms for certain classes of CSPs.  Similarly as usual Datalog, $\mathbb{G}$-Datalog can be further restricted syntactically in order to (hopefully) capture CSPs that fall into a certain complexity class. 

$\mathbb{G}$-Datalog programs are defined in a fixed finite number of stages. 
A $\mathbb{G}$-Datalog-program of stage one 
has IDBs $\sigma \cup \sigma'$, where $\sigma$ is a set of ordinary IDBs, and $\sigma' \subseteq \{L^k_{a_1,a_2,a_3,b} \mid a_1,a_2,a_3,b \in G\}$ 
is a special set of IDBs, for some 
$k \in {\mathbb N}$, 
where $L^k_{a_1,a_2,a_3,b}$
has arity $k+3$. 
For evaluating such a program, 
we first evaluate the Datalog program as usual on a given finite $\tau$-structure $\bA$, reaching a fixed point, which is a  $(\tau \cup \sigma \cup \sigma')$-structure $\bA^*$.
From $\bA^*$, we compute a 
$(\tau \cup \sigma \cup \{L^{\top},L^\bot\})$-structure $\bB$, 
where $L^{\top}$ and $L^{\bot}$ are $k$-ary relation symbols, as follows. 
We consider the system 
consisting of all linear equations 
of the form 
$$a_1 x_{v_1} + a_2 x_{v_2} + a_3 x_{v_3} = b_0$$ for every $(u_1,\dots,u_k,v_1,v_2,v_3) \in A^{k+3}$ such that $\bA^* \models  L^k_{a_1,a_2,a_3,b}(u_1,\dots,u_k,v_1,v_2,v_3)$.
A tuple $(u_1,\dots,u_k) \in A^k$ is contained in 
the relation $(L^\top)^{\bB}$ if this system is satisfiable, otherwise it is contained in $(L^\bot)^{\bB}$. We then allow for a final derivation of the goal predicate from the relation $L^\bot$. As with ordinary Datalog, we may specify an output signature $\rho$, which is a subset of $\sigma \cup \{L^\top, L^\bot\}$.

A $\mathbb{G}$-Datalog program $\Pi_i$ of stage $i$ with EDBs $\tau$, for $i \geq 2$, consists of a $\mathbb{G}$-Datalog program $\Pi_{i-1}$ of stage $i-1$
with input signature $\tau$ and output signature $\rho \subseteq \sigma \cup \{L^\top, L^\bot\}$, and a $\mathbb{G}$-Datalog program $\Pi$ of stage one 
with EDBs $\rho$. On input $\bA$, we first compute
the output $\rho$-structure 
$\bB$ of $\Pi_{i-1}$ on input $\bA$, and then 
the output structure of $\Pi$ on the input structure $\bB$; the output of $\Pi$ on $\bB$ is then defined to be the output of $\Pi_i$ on the input structure $\bA$.

As before, we say that a $\mathbb{G}$-Datalog program \emph{solves} $\Csp(\bB)$ if it derives on a given $\tau$-structure $\bA$ a distinguished goal predicate of arity $0$ in the final stage if and only if there
is no homomorphism from $\bA$ to $\bB$. 
For the CSPs studied in this text, 
it will suffice to work with ${\mathbb Z}_2$-Datalog.  

\begin{lemma}
\label{lem:lin}
    The problem $\Csp(\{0,1\};L,G,E,R_{\{0\}},R_{\{1\}})$,
    where $L$ is defined in 
    Lemma~\ref{lem:brady1ishsimple}, 
    $G := \{(0,1),(1,0)\}$, 
    $E := \{(0,0),(1,1)\}$, 
    and $R_X$ is a unary relation that denotes $X$, 
    can be solved by a symmetric linear ${\mathbb Z}_2$-Datalog $\Pi$ of stage one. 
\end{lemma}
\begin{proof}  
    We set $k=0$, $\sigma = \emptyset$, and use $U$ as the goal predicate. The rules of $\Pi$ are 
\begin{align*} L^0_{1,1,1,1}(x_1,x_2,x_3) & \; {:}{-} \; L(x_1,x_2,x_3) \\
L^0_{1,1,0,1}(x_1,x_2,x_3) & \; {:}{-} \; G(x_1,x_2) \\
L^0_{1,1,0,0}(x_1,x_2,x_3) & \; {:}{-} \; E(x_1,x_2) \\
L^0_{1,0,0,0}(x_1,x_2,x_3) & \; {:}{-} \; R_{0}(x_1) \\
L^0_{1,0,0,1}(x_1,x_2,x_3) & \; {:}{-} \; R_{1}(x_1) \qedhere 
\end{align*}
\end{proof}

We may impose the same syntactic restrictions \emph{linear} and \emph{symmetric linear} also for $\mathbb{G}$-Datalog programs.


\subsection{Symmetric linear ${\mathbb Z}_2$-Datalog and $\oplus$L}
Mod$_k$L is the complexity class which consists of all problems that can be solved by a non-deterministic Turing machine with logarithmic work space such that the answer is yes if and only if   the number of accepting paths is  congruent to 
$0$ mod $k$. Mod$_2$L is commonly known as $\oplus$L (parity-L). 
We denote by $\FL^{\oplus\text{L}}$ the class of functions that are computable with a deterministic logspace Turing machine with an oracle belonging to $\oplus$L, and by $\text{L}^{\oplus{\text{L}}}$ the decision problems in this class.

\begin{theorem}\label{thm:Z2Datalog}
    The problem of deciding whether a fixed 
    symmetric linear ${\mathbb Z}_2$-Datalog program derives on a given finite structure the goal predicate is in $\oplus L$. 
\end{theorem}
\begin{proof} 
The statement follows from the facts that 
    \begin{itemize}
   \item symmetric linear Datalog programs can be evaluated in L~\cite{EgriLaroseTessonLogspace},
   \item 
    L is obviously contained in $\oplus$L, 
    \item the system of ${\mathbb Z}_2$-linear equations from the semantics of rank one ${\mathbb Z}_2$-Datalog can be constructed in $L$, 
   \item solving systems of linear equations over ${\mathbb Z}_2$ is in $\oplus$L (it is even complete for this class~\cite{Damm}), 
   \item the class of functions belonging to $\FL^{\oplus\text{L}}$ with a fixed $\oplus$L oracle is closed under composition (the proof of this is similar to the standard argument that logspace computable functions compose), and 
    \item 
    $L^{\oplus\text{L}}$ is contained in $\oplus$L~\cite{HertrampfReithVollmer}.   \qedhere 
     \end{itemize} 
\end{proof}

\begin{remark}
Given that solving systems of linear equations over ${\mathbb Z}_p$ is complete for Mod$_p$L 
we similarly obtain that deciding whether a fixed 
    symmetric linear ${\mathbb Z}_p$-Datalog program derives goal on a given finite structure is in Mod$_p$L. 
\end{remark} 

\begin{remark}
   We mention that \emph{linear} (rather than symmetric linear)  ${\mathbb Z}_p$-Datalog programs can
    be evaluated in the complexity class $L^{\#L}$, which is the class of problems that can be decided in deterministic logarithmic space with access to an oracle for problems in $\#L$. The class $\#L$ is the class of functions that returns for a given non-deterministic logspace algorithm the number of accepting paths. 
    ${\mathbb Z}_p$-Datalog programs
    can be evaluated in $L^{\#L}$, because finding solutions of linear equations over $\mathbb Z_p$ is in $\Mod_p \Lclass$, which is contained in $L^{\#L}$, 
    and linear Datalog programs can be evaluated in NL, which is also contained in $L^{\#L}$. The class 
    $L^{\#L}$ is contained in NC$^2$.
    In contrast, it is not known whether 
    systems of ${\mathbb Z}$-linear equations can be solved in NC~\cite{AllenderBealsOgihara}. 
\end{remark}

 The following can be shown similarly as for symmetric linear Datalog~\cite[Lemma 4.2]{StarkeDiss}. 

\begin{lemma}\label{lem:pp-construct}
    Let $\bC$ be a structure with a pp-construction in $\bB$,
    and $\Csp(\bB)$ can be solved by (linear, or symmetric linear) $\mathbb{G}$-Datalog. 
    Then $\Csp(\bC)$ can also be solved by
(linear or symmetric linear, respectively) $\mathbb{G}$-Datalog.
\end{lemma}

\begin{question}
   Is there a finite structure $\bB$  such that $\Csp(\bB)$ is in $\oplus L$,
   but 
    $\Csp(\bB)$ 
    cannot be solved by a symmetric linear ${\mathbb Z}_2$-Datalog program? 
\end{question}

We will answer this question negatively for finite structures 
$\bB$ with a conservative Maltsev polymorphism in the next section.

\subsection{Implementations in symmetric linear ${\mathbb Z}_2$-Datalog}\label{sect:implement}
When implementing the procedures from Section~\ref{sect:alg}, we follow the diagram in Figure~\ref{fig:alg} bottom-up, starting with $\Solve_{n,1}$.

\begin{proposition}\label{prop:solven1}
    The procedure $\Solve_{n,1}$ (Algorithm~\ref{alg:solve}) can be implemented in symmetric linear ${\mathbb Z}_2$-Datalog; only one stage is needed.
\end{proposition}
\begin{proof}
    The program computes a system of ${\mathbb Z}_2$-linear equations $E$ and a permutational system ${\mathfrak T}$;
    this can be achieved with a recursion-free Datalog program (which is in particular symmetric and linear). The system $E$ can be viewed as an instance of $\Csp(\{0,1\};L,G,R_{\{0\}},R_{\{1\}})$
as in Lemma~\ref{lem:lin} which is in symmetric linear ${\mathbb Z}_2$-Datalog.
The satisfiability of the instance ${\mathfrak T}$ of $\Csp({\mathfrak S})$ (Definition~\ref{def:S}) can even be checked in symmetric linear Datalog, since $\bS$ has a majority polymorphism and a Maltsev polymorphism (Theorem~\ref{thm:DL}). Note that this program can derive the goal predicate before checking satisfiability of the permutational system (in which case a contradiction is found without considering the linear relation), or after (in which case the permutational system is inconsistent).
\end{proof}


\begin{proposition}\label{prop:reduce-symlin}
    The procedure $\Reduce_{n,k}$ (Algorithm~\ref{alg:reducebasis} and Algorithm~\ref{alg:reduce}) can be implemented in symmetric linear ${\mathbb Z}_2$-Datalog using at most $(4(k-1) + 3)$-many stages. 
\end{proposition}
\begin{proof}


For Algorithm~\ref{alg:reducebasis}, we compose the following symmetric linear $\mathbb{Z}_2$-Datalog procedures.
\begin{enumerate}
    \item On input instance $\mathfrak{I}$, find the connected components $C_1, \dots, C_v$ of $\mathfrak{I}$. Output $\mathfrak{I}$, along with a binary equivalence relation $C$ whose classes are $C_1, \dots, C_v$.
    \item Now run the following modification of the program from Proposition~\ref{prop:solven1}. Every rule $$P(x_1,\dots,x_n) \; {:}{-} \; \phi(x_1,\dots,x_n,y_1,\dots,y_m)$$ of the original program is replaced by the rule
    \[
    P'(x_1,\dots,x_n) \; {:}{-} \; \phi(x_1,\dots,x_n,y_1,\dots,y_m) \wedge C(z,x_1) \wedge \dots \wedge C(z,x_1) \wedge C(z,y_1) \wedge \dots \wedge C(z,y_m)
    \]
    
    Hence, the primed version of the goal predicate of the program in Proposition~\ref{prop:solven1} is now a unary relation which is the union of those connected components $C_i$ whose induced subinstance $\mathfrak{C}_i$ is unsatisfiable. We call this unary predicate $G'$. Output the original instance $\mathfrak{I}$ and the predicate $G'$.
    \item The predicate $G'$ contains the information necessary to assign empty unary constraints $R_{\emptyset}$ to variables of $\mathfrak{I}$. Also in the program is a rule with empty body and head $R_{\{\epsilon \}}$, hence all variables will belong to this unary relation. The binary and ternary constraints for $\overline{\mathfrak{I}}$ can be inferred from the original instance using symmetric linear Datalog rules. Furthermore, the instance can now be made concise with symmetric linear Datalog rules. We output this obtained instance $\overline{\mathfrak{I}}$ along with a unary predicate $D_0$  that holds for every variable in $\mathfrak{I}$, which codes the constant value $0$ reduction atlas output by $\Reduce_{n,1}$.

\end{enumerate}

Now suppose that $\Reduce_{n,k}$ can be implemented in symmetric linear $\mathbb{Z}_2$-Datalog. Let $\mathfrak{I}$ be an instance of $\Csp(\mathfrak{P}^{n,k+1})$. To implement Algorithm~\ref{alg:reduce}, we compose the following procedures. 

\begin{enumerate}
    \item Create a binary equivalence relation $C$ which codes the uniform connected components of $\mathfrak{I}$. Create two unary relations $S$ and $T$ for transfer source and transfer target variables. Make the instance transfer target compatible and then concise. Output the modified instance $\mathfrak{I}$ along with the relations $C$ and $S$.
    \item Make a unary relation $Z$ which is equal to the union of all transfer source components. We can find the complement of $Z$ in symmetric linear $\mathbb{Z}_2$-Datalog with rules which assert the linear equations $x =0$ for all $x \in I$ and $x=1$ for all $x \in Z$. Output the instance $\mathfrak{I}$, the unary relation $Z$, and its complement $\lnot Z$. 

    \item The next program forms the induced subinstances $\mathfrak{S}^U$, $\mathfrak{M}^U$, and $\mathfrak{I}^N$. Since this program has input unary EDBs $Z$ and $\lnot Z$, this is possible with symmetric linear Datalog rules. The program also distinguishes between the IDBs corresponding to the relations of each induced subinstance with two disjoint copies of the symbols which name the uniform relations, hence we obtain three distinguishable structures $\mathfrak{S}^U, \mathfrak{M}^U$, and $\mathfrak{I}^U$. Output each induced subinstance along with the unary relation $Z$.

    \item Now form $\overline{\mathfrak{S}^U} = (\mathfrak{S}^U)^{(0)}$ and $(\mathfrak{M}^U)^{(0)}$ in the obvious way, preserving the ability to distinguish the two instances. Output $\overline{\mathfrak{S}^U}$, $(\mathfrak{M}^U)^{(0)}$, $\mathfrak{M}^U$, $\mathfrak{I}^N$, and $Z$.
    
    \item Now form $\overline{\mathfrak{M}^U}$. To accomplish this, we run a copy of the symmetric linear $\mathbb{Z}_2$-datalog procedure for $\Reduce_{n,k}((\mathfrak{M}^U)^{(0)})$ which we inductively suppose to exist, but with all symbols and rules recognizing only the relations for $(\mathfrak{M}^U)^{(0)}$. This subprocedure will consist of several stages, so at the end of each stage of the computation we carry along the instances $\overline{\mathfrak{S}^U}$, $\mathfrak{I}^N$, and the predicate $Z$. Note that in the final stage we will obtain some IDBs $Z_1, \dots, Z_{k}$ which partition $M$ and encode the reduction atlas output by $\Reduce_{n,k}(\mathfrak{M}^U)$, so by setting $Z_0 =Z$, we now have a reduction atlas. 
    
    The instance $\overline{\mathfrak{M}^U}= \overline{(\mathfrak{M}^U)^{(0)}} \cdot (\mathfrak{M}^U)^{\lnot(0)}$ can now be formed, again with symmetric linear rules. Furthermore, using the IDBs $Z_0, \dots, Z_k$ which encode the reduction atlas, it now possible to implement Algorithm~\ref{alg:outerreduce} to obtain $\overline{\mathfrak{I}^N}$. 
    
    Form the union of the instances $\overline{\mathfrak{I}} = \overline{\mathfrak{S}^U}$, $\overline{\mathfrak{M}^U}$, $\mathfrak{I}^N$ with symmetric linear rules which remove all distinctions between symbols which name the same relation.
    Output $\overline{\mathfrak{I}}$ and the unary relations $Z_0, \dots, Z_k$. 
\end{enumerate}
Calculating the upper bound on the number of stages necessary in the implementation is straightforward.
\qedhere
\end{proof}

    \begin{theorem}\label{thm:main}
        For every $n,k \in {\mathbb N}$, the problem 
        $\Csp(\bP^{n,k})$ can be solved by a symmetric linear ${\mathbb Z}_2$-Datalog program with at most $2k^2 + k -2$-many stages.
    \end{theorem}
    \begin{proof} 
    We implement 
    $\Solve_{n,k}$ 
    (Algorithm~\ref{alg:solve}) 
    with a symmetric linear $\mathbb Z_2$-Datalog as follows.
    The special case $k=1$ has already been treated 
    in Proposition~\ref{prop:solven1}. 
    Given that $\Reduce_{n,k}$ can be implemented in symmetric linear ${\mathbb Z}_2$-Datalog (Proposition~\ref{prop:reduce-symlin}), we may therefore formulate $\Solve_{n,k}$ 
    (Algorithm~\ref{alg:solve}) 
    with a symmetric linear $\mathbb Z_2$-Datalog program. The number of stages necessary is bounded above by
$1+ \sum_{l=2}^k (3 + 4(l-1))$,
    where the sum adds the number of stages for each run of $\Reduce_{n,l}$ for $2 \leq l \leq k$ and a final stage accounts for the final run of $\Solve_{n,1}$.
    \end{proof} 
    
    \begin{corollary}\label{cor:OplusAlgforPnk}
        For every $n,k \in {\mathbb N}$, the problem 
        $\Csp(\bP^{n,k})$ is in $\oplus$L. 
    \end{corollary}

\section{Conclusion}
Our results show that the CSP for every structure with a finite domain, finite signature, and a conservative Maltsev polymorphism is in the complexity class $\oplus L$. To this end, we have introduced $\mathbb{G}$-Datalog, which extends 
Datalog by a mechanism for solving linear equations over a cyclic group $\mathbb{G}$. We show that symmetric linear ${\mathbb Z}_2$-Datalog programs can be solved in $\oplus L$ (Theorem~\ref{thm:Z2Datalog}).  

To show that our class of CSPs can be solved by 
${\mathbb Z}_2$-Datalog programs, we have presented  structural results about finite-domain structures with a conservative Maltsev polymorphism. Such structures also have a conservative minority polymorphism (Lemma~\ref{lem:Carbonnel}) and their polymorphism algebras possess a tree-like decomposition along their congruences (Theorem~\ref{thm:representwithtreealgebras}). We have presented a finite set of relations from which all other relations can be defined primitively positively (Theorem~\ref{thm:consminorityrelbasis}). 
We then focused on a fundamental family of structures $\bP_{n,k}$ in our class. 
The importance of this family comes from the fact that all other finite-domain structures with a conservative Maltsev polymorphism have a \emph{pp-construction} in a structure from this family (Theorem~\ref{thm:FinalPPConstructionTheorem}), 
and that the class of CSPs that can be solved by (linear, symmetric linear) $\mathbb{G}$-Datalog programs is closed under primitive positive constructability (Lemma~\ref{lem:pp-construct}).

We also proved that the CSP for every structure $\bP_{n,k}$ can be solved by a ${\mathbb Z}_2$-Datalog program (Theorem~\ref{thm:main}). 
We may now determine the descriptive and computational complexity of conservative Maltsev CSPs. 

\begin{theorem}\label{thm:dicho}
    Let $\bB$ be a finite structure with a conservative Maltsev polymorphism. Then either $\Csp(\bB)$ is $\oplus$L-complete, or in L. 
\end{theorem}
\begin{proof}
    Containment in $\oplus$L is the main result of this paper. 
    Suppose that $\bB$ contains two elements $a,b$ such that the relation $L_{\{a,b\}}$ is primitively positively definable in $\bB$. 
    Then there is a logspace reduction from solving linear equations over ${\mathbb Z}_2$ to $\Csp(\bB)$. Since solving linear equations over ${\mathbb Z}_2$ is $\oplus$L-complete~\cite{Damm}, we have that $\Csp(\bB)$ is $\oplus$L-complete. 
    Now suppose that for any two elements $a,b$ of $\bB$, there is no primitive positive definition of  $L_{\{a,b\}}$ in $\bB$. 
    Then all relations over $\{a,b\}$ with a primitive positive definition in $\bB$ are preserved by the (unique) majority operation on $\{a,b\}$, from which it follows that 
    $\bB$ has a (global) majority operation~\cite{Conservative}. 
    Hence, $\Csp(\bB)$ can be expressed in linear Datalog~\cite{DalmauKrokhin08}.
    Hence, since $\bB$ by assumption also has  a Maltsev polymorphism, a result of Larose and Dalmau~\cite{DalmauLarose08} implies that $\Csp(\bB)$ can be solved by a symmetric linear Datalog program, 
    and hence is in L~\cite{EgriLaroseTessonLogspace}. 
\end{proof}


\bibliographystyle{abbrv}
\bibliography{global}

\end{document}